\date{} 
\title{Imaginary Geometry I: Interacting $\SLE$s}
\author{Jason Miller and Scott Sheffield}
\documentclass[12pt,naturalnames]{article}
\usepackage{amsmath}
\usepackage{amssymb}
\usepackage{amsthm}
\usepackage{amsfonts}
\usepackage{graphicx}
\usepackage{tabularx}
\usepackage{caption}
\usepackage{enumerate}
\usepackage[protrusion=true,expansion=true]{microtype}
\usepackage{color}
\usepackage{morefloats}
\usepackage[normalem]{ulem}
\input epsf.tex

\usepackage[margin=1.2in]{geometry}
\allowdisplaybreaks

\newif\ifhyper\IfFileExists{hyperref.sty}{\hypertrue}{\hyperfalse}
\ifhyper\usepackage{hyperref}\fi

\newif\ifdraft
\drafttrue
\def\note#1/{\ifdraft {\bf [#1]}\fi}
\long\def\comment#1{}
\numberwithin{equation}{section}
\numberwithin{figure}{section}

\newtheorem{theorem}{Theorem}
\numberwithin{theorem}{section}
\newtheorem{corollary}[theorem]{Corollary}
\newtheorem{lemma}[theorem]{Lemma}
\newtheorem{proposition}[theorem]{Proposition}

\theoremstyle{remark}\newtheorem{definition}[theorem]{Definition}
\theoremstyle{remark}\newtheorem{remark}[theorem]{Remark}

\usepackage{subfigure}
\usepackage{graphicx}

\newcommand{\C}{\mathbf{C}}
\newcommand{\D}{\mathbf{D}}
\newcommand{\E}{\mathbf{E}}

\newcommand{\N}{\mathbf{N}}
\newcommand{\Z}{\mathbf{Z}}
\newcommand{\p}{\mathbf{P}}
\newcommand{\Q}{\mathbf{Q}}
\newcommand{\R}{\mathbf{R}}
\newcommand{\s}{\mathbf{S}}
\newcommand{\h}{\mathbf{H}}

\newcommand{\Fg}{\mathfrak {g}}
\newcommand{\Fh}{\mathfrak {h}}

\newcommand{\CA}{\mathcal {A}}
\newcommand{\CB}{\mathcal {B}}
\newcommand{\CC}{\mathcal {C}}
\newcommand{\CD}{\mathcal {D}}

\newcommand{\CF}{\mathcal {F}}

\newcommand{\CJ}{\mathcal {J}}
\newcommand{\CK}{\mathcal {K}}

\newcommand{\CS}{\mathcal {S}}

\newcommand{\CU}{\mathcal {U}}

\newcommand{\CZ}{\mathcal {Z}}

\newcommand{\CG}{\mathcal {G}}

\newcommand{\dist}{{\rm dist}}

\newcommand{\diam}{{\rm diam}}
\newcommand{\var}{{\rm Var}}
\newcommand{\im}{{\rm Im}}
\newcommand{\re}{{\rm Re}}
\newcommand{\cov}{{\rm Cov}}

\newcommand{\supp}{{\rm supp }}

\newcommand{\SLE}{{\rm SLE}}
\newcommand{\CLE}{{\rm CLE}}
\newcommand{\BES}{{\rm BES}}

\newcommand{\strip}{\CS}
\newcommand{\striptop}{\partial_U \CS}
\newcommand{\stripbot}{\partial_L \CS}
\newcommand{\lightcone}{{\mathbf L}}
\newcommand{\fan}{{\mathbf F}}

\newcommand{\wh}{\widehat}
\newcommand{\wt}{\widetilde}
\newcommand{\ol}{\overline}
\newcommand{\ul}{\underline}
\newcommand{\one}{{\bf 1}}

\newcommand{\giv}{\,|\,}

\def\diam{\mathop{\mathrm{diam}}}

\def\dist{\mathop{\mathrm{dist}}}

\def\Im{{\rm Im}\,}

\newcommand{\Ito}{It\^o}

\def \E {{\bf E}}

\def\hcap{{\rm hcap}}

\def\@rst #1 #2other{#1}
\newcommand\MR[1]{\relax\ifhmode\unskip\spacefactor3000 \space\fi
  \MRhref{\expandafter\@rst #1 other}{#1}}
\newcommand{\MRhref}[2]{\href{http://www.ams.org/mathscinet-getitem?mr=#1}{MR#2}}

\begin{document}
\maketitle

\begin{abstract}
Fix constants $\chi >0$ and $\theta \in [0,2\pi)$, and let $h$ be an instance of the Gaussian free field on a planar domain.  We study flow lines of the vector field $e^{i(h/\chi+\theta)}$ starting at a fixed boundary point of the domain.  Letting $\theta$ vary, one obtains a family of curves that look locally like $\SLE_\kappa$ processes with $\kappa \in (0,4)$ (where $\chi = \tfrac{2}{\sqrt{\kappa}} -\tfrac{ \sqrt{\kappa}}{2}$), which we interpret as the rays of a random geometry with purely imaginary curvature.

We extend the fundamental existence and uniqueness results about these paths to the case that the paths intersect the boundary.  We also show that flow lines of different angles cross each other at most once but (in contrast to what happens when $h$ is smooth) may bounce off of each other after crossing.  Flow lines of the same angle started at different points merge into each other upon intersecting, forming a tree structure.  We construct so-called {\em counterflow lines} ($\SLE_{16/\kappa}$) within the same geometry using ordered ``light cones'' of points accessible by angle-restricted trajectories and develop a robust theory of flow and counterflow line interaction.

The theory leads to new results about $\SLE$.  For example, we prove that $\SLE_\kappa(\rho)$ processes are almost surely continuous random curves, even when they intersect the boundary, and establish Duplantier duality for general $\SLE_{16/\kappa}(\rho)$ processes.
\end{abstract}

\newpage
{\setlength{\parskip}{0.05cm plus1mm minus1mm}
\tableofcontents}
\newpage

\medbreak {\noindent\bf Acknowledgments.}  We thank Bertrand Duplantier, Steffen Rohde, Oded Schramm, Wendelin Werner, and David Wilson for helpful discussions.  We also thank Nike Sun, Sam Watson, and several anonymous referees for very helpful comments which have led to many significant improvements.

\section{Introduction}
\label{sec::intro}

All readers are familiar with two dimensional Riemannian geometries whose Gaussian curvature is purely positive (the sphere), purely negative (hyperbolic space), or zero (the plane).  In this paper, we study ``geometries'' whose Gaussian curvature is purely {\em imaginary}.  We call them {\em imaginary geometries}.

Imaginary geometries have zero real curvature, which means (informally) that when a small bug slides without twisting around a closed loop, the bug's angle of rotation is unchanged.  However, the bug's {\em size} may change (an {\em Alice in Wonderland} phenomenon that further justifies the term ``imaginary'').\footnote{In both real and imaginary geometries, parallel transport about a simple loop multiplies a $\C$-identified tangent space by $e^{i C}$ where $C$ is the integral of the enclosed curvature; these transformations are rotations when $C$ is real, dilations when $C$ is imaginary.}  ``Straight lines'' and ``angles'' are well-defined in imaginary geometry, and the angles of a triangle always sum to $\pi$, but ``distance'' is not defined.

A simply connected imaginary geometry can be described by a simply connected subdomain $D$ of the complex plane $\C$ and a function $h\colon D \to \R$.\footnote{In the language of differential geometry, an {\em imaginary geometry} is a two dimensional manifold endowed with a torsion-free affine connection whose holonomy group consists entirely of dilations (c.f.\ ordinary Riemannian surfaces, whose Levi-Civita holonomy groups consist entirely of rotations), and straight lines are geodesic flows of the connection.  The connection endows the manifold with a conformal structure, and by the uniformization theorem one can conformally map the geometry to a planar domain on which the geodesics are determined by some function $h$ in the manner described here \cite{SHE_WELD}.}  The angle-$\theta$ ray beginning at a point $z \in D$ is the flow line of $e^{i(h + \theta)}$ beginning at $z$, i.e., the solution to the ODE
\begin{equation}
\eta'(t) = e^{i \left(h(\eta(t))+\theta\right)} \quad\text{for}\quad t > 0, \quad \eta(0) = z
\end{equation}
as in Figure~\ref{fig::flow_lines_example}.\footnote{Imaginary geometries have also been called ``altimeter-compass'' geometries \cite{She_SLE_lectures}. If the graph of $h$ is viewed as a mountainous terrain, then a hiker holding an analog altimeter---with a needle indicating altitude modulo $2\pi$---in one hand and a compass in the other can trace a ray by walking at constant speed (continuously changing direction as necessary) in such a way that the two needles always point in the same direction.}  In this paper we concern ourselves only with these rays, which we view as a simple and complete description of the imaginary geometry.\footnote{This description is canonical up to conformal coordinate change, see Figure~\ref{fig::coordinatechange}.}
Our goal is to make sense of and study the properties of these flow lines when $h$ is a constant multiple of a random {\em generalized} function called the Gaussian free field.

\subsection{Overview} 
\label{ss::overview}

Given an instance $h$ of the Gaussian free field (GFF), constants $\chi>0$ and $\theta \in [0,2\pi)$, and an initial point $z$, is there always a canonical way to define the flow lines of the complex vector field $e^{i(h/\chi + \theta)}$, i.e., solutions to the ODE
\begin{equation}
\label{eqn::flowode}
\eta'(t) = e^{i(h(\eta(t))/\chi+\theta)} \quad\text{for}\quad t > 0,
\end{equation}
beginning at $z$?  The answer would obviously be yes if $h$ were a smooth function (Figure~\ref{fig::flow_lines_example}), but it is less obvious for an instance of the GFF, which is a distribution (a.k.a.\ a generalized function), not a function (Figures~\ref{fig::flowlines}--\ref{fig::flowlines4}).

\captionsetup{labelsep=none}
\begin{figure}[ht!]
\begin{center}
\subfigure[The vector field $e^{ih(z)}$ where $h(z) = |z|^2$, together with a flow line started at zero.]{
\includegraphics[width=0.46\textwidth]{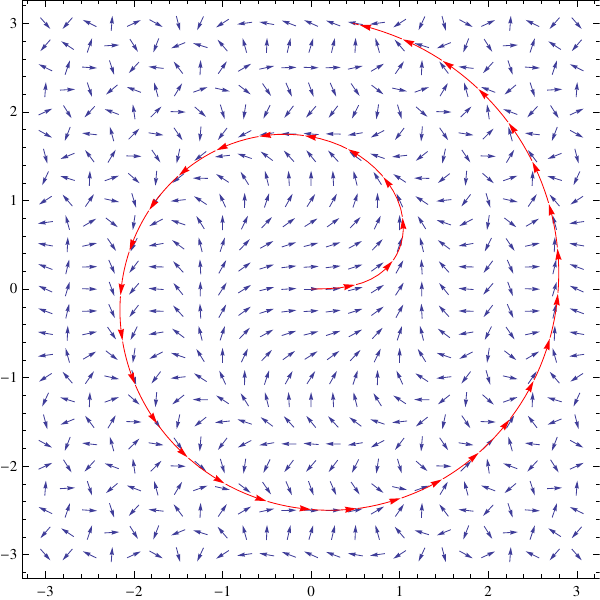}}
\hspace{0.04\textwidth}
\subfigure[Flow lines of $e^{i(h(z)+\theta)}$ for $12$ uniformly spaced $\theta$ values.]{
\includegraphics[width=0.46\textwidth]{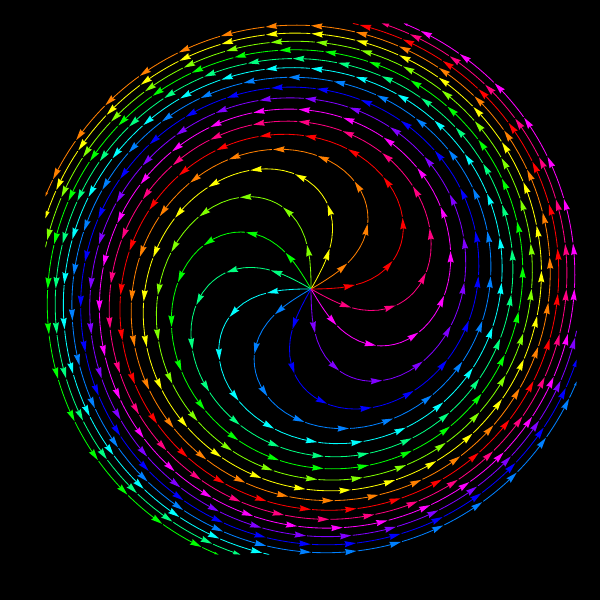}}
\caption{\label{fig::flow_lines_example}}
\end{center}
\end{figure}
\captionsetup{labelsep=colon}

Several works in recent years have addressed special cases and variants of this question \cite{She_SLE_lectures, DUB_PART, MakarovSmirnov09, SchrammShe10, HagendorfBauerBernard10, IzyurovKytola10, SHE_WELD} and have shown that in certain circumstances there is a sense in which the paths are well-defined (and uniquely determined) by $h$, and are variants of the Schramm-Loewner evolution (SLE).  In this article, we will focus on the case that $z$ is point on the boundary of the domain where $h$ is defined and establish a more general set of results.  (Flow lines beginning at interior points will be addressed in a subsequent paper.)  In particular, we show that the paths exist and are determined by $h$ even in settings where they hit and bounce off of the boundary, and we will also describe the interaction of multiple flow lines that hit the boundary and cross or bounce off each other.  These topics have never been previously addressed.  Ultimately, our goal is to establish a robust theory of the imaginary geometry of the GFF, with a complete description of all the rays and the way they interact with each other.  This will have a range of applications in $\SLE$ theory: in particular, this paper will establish continuity results for $\SLE_\kappa(\ul{\rho})$ curves and generalizations of so-called Duplantier duality (i.e., descriptions of the boundaries of $\SLE_{16/\kappa}(\ul{\rho})$ curves), along with a ``light cone'' interpretation of $\SLE_{16/\kappa}(\ul{\rho})$ that allows these curves to be constructed and decomposed in surprising ways.

This paper is the first in a four-paper series that also includes \cite{MS_IMAG2, MS_IMAG3,MS_INTERIOR}.  Among other things, the later papers will use the theory established here to produce descriptions of the time-reversals of $\SLE_\kappa(\ul{\rho})$ for all values of $\kappa$, a complete construction of trees of flow lines started from {\em interior} points, the first proof that conformal loop ensembles $\CLE_{\kappa'}$ are canonically defined when $\kappa' \in (4,8)$, and a geometric interpretation of these loop ensembles.  In subsequent works, we expect these results to be useful to the theory of Liouville quantum gravity, allowing one to generalize the results about ``conformal weldings'' of random surfaces that appear \cite{SHE_WELD}, and to complete the program outlined in \cite{2011arXiv1108.2241S} for showing that discrete loop-decorated random surfaces have $\CLE$-decorated Liouville quantum gravity as a scaling limit, at least in a certain topology.  We will find that many basic $\SLE$ and $\CLE$ properties can be established more easily and in more generality using the theory developed here.

\begin{figure}[ht!]
\begin{center}
\includegraphics[width=0.95\textwidth, height=0.95\textwidth ,clip=true, trim = 1mm 1mm 1mm 1mm]{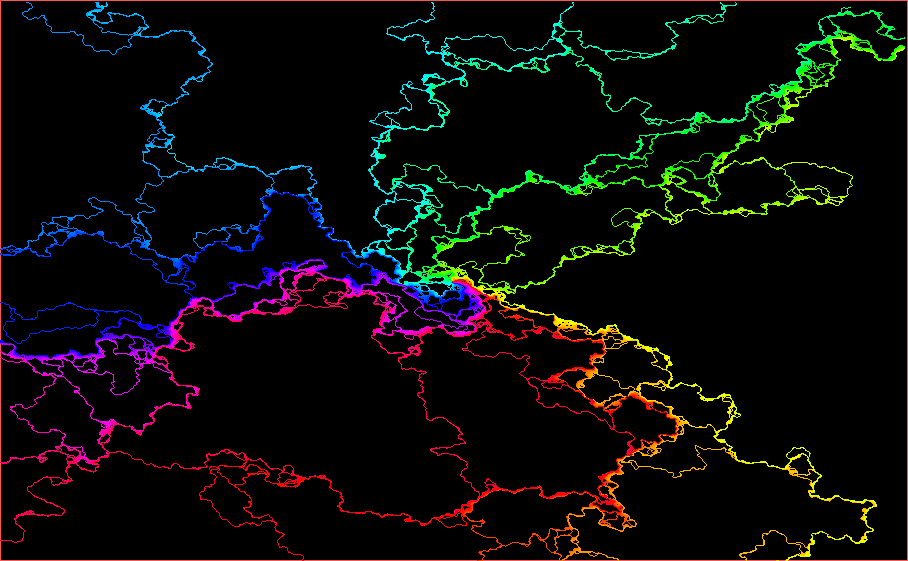}
\end{center}
\caption{ \label{fig::flowlines} Numerically generated flow lines, started at a common point, of $e^{i(h/\chi+\theta)}$ where $h$ is the projection of a GFF onto the space of functions piecewise linear on the triangles of a $300 \times 300$ grid; $\kappa=4/3$ and $\chi = 2/\sqrt{\kappa} - \sqrt{\kappa}/2 = \sqrt{4/3}$.  Different colors indicate different values of $\theta \in [0,2\pi)$.  We expect but do not prove that if one considers increasingly fine meshes (and the same instance of the GFF) the corresponding paths converge to limiting continuous paths (an analogous result was proven for $\kappa = 4$ \cite{SS09,SchrammShe10}).}
\end{figure}

\begin{figure}[ht!]
\begin{center}
\includegraphics[width=\textwidth ,height=0.8\textwidth,clip=true, trim = 1mm 1mm 1mm 1mm]{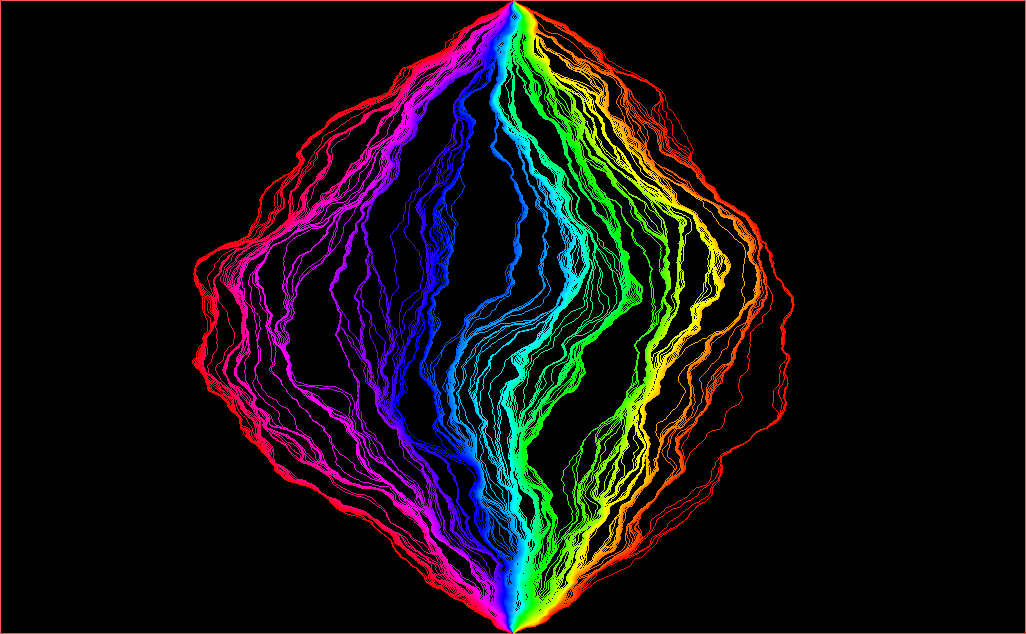}
\end{center}
\caption{ \label{fig::flowlines2} Numerically generated flow lines, started at $-i$ of $e^{i(h/\chi+\theta)}$ where $h$ is the projection of a GFF on $[-1,1]^2$ onto the space of functions piecewise linear on the triangles of a $300 \times 300$ grid; $\kappa=1/8$.  Different colors indicate different values of $\theta \in [-\tfrac{\pi}{2},\tfrac{\pi}{2}]$.  The boundary data for $h$ is chosen so that the central (``north-going'') curve shown should approximate an $\SLE_{1/8}$ process.}
\end{figure}

\begin{figure}[h!]
\begin{center}
\includegraphics[width=\textwidth ,height=0.8\textwidth,clip=true, trim = 1mm 1mm 1mm 1mm]{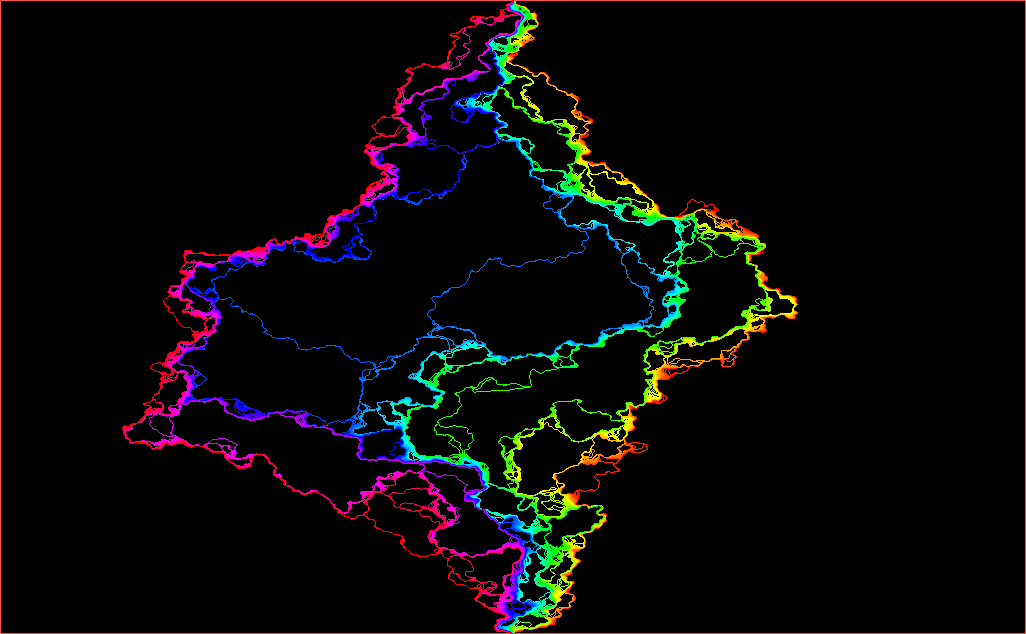}
\end{center}
\caption{\label{fig::flowlines3} Numerically generated flow lines, started at $-i$ of $e^{i(h/\chi+\theta)}$ where $h$ is the projection of a GFF on $[-1,1]^2$ onto the space of functions piecewise linear on the triangles of a $300 \times 300$ grid; $\kappa=1$.  Different colors indicate different values of $\theta \in [-\tfrac{\pi}{2},\tfrac{\pi}{2}]$.  The boundary data for $h$ is chosen so that the central (``north-going'') curve shown should approximate an $\SLE_1$ process.}
\end{figure}

\begin{figure}[h!]
\begin{center}
\includegraphics[width=\textwidth ,height=0.8\textwidth,clip=true, trim = 1mm 1mm 1mm 1mm]{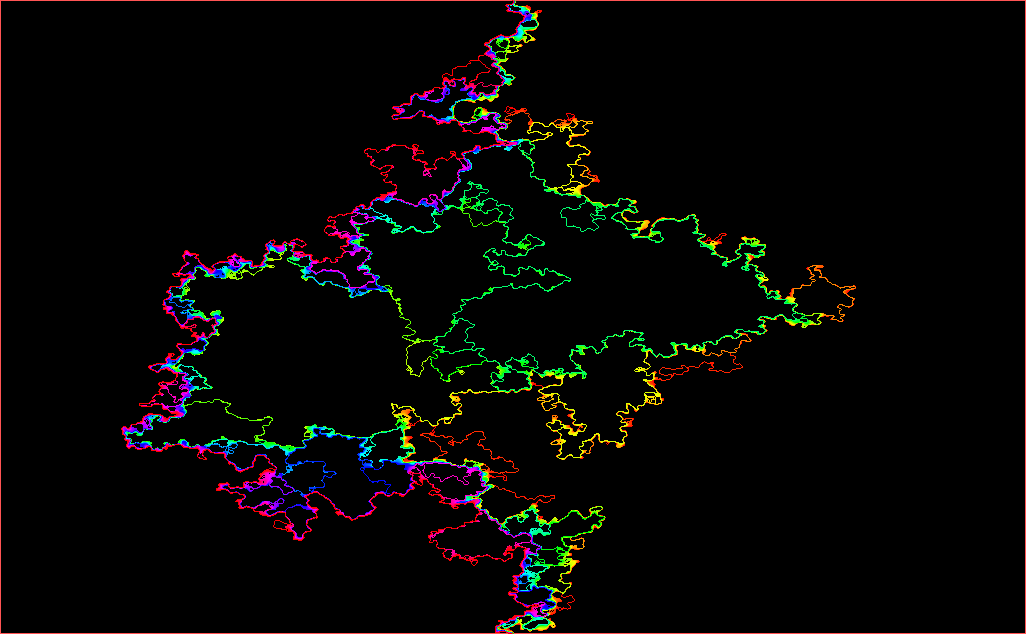}
\end{center}
\caption{\label{fig::flowlines4} Numerically generated flow lines, started at $-i$ of $e^{i(h/\chi+\theta)}$ where $h$ is the projection of a GFF on $[-1,1]^2$ onto the space of functions piecewise linear on the triangles of a $300 \times 300$ grid; $\kappa=2$.  Different colors indicate different values of $\theta \in [-\tfrac{\pi}{2},\tfrac{\pi}{2}]$.  The boundary data for $h$ is chosen so that the central (``north-going'') curve shown should approximate an $\SLE_2$ process.}
\end{figure}

We will fix $\chi > 0$ and interpret the paths corresponding to different $\theta$ values as ``rays of a random geometry'' angled in different directions and show that different paths started at a common point never cross one another.  Note that these are the rays of ordinary Euclidean geometry when $h$ is a constant.

Theorem~\ref{thm::coupling_existence} and Theorem~\ref{thm::coupling_uniqueness} establish the fact that the flow lines are well-defined and uniquely determined by $h$ almost surely.  Theorem~\ref{thm::coupling_existence} is the same as a theorem proved in \cite{DUB_PART}.  For convenience, we have restated it here and provided a proof in Section~\ref{subsec::existence_proof}.  (As stated in \cite{DUB_PART}, the theorem was conditional on the existence of solutions to a certain SDE, but we will prove this existence in Section~\ref{sec::sle}.)  This theorem establishes the existence of a coupling between $h$ and the path with certain properties.  Theorem~\ref{thm::coupling_uniqueness} then shows that in this coupling, the path is almost surely determined by the field.  Theorem~\ref{thm::coupling_uniqueness} is an extension of a result in \cite{DUB_PART}.  Unlike the result in \cite{DUB_PART}, our Theorem~\ref{thm::coupling_uniqueness} applies to paths that interact with the domain boundaries in non-trivial ways, and this requires new tools.

The boundary-intersecting case of Theorem~\ref{thm::coupling_uniqueness} and other ideas will then be used to describe the way that distinct flow lines interact with one another when they intersect (see Figure~\ref{fig::flow_line_interaction}).  We show that the flow lines started at the same point, corresponding to different $\theta$ values, may bounce off one another (depending on the angle difference) but almost surely do not cross one another (see Proposition~\ref{prop::angle_varying_monotonicity}), that flow lines started at distinct points with the same angle can ``merge'' with each other, and that flow lines started at distinct points with distinct angles almost surely cross at most once.  We give a complete description of the {\em conditional} law of $h$ given a finite collection of (possibly intersecting) flow lines.  (The conditional law of $h$ given multiple flow line segments is discussed in \cite{DUB_PART}, but the results there only apply to {\em non-intersecting} segments.  Extending these results requires, among other things, ruling out pathological behavior of the conditional expectation of the field --- given the paths --- near points where the paths intersect.)  These are some of the fundamental results one needs to begin to understand (continuum analogs of) Figures~\ref{fig::flowlines}--\ref{fig::flowlines4},~\ref{fig::grid}, and~\ref{fig::two_fans}.

As mentioned above, we also establish some new results in classical SLE theory.  For example, the flow line technology enables us to show in Theorem~\ref{thm::continuity} that the so-called $\SLE_\kappa(\ul{\rho})$ curves are a.s.\ continuous even when they hit the boundary.  Rohde and Schramm proved that ordinary $\SLE_\kappa$ on a Jordan domain is continuous when $\kappa \not = 8$ \cite{RS05}; the continuity of $\SLE_8$ was proved by Lawler, Schramm, and Werner in \cite{LSW04}  (extensions to more general domains are proved in \cite{2008arXiv0810.4327G}) but their techniques do not readily apply to boundary intersecting $\SLE_\kappa(\ul{\rho})$, and the lack of a proof for $\SLE_\kappa(\ul{\rho})$ has been a persistent gap in the literature.  Another approach to proving Theorem~\ref{thm::continuity} in the case of a single force point, based on extremal length arguments, has been proposed (though not yet published) by Kemppainen, Schramm, and Sheffield \cite{KSS}.

The random geometry point of view also gives us a new way of understanding other random objects with conformal symmetries.  For example, we will use the flow-line geometry to construct so-called {\em counterflow lines}, which are forms of $\SLE_{16/\kappa}$ ($\kappa \in (0,4)$) that arise as the ``light cones'' of points accessible by certain angle-restricted $\SLE_{\kappa}$ trajectories.  To use another metaphor, we say that a point $y$ is ``downstream'' from another point $x$ if it can be reached from $x$ by an angle-varying flow line whose angles lie in some allowed range; the counterflow line is a curve that traces through all the points that are downstream from a given boundary point $x$, but it traces them in an ``upstream'' (or ``counterflow'') direction.   This is the content of Theorem~\ref{thm::lightconeroughstatement}, which is stated somewhat informally.  (A more precise statement of Theorem~\ref{thm::lightconeroughstatement}, which applies to $\SLE_{16/\kappa}(\ul{\rho})$ processes that are not boundary intersecting, appears in Proposition~\ref{prop::light_cone_construction}; the general version is explained precisely in Section~\ref{subsubsec::light_cone_general}.)  In contrast to what happens when $h$ is smooth, the light cones thus constructed are not simply connected sets when $\kappa \in(2,4)$.  It also turns out that one can reach all points in the light cone by considering paths that alternate between the two extreme angles.  See Figures~\ref{fig::sle6_lightcone}--\ref{fig::sle6_decomposition} for discrete simulations of light cones generated in this manner (the two extreme angles differ by $\pi$; see also Figure~\ref{fig::pi_turn} for an explanation of the fact that a path with angle changes of size $\pi$ does not just retrace itself).

We will also show in Proposition~\ref{prop::fan_does_not_hit} that, for any $\kappa \in (0,4)$, the closure of the union of all the flow lines starting at a given point $z$ with angles in a countable, dense set (as depicted in Figures~{\ref{fig::flowlines}--\ref{fig::flowlines4}}) almost surely has Lebesgue measure zero.  (It is easy to see that the resulting object does not depend on the choice of countable, dense set.)  Put somewhat fancifully, this states that when a person holds a gun at a point $z$ in the imaginary geometry, there are certain other points (in fact, almost all points) that the gun cannot hit no matter how carefully it is aimed.  (One might guess this to be the case from the amount of black space in Figures~\ref{fig::flowlines}--\ref{fig::flowlines4},~\ref{fig::sle64_fan}.) Generally, random imaginary geometry yields many natural ways of coupling and understanding multiple SLEs on the same domain, as well as SLE variants on non-simply-connected domains.

The flow lines constructed here also turn out to be relevant to the study of Liouville quantum gravity.  For example, we plan to show in a subsequent joint work with Duplantier that the rays in Figures~\ref{fig::flowlines}--\ref{fig::flowlines4} arise when gluing together independent Liouville quantum gravity surfaces via the conformal welding procedure presented in  \cite{SHE_WELD}.  The tools developed here are essential for that program.

\subsection{Background and setting}

Let $D \subseteq \C$ be a domain with harmonically non-trivial boundary (i.e., a Brownian motion started at a point $z \in D$ almost surely hits $\partial D$) and let $C_0^\infty(D)$ denote the space of compactly supported $C^\infty$ functions on $D$.  For $f,g \in C_0^\infty(D)$, let
\[ (f,g)_\nabla := \frac{1}{2\pi} \int_D \nabla f (x) \cdot \nabla g(x)dx\]
denote the Dirichlet inner product of $f$ and $g$ where $dx$ is the Lebesgue measure on~$D$.  Let $H(D)$ be the Hilbert space closure of $C_0^\infty(D)$ under $(\cdot,\cdot)_\nabla$.  The continuum Gaussian free field $h$ (with zero boundary conditions) is the so-called standard Gaussian on $H(D)$.  It is given formally as a random linear combination
\begin{equation}
\label{eqn::gff_definition}
 h = \sum_n \alpha_n \phi_n,
\end{equation}
where $(\alpha_n)$ are i.i.d.\ $N(0,1)$ and $(\phi_n)$ is an orthonormal basis of $H(D)$.  (We will give a more formal introduction to the GFF in Section~\ref{sec::gff}.)

The GFF is a two-dimensional-time analog of Brownian motion.  Just as many random walk models have Brownian motion as a scaling limit, many random (real or integer valued) functions on two dimensional lattices have the GFF as a scaling limit \cite{BAD96, NS97,  KEN01, RV08, MillerGLCLT}.

The GFF can be used to generate various kinds of random geometric structures, including both Liouville quantum gravity and the imaginary geometry discussed here \cite{SHE_WELD}.  Roughly speaking, the former corresponds to replacing a Euclidean metric $dx^2 + dy^2$ with $e^{\gamma h}(dx^2 + dy^2)$ (where $\gamma \in (0,2)$ is a fixed constant and $h$ is the Gaussian free field).  The latter is closely related, and corresponds to considering $e^{ih/\chi}$, for a fixed constant $\chi >0$.  Informally, as discussed above, the ``rays'' of the imaginary geometry are flow lines of the complex vector field $e^{i(h/\chi+\theta)}$, i.e., solutions to the ODE~\eqref{eqn::flowode},
for given values of $\eta(0)$ and $\theta$.

A brief overview of imaginary geometry (as defined for general functions $h$) appears in \cite{SHE_WELD}, where the rays are interpreted as geodesics of a variant of the Levi-Civita connection associated with Liouville quantum gravity.  One can interpret the $e^{ih}$ direction as ``north'' and the $e^{i(h + \pi/2)}$ direction as ``west'', etc.  Then $h$ determines a way of assigning a set of compass directions to every point in the domain, and a ray is determined by an initial point and a direction.  (We have not described a Riemannian geometry, since we have not introduced a notion of length or area.)  When $h$ is constant, the rays correspond to rays in ordinary Euclidean geometry.  For more general continuous~$h$, one can still show that when three rays form a triangle, the sum of the angles is always $\pi$ \cite{SHE_WELD}.

Throughout the rest of this article, when we say that $\eta$ is a flow line of $h$ it is to be interpreted that $\eta$ is a flow line of the vector field $e^{ih/\chi}$; both $h$ and $\chi$ will be clear from the context.  In particular, the statement that $\eta$ is a flow line of $h$ with angle $\theta$ is equivalent to the statement that $\eta$ is a flow line of $h+\theta \chi$.

We next remark that if $h$ is a smooth function on $D$, $\eta$ a flow line of $e^{ih/\chi}$, and $\psi \colon \wt D \to D$ a conformal transformation, then by the chain rule, $\psi^{-1} \circ \eta$ is a flow line of $h \circ \psi - \chi \arg \psi'$ (note that a reparameterization of a flow line remains a flow line), as in Figure~\ref{fig::coordinatechange}. With this in mind, we define an {\bf imaginary surface} to be an equivalence class of pairs $(D,h)$ under the equivalence relation
\begin{equation}
\label{eqn::ac_eq_rel}
 (D,h) \rightarrow (\psi^{-1}(D), h \circ \psi - \chi \arg \psi') = (\wt{D},\wt{h}).
\end{equation}
Note that this makes sense even for $h$ which are not necessarily smooth.  We interpret~$\psi$ as a (conformal) {\em coordinate change} of the imaginary surface.  In what follows, we will generally take $D$ to be the upper half plane, but one can map the flow lines defined there to other domains using~\eqref{eqn::ac_eq_rel}.

\begin{figure}[h]
\begin{center}
\includegraphics[scale=0.85]{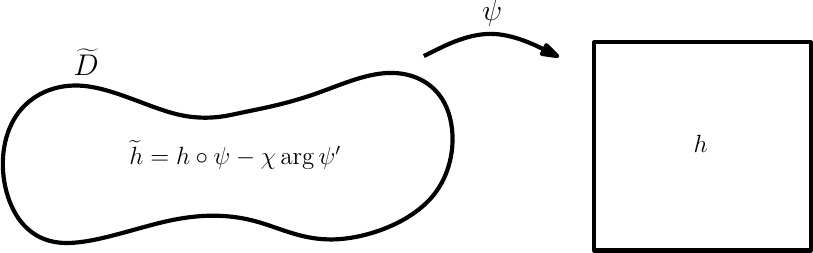}
\caption{\label{fig::coordinatechange} The set of flow lines in $\wt{D}$ will be the pullback via a conformal map $\psi$ of the set of flow lines in $D$ provided $h$ is transformed to a new function $\wt{h}$ in the manner shown.}
\end{center}
\end{figure}
When $h$ is an instance of the GFF on a planar domain, the ODE~\eqref{eqn::flowode} is not well-defined, since $h$ is a distribution-valued random variable and not a continuous function.  One could try to approximate one of these rays by replacing the $h$ in~\eqref{eqn::flowode} by its projection onto a space of continuous functions --- for example, the space of functions that are piecewise linear on the triangles of some very fine lattice.  This approach (and a range of~$\theta$ values) was used to generate the rays in Figures~\ref{fig::flowlines}--\ref{fig::flowlines4},~\ref{fig::grid},~\ref{fig::two_fans},~\ref{fig::sle6_lightcone}-\ref{fig::sle6_decomposition}, and~\ref{fig::flow_line_interaction}.  We expect that these rays will converge to limiting path-valued functions of $h$ as the mesh size gets finer.  This has not been proved, but an analogous result has been shown for level sets of $h$ \cite{SS09,SchrammShe10}.

\begin{figure}[h!]
\begin{center}
\includegraphics[width=\textwidth ,height=0.8\textwidth,clip=true, trim = 1mm 1mm 1mm 1mm]{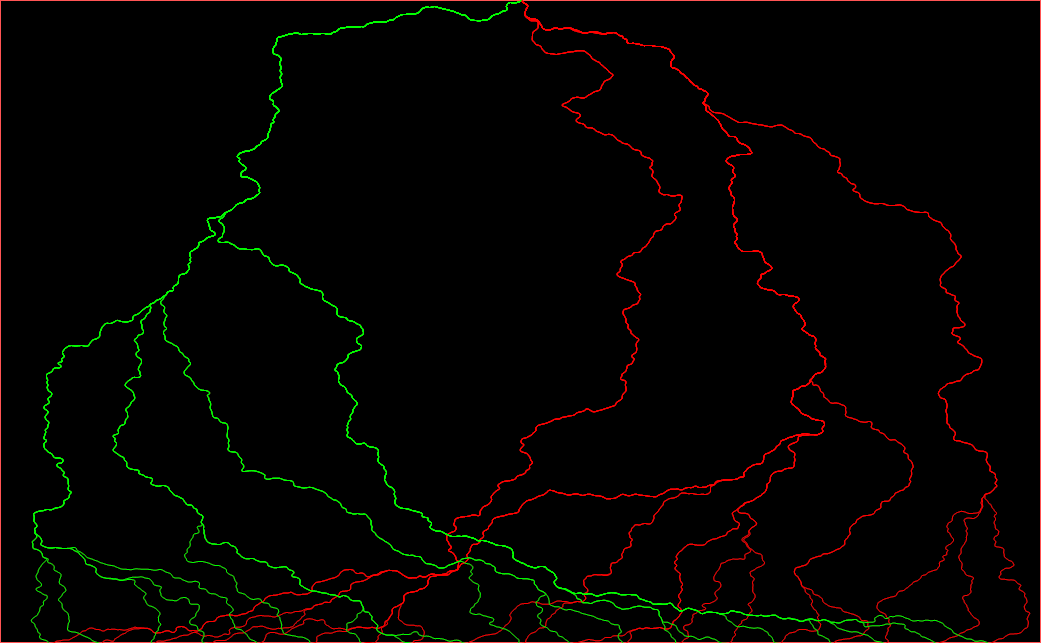}
\end{center}
\caption{ \label{fig::grid} Numerically generated flow lines, started at evenly spaced points on $[-1-i,1-i]$ of $e^{i h/\chi}$ where $h$ is the projection of a GFF on $[-1,1]^2$ onto the space of functions piecewise linear on the triangles of a $300 \times 300$ grid; $\kappa=1/2$.  The angle of the green lines is $\tfrac{\pi}{4}$ and the angle of the red lines is $-\tfrac{\pi}{4}$.  Flow lines of the same color appear to merge, but the red and green lines always cross at right angles.  The boundary data of $h$ was given by taking $0$ boundary conditions on $\h$ and then applying the transformation rule~\eqref{eqn::ac_eq_rel} with a conformal map $\psi \colon \h \to [-1,1]^2$ where $\psi(0) = -i$ and $\psi(\infty) = i$.}
\end{figure}

\begin{figure}[h!]
\begin{center}
\includegraphics[width=\textwidth ,height=0.8\textwidth,clip=true, trim = 1mm 1mm 1mm 1mm]{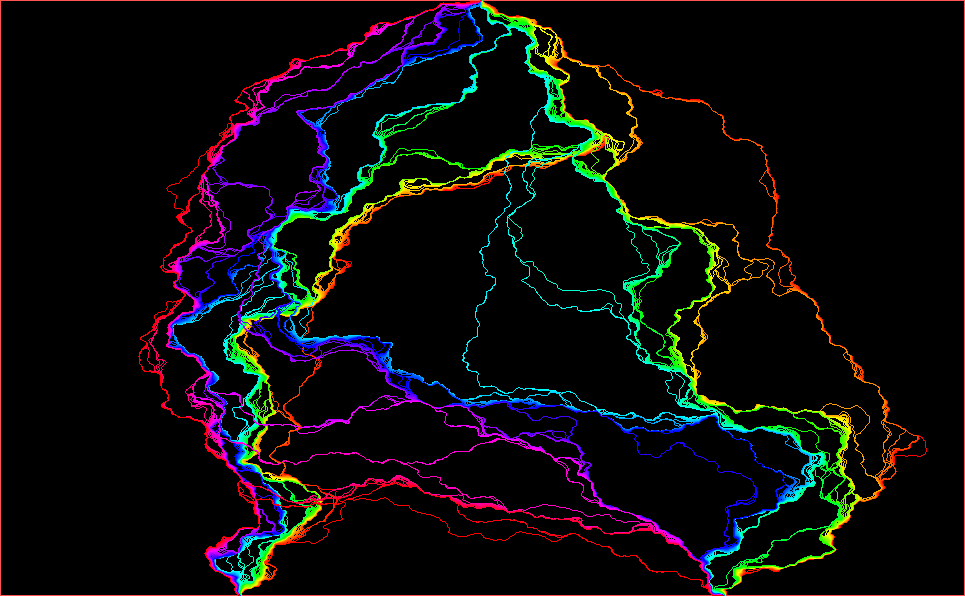}
\end{center}
\caption{ \label{fig::two_fans} Numerically generated flow lines, started at $-1/2-i$ and $1/2-i$ of $e^{i(h/\chi+\theta)}$ with angles evenly spaced in $[-\tfrac{\pi}{4},\tfrac{\pi}{4}]$ where $h$ is the projection of a GFF on $[-1,1]^2$ onto the space of functions piecewise linear on the triangles of a $300 \times 300$ grid; $\kappa=1/2$.  Flow lines of different colors appear to cross at most once and flow lines of the same color appear to merge.  The boundary data for $h$ is the same as in Figure~\ref{fig::grid}.}
\end{figure}

As we discussed briefly in Section~\ref{ss::overview}, it turns out that it is possible to make sense of these flow lines and level sets directly in the continuum, without the discretizations mentioned above.  The construction is rather interesting.  One begins by constructing explicit couplings of $h$ with variants of the Schramm-Loewner evolution and showing that these couplings have certain properties.  Namely, if one conditions on part of the curve, then the conditional law of $h$ is that of a GFF in the complement of the curve with certain boundary conditions.  Examples of these couplings appear in \cite{She_SLE_lectures, DUB_PART, SchrammShe10, SHE_WELD} as well as variants in \cite{MakarovSmirnov09,HagendorfBauerBernard10,IzyurovKytola10}.  This step is carried out in some generality in \cite{DUB_PART, SHE_WELD}.  A second step (implemented only for some particular boundary value choices in \cite{DUB_PART} and \cite{SchrammShe10}) is to show that in such a coupling, the path is actually completely {\em determined} by $h$, and thus can be interpreted as a path-valued function of $h$.

Before we describe the rigorous construction of the flow lines of $e^{i(h/\chi+\theta)}$, let us offer some geometric intuition.  Suppose that $h$ is a continuous function and consider a flow line of the complex vector field $e^{ih/\chi}$ in $\h$ beginning at $0$.  That is, $\eta \colon [0,\infty) \to \h$ is a solution to the ODE
\begin{equation}
\label{eqn::flowode_ic}
\eta'(t) = e^{ih(\eta(t))/\chi} \quad\text{for}\quad t > 0,\ \ \ \eta(0) = 0.
\end{equation}
Note that $\| \eta'(t) \| = 1$.  Thus, the time derivative $\eta'(t)$ moves continuously around the unit circle $\s^1$ and $\big(h(\eta(t)) - h(\eta(0))\big) / \chi$ describes the net amount of \emph{winding} of $\eta'$ around $\s^1$ between times $0$ and $t$.  Let $g_t$ be the Loewner map of $\eta$.  That is, for each $t$, $g_t$ is the unique conformal transformation of the unbounded connected component of $\h \setminus \eta([0,t])$ to $\h$ that looks like the identity at infinity: $\lim_{z \to \infty} |g_t(z)-z| = 0$.  Loewner's theorem says that $g_t$ is a solution to the equation
\begin{equation} \label{eqn::loewner_ode}
 \partial_t g_t(z) = \frac{2}{g_t(z) - W_t},\ \ \ g_0(z) = z,
\end{equation}
where $W_t = g_t(\eta(t))$, provided $\eta$ is parameterized appropriately.  It will be convenient for us to consider the centered Loewner flow $f_t = g_t - W_t$ of $\eta$ in place of $g_t$.  The reason for this particular choice is that $f_t$ maps the tip of $\eta|_{[0,t]}$ to $0$.  Note that
\begin{equation}
\label{eqn::loewner_ode_centered}
  d f_t(z) = \frac{2}{f_t(z)} dt - dW_t.
\end{equation}
We may assume that $\eta$ starts out in the vertical direction, so that the winding number is approximately $\pi/2$ as $t \downarrow 0$.  We claim that the statement that $\eta|_{[0,t]}$ is a flow line of $e^{i(h/\chi + \pi/2)}$ is equivalent to the statement that for each $x$ on $\eta((0,t))$, we have
\begin{equation}
\label{eqn::flow_left}
 \chi \arg f_t'(z) \to -h(x) - \chi \pi/2
\end{equation}
as $z$ approaches from the left side of $\eta$ and
\begin{equation}
\label{eqn::flow_right}
 \chi \arg f_t'(z) \to -h(x) + \chi \pi/2
\end{equation}
as $z$ approaches from the right side of $\eta$.  To see this, first note that both $s \mapsto f_t^{-1}(s)|_{(0,s_+)}$ and $s \mapsto f_t^{-1}(-s)|_{(s_-,0)}$ are parameterizations of $\eta|_{[0,t]}$ where $s_-,s_+$ are the two images of~$0$ under $f_t$.  One then checks~\eqref{eqn::flow_left} (and~\eqref{eqn::flow_right} analogously) by using that $\eta(s) = f_t^{-1}(\phi(s))$ for $\phi \colon (0,\infty) \to (0,\infty)$ a smooth decreasing function and applying~\eqref{eqn::flowode_ic}.   If $\chi = 0$, then~\eqref{eqn::flow_left} and~\eqref{eqn::flow_right} hold if and only if $h$ is identically zero along the path, which is to say that $\eta$ is a zero-height contour line of $h$.  Roughly speaking, the flow lines of $e^{i(h/\chi + \pi/2)}$ and level sets of $h$ are characterized by~\eqref{eqn::flow_left} and~\eqref{eqn::flow_right}, though it turns out that the ``angle gap'' must be modified by a constant factor in order to account for the roughness of the field.  In a sense there is a constant ``height gap'' between the two sides of the path, analogous to what was shown for level lines of the GFF in \cite{SS09, SchrammShe10}.  The law of the flow line of $h$ starting at $0$ is determined by the boundary conditions of $h$.  It turns out that if the boundary conditions of $h$ are those shown in Figure~\ref{fig::flowlineheights}, then the flow line starting at $0$ is an $\SLE_\kappa$ process (with $\ul{\rho} \equiv 0$).  Namely, one has $-\lambda$ and $\lambda$ along the left and right sides of the axis and along the path one has $-\lambda'$ plus the winding on the left and $\lambda'$ plus the winding on the right, for the particular values of $\lambda$ and $\lambda'$ described in the caption.  Each time the path makes a quarter turn to the left, heights go up by $\tfrac{\pi}{2} \chi$.  Each time the path makes a quarter turn to the right, heights go down by $\tfrac{\pi}{2} \chi$.

\begin{figure}[h!]
\begin{center}
\includegraphics[scale=0.85]{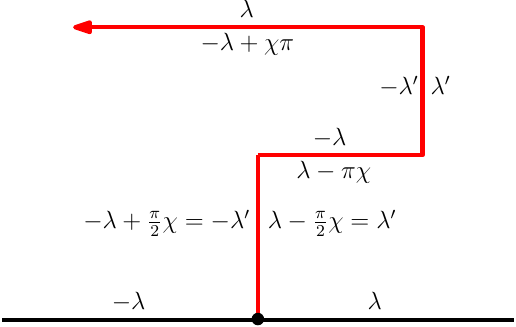}
\caption{\label{fig::flowlineheights} Fix $\kappa \in (0,4)$ and set $\lambda = \lambda(\kappa) = \frac{\pi}{\sqrt{\kappa}}$.  Write $\lambda' = \lambda(16/\kappa) = \frac{\pi \sqrt{\kappa}}{4}$.
Conditioned on a flow line, the heights of the field are given by (a constant plus)
$\chi$ times the winding of the path minus $\lambda'$ on the left side and $\chi$ times the
winding plus $\lambda'$ on the right side. For a fractal curve, these heights are not
pointwise defined (though their harmonic extension is well-defined). The figure
illustrates these heights for a piecewise linear curve.  In Figure~\ref{fig::winding}, we will describe a more compact notation for indicating the boundary heights in figures.
}
\end{center}
\end{figure}

\begin{figure}[h!]
\begin{center}
\subfigure[\label{fig::winding_left}]{\includegraphics[scale=0.85]{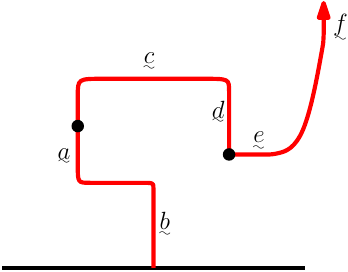}}
\hspace{0.02\textwidth}
\subfigure[\label{fig::winding_right}]{\includegraphics[scale=0.85]{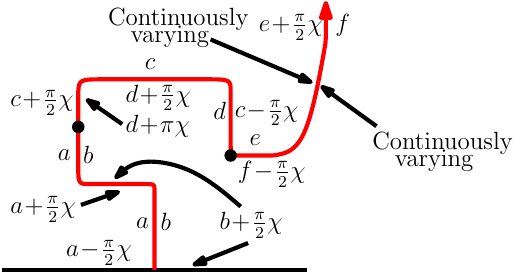}}
\caption{\label{fig::winding}  Throughout this article, we will need to consider Gaussian free fields whose boundary data changes with the winding of the boundary.  In order to indicate this succinctly, we will often make use of the notation depicted on the left hand side.  Specifically, we will delineate the boundary $\partial D$ of a Jordan domain $D$ with black dots.  On each arc $L$ of $\partial D$ which lies between a pair of black dots, we will draw either a horizontal or vertical segment $L_0$ and label it with \uwave{$x$} where $x \in \R$.  This serves to indicate that the boundary data along $L_0$ is given by $x$ as well as describe how the boundary data depends on the winding of $L$.  Whenever $L$ makes a quarter turn to the right, the height goes down by $\tfrac{\pi}{2} \chi$ and whenever $L$ makes a quarter turn to the left, the height goes up by $\tfrac{\pi}{2} \chi$.  More generally, if $L$ makes a turn which is not necessarily at a right angle, the boundary data is given by $\chi$ times the winding of $L$ relative to $L_0$.  When we just write $x$ next to a horizontal or vertical segment, we mean to indicate the boundary data at that segment and nowhere else.  The right panel above has exactly the same meaning as the left panel, but in the former the boundary data is spelled out explicitly everywhere.  Even when the curve has a fractal, non-smooth structure, the {\em harmonic extension} of the boundary values still makes sense, since one can transform the figure via the rule in Figure~\ref{fig::coordinatechange} to a half plane with piecewise constant boundary conditions. The notation above is simply a convenient way of describing the values of the constants.  We will often include horizontal or vertical segments on curves in our figures (even if the whole curve is known to be fractal) so that we can label them this way.
}
\end{center}
\end{figure}

\subsection{Coupling of paths with the GFF} \label{subsec::coupling}

We will now review some known results about coupling the GFF with $\SLE$.  For convenience and concreteness, we take $D$ to be the upper half-plane $\h$.  Couplings for other simply connected domains are obtained using the change of variables described in Figure~\ref{fig::coordinatechange}.  Recall that $\SLE_\kappa$ is the random curve described by the centered Loewner flow~\eqref{eqn::loewner_ode_centered} where $W_t = \sqrt{\kappa} B_t$ and $B_t$ is a standard Brownian motion.  More generally, an $\SLE_\kappa(\ul{\rho})$ process is a variant of $\SLE_\kappa$ in which one keeps track of multiple additional points, which we refer to as force points.  Throughout the rest of the article, we will denote configurations of force points as follows.  We suppose $\ul{x}^L = (x^{k,L} < \cdots < x^{1,L})$ where $x^{1,L} \leq 0$, and $\ul{x}^R = (x^{1,R} < \cdots < x^{\ell,R})$ where $x^{1,R} \geq 0$.  The superscripts $L,R$ stand for ``left'' and ``right,'' respectively.  If we do not wish to refer to the elements of $\ul{x}^L,\ul{x}^R$, we will denote such a configuration as $(\ul{x}^L;\ul{x}^R)$.  Associated with each force point $x^{i,q}$, $q \in \{L,R\}$ is a weight $\rho^{i,q} \in \R$ and we will refer to the vector of weights as $\ul{\rho} = (\ul{\rho}^L;\ul{\rho}^R)$.  An $\SLE_\kappa(\ul{\rho})$ process with force points $(\ul{x}^L;\ul{x}^R)$ corresponding to the weights $\ul{\rho}$ is the measure on continuously growing compact hulls $K_t$ --- compact subsets of $\ol{\h}$ so that $\h \setminus K_t$ is simply connected --- such that the conformal maps $g_t \colon \h \setminus K_t \to \h$, normalized so that $\lim_{z \to \infty} |g_t(z) - z| = 0$, satisfy~\eqref{eqn::loewner_ode_centered} with $W_t$ replaced by the solution to the system of (integrated) SDEs
\begin{align}
W_t &= \sqrt{\kappa} B_t + \sum_{q \in \{L,R\}} \sum_{i} \int_0^t \frac{\rho^{i,q}}{W_s - V_s^{i,q}} ds, \label{eqn::sle_kappa_rho_sde}\\
V_t^{i,q} &= \int_0^t \frac{2}{V_s^{i,q} - W_s} ds + x^{i,q},\ \ \ q \in \{L,R\}. \label{eqn::force_point_evolution}
\end{align}
We will provide some additional discussion of both $\SLE_\kappa$ and $\SLE_\kappa(\ul{\rho})$ processes in Section~\ref{sec::sle}.  The general coupling statement below applies for all $\kappa>0$.  Theorem~\ref{thm::coupling_existence} below gives a general statement of the existence of the coupling.  Essentially, the theorem states that if we sample a particular random curve on a domain $D$ --- and then sample a Gaussian free field on $D$ minus that curve with certain boundary conditions --- then the resulting field (interpreted as a distribution on all of $D$) has the law of a Gaussian free field on $D$ with certain boundary conditions.

It is proved in \cite{DUB_PART} that Theorem~\ref{thm::coupling_existence} holds for any $\kappa$ and $\ul{\rho}$ for which a solution to~\eqref{eqn::sle_kappa_rho_sde} exists (this can also be extended to a continuum of force points; this is done for a time-reversed version of $\SLE$ in \cite{SHE_WELD}).  The special case of $\pm \lambda$ boundary conditions also appears in \cite{She_SLE_lectures}.  (See also \cite{SHE_WELD} for a more detailed version of the argument in \cite{She_SLE_lectures} with additional figures and explanation.)

The question of when~\eqref{eqn::sle_kappa_rho_sde} has a solution is not explicitly addressed in \cite{DUB_PART}.  In Section~\ref{sec::sle}, we will prove the existence of a unique solution to~\eqref{eqn::sle_kappa_rho_sde} up until the {\bf continuation threshold} is hit --- the first time $t$ that $W_t = V_t^{j,q}$ where $\sum_{i=1}^j \rho^{i,q} \leq -2$, for some $q \in \{L,R\}$.  This is the content of Theorem~\ref{thm::slekrdef}.  We will reprove Theorem~\ref{thm::coupling_existence} here for the convenience of the reader.  It is a straightforward consequence of Theorem~\ref{thm::slekrdef} and \cite[Theorem~6.4]{DUB_PART}.

All of our results will hold for $\SLE_\kappa(\ul{\rho})$ processes up until (and including) the continuation threshold.  It turns out that the continuation threshold is infinite almost surely if and only if
\[ \sum_{i=1}^j \rho^{i,L} > -2 \quad\text{for all}\quad 1 \leq j \leq k \quad\text{and}\quad \sum_{i=1}^j \rho^{i,R} > -2 \quad\text{for all}\quad 1 \leq j \leq \ell.\]

\begin{theorem}
\label{thm::coupling_existence}
Fix $\kappa > 0$ and a vector of weights $(\ul{\rho}^L;\ul{\rho}^R)$.  Let $K_t$ be the hull at time $t$ of the $\SLE_\kappa(\ul{\rho})$ process generated by the Loewner flow~\eqref{eqn::loewner_ode_centered} where $(W,V^{i,q})$ solves~\eqref{eqn::sle_kappa_rho_sde},~\eqref{eqn::force_point_evolution}.  Let $\Fh_{t}^0$ be the function which is harmonic in $\h$ with boundary values
\begin{align*}
   -&\lambda\left(1+ \sum_{i=0}^j \rho^{i,L}\right) \quad\text{if}\quad  s \in [V_t^{j+1,L},V_t^{j,L}),\\
   &\lambda\left(1+\sum_{i=0}^j \rho^{i,R}\right)  \quad\text{if}\quad  s \in [V_t^{j,R},V_t^{j+1,R}),
\end{align*}
where $\rho^{0,L} = \rho^{0,R} = 0$, $x^{0,L} = 0^-$, $x^{k+1,L} = -\infty$, $x^{0,R} = 0^+$, and $x^{\ell+1,R} = \infty$.  (See Figure~\ref{fig::conditional_boundary_data}.)
Let
\[ \Fh_t(z) = \Fh_{t}^0(f_t(z)) - \chi \arg f_t'(z),\ \ \ \chi = \frac{2}{\sqrt{\kappa}} - \frac{\sqrt{\kappa}}{2}.\]
Let $(\CF_t)$ be the filtration generated by $(W,V^{i,q})$.  There exists a coupling $(K,h)$ where $\wt{h}$ is a zero boundary GFF on $\h$ and $h = \wt{h} + \Fh_0$ such that the following is true.  Suppose~$\tau$ is any $\CF_t$-stopping time which almost surely occurs before the continuation threshold is reached.  Then $K_\tau$ is a local set for $h$ and the conditional law of $h|_{\h \setminus K_\tau}$ given $\CF_\tau$ is equal to the law of $\Fh_\tau + \wt{h} \circ f_\tau$.
\end{theorem}

We will give a review of the theory of local sets \cite{SchrammShe10} for the GFF in Section~\ref{subsec::local_sets}.

\begin{figure}[h!]
\begin{center}
\includegraphics[scale=0.85]{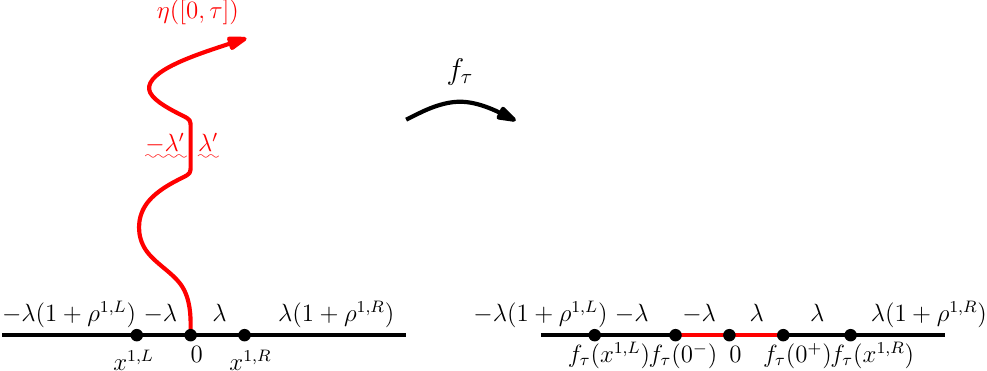}
\caption{\label{fig::conditional_boundary_data}The function $\Fh_{\tau}^0$ in Theorem~\ref{thm::coupling_existence} is the harmonic extension of the boundary values depicted in the right panel in the case that there are two boundary force points, one on each side of $0$.  The function $\Fh_\tau = \Fh_{\tau}^0 \circ f_\tau - \chi \arg f_\tau'$ in Theorem~\ref{thm::coupling_existence} is the harmonic extension of the boundary data specified in the left panel.  (Recall the relationship between $\lambda$ and $\lambda'$ indicated in Figure~\ref{fig::flowlineheights}.)}
\end{center}
\end{figure}

Notice that $\chi > 0$ when $\kappa \in (0,4)$, $\chi < 0$ when $\kappa > 4$, and that $\chi(\kappa) = -\chi(\kappa')$ for $\kappa' = 16/\kappa$ (though throughout the rest of this article, whenever we write $\chi$ it will be assumed that $\kappa \in (0,4)$).  This means that in the coupling of Theorem~\ref{thm::coupling_existence}, the conditional law of $h$ given either an $\SLE_{\kappa}$ or an $\SLE_{\kappa'}$ curve transforms in the same way under a conformal map, up to a change of sign.  Using this, we are able to construct $\eta \sim \SLE_\kappa$, $\kappa \in (0,4)$, and $\eta' \sim \SLE_{\kappa'}$ curves within the same imaginary geometry (see Figure~\ref{fig::counterflowline}).  We accomplish this by taking $\eta$ to be coupled with $h$ and $\eta'$ to be coupled with $-h$, as in the statement of Theorem~\ref{thm::coupling_existence} (this is the reason we can always take $\chi > 0$).

{\bf Definition.}
When $\kappa \in (0,4)$, we will refer to an $\SLE_\kappa(\ul \rho)$ curve (if it exists) coupled with a GFF~$h$ on~$\h$ with boundary conditions as in Theorem~\ref{thm::coupling_existence} as a {\bf flow line} of $h$. One can use the conformal coordinate change of  Figure~\ref{fig::coordinatechange} to extend this definition to simply connected domains other than~$\h$. To spell out this point explicitly, suppose that~$D$ is a simply connected domain homeomorphic to the disk, $x,y \in \partial D$ are distinct, and $\psi \colon D \to \h$ is a conformal transformation with $\psi(x) = 0$ and $\psi(y) = \infty$.  Let us assume that we have fixed a branch of $\arg \psi'$ that is defined continuously on all of $D$. We assume further that $\ul{x}^L$ (resp.\ $\ul{x}^R$) consists of $k$ (resp.\ $\ell$) distinct marked prime ends in the clockwise (resp.\ counterclockwise) segment of $\partial D$ (as defined by $\psi$) which are in clockwise (resp.\ counterclockwise) order.  We take $x^{0,L} = x = x^{0,R} = x$ and $x^{k+1,L} = x^{\ell+1,R} = y$.  We then suppose that $h$ is a GFF on $D$ with boundary conditions in the clockwise (resp.\ counterclockwise) segment of $\partial D$ from $x^{j,L}$ to $x^{j+1,L}$ (resp.\ $x^{j,R}$ to $x^{j+1,R}$) given by $-\lambda \left(1+ \sum_{i=0}^j \rho^{i,L}\right) - \chi \arg \psi'$ (resp.\ $\lambda\left(1+\sum_{i=0}^j \rho^{i,R} \right) - \chi \arg \psi'$).  We refer to an $\SLE_\kappa(\ul{\rho})$ curve $\eta$ (if it exists) from $x$ to $y$ on $D$, $\kappa \in (0,4)$, coupled with $h$ as a {\bf flow line} of $h$ if the curve $\psi(\eta)$ in $\h$  is coupled as a flow line of the GFF $h \circ \psi^{-1} - \chi \arg (\psi^{-1})'$ on $\h$.  (Recall~\eqref{eqn::ac_eq_rel} and Figure~\ref{fig::coordinatechange}.)

{\bf Remark.} Observe that in the discussion above, the choice of the branch of $\arg \psi'$ was important. Changing the branch chosen would in some sense correspond to adding a multiple of $2 \pi \chi$ to either side of the  $\SLE_\kappa(\ul \rho)$ curve, and if one did this then (in order for the curve to remain a flow line) one would have to compensate by adding the same quantity to the boundary data. In some sense, changing the branch of $\arg \psi'$ is equivalent to adding a multiple of $2 \pi \chi$ to the boundary data.
If one wishes to be fully concrete, one can fix the branch of $\arg \psi'$ in an arbitrary way --- say, so that $\arg \psi'(\psi^{-1}(i)) \in (-\pi,\pi]$ --- and then assume that the boundary data is adjusted accordingly. In practice, when we discuss flow lines (in the half plane or elsewhere) we will usually specify boundary data using a figure and the notation explained in Figure~\ref{fig::winding} (or in Figure~\ref{fig::conditional_boundary_data}). This approach will avoid any ``multiple of $2\pi\chi$'' ambiguity and will make it completely clear exactly what the boundary data is along the curve. This remark also applies to the definition of counterflow line given below.

We will give several examples of coordinate changes in Section~\ref{sec::dubedat}.  See also Figure~\ref{fig::flowlineheights} and Figure~\ref{fig::winding} for an illustration of how the boundary data for the GFF changes when applying~\eqref{eqn::ac_eq_rel}.

The fact that $\SLE_\kappa(\ul \rho)$ is generated by a continuous curve up until hitting the continuation threshold will be established for general $\rho$ values in Theorem~\ref{thm::continuity}.  It is not obvious from the coupling described in Theorem~\ref{thm::coupling_existence} that such paths are deterministic functions of $h$.  That this is in fact the case is given in Theorem~\ref{thm::coupling_uniqueness}.

As mentioned earlier, we will sometimes use the phrase {\em flow line of angle $\theta$} to denote the corresponding curve that one obtains when $\theta \chi$ is added to the boundary data (so that $h$ is replaced by $h + \theta \chi$).

{\bf Definition.}
We will refer to an $\SLE_{\kappa'}(\ul \rho)$ curve (if it exists), $\kappa' \in (4,\infty)$, coupled with a GFF $-h$ (note the sign change here; this accounts for the $\chi(\kappa)$ vs.\ $\chi(\kappa')$ issue discussed just above) as in Theorem~\ref{thm::coupling_existence} as a {\bf counterflow line} of $h$.  
Again, one can use conformal maps to extend this definition to simply connected domains other than~$\h$. Suppose that~$D$ is a non-trivial simply connected domain, $x,y \in \partial D$ are distinct, and $\psi \colon D \to \h$ is a conformal transformation with $\psi(x) = 0$ and $\psi(y) = \infty$, and that a branch of $\arg \psi'$ has been fixed (as in the flow line definition above).  We assume further that $\ul{x}^L$ (resp.\ $\ul{x}^R$) consists of $k$ (resp.\ $\ell$) distinct marked prime ends in the clockwise (resp.\ counterclockwise) segment of $\partial D$ (as defined by $\psi$) which are in clockwise (resp.\ counterclockwise) order.  We take $x^{0,L} = x = x^{0,R} = x$ and $x^{k+1,L} = x^{\ell+1,R} = y$.  We then suppose that $h$ is a GFF on $D$ with boundary conditions in the clockwise (resp.\ counterclockwise) segment of $\partial D$ from $x^{j,L}$ to $x^{j+1,L}$ (resp.\ $x^{j,R}$ to $x^{j+1,R}$) given by $\lambda' \left(1+ \sum_{i=0}^j \rho^{i,L}\right) - \chi \arg \psi'$ (resp.\ $-\lambda'\left(1+\sum_{i=0}^j \rho^{i,R} \right) - \chi \arg \psi'$); here $\chi = \chi(\kappa) > 0$.  We refer to an $\SLE_{\kappa'}(\ul{\rho})$ curve $\eta'$ (if it exists) from $x$ to $y$ on $D$, $\kappa' \in (4,\infty)$, coupled with $h$ as a {\bf counterflow line} of $h$ if the curve $\psi(\eta')$ in $\h$ is coupled as a counterflow line of the GFF $h \circ \psi^{-1} - \chi \arg (\psi^{-1})'$ on $\h$; here $\chi = \chi(\kappa) > 0$.  (Recall~\eqref{eqn::ac_eq_rel} and Figure~\ref{fig::coordinatechange}.)

Again, the fact that $\SLE_{\kappa'}(\ul \rho)$ is generated by a continuous curve up until hitting the continuation threshold is established for general $\rho$ values in Theorem~\ref{thm::continuity}.

As in the setting of flow lines, it is not obvious from the coupling described in Theorem~\ref{thm::coupling_existence} that such paths are deterministic functions of $h$.  That this is in fact the case is given in Theorem~\ref{thm::coupling_uniqueness}.  The reason for the terminology ``counterflow line'' is that, as briefly mentioned earlier, it will turn out that the set of the points hit by an $\SLE_{\kappa'}$ counterflow line can be interpreted as a ``light cone'' of points accessible by certain angle-restricted $\SLE_\kappa$ flow lines; the $\SLE_{\kappa'}$ passes through the points on each of these flow lines in the opposite (``counterflow'') direction.  We will provide some additional explanation near the statement of Theorem~\ref{thm::lightconeroughstatement}.

The correction $-\chi \arg f_t'$ which appears in the statement Theorem~\ref{thm::coupling_existence} has the interpretation of being the harmonic extension of $\chi$ times the {\bf winding} of $\partial (\h \setminus \eta([0,\tau]))$.  We will use the informal notation $\chi \cdot {\rm winding}$ for this function throughout this article and employ a special notation to indicate this in figures.  See Figure~\ref{fig::winding} for further explanation of this point.

Similar couplings are constructed in \cite{IzyurovKytola10} for the GFF with Neumann boundary data on part of the domain boundary, and \cite{HagendorfBauerBernard10} couples the GFF on an annulus with annulus $\SLE$.  Makarov and Smirnov extend the $\SLE_4$ results of \cite{She_SLE_lectures,SchrammShe10} to the setting of the massive GFF and a massive version of $\SLE$ in \cite{MakarovSmirnov09}.

\begin{figure}[h!]
\begin{center}
\includegraphics[scale=0.85]{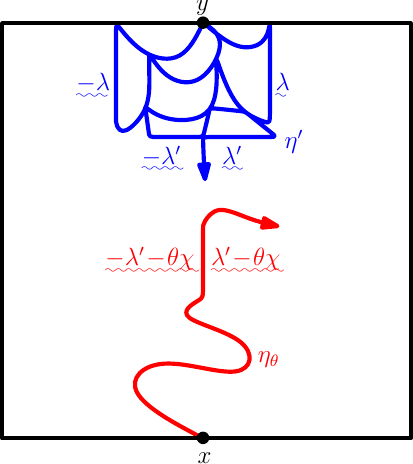}
\end{center}
\caption{\label{fig::counterflowline} We can construct $\SLE_\kappa$ flow lines, $\kappa \in (0,4$), and $\SLE_{\kappa'}$, $\kappa' = 16/\kappa$, counterflow lines within the same imaginary geometry.  This is depicted above for a single counterflow line $\eta'$ emanating from $y$ and a flow line $\eta_\theta$ with angle $\theta$ starting from $x$.  In this coupling, $\eta_\theta$ is coupled with $h+\theta \chi$ and $\eta'$ is coupled with $-h$ as in Theorem~\ref{thm::coupling_existence}.  Also shown is the boundary data for $h$ in $D \setminus (\eta'([0,\tau']) \cup \eta_\theta([0,\tau]))$ conditional on $\eta_\theta([0,\tau])$ and $\eta'([0,\tau'])$ where $\tau$ and $\tau'$ are stopping times for $\eta_\theta$ and $\eta'$ respectively (we intentionally did not specify the boundary data of $h$ on $\partial D$).  Assume that $\eta'$ is non-boundary filling.  Then if $\theta = \tfrac{1}{\chi}(\lambda'-\lambda) = -\tfrac{\pi}{2}$ so that the boundary data on the right side of $\eta_\theta$ matches that on the right side of $\eta'$, then $\eta_\theta$ will almost surely hit and then ``merge'' into the right boundary of $\eta'$.  The analogous result holds if $\theta = \tfrac{1}{\chi}(\lambda-\lambda') = \tfrac{\pi}{2}$ so that the boundary data on the left side of $\eta_\theta$ matches that on the left side of $\eta'$.  This fact is known as Duplantier duality (or $\SLE$ duality).  More generally, if $\theta \in [-\tfrac{\pi}{2},\tfrac{\pi}{2}]$ then $\eta_\theta$ is almost surely contained in $\eta'$ but the union of the traces of $\eta_\theta$ as $\theta$ ranges over the entire interval $[-\tfrac{\pi}{2},\tfrac{\pi}{2}]$ is almost surely a strict subset of the range of $\eta'$.  We will show, however, that the range of $\eta'$ can be constructed as a ``light cone'' of $\SLE_\kappa$ trajectories whose angle is allowed to vary in time but is restricted to $[-\tfrac{\pi}{2},\tfrac{\pi}{2}]$.
}
\end{figure}

\subsection{Main results} \label{subsec::mainresult}

In the case that $\rho = 0$ and $\eta$ is ordinary $\SLE$, Dub\'edat showed in \cite{DUB_PART} that in the coupling of Theorem~\ref{thm::coupling_existence} the path is actually a.s.\ determined by the field.  A $\kappa = 4$ analog of this statement was also shown in \cite{SchrammShe10}. In this paper, we will extend these results to the more general setting of Theorem~\ref{thm::coupling_existence}.

\begin{theorem}
\label{thm::coupling_uniqueness}  Suppose that $h$ is a GFF on $\h$ and that $\eta \sim \SLE_\kappa(\ul{\rho})$.  If $(\eta,h)$ are coupled as in the statement of Theorem~\ref{thm::coupling_existence}, then $\eta$ is almost surely determined by~$h$.
\end{theorem}

The basic idea of our proof is as follows.  First, we extend the argument of \cite{DUB_PART} for $\SLE_\kappa$, $\kappa \in (0,4]$, to the case of $\eta \sim \SLE_\kappa(\ul{\rho})$ with $\ul{\rho} = (\rho^L;\rho^R)$ where $\rho^L$ and $\rho^R$ are real numbers satisfying $\rho^L \geq \tfrac{\kappa}{2}-2$ and $\rho^R \geq 0$.  This condition implies that $\eta$ almost surely does not intersect $\partial \h$ after time $0$ and allows us to apply the argument from \cite{DUB_PART} with relatively minor modifications.  We then reduce the more general case that $\rho^L,\rho^R > -2$ to the former setting by studying the flow lines $\eta_\theta$ of $e^{i(h/\chi + \theta)}$ emanating from $0$.  In this case, these are also $\SLE_\kappa(\ul{\rho})$ curves with force points at $0^-$ and $0^+$.  We will prove that if $\theta_1 < 0 < \theta_2$, then $\eta_{\theta_1}$ almost surely lies to the right of $\eta$ which in turn almost surely lies to the right of $\eta_{\theta_2}$.  We will next show that the conditional law of $\eta$ given $\eta_{\theta_1},\eta_{\theta_2}$ is an $\SLE_\kappa(\rho^L(\theta_1);\rho^R(\theta_2))$ process independently in each of the connected components of $\h \setminus (\eta_{\theta_1} \cup \eta_{\theta_2})$ which lie between $\eta_{\theta_1}$ and $\eta_{\theta_2}$.  By adjusting $\theta_1,\theta_2$, we can obtain any combination of $\rho^L(\theta_1),\rho^R(\theta_2) > -2$.  We then extend this result to the setting of many force points by systematically studying the case with two boundary force points which are both to the right of $0$ and then employing the absolute continuity properties of the GFF combined with an induction argument.  The idea for $\kappa > 4$ follows from a more elaborate variant of this general strategy.

By applying the same set of techniques used to prove Theorem~\ref{thm::coupling_uniqueness}, we also obtain the continuity of the $\SLE_\kappa(\ul{\rho})$ trace.

\begin{theorem}
\label{thm::continuity}
Suppose that $\kappa > 0$.  If $\eta \sim \SLE_\kappa(\ul{\rho})$ on $\h$ from $0$ to $\infty$ then $\eta$ is almost surely a continuous path, up to and including the continuation threshold. On the event that the continuation threshold is not hit before $\eta$ reaches $\infty$, we have a.s.\ that $\lim_{t \to \infty} |\eta(t)| = \infty$.
\end{theorem}

The continuity of $\SLE_\kappa$ (with $\rho = 0)$ was first proved by Rohde and Schramm in \cite{RS05}.  By invoking the Girsanov theorem, one can deduce from \cite{RS05} that $\SLE_\kappa(\ul{\rho})$ processes are also continuous, but only up until just before the first time that a force point is absorbed.  The main idea of the proof in \cite{RS05} is to control the moments of the derivatives of the reverse $\SLE_\kappa$ Loewner flow near the origin.  These estimates involve martingales whose corresponding PDEs become complicated when working with $\SLE_\kappa(\ul{\rho})$ in place of usual $\SLE_\kappa$.  Our proof uses the Gaussian free field as a vehicle to construct couplings which allow us to circumvent these technicalities.

Another achievement of this paper will be to show how to jointly construct all of the flow lines emanating from a single boundary point.  This turns out to give us a flow-line based construction of $\SLE_{16/\kappa}(\ul{\rho})$, $\kappa \in (0,4)$.  That is, $\SLE_{16/\kappa}$ variants occur naturally within the same imaginary geometry as $\SLE_{\kappa}$.  Note that $16/\kappa$ assumes all possible values in $(4, \infty)$ as $\kappa$ ranges over $(0,4)$.  Imprecisely, we have that the set of all points reachable by proceeding from the origin in a possibly varying but always ``northerly'' direction (the so-called ``light cone'') along $\SLE_\kappa$ flow lines is a form of $\SLE_{16/\kappa}$ for $\kappa \in (0,4)$ generated in the reverse direction (see Figure~\ref{fig::counterflowline}).

\begin{figure}[h!]
\begin{center}
\subfigure{
\includegraphics[width=0.63\textwidth]{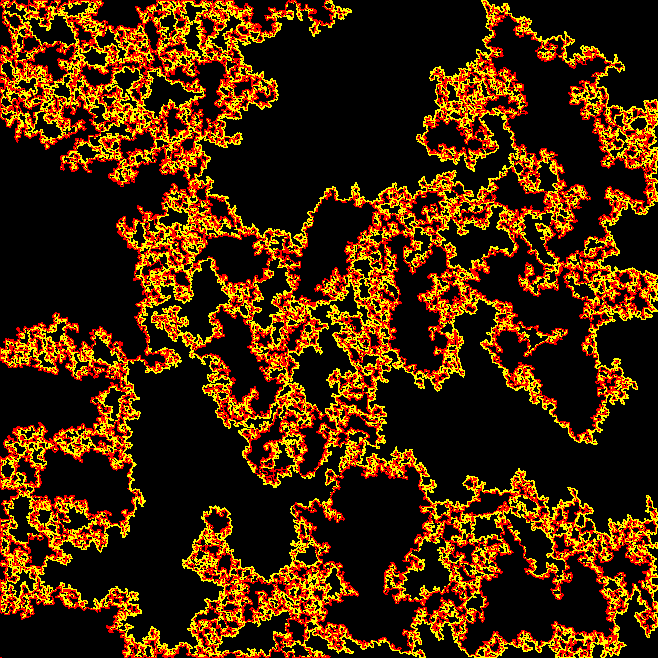}}

\subfigure{
\includegraphics[width=0.31\textwidth]{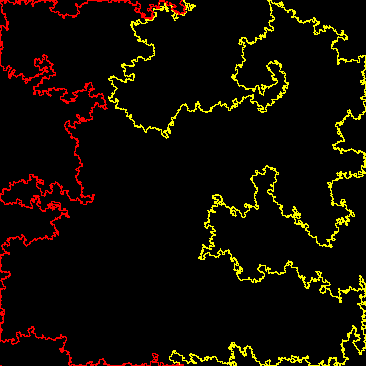}}
\subfigure{
\includegraphics[width=0.31\textwidth]{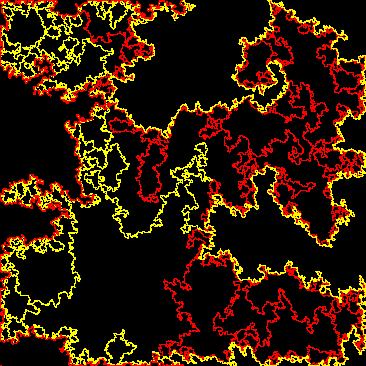}}
\subfigure{
\includegraphics[width=0.31\textwidth]{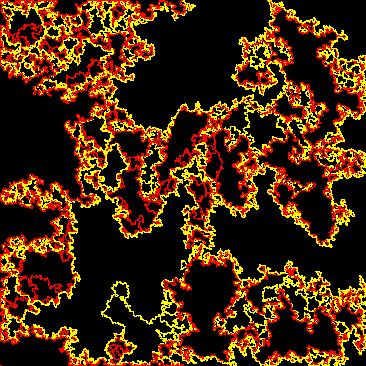}}
\end{center}
\caption{\label{fig::sle6_lightcone} {Simulation of the light cone construction of an $\SLE_6$ curve $\eta'$ in $[-1,1]^2$ from $i$ to $-i$, generated using a projection $h$ of a GFF on $[-1,1]^2$ onto the space of functions piecewise linear on the triangles of an $800 \times 800$ grid.
The lower left panel shows left and right boundaries of $\eta'$, which consist of points accessible by flowing in the vector field $e^{i h/\chi}$ for $\chi = 2/\sqrt{8/3} - \sqrt{8/3}/2$ at angle $\tfrac{\pi}{2}$ (red) and $-\tfrac{\pi}{2}$ (yellow), respectively, from $-i$.  The lower middle panel shows points accessible by flowing at angle $\tfrac{\pi}{2}$ (red) or angle $-\tfrac{\pi}{2}$ (yellow) from the yellow and red points, respectively, of the left picture; the lower right shows another iteration of this.  The top picture illustrates the light cone, the limit of this procedure.  (All paths are red or yellow;
any shade variation is a rendering artifact.)
}}
\end{figure}

\begin{figure}[h!]
\begin{center}
\subfigure{
\includegraphics[width=0.63\textwidth]{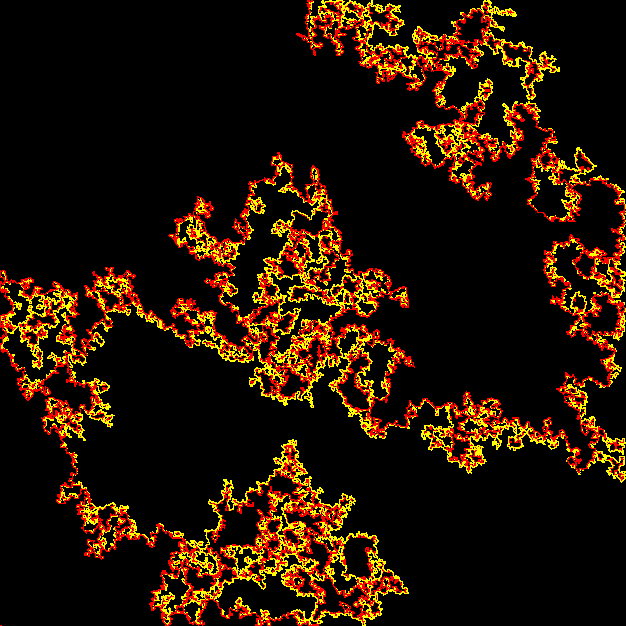}}

\subfigure{
\includegraphics[width=0.31\textwidth]{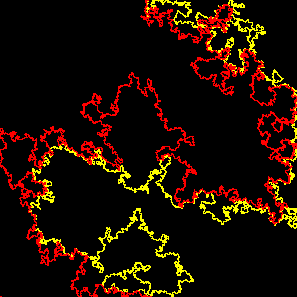}}
\subfigure{
\includegraphics[width=0.31\textwidth]{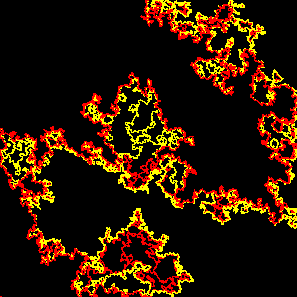}}
\subfigure{
\includegraphics[width=0.31\textwidth]{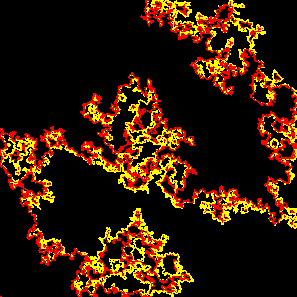}}
\end{center}
\caption{\label{fig::sle16_3_lightcone} {Numerical simulation of the light cone construction of an $\SLE_{16/3}$ process $\eta'$ in $[-1,1]^2$ from $i$ to $-i$ generated using a projection $h$ of a GFF on $[-1,1]^2$ onto the space of functions piecewise linear on the triangles of an $800 \times 800$ grid.  The lower left panel depicts the left and right boundaries of $\eta'$, which correspond to the set of points accessible by flowing in the vector field $e^{i h/\chi}$ for $\chi = 2/\sqrt{3} - \sqrt{3}/2$ at angle $\tfrac{\pi}{2}$ (red) and $-\tfrac{\pi}{2}$ (yellow), respectively, from $-i$.  The lower middle panel shows the set of points accessible by flowing at angle $\tfrac{\pi}{2}$ (red) or angle $-\tfrac{\pi}{2}$ (yellow) from the yellow and red points, respectively, of the left picture and the lower right panel depicts another iteration of this.  The top picture illustrates the light cone, which is the limit of this procedure.
}}
\end{figure}

\begin{figure}[h!]
\begin{center}
\subfigure{
\includegraphics[width=0.58\textwidth]{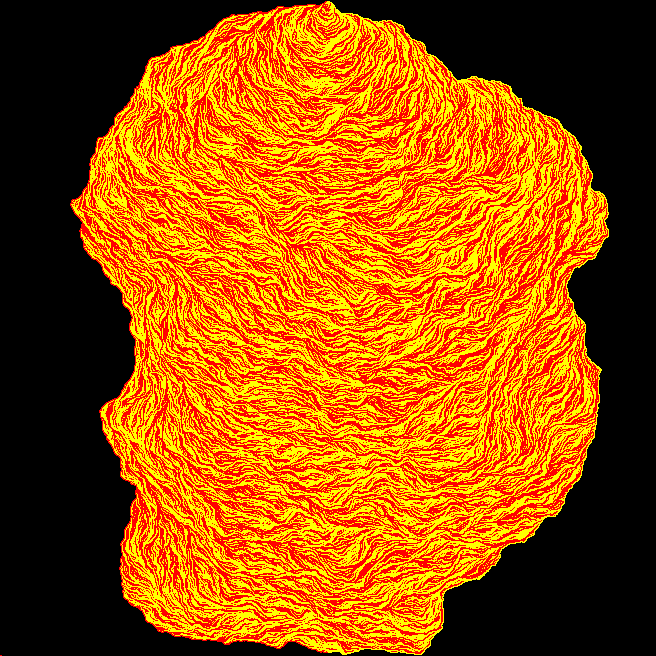}}

\subfigure{
\includegraphics[width=0.31\textwidth]{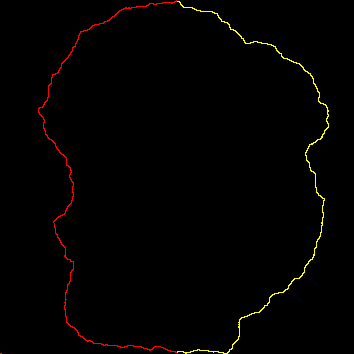}}
\subfigure{
\includegraphics[width=0.31\textwidth]{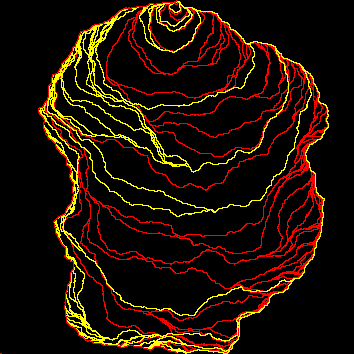}}
\subfigure{
\includegraphics[width=0.31\textwidth]{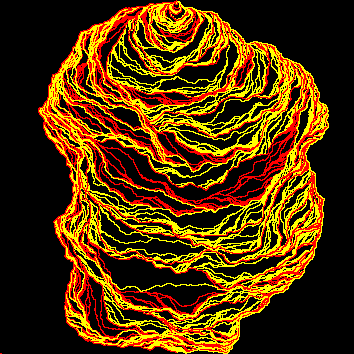}}
\end{center}
\caption{\label{fig::sle64_lightcone} { Numerical simulation of the light cone construction of an $\SLE_{64}(32;32)$ process $\eta'$ in $[-1,1]^2$ from $i$ to $-i$ generated using a projection $h$ of a GFF on $[-1,1]^2$ onto the space of functions piecewise linear on the triangles of an $800 \times 800$ grid. (It turns out that $\SLE_{64}(\rho_1;\rho_2)$ processes are boundary filling only when $\rho_1,\rho_2 \leq 28$.)  The lower left panel depicts the left and right boundaries of $\eta'$, which correspond to the set of points accessible by flowing in the vector field $e^{i h/\chi}$ for $\chi = 2/\sqrt{1/4} - \sqrt{1/4}/2$ at angle $\tfrac{\pi}{2}$ (red) and $-\tfrac{\pi}{2}$ (yellow), respectively, from $-i$.  The lower middle panel shows the set of points accessible by flowing at angle $\tfrac{\pi}{2}$ (red) or angle $-\tfrac{\pi}{2}$ (yellow) from the yellow and red points, respectively, of the left picture and the lower right panel depicts another iteration of this.  The top picture illustrates the light cone, which is the limit of this procedure.
}}
\end{figure}

\begin{figure}[h!]
\begin{center}
\includegraphics[width=0.85\textwidth]{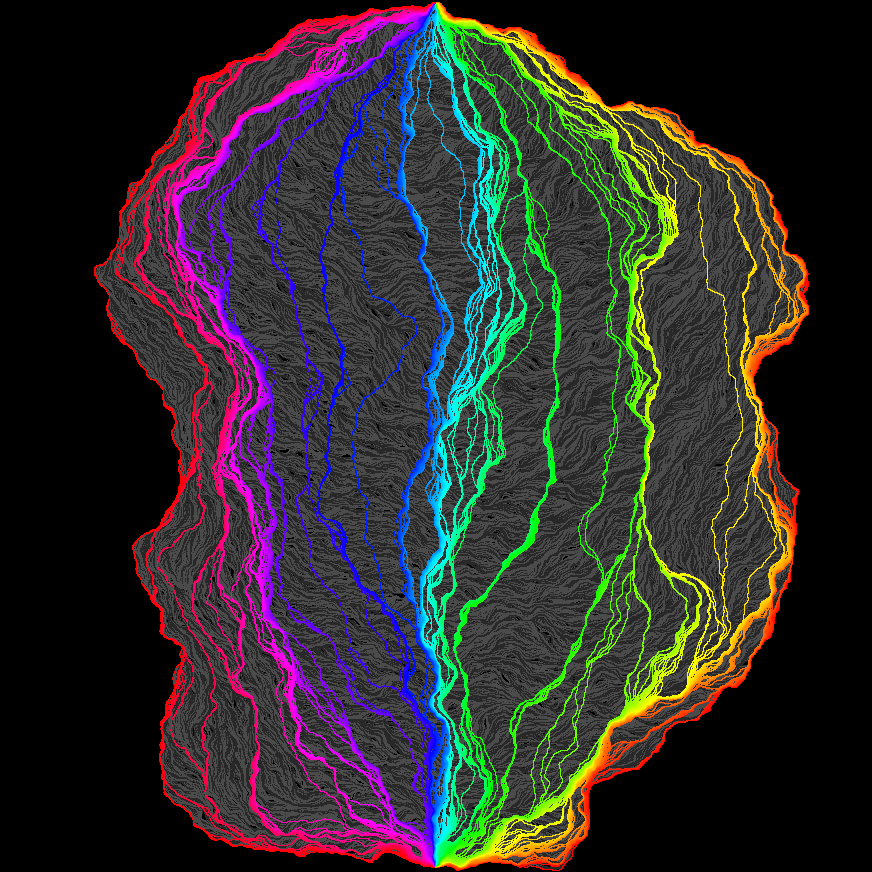}
\end{center}
\caption{\label{fig::sle64_fan} The simulation of the light cone from the top panel of Figure~\ref{fig::sle64_lightcone} where trajectories which flow at angle $\tfrac{\pi}{2}$ are dark gray and those which flow at angle $-\tfrac{\pi}{2}$ are depicted in a medium-dark gray.  The fan from $-i$ --- the set of all points accessible by fixed-angle trajectories with angles in $[-\tfrac{\pi}{2},\tfrac{\pi}{2}]$ starting at $-i$ --- is drawn on top of the light cone.  The different colors indicate trajectories with different angles.  The simulation shows that the fan does not fill the light cone; we establish this fact rigorously in Proposition~\ref{prop::fan_does_not_hit}.}
\end{figure}

\begin{figure}[h!]
\begin{center}
\includegraphics[width=0.95\textwidth,clip=true, trim = 1mm 1mm 1mm 1mm]{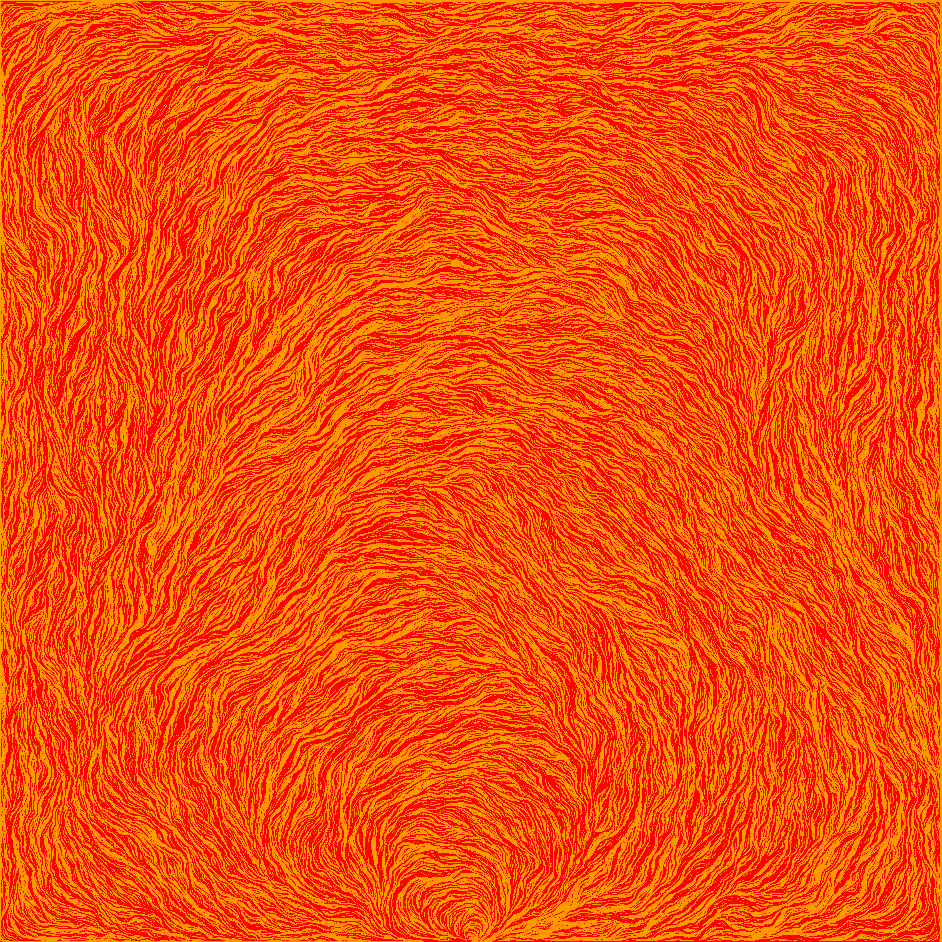}
\end{center}
\caption{\label{fig::sle128_lightcone} {Numerical simulation of the light cone construction of an $\SLE_{128}$ process $\eta'$ in $[-1,1]^2$ from $i$ to $-i$ generated using a projection $h$ of a GFF on $[-1,1]^2$ onto the space of functions piecewise linear on the triangles of an $800 \times 800$ grid.  The red and yellow curves depict the left and right boundaries, respectively, of the time evolution of $\eta'$ as it traverses $[-1,1]^2$.
}}
\end{figure}

\begin{figure}[h!]
\begin{center}
\subfigure[An $\SLE_6$ process $\eta'$ from $i$ to $-i$ generated using the light cone construction.]{
\includegraphics[width=.45\textwidth]{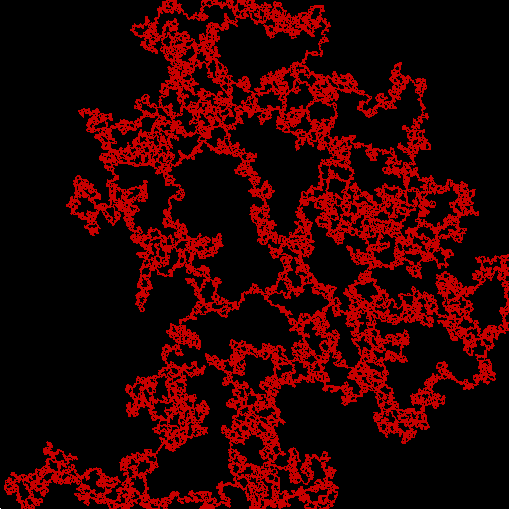}}
\hspace{0.03\textwidth}
\subfigure[The zero angle flow line $\eta$ from $-i$ to $i$ drawn on top of $\eta'$.]{
\includegraphics[width=.45\textwidth]{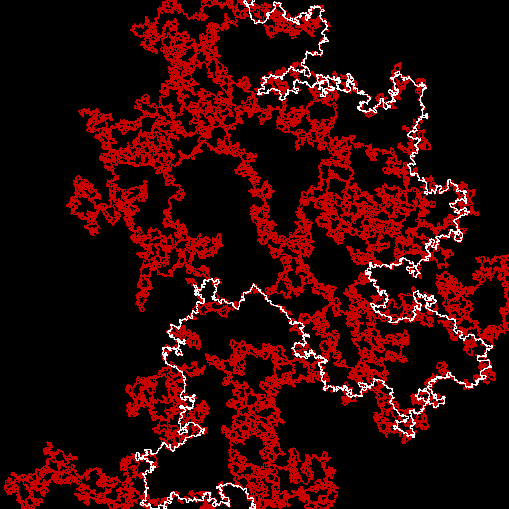}}
\subfigure[The fan from $-i$ to $i$.  The rays are $\SLE_{8/3}(\rho_1;\rho_2)$ processes.]{
\includegraphics[width=.45\textwidth]{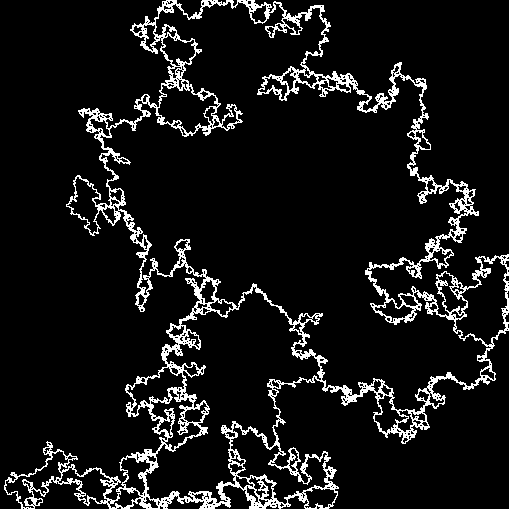}}
\hspace{0.03\textwidth}
\subfigure[The fan drawn on top of $\eta'$.  It does not cover the range of $\eta'$.]{
\includegraphics[width=.45\textwidth]{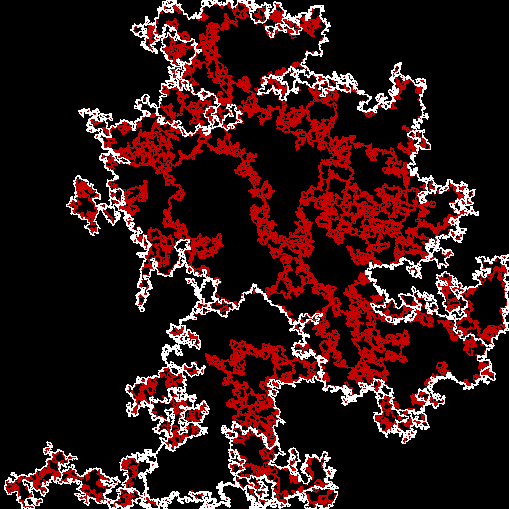}}
\end{center}
\caption{\label{fig::sle6_decomposition} { Numerical simulation of the light cone construction of an $\SLE_6$ process $\eta'$ in $[-1,1]^2$ from $i$ to $-i$ and its interaction with the zero angle flow line $\eta \sim \SLE_{8/3}(-1;-1)$ and the fan starting from $-i$, generated using a projection $h$ of a GFF on $[-1,1]^2$ onto the space of functions piecewise linear on the triangles of an $800 \times 800$ grid.  In the top right panel, the conditional law of the restrictions of $\eta'$ given $\eta$ to the left and right sides of $[-1,1]^2 \setminus \eta$ are independent $\SLE_6(-\tfrac{3}{2})$ processes.}}
\end{figure}

\begin{figure}[h!!!]
\begin{center}
\includegraphics[scale=0.85]{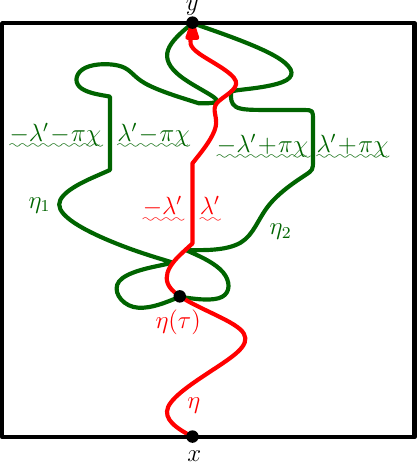}
\end{center}
\caption{\label{fig::pi_turn} Let $h$ be a GFF on a Jordan domain $D$, fix $x,y \in \partial D$ distinct, and let $\eta$ be the flow line of $h$ starting at $x$ targeted at $y$.  Let $\tau$ be any stopping time for $\eta$ and let $\eta_1$ and $\eta_2$ be the flow lines of $h$ conditional on $\eta$ starting at $\eta(\tau)$ with angles $\pi$ and $-\pi$, respectively, in the sense shown in the figure.  If $h$ were a smooth function, then we would have $\eta_1 = \eta_2$ and since $\pi$ and $-\pi$ are the same modulo $2\pi$, both paths would trace $\eta([0,\tau])$ in the reverse direction.  For the GFF, we think of $\eta_1$ (resp.\ $\eta_2$) as starting infinitesimally to the left (resp.\ right) of $\eta(\tau)$; due to the roughness of the field, $\eta_1$ and $\eta_2$ do not merge into (and in fact cannot hit) $\eta([0,\tau])$.  If $\kappa \in (2,4)$, then $\eta_1$ and $\eta_2$ can hit $\eta|_{(\tau,\infty)}$ and if $\kappa \in (0,2]$ then $\eta_1$ and $\eta_2$ do not hit $\eta|_{(\tau,\infty)}$.  If $\kappa \in (8/3,4)$, then $\eta_1$ can hit $\eta_2$ and if $\kappa \in (0,8/3]$ then $\eta_1$ cannot hit $\eta_2$.  This, in particular, explains why the yellow and red curves of Figures~\ref{fig::sle6_lightcone}--\ref{fig::sle64_lightcone} do not trace each other.}
\end{figure}

Theorem~\ref{thm::lightconeroughstatement} below is stated somewhat informally.  As mentioned earlier, precise statements will appear in Proposition~\ref{prop::light_cone_construction} and in Section~\ref{subsubsec::light_cone_general}.

\begin{theorem}
\label{thm::lightconeroughstatement}
Suppose that $h$ is a GFF on $\h$ with piecewise constant boundary data.  Let $\eta'$ be the counterflow line of $h$ starting at $\infty$ targeted at $0$.  Assume that the continuation threshold for $\eta'$ is almost surely not hit.  Then the range of $\eta'$ is almost surely equal to the set of points accessible by $\SLE_\kappa$ trajectories of $h$ starting at $0$ whose angles are restricted to be in $[-\tfrac{\pi}{2},\tfrac{\pi}{2}]$ but may change in time.  Let $\eta_L$ be the flow line of $h$ with angle $\tfrac{\pi}{2}$ starting at $0$ and $\eta_R$ the flow line of $h$ with angle $-\tfrac{\pi}{2}$.  It is almost surely the case that if $\eta'$ is nowhere boundary filling (i.e., $\eta' \cap \R$ has empty interior), then $\eta_L$ and $\eta_R$ do not hit the continuation threshold before reaching $\infty$ and are the left and right boundaries of $\eta'$.

A similar statement holds on the event that $\eta'$ is boundary filling on one or more segments of $\R$.  In this case, $\eta_L$ and $\eta_R$ hit their continuation thresholds before reaching $\infty$, but they can be extended to describe the entire left and right boundaries of $\eta'$ in the manner explained in Figure~\ref{fig::lightcone_boundary_filling}.
\end{theorem}

The light cone construction of $\SLE_{16/\kappa}$ processes described in the statement of Theorem~\ref{thm::lightconeroughstatement} includes what is known as \emph{Duplantier duality} or \emph{$\SLE$ duality} --- that the outer boundary of an $\SLE_{16/\kappa}$ process is equal in law to a kind of $\SLE_\kappa$ process.  This was proven in certain cases by Zhan \cite{ZHAN_DUALITY_1,ZHAN_DUALITY_2} and Dub\'edat \cite{DUB_DUAL}.  Theorem~\ref{thm::lightconeroughstatement} provides a more general version of this duality. It shows that the law of the right boundary of any $\SLE_{16/\kappa}(\ul{\rho}')$ process $\eta'$ from $\infty$ to $0$ in $\h$ is given by the flow line of angle $-\tfrac{\pi}{2}$ in the same imaginary geometry.  Analogously, the law of the left boundary of any $\SLE_{16/\kappa}(\ul{\rho}')$ process $\eta'$ is given by the flow line of angle $\tfrac{\pi}{2}$ in the same imaginary geometry.

We can also compute the conditional law of $\eta'$ given either $\eta_L$ or $\eta_R$.
These results are described in more detail in Section~\ref{subsubsec::light_cone_general}.  (One version of this statement also appears in \cite[Section~8]{DUB_PART}, where it is called ``strong duality''.)   We will also describe the law of $\eta'$ conditioned on the boundaries of the portions of $\eta'$ traced before and after $\eta'$ hits a given boundary point.  This result will be of particular interest to us in a subsequent work, in which we will prove the time reversal symmetry of $\SLE_{16/\kappa}$ processes when $\kappa \in (2,4)$ (so that $16/\kappa \in (4,8)$).

\begin{figure}[h!]
\begin{center}
\includegraphics[scale=0.85]{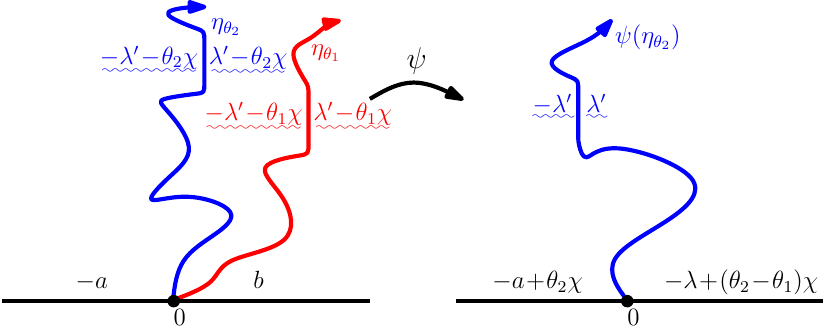}
\caption{\label{fig::monotonicity}  Suppose that $h$ is a GFF on $\h$ with the boundary data on the left panel.  For each $\theta \in \R$, let $\eta_\theta$ be the flow line of the GFF $h+\theta \chi$.  This corresponds to setting the angle of $\eta_\theta$ to be $\theta$.  Just as if $h$ were a smooth function, if $\theta_1 < \theta_2$ then $\eta_{\theta_1}$ lies to the right of $\eta_{\theta_2}$.  The conditional law of $h$ given $\eta_{\theta_1}$ and $\eta_{\theta_2}$ is a GFF on $\h \setminus \bigcup_{i=1}^2 \eta_{\theta_i}$ whose boundary data is shown above.  By applying a conformal mapping and using the transformation rule~\eqref{eqn::ac_eq_rel}, we can compute the conditional law of $\eta_{\theta_2}$ given the realization of $\eta_{\theta_1}$ and vice-versa.  That is, $\eta_{\theta_2}$ given $\eta_{\theta_1}$ is an $\SLE_\kappa((a-\theta_2 \chi)/\lambda -1; (\theta_2-\theta_1)\chi/\lambda-2)$ process independently in each of the connected components of $\h \setminus \eta_{\theta_1}$ which lie to the left of $\eta_{\theta_1}$.  Moreover, $\eta_{\theta_1}$ given $\eta_{\theta_2}$ is an $\SLE_\kappa((\theta_2-\theta_1) \chi/\lambda -2;(b+\theta_1\chi)/\lambda-1)$ independently in each of the connected components of $\h \setminus \eta_{\theta_2}$ which lie to the right of~$\eta_{\theta_2}$.  Versions of this result also hold for flow lines which start at different points as well as in the setting where the boundary data is piecewise constant (see Theorem~\ref{thm::monotonicity_crossing_merging}).}
\end{center}
\end{figure}

\begin{figure}[h!]
\begin{center}
\subfigure[If $\theta_1 < \theta_2$, then $\eta_{\theta_1}^{x_1}$ stays to the right of $\eta_{\theta_2}^{x_2}$.]{
\includegraphics[width=0.45\textwidth,height=0.4\textwidth,clip=true, trim = 1mm 1mm 1mm 1mm]{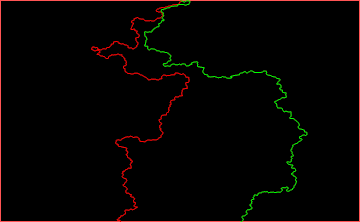}}
\hspace{0.02\textwidth}
\subfigure[If $\theta_1 = \theta_2$, then $\eta_{\theta_1}^{x_1}$ merges with $\eta_{\theta_2}^{x_2}$ upon intersecting.]{
\includegraphics[width=0.45\textwidth,height=0.4\textwidth,clip=true, trim = 1mm 1mm 1mm 1mm]{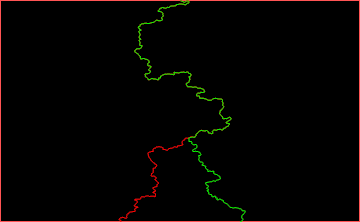}}
\subfigure[If $\theta_2 < \theta_1 < \theta_2+\pi$, then $\eta_{\theta_1}^{x_1}$ crosses $\eta_{\theta_2}^{x_2}$ upon interesting but does not cross back.]{
\includegraphics[width=0.45\textwidth,height=0.4\textwidth,clip=true, trim = 1mm 1mm 1mm 1mm]{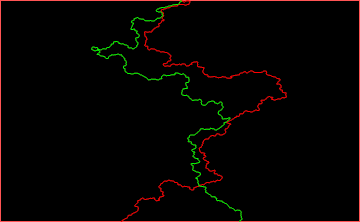}}
\end{center}
\caption{\label{fig::flow_line_interaction}  Numerical simulations which depict the three types of flow line interaction, as described in the statement of Theorem~\ref{thm::monotonicity_crossing_merging}.  In each of the simulations, we fixed $x_2 < x_1$ in $[-1-i,1-i]$, $\theta_1,\theta_2 \in \R$, and took $\eta_{\theta_1}^{x_1}$ (resp.\ $\eta_{\theta_2}^{x_2}$) to be the flow line of a projection of a GFF on $[-1,1]^2$ to the space of functions piecewise linear on the triangles of a $300 \times 300$ grid starting at $x_1$ (resp.\ $x_2$) with angle $\theta_1$ (resp.\ $\theta_2$).}
\end{figure}

The final result we wish to state concerns the interaction of imaginary rays with different angle and starting point.  In contrast with the case that $h$ is smooth, these rays may bounce off of each other and even merge, but they have the same monotonicity behavior in their starting point and angle as in the smooth case.  This result leads to a theoretical understanding of the phenomena simulated in Figures~\ref{fig::flowlines}--\ref{fig::flowlines4},~\ref{fig::grid}, and~\ref{fig::two_fans}.  The following statement is somewhat imprecise (as it does not describe all the constraints on boundary data that affect whether the distinct flow lines are certain to intersect before getting trapped at other boundary points) but a more detailed discussion appears in Section~\ref{sec::uniqueness}; see also Figure~\ref{fig::monotonicity} and Figure~\ref{fig::flow_line_interaction}.

\begin{theorem}
\label{thm::monotonicity_crossing_merging}
Suppose that $h$ is a GFF on $\h$ with piecewise constant boundary data.  For each $\theta \in \R$ and $x \in \partial \h$ we let $\eta_\theta^x$ be the flow line of $h$ starting at $x$ with angle $\theta$.  Fix $x_1,x_2 \in \partial \h$ with $x_1 \geq x_2$.
\begin{enumerate}[(i)]
\item If $\theta_1 < \theta_2$ then $\eta_{\theta_1}^{x_1}$ almost surely stays to the right of $\eta_{\theta_2}^{x_2}$.  If, in addition, $\theta_2-\theta_1 <  \pi \kappa/(4-\kappa)$, then $\eta_{\theta_1}^{x_1}$ and $\eta_{\theta_2}^{x_2}$ can bounce off of each other; otherwise the paths almost surely do not intersect (except possibly at their starting point).
\item If $\theta_1 = \theta_2$, then $\eta_{\theta_1}^{x_1}$ may intersect $\eta_{\theta_2}^{x_2}$ and, upon intersecting, the two curves merge and never separate.
\item Finally, if $\theta_2+\pi > \theta_1 > \theta_2$, then $\eta_{\theta_1}^{x_1}$ may intersect $\eta_{\theta_2}^{x_2}$ and, upon intersecting, crosses and then never crosses back.  If, in addition, $\theta_1-\theta_2 < \pi \kappa/(4-\kappa)$, then $\eta_{\theta_1}^{x_1}$ and $\eta_{\theta_2}^{x_2}$ can bounce off of each other; otherwise the paths almost surely do not subsequently intersect.
\end{enumerate}
\end{theorem}

The monotonicity component of Theorem~\ref{thm::monotonicity_crossing_merging} (i.e., the fact that $\eta_{\theta_1}^{x_1}$ almost surely stays to the right of $\eta_{\theta_2}^{x_2}$) will be first proved in settings where $\eta_{\theta_1}^x,\eta_{\theta_2}^x$ almost surely do not intersect $\partial \h$ after time $0$ (and have the same starting point) in Section~\ref{sec::non_boundary_intersecting}.  In Section~\ref{sec::uniqueness}, we will extend this result to the boundary intersecting regime and establish the merging and crossing statements.  We will also explain in Section~\ref{sec::interacting} and Section~\ref{sec::uniqueness} how in the setting of Theorem~\ref{thm::monotonicity_crossing_merging} one can compute the conditional law of $\eta_{\theta_1}^{x_1}$ given $\eta_{\theta_2}^{x_2}$ and vice-versa (see Figure~\ref{fig::monotonicity} for an important special case of this).

Note that the angle restriction $\theta_2 < \theta_1 < \theta_2+\pi$ is also the one that allows the Euclidean lines to cross (i.e., would allow for $\eta_{\theta_2}$ to cross from the left side of $\eta_{\theta_1}$ to the right side if $h$ were constant).  Although we will not explore this issue here, we remark that it is also interesting to consider what would happen if we took $\theta_1 \geq \theta_2 + \pi$.  It turns out that in this regime extra crossings can occur at points where both paths intersect $\R$, which is somewhat more complicated to describe.

\subsection{Outline}

The remainder of this article is structured as follows.  In Section~\ref{sec::sle}, we will prove the existence and uniqueness of solutions to the $\SLE_\kappa(\ul{\rho})$ equation~\eqref{eqn::sle_kappa_rho_sde}, even with force points starting at $0^-,0^+$.  We will also show that solutions to~\eqref{eqn::sle_kappa_rho_sde} are characterized by a certain martingale property.  Next, in Section~\ref{sec::gff}, we will review the construction and properties of the Gaussian free field which will be relevant for this work.  The notion of a ``local set,'' first introduced in \cite{SchrammShe10}, will be of particular importance to us.  We will also provide an independent proof of Theorem~\ref{thm::coupling_existence}.  In Section~\ref{sec::dubedat}, we will give a new presentation of Dub\'edat's proof of $\SLE$-duality --- that the outer boundary of an $\SLE_{16/\kappa}$ process is described by a certain $\SLE_\kappa$ process for $\kappa \in (0,4)$.  Following Dub\'edat, we explain how this result (and a slight generalization) implies Theorem~\ref{thm::coupling_uniqueness} for flow lines which are non-boundary intersecting.  The purpose of Section~\ref{sec::non_boundary_intersecting} is to establish the monotonicity of flow lines in their angle and to prove Theorem~\ref{thm::lightconeroughstatement} --- that the range of an $\SLE_{16/\kappa}$ trace can be realized as a light cone of points which are accessible by angle restricted $\SLE_\kappa$ trajectories, $\kappa \in (0,4)$ --- in a certain special case.  Then, in Section~\ref{sec::interacting}, we will prove a number of technical estimates which allow us to rule out pathological behavior in the conditional mean of the GFF when multiple flow and counterflow lines interact.  This will allow us to compute the conditional law of one path given several others.  Finally, in Section~\ref{sec::uniqueness} we will complete the proofs of our main theorems.

The general strategy in Sections~\ref{sec::dubedat}--\ref{sec::uniqueness} is the following:
\begin{enumerate}
\item We first show that non-boundary intersecting flow and counterflow lines are deterministic functions of the field and respect certain monotonicity properties (Section~\ref{sec::dubedat} and Section~\ref{sec::non_boundary_intersecting}).
\item We will then explain how to compute the conditional law of one non-boundary-intersecting path given several others in Section~\ref{sec::interacting}.  The conditional law will always be an $\SLE_\kappa(\rho)$ type process.  Even though the paths we consider in Sections~\ref{sec::dubedat}--\ref{sec::interacting} do not intersect the boundary, they can intersect each other. 
\item We will use this in Section~\ref{sec::uniqueness} to derive the corresponding statements (as well as continuity of the trajectories) for boundary intersecting paths from our results in the case that the paths are not boundary intersecting using conditioning arguments.
\end{enumerate}

\section{The Schramm-Loewner evolution}
\label{sec::sle}

\subsection{Overview of $\SLE_\kappa$}

$\SLE_\kappa$ is a one-parameter family of conformally invariant random curves, introduced by Oded Schramm in \cite{S0} as a candidate for (and later proved to be) the scaling limit of loop erased random walk \cite{LSW04} and the interfaces in critical percolation \cite{S01, CN06}.  Schramm's curves have been shown so far also to arise as the scaling limit of the macroscopic interfaces in several other models from statistical physics: \cite{S07,CS10U,SS05,SS09,MillerSLE}.  More detailed introductions to $\SLE$ can be found in many excellent survey articles of the subject, e.g., \cite{W03, LAW05}.

An $\SLE_\kappa$ in $\h$ from $0$ to $\infty$ is defined by the random family of conformal maps $g_t$ obtained by solving the Loewner ODE~\eqref{eqn::loewner_ode} with $W = \sqrt{\kappa} B$ and $B$ a standard Brownian motion.  Write $K_t := \{z \in \h: \tau(z) \leq t \}$ where $\tau(z) = \sup\{t \geq 0 : \im(g_t(z)) > 0\}$.  Then $g_t$ is the unique conformal map from $\h_t := \h \setminus K_t$ to $\h$ satisfying $\lim_{|z| \to \infty} |g_t(z) - z| = 0$.

Rohde and Schramm showed that there a.s.\ exists a curve $\eta$ (the so-called $\SLE$ \emph{trace}) such that for each $t \geq 0$ the domain $\h_t$ of $g_t$ is the unbounded connected component of $\h \setminus \eta([0,t])$, in which case the (necessarily simply connected and closed) set $K_t$ is called the ``filling'' of $\eta([0,t])$ \cite{RS05}.  An $\SLE_\kappa$ connecting boundary points $x$ and $y$ of an arbitrary simply connected Jordan domain can be constructed as the image of an $\SLE_\kappa$ on $\h$ under a conformal transformation $\varphi \colon \h \to D$ sending $0$ to $x$ and $\infty$ to $y$.  (The choice of $\varphi$ does not affect the law of this image path, since the law of $\SLE_\kappa$ on $\h$ is scale invariant.)

\subsection{Definition of SLE$_\kappa(\rho)$}
\label{subsubsec::sle_kappa_rho_sde}
The so-called $\SLE_\kappa(\ul{\rho})$ processes are an important variant of $\SLE_\kappa$ in which one keeps track of additional marked points.  Just as with regular $\SLE_\kappa$, one constructs $\SLE_\kappa(\ul{\rho})$ using the Loewner equation except that the driving function $W$ is replaced with a solution to the SDE~\eqref{eqn::sle_kappa_rho_sde}.  The purpose of this section is to construct solutions to~\eqref{eqn::sle_kappa_rho_sde} in a careful and canonical way.  We will not actually need to think about the Loewner evolution on the half plane for any of the discussion in this subsection.  It will be enough for now to think about the Loewner evolution restricted to the real line.

We first recall that the Bessel process of dimension $\delta > 0$, also written $\BES^\delta$, is in some sense a (non-negative) solution to the SDE
\begin{equation}
\label{eqn::bes_definition}
dX_t = dB_t + \frac{\delta-1}{2 X_t} dt,\quad X_0 \geq 0
\end{equation}
where $B$ is a standard Brownian motion.  A detailed construction of the Bessel processes appears, for example, in \cite[Chapter~XI]{RY04}. We review a few of the basic facts here.  When $\delta > 1$,~\eqref{eqn::bes_definition} holds in the sense that $X$ is a.s.\ {\em instantaneously reflecting} at $0$ (i.e., the set of times for which $X_t =0$ has Lebesgue measure zero) and a.s.\ satisfies
\begin{equation}
\label{eqn::integral_bes_definition}
X_t = X_0 + B_t + \int_0^t \frac{\delta-1}{2 X_s} ds,\quad X_0 \geq 0.
\end{equation}
In particular, assuming $\delta > 1$, the integral in~\eqref{eqn::integral_bes_definition} is finite a.s.\ so that $X_t$ is a semi-martingale.  The solution is a strong solution in the sense of \cite{RY04}, which means that~$X$ is adapted to the filtration generated by the Brownian motion~$B$.  The law of~$X$ is determined by the fact that it is a solution to~\eqref{eqn::bes_definition} away from times where $X_t = 0$, instantaneously reflecting where $X_t = 0$, and adapted to the filtration generated by $B$.

Regardless of $\delta$, standard SDE results imply that~\eqref{eqn::integral_bes_definition} has a unique solution up until the first time $t$ that $X_t = 0$.  When $\delta < 1$, however,~\eqref{eqn::integral_bes_definition} cannot hold beyond times at which $X_t = 0$ without a so-called principal value correction, because the integral in~\eqref{eqn::integral_bes_definition} is almost surely infinite beyond such times (see \cite[Section~3.1]{SHE_CLE} for additional discussion of this point).  Bessel processes can be defined for all time whenever $\delta > 0$ but they are not semi-martingales when $\delta \in (0,1)$.  For this paper, it turns out not to be necessary to consider settings that require a principal value correction.  We may always assume that either $\delta > 1$ or that $\delta \leq 1$ but we only consider the process up to the first time that $X$ reaches zero.

Fix a value $\rho > -2$ and write
\[ \delta = 1+ \frac{2(\rho+2)}{\kappa},\]
noting that $\delta > 1$.  Let $X$ be an instantaneously reflecting solution to~\eqref{eqn::bes_definition} for some $X_0 = x_0 \geq 0$.  We would like to define a pair $W$ and $V^R$ that solves the
SDE~\eqref{eqn::sle_kappa_rho_sde} with $W_0 = 0$ and some fixed initial value $V_0^R = x_0^R \geq 0$.  To motivate the definition, note that~\eqref{eqn::sle_kappa_rho_sde} formally implies that the difference $V^R - W$ solves the same SDE as $\sqrt \kappa X$, away from times where it is equal to zero.  Thus it is natural to write
\begin{equation}
\label{eqn::sle_bessel}
\begin{split}
V_t^R &= x_0^R + \int_0^t \frac{2}{\sqrt \kappa X_s} ds,\\
 W_t &= V_t^R - \sqrt{\kappa} X_t.
\end{split}
\end{equation}
The standard definition of (single-force-point) $\SLE_\kappa(\rho)$ is the Loewner evolution driven by the process $W$ defined in~\eqref{eqn::sle_bessel}.

Let us now extend the definition to the multiple-force-point setting.  Although the definition is straightforward, we have not found a construction of the law of multiple-force-point $\SLE$ in the literature that applies in the generality we consider here.  There are some minor technicalities that arise when solving the SDE~\eqref{eqn::sle_kappa_rho_sde} that do not seem to have been fully addressed previously.  One definition of $\SLE_\kappa(\ul{\rho})$ in the case of two force points, one left and one right, both starting at zero (constructed by continuously rescaling so that the force points stay at~$0$ and~$1$ and using a time change to reduce the problem to a one-dimensional diffusion) appeared in \cite{SS09}, and it was shown that when $\kappa = 4$ the process defined this way is a scaling limit of discrete GFF level lines with certain boundary conditions.  However, \cite{SS09} did not provide a general-$\kappa$ explanation of the sense in which the definition was canonical.  One could worry that subtle changes to the way that the process gets started, or the way the process behaves when force points collide with~$W_t$, could lead to different but equally valid definitions of~$\SLE_\kappa(\ul{\rho})$.

\begin{definition}
\label{def::slekrdef}
Let $B$ be a standard Brownian motion.  We will say that the continuous processes $W$ and $V^{i,q}$ describe an SLE$_\kappa(\ul{\rho})$ evolution corresponding to~$B$ (up to some stopping time) if $B$ is a Brownian motion with respect to the filtration $\CF_t = \sigma( B_s, W_s, V_s^{i,q} : s \leq t)$ and the following hold (up to that stopping time):
\begin{enumerate}
\item \label{lbl::awayfromcollision} For every stopping time $\tau$ for $(W,V^{i,q})$ which is almost surely a non-collision time for $W$ and the $V^{i,q}$, we have that the processes $W$, $V^{i,q}$, and $B$ satisfy~\eqref{eqn::sle_kappa_rho_sde} in the time interval $[\tau,\sigma]$ where $\sigma$ is the first time after $\tau$ that $W$ collides with one of the $V^{i,q}$.  Moreover, $(W,V^{i,q})$ in $[\tau,\sigma]$ is adapted to the filtration generated by $(W_\tau,V_\tau^{i,q})$ and $B|_{[\tau,\sigma]}$.
\item \label{lbl::instantaneousreflection} We have instantaneous reflection of $W$ off the $V^{i,q}$, i.e., it is almost surely the case that for Lebesgue almost all times $t$ we have $W_t \neq V_t^{i,q}$ for each $q$ and $i$.
\item \label{lbl::nolocalpushes} We also have almost surely that $V_t^{i,q} = x^{i,q} + \int_0^t \frac{2}{V_s^{i,q}- W_s}ds$ for each~$q$ and~$i$.
\end{enumerate}
\end{definition}

The three conditions are equivalent to the integral form of~\eqref{eqn::sle_kappa_rho_sde} (as explained just below), but it will be convenient to treat them separately.  The definition stated above is motivated by but does not make any reference to Loewner evolution.

Once we are given the first two conditions, Condition~\ref{lbl::nolocalpushes} rules out extraneous ``local time pushes'' that might be made to {\em both} $W$ and $V^{i,q}$ on the set of collision times.  Condition~\ref{lbl::nolocalpushes} actually implies Condition~\ref{lbl::instantaneousreflection} (since instantaneous reflection is required in order for the integral in Condition~\ref{lbl::nolocalpushes} to be defined).  We will use the term $\SLE_\kappa(\ul{\rho})$ to describe the Loewner evolution $(g_t)$ driven by $W$ or the corresponding trace (which we will eventually prove to be a continuous path almost surely).

We allow for the possibility that some of the $V^{i,L}$ may be equal to one another when $t=0$ or that they may merge into each other at some $t>0$ (and similarly for the $V^{i,R}$).  We define the {\bf continuation threshold} to be the infimum of the $t$ values for which either
\[ \sum_{i: V_t^{i,L} = W_t} \rho^{i,L} \leq -2 \quad\text{or}\quad \sum_{i: V_t^{i,R} = W_t} \rho^{i,R} \leq -2.\]
We will only construct $\SLE_\kappa(\ul{\rho})$ for $t$ below the continuation threshold.

We will now explain why Conditions~\ref{lbl::awayfromcollision}--\ref{lbl::nolocalpushes} from Definition~\ref{def::slekrdef} imply that the processes~$W_t$ and~$V_t^{i,q}$ satisfy the integral form of~\eqref{eqn::sle_kappa_rho_sde}.  

Fix $T, \wt{\epsilon} > 0$ and~$S \in (0,T)$ (non-random).  Let $S_{\wt{\epsilon}}$ be the first time $t$ after $S$ that both $V_t^{1,L} - W_t \leq -\wt{\epsilon}$ and $V_t^{1,R} - W_t \geq \wt{\epsilon}$ and let $T_{\wt{\epsilon}}$ be the minimum of $T$ and the first time after~$S_{\wt{\epsilon}}$ that either
\begin{enumerate}
\item there are at least two force points within distance~$\wt{\epsilon}$ of~$W$ or
\item $W$ is within distance $\wt{\epsilon}$ of a force point with weight less than or equal to $-2$.
\end{enumerate}
Note that $T_{\wt{\epsilon}}$ occurs before the continuation threshold is hit.  We are going to show that Conditions~\ref{lbl::awayfromcollision}--\ref{lbl::nolocalpushes} imply that $W_t$ and $V_t^{i,q}$ satisfy the integral form of~\eqref{eqn::sle_kappa_rho_sde} in the time interval $[S_{\wt{\epsilon}},T_{\wt{\epsilon}}]$.  Once we have shown this, it is then clear that $W_t$ and $V_t^{i,q}$ satisfy the integral form of~\eqref{eqn::sle_kappa_rho_sde} up until time~$T$ (or the continuation threshold is hit).  Indeed, by sending $\wt{\epsilon} \to 0$, we see that the integrated version of the equation is solved in the time interval from $S$ up until the first time after $S$ that there is a collision of force points (in which case the force points merge), the continuation threshold is hit, or time $T$ is reached.  The result thus follows by inducting on the number of force points and then taking a limit as $S \to 0$.

Fix $\epsilon \in (0,\wt{\epsilon})$.  Let $\sigma_1 = \inf\{t \geq S_{\wt{\epsilon}} : \min_{i,q} | W_t - V_t^{i,q}| = 0\}$ and let $i_1,q_1$ be such that $W_{\sigma_1} = V_{\sigma_1}^{i_1,q_1}$.  Let $\tau_1 = \inf\{t \geq \sigma_1 : |W_t - V_t^{i_1,q_1}| \geq \epsilon\}$ and note by the monotonicity of the force points (i.e., the $V_t^{i,L}$ are decreasing in $i$ and the $V_t^{i,R}$ are increasing in $i$) that $\min_{i,q} |W_{\tau_1} - V_{\tau_1}^{i,q}| = \epsilon > 0$.  Suppose that $\sigma_j$, $\tau_j$ have been defined for $1 \leq j \leq k$.  We then let $\sigma_{k+1} = \inf\{t \geq \tau_k : \min_{i,q} |W_t - V_t^{i,q}| = 0\}$ and let $i_{k+1},q_{k+1}$ be such that $W_{\sigma_{k+1}} = V_{\sigma_{k+1}}^{i_{k+1},q_{k+1}}$.  Let $\tau_{k+1} = \inf\{t \geq \sigma_{k+1} : |W_t - V_t^{i_{k+1},q_{k+1}}| \geq \epsilon\}$ and note by the monotonicity of the force points that $\min_{i,q} |W_{\tau_{k+1}} - V_{\tau_{k+1}}^{i,q}| = \epsilon > 0$.

Condition~\ref{lbl::awayfromcollision} implies that there exists a standard Brownian motion $B$ such that
\begin{equation}
\label{eqn::initial_sde}
\begin{split}
  &\sum_j (W_{\sigma_{j+1} \wedge T_{\wt{\epsilon}}} - W_{\tau_j \wedge T_{\wt{\epsilon}}} ) - \sum_{i,j,q} \int_{\tau_j \wedge T_{\wt{\epsilon}}}^{\sigma_{j+1} \wedge T_{\wt{\epsilon}}} \frac{\rho^{i,q}}{W_s - V_s^{i,q}} ds\\
= &\sum_j \sqrt{\kappa} (B_{\sigma_{j+1} \wedge T_{\wt{\epsilon}}} - B_{\tau_j \wedge T_{\wt{\epsilon}}}).
\end{split}
\end{equation}
Let $N_{\wt{\epsilon}} = \min\{j \geq 1 : \tau_j \geq T_{\wt{\epsilon}}\}$.  By the definition of the stopping times, we have that
\begin{align}
      &\sum_j |W_{\tau_j \wedge T_{\wt{\epsilon}}} - W_{\sigma_j \wedge T_{\wt{\epsilon}}}|
 \leq N_{\wt{\epsilon}} \epsilon + \sum_j |V_{\tau_j \wedge T_{\wt{\epsilon}}}^{i_j,q_j} - V_{\sigma_j \wedge T_{\wt{\epsilon}}}^{i_j,q_j}| \notag\\
 \leq& N_{\wt{\epsilon}} \epsilon + \sum_{i,j,q} |V_{\tau_j \wedge T_{\wt{\epsilon}}}^{i,q} - V_{\sigma_j \wedge T_{\wt{\epsilon}}}^{i,q}|. \label{eqn::w_v_change}
\end{align}
Condition~\ref{lbl::nolocalpushes} implies that the $V_t^{i,q}$ are absolutely continuous, hence the sum on the right hand side of~\eqref{eqn::w_v_change} almost surely tends to $0$ as $\epsilon \to 0$.

We turn to explain why $N_{\wt{\epsilon}} \epsilon \to 0$ almost surely as $\epsilon \to 0$ (at least along a positive sequence $(\epsilon_k)$ tending to $0$ sufficiently quickly).  As we will explain momentarily in more detail, this follows in the case that we have a single force point with weight $\rho > -2$ because of the tail for the amount of time it takes for a Bessel process of dimension $\delta > 1$ to exit $[0,\wt{\epsilon}]$ when starting from $\epsilon$.  We can reduce the case of many force points to this case in the following manner.  

We group the intervals $[\tau_j,\sigma_{j+1}]$ up until time $T_{\wt{\epsilon}}$ into two different types: those intervals in which
\begin{enumerate}
\item $W$ starts at distance $\epsilon$ of $V^{1,L}$ and
\item $W$ starts at distance $\epsilon$ of $V^{1,R}$. 
\end{enumerate}
(By relabeling the $V^{i,q}$ due to merging, we call $V^{1,L}$ and $V^{1,R}$ the rightmost and leftmost force point, respectively, which is to the left and right, respectively, of~$W$ after time~$S_{\wt{\epsilon}}$.)  Let $\CJ^L$ (resp.\ $\CJ^R$) consist of those $j$ of the first (resp.\ second) type.  For each $j$, we let $\xi_{j+1} = \inf\{t \geq \sigma_j : |W_t - V_t^{i_j,q_j}| \geq \wt{\epsilon}\}$ and let $\wt{\sigma}_{j+1} = \min(\sigma_{j+1},\xi_{j+1},T_{\wt{\epsilon}})$.  By the Girsanov theorem, the law of the sequence $\big( W_{t - \tau_j} - V_{t-\tau_j}^{1,L} : t \in [\tau_j,\wt{\sigma}_{j+1}] \big)$ for $j \in \CJ^L$ up until the first $j \in \CJ^L$ with $\tau_j \geq T_{\wt{\epsilon}}$ is absolutely continuous with respect to the corresponding sequence for single-force-point $\SLE_\kappa(\rho)$ with $\rho > -2$ restricted to the corresponding intervals of time (i.e., when the driving function starts from distance~$\epsilon$ of its force point and then is run until either hitting the force point or reaching distance at least $\wt{\epsilon}$ from the force point).

We are now going to explain why the Radon-Nikodym derivative $Z^L$ between these two sequences does not degenerate when we take a limit as $\epsilon \to 0$ (with $S,T,\wt{\epsilon}$ fixed).  This argument will likewise give that the same is true when we consider $j \in \CJ^R$ in place of $j \in \CJ^L$.  Let
\[ M^L = -\frac{1}{\sqrt{\kappa}} \sum_{j \in \CJ^L} \int_{\tau_j}^{\wt{\sigma}_{j+1}} \sum_{(i,q) \neq (1,L)} \frac{\rho^{i,q}}{W_s - V_s^{i,q}} d B_s \]
and also let
\[ \langle M^L \rangle = \frac{1}{\kappa} \sum_{j \in \CJ^L} \int_{\tau_j}^{\wt{\sigma}_{j+1}} \left( \sum_{(i,q) \neq (1,L)} \frac{\rho^{i,q}}{W_s - V_s^{i,q}} \right)^2 ds\]
be the quadratic variation of $M^L$.  The Girsanov theorem implies that the result of weighting the law of $(W,V^{i,q})$ by
\begin{equation}
\label{eqn::rn_bessel_form}
Z^L = \exp(M^L - \tfrac{1}{2} \langle M^L \rangle)
\end{equation}
(as a consequence of the deterministic upper bound on $\langle M^L \rangle$ that we will momentarily obtain, we will see that there are not any integrability issues with $Z^L$) is a process which evolves as a single-force-point $\SLE_\kappa(\rho)$ with $\rho > -2$ in each of the intervals $[\tau_j,\wt{\sigma}_{j+1}]$ for $j \in \CJ^L$ where the force point and driving function start at distance $\epsilon$ from each other and the evolution is stopped once they either collide or reach distance $\wt{\epsilon}$ from each other.  Let $\p_{\epsilon,\wt{\epsilon}}^*$ denote the resulting law.  That is, $d\p_{\epsilon,\wt{\epsilon}}^* / d\p = Z^L$.  Since in each of the intervals $[\tau_j,\wt{\sigma}_{j+1}]$, we know that the distance of $W$ to $V^{i,q}$ for $(i,q) \neq (1,L)$ is at least $\wt{\epsilon}$, it then follows that with $n$ given by the number of force points and $C = \kappa^{-1} \max_{i,q} (\rho^{i,q})^2$ we have the deterministic bound
\[ \langle M^L \rangle \leq C n^2 T \wt{\epsilon}^{-2}.\]
Using that $\exp(- M^L - \tfrac{1}{2} \langle M^L \rangle)$ has mean $1$ in the third step below (as a consequence of the bound on $\langle M^L \rangle$ given just above), we have that
\begin{align*}
   \E[ (Z^L)^{-1} ]
&\leq e^{C n^2 T \wt{\epsilon}^{-2}/2} \E[ \exp(-M^L)]
  \leq e^{C n^2 T \wt{\epsilon}^{-2}} \E[ \exp(-M^L - \tfrac{1}{2}\langle M^L \rangle)]
  = e^{C n^2 T \wt{\epsilon}^{-2}}.
\end{align*}
We thus have that
\[ \E_{\epsilon,\wt{\epsilon}}^*[ (Z^L)^{-2}] = \E[ (Z^L)^{-2} \cdot Z^L] = \E[ (Z^L)^{-1}] \leq e^{C n^2 T \wt{\epsilon}^{-2}}.\]
Then for any event $A$, we have that
\begin{equation}
\label{eqn::rn_prob_bound}
 \p[A] = \E_{\epsilon,\wt{\epsilon}}^*[\one_A (Z^L)^{-1}] \leq (\p^*[A])^{1/2} (\E_{\epsilon,\wt{\epsilon}}^*[ (Z^L)^{-2} ] )^{1/2} \leq e^{C n^2 T \wt{\epsilon}^{-2}/2} (\p_{\epsilon,\wt{\epsilon}}^*[A])^{1/2}.
\end{equation}

We claim that for each fixed $\zeta > 0$ we have that
\begin{equation}
\label{eqn::n_eps_eps_bound}
\limsup_{\epsilon \to 0} \p_{\epsilon,\wt{\epsilon}}^*[ N_{\wt{\epsilon}} \epsilon \geq \zeta] = 0.
\end{equation}
Indeed,~\eqref{eqn::n_eps_eps_bound} follows because of the tail of the amount of time it takes for a Bessel process of dimension $\delta \in (1,2)$ starting from $\epsilon$ to hit either $0$ or $\wt{\epsilon}$.  To make this precise, we suppose that $X$ is a Bessel process of dimension $\delta \in (1,2)$ starting from $\epsilon$.  Fix $\alpha > \epsilon$.  Let $\tau_0$ (resp.\ $\tau_\alpha$) be the first time that $X$ hits $0$ (resp.\ $\alpha$) and let $\tau = \tau_0 \wedge \tau_\alpha$ be the first time that $X$ exits $[0,\alpha]$.  Using that $X_t^{2-\delta}$ is a continuous local martingale, the optional stopping theorem implies that
\begin{equation}
\label{eqn::bessel_loc_mg}
\epsilon^{2-\delta}  = X_0^{2-\delta} = \E[ X_\tau^{2-\delta} ] = \alpha^{2-\delta} \p[ \tau_\alpha < \tau_0].
\end{equation}
Rearranging~\eqref{eqn::bessel_loc_mg} implies that
\begin{equation}
\label{eqn::bessel_hit_alpha}
\p[ \tau_\alpha < \tau_0 ] = \left( \frac{\epsilon}{\alpha} \right)^{2-\delta}.
\end{equation}
Applying~\eqref{eqn::bessel_hit_alpha} for $\alpha = \wt{\epsilon}/2$ and using that Bessel processes satisfy Brownian scaling, it follows that the probability that $X$ takes at least $\wt{\epsilon}^2$ time to exit $[0,\wt{\epsilon}]$ is at least an $\wt{\epsilon}$-dependent constant times $\epsilon^{2-\delta}$.  Since $2-\delta \in (0,1)$ so that the exponent of $\epsilon^{2-\delta}$ is in $(0,1)$, we see that~\eqref{eqn::n_eps_eps_bound} follows from Chebyshev's inequality as one can stochastically dominate from below the sum of the lengths of time required by the process starting from $\epsilon$ to exit $[0,\wt{\epsilon}]$ by a sum of independent random variables which take the value $\wt{\epsilon}^2$ with probability proportional to $\epsilon^{2-\delta}$ and $0$ otherwise.

Combining,~\eqref{eqn::rn_prob_bound} and~\eqref{eqn::n_eps_eps_bound} along with the Borel-Cantelli lemma implies that $N_{\wt{\epsilon}} \epsilon \to 0$ as $\epsilon \to 0$ almost surely (at least along a positive sequence $\epsilon_k$ tending to $0$ sufficiently quickly).

Since $N_{\wt{\epsilon}} \epsilon \to 0$ almost surely as $\epsilon \to 0$, by combining~\eqref{eqn::initial_sde} and~\eqref{eqn::w_v_change} we consequently have that
\begin{align*}
 &\sum_j (W_{\sigma_{j+1} \wedge T_{\wt{\epsilon}}} - W_{\tau_j \wedge T_{\wt{\epsilon}}}) - \sum_{i,j,q} \int_{\tau_j \wedge T_{\wt{\epsilon}}}^{\sigma_{j+1} \wedge T_{\wt{\epsilon}}} \frac{\rho^{i,q}}{W_s - V_s^{i,q}} ds\\
\to& W_{T_{\wt{\epsilon}}} - W_{S_{\wt{\epsilon}}} - \sum_{i,q} \int_{S_{\wt{\epsilon}}}^{T_{\wt{\epsilon}}} \frac{\rho^{i,q}}{W_s - V_s^{i,q}} ds \quad\text{as}\quad \epsilon \to 0.
\end{align*}
We also have that
\begin{align*}
&     \left( \sum_j (B_{\tau_j \wedge T_{\wt{\epsilon}}} - B_{\sigma_j \wedge T_{\wt{\epsilon}}}) \right)^2
= \sum_{i,j} (B_{\tau_j \wedge T_{\wt{\epsilon}}} - B_{\sigma_j \wedge T_{\wt{\epsilon}}})(B_{\tau_i \wedge T_{\wt{\epsilon}}} - B_{\sigma_i \wedge T_{\wt{\epsilon}}})\\ 
=& \sum_j (B_{\tau_j \wedge T_{\wt{\epsilon}}} - B_{\sigma_j \wedge T_{\wt{\epsilon}}})^2 + 2 \sum_{i < j} (B_{\tau_j \wedge T_{\wt{\epsilon}}} - B_{\sigma_j \wedge T_{\wt{\epsilon}}})(B_{\tau_i \wedge T_{\wt{\epsilon}}} - B_{\sigma_i \wedge T_{\wt{\epsilon}}}).
\end{align*}
The second summand above has zero expectation by the optional stopping theorem since the stopping times are bounded by $T$.  Since
\[ (B_{t \wedge T_{\wt{\epsilon}}} - B_{\sigma_j \wedge T_{\wt{\epsilon}}})^2 - (t \wedge T_{\wt{\epsilon}} - \sigma_j \wedge T_{\wt{\epsilon}}) \quad\text{for}\quad t \geq \sigma_j\]
is a martingale for each $j$, it follows from the optional stopping theorem that the expectation of the first summand above is equal to
\[ \E\left[ \sum_j \big(\tau_j \wedge T_{\wt{\epsilon}} - \sigma_j \wedge T_{\wt{\epsilon}}\big) \right].\]
Note that the quantity inside of the expectation is bounded from above by the minimum of $T$ and the amount of time that~$W$ spends within distance~$\epsilon$ of the~$V^{i,q}$.  Condition~\ref{lbl::instantaneousreflection} implies that this latter quantity tends to~$0$ almost surely as~$\epsilon \to 0$.  Therefore the expectation tends to zero as $\epsilon \to 0$ by the dominated convergence theorem (we may use the constant function $T$ as our dominating function).  Consequently, it follows that the sum of the changes to $B$ in the intervals $[\sigma_j \wedge T_{\wt{\epsilon}},\tau_j \wedge T_{\wt{\epsilon}}]$ for $1 \leq j \leq N_{\wt{\epsilon}}$ tends to $0$ in probability as $\epsilon \to 0$.  Combining, we have that
\[ W_{T_{\wt{\epsilon}}} - W_{S_{\wt{\epsilon}}} - \sum_{i,q} \int_{S_{\wt{\epsilon}}}^{T_{\wt{\epsilon}}} \frac{\rho^{i,q}}{W_s - V_s^{i,q}} ds = \sqrt{\kappa} (B_{T_{\wt{\epsilon}}} - B_{S_{\wt{\epsilon}}}) + o(1)\]
where the $o(1)$ term tends to $0$ in probability as $\epsilon \to 0$.  By passing along a subsequence as $\epsilon \to 0$, we have that the $o(1)$ term tends to $0$ almost surely, hence
\[ W_{T_{\wt{\epsilon}}} - W_{S_{\wt{\epsilon}}} - \sum_{i,q} \int_{S_{\wt{\epsilon}}}^{T_{\wt{\epsilon}}} \frac{\rho^{i,q}}{W_s - V_s^{i,q}} ds = \sqrt{\kappa} (B_{T_{\wt{\epsilon}}} - B_{S_{\wt{\epsilon}}}).\]
By sending $\wt{\epsilon} \to 0$ and then repeating the argument at the successive merging times of the force points, we thus have almost surely with $\tau$ given by the continuation threshold that
\begin{equation}
\label{eqn::w_t_w_s_diff}
W_{T \wedge \tau} - W_{S \wedge \tau} - \sum_{i,q} \int_{S \wedge \tau}^{T \wedge \tau} \frac{\rho^{i,q}}{W_s - V_s^{i,q}} ds = \sqrt{\kappa}(B_{T \wedge \tau} - B_{S \wedge \tau}).
\end{equation}
Since both sides of~\eqref{eqn::w_t_w_s_diff} are continuous, it follows that~\eqref{eqn::w_t_w_s_diff} holds almost surely for all $0 < S < T$.  It therefore follows that~\eqref{eqn::sle_kappa_rho_sde} is satisfied in integrated form.

\begin{theorem}
\label{thm::slekrdef}
Given the vector $\ul{\rho}$ and the initial values $V_0^{i,q}$, Definition~\ref{def::slekrdef} uniquely determines a joint law for $W_t$, $B_t$, and the $V_t^{i,q}$ --- each defined for all $t$ up to the continuation threshold.  Under this law, the values $W_t$, $B_t$, and $V_t^{i,q}$ taken together are a continuous multidimensional Markovian process indexed by $t$.
\end{theorem}
\begin{proof}
When there is only a single force point, Theorem~\ref{thm::slekrdef} follows from standard facts about Bessel processes (see Chapter XI of \cite{RY04}; recall also~\eqref{eqn::sle_bessel}) and the definition coincides with the standard definition of SLE$_\kappa(\ul{\rho})$.

If there are multiple force points but all of the $V_0^{i,q}$ are non-zero except
for one (without loss of generality, we may suppose that only $V_0^{1,R}$ is possibly zero) then one can obtain existence of a process with the properties above, defined up until the first time that one of the {\em other} force points collides with $W_t$, using a Girsanov transformation (see the discussion of Girsanov's Theorem, e.g., in \cite{KS98,RY04}) applied to the standard one-force-point SLE$_\kappa(\rho)$ that one would obtain if $V_0^{1,R}$ were the only force point.  Girsanov's theorem applies because the remaining force points introduce a smooth drift to the Brownian motion, and the new process obtained is absolutely continuous with respect to the one-force-point process (as long as one stops at a bounded stopping time that occurs before $W_t$ gets within some fixed constant distance of one the other force points).

One can also reverse this procedure (starting with a process defined for multiple force points and applying Girsanov's theorem to produce the process corresponding to one force point).  If there were multiple possibilities for the joint laws of the $W_t$, $B_t$, and $V_t^{i,q}$ in the multiple-force-point case, then this would produce multiple possibilities for the joint laws in the single-force-point case, contradicting what we have already established.

This gives us existence and uniqueness of the law up until the first time that one of the {\em other} force points (besides $V_0^{1,R}$) hits $W_t$.  When this happens, one can use this other force point in place of $V_0^{1,R}$ (or if this other point is on the right, it will have merged with $V_0^{1,R}$ and one can subsequently treat the two force points together) and continue until a force point other than this new one is hit.  Iterating this process uniquely defines the law all the way up to the continuation threshold.  (To check this formally, one has to rule out the possibility that infinitely many of these iterations may occur in a finite period of time.  Since there are only finitely many force points, the number of times at which two right force points merge, or two left points merge, is finite.  Thus, one needs only to check that it takes an infinite amount of time almost surely for $W_t$ to alternate between hitting a left force point and hitting a right force point infinitely often, which is a simple exercise, given that the $V_t^{i,L}$ are decreasing in time while the $V_t^{i,R}$ are increasing.)

The only remaining case to treat is the possibility that there are two force points immediately to the left and right of the origin at time zero.  Let us first consider the case that these are the only two force points.  We then need to construct a triple of processes $V^{1,L}_t \leq W_t \leq V^{1,R}_t$ starting at zero.  The hypotheses imply that in any such construction, we cannot have equality of all three processes at any positive time.  Thus, if we know the processes up to any positive time, then the results above imply that the law of the continuation is uniquely determined thereafter.  In a sense, the problem is figuring out how to ``get the process started''.
Since both force points start at the origin, we will be able to use scale invariance to help us deduce existence and uniqueness of the law.

As an alternative warm-up problem, suppose we start the process off at time zero with $V^{1,L}_0 = -1$ and $W_0 = 0$ and $V^{1,R}_0 = 1$.   The previous discussion yields existence and uniqueness of the law in this case.  Let $R$ be the set of values $r$ for which there exists a $t$ such that $e^r = |V^{1,L}_t -  W_t| = |W_t - V^{1,R}_t|$.  Then $R$ is a subset of $[0,\infty)$ that contains $0$.  By scale invariance and the Markovian property, $R$ has a certain renewal property: namely, for each fixed $a$, we have that conditioned on $a' = \inf \{ R \cap [a, \infty) \}$, the conditional law of $R \cap [a', \infty)$ is the same as the original law of $R$ {\em translated} by $a'$ units to the right.

Moreover, we claim that $R$ possesses an additional expectation-boundedness property: namely, that the expectation of $|a' - a|$ is bounded independently of $a$.  In fact, we claim a stronger result: namely, given {\em any} choices for $V^{1,L}_t$ and $V^{1,R}_t$ and $W_t$ at a fixed starting time $t$, the expected value of $\log \frac{|V^{1,L}_\tau - V^{1,R}_\tau|}{|V^{1,L}_t - V^{1,R}_t|}$ (i.e., the amount that the $\log$ distance between the force points changes between times $t$ and $\tau$), where $\tau$ is the smallest value greater than $t$ satisfying $|V^{1,L}_\tau -  W_\tau| = |W_\tau -  V^{1,R}_\tau|$, is at most some fixed constant.  This follows from the fact that, no matter where the force points begin at some fixed starting time, there is a uniformly positive probability that $W_t$ will be exactly between those two force points before the distance between them doubles.  It is enough to show this for the worst case in which $W_t$ starts out equal to one of the two force points, and this follows from absolute continuity with respect to Bessel processes.

We remark that renewal property described above is enjoyed by other random sets familiar to the reader: e.g., the zero set of a Brownian motion or more generally the zero set of a Bessel process with dimension in $(0,2)$.  However, these random sets do not enjoy the expectation-boundedness property described above.   We further recall the well known fact that each of these random sets can be written as the range of an increasing stable L\'evy process, where the L\'evy jump measure $\nu$ is an infinite measure on $(0,\infty)$ whose density function is a power law.  It is not hard to deduce from the renewal property that the random set $R$ is also the range of an increasing stable L\'evy process, albeit with a different (not necessarily power law) measure $\nu$.

Given a fixed $a>0$ and the largest value $a''$ in $(-\infty, a] \cap R$, the law of $a' - a''$ is given by the measure $\nu$ restricted to the interval $(a - a', \infty)$, and normalized to be a probability measure.  The finite expectation argument above implies that $\int_b^\infty r d \nu(r) < \infty$ for any $b > 0$.  Now a natural way to construct the $\SLE_\kappa(\ul{\rho})$ process is to take a very negative value $r$ and start the process with $V^{1,L}_0 = -e^r$, $W_0 = 0$, and $V^{1,R}_0 = e^r$.  Taking the limit as $r \to -\infty$, the law of the corresponding sets $R$ (and of the entire triple of processes $V^{1,L}_t$, $W_t$, $V^{1,R}_t$) converges to a limit w.r.t.\ the Hausdorff topology on compact subsets of $\R$.  One can show this by considering two very small values $r < r'$, generating corresponding sets $R$ and $R'$, and then taking $r''$ to be the smallest value which lies in one of the sets $R$ and $R'$ and is of distance at most $\delta$ from the other set.  One can discover this point via a sort of ``leapfrog'' exploration.  Namely, one first explores the points in $R$ (following the L\'evy process) until the first time one discovers a point larger than (or within $\delta$ of) $r'$.  One then observes the points in $R'$ until discovering a point larger than (or within $\delta$ of) the set of discovered points of $R$, and so forth.  After discovering the point $r''$, one can then couple $R$ and $R'$ to be translations of each other (by an amount less than $\delta$) ever after; by scale invariance, one can take the corresponding $V^{1,L}_t$, $W_t$, $V^{1,R}_t$ processes (after the corresponding small times) to be rescalings of each other by a factor close to $1$.  If we fix some $K > 0$, then Hausdorff-metric compactness implies the existence of subsequential limits of the laws of $R \cap [-K,K]$ and $R' \cap [-K,K]$ exist.  The argument above shows that any such limits can be coupled in such a way that they agree (up to Hausdorff distance~$\delta$) with probability arbitrarily close to one; since $\delta$ is arbitrary, this implies that there must be a unique limit.  Since $K$ can also be arbitrary, we obtain both the existence of a limiting random set on $\R$ and the fact that there is a unique process satisfying the hypotheses of Definition~\ref{def::slekrdef}.

The extension from two force points (both at the origin) to many force points (two at the origin) is the same Girsanov argument given above.
\end{proof}

\begin{remark}
\label{rem::continuity_non_boundary} 
Suppose that $\eta$ is an $\SLE_\kappa(\ul{\rho}^L;\ul{\rho}^R)$ process where $\sum_{i=1}^j \rho^{i,q} \geq \tfrac{\kappa}{2}-2$ for all $1 \leq j \leq |\ul{\rho}^q|$ and $q \in \{L,R\}$.  Assume further that $x^{1,L} < 0 < x^{1,R}$.  Then $\eta$ is almost surely a continuous curve because its law is mutually absolutely continuous with respect to the law of an $\SLE_\kappa$ process (with no force points) up to every fixed time~$t$.  The reason is that, in this case, a comparison with Bessel processes implies that $W_t \neq V_t^{i,q}$ for all $1 \leq i \leq |\ul{\rho}^q|$ and $q \in \{L,R\}$, so one can compute the Radon-Nikodym derivative explicitly using Girsanov's theorem.  Moreover, $\eta$ is almost surely continuous even if $x^{1,L} = 0^-$ and $x^{1,R} = 0^+$, the reason being that we can apply the same Girsanov argument to $\eta|_{[s,t]}$ for every $0 < s < t$.  We will use this fact repeatedly, often without reference, throughout the article.
\end{remark}

\subsection{Martingale characterization of SLE$_\kappa(\rho)$}

The $\SLE_\kappa(\ul{\rho})$ processes are singled out by the following martingale characterization, which we will use repeatedly.  A version of this result for $\SLE_\kappa$ ($\ul{\rho} \equiv 0$) appears in \cite[Section~7.2]{DUB_PART}.  The argument that we present here is similar to the one in \cite{DUB_PART}.

If we are given any process $W_t$ with $W_0=0$ we can define the Loewner evolution $g_t$.  If we are also given a set of points $x^{i,L} \leq 0$ and $x^{i,R} \geq 0$ then we can define processes $V^{i,q}_t$ such that if $x^{i,q}$ has not yet been absorbed by the Loewner hull $K_t$ then $V^{i,q}_t = g_t(x^{i,q})$, and otherwise $x^{i,L}$ (resp., $x^{i,R}$) is the $g_t$ image of the left (resp., right) endpoint of $\R \cap K_t$.

\begin{theorem}
\label{thm::martingale}
Suppose we are given a random continuous curve $\eta$ on $\overline \h$ from $0$ to $\infty$ whose Loewner driving function $W_t$ is almost surely continuous.  Suppose that $x^{i,q}$ and $\rho^{i,q}$ values are given and that the $V_t^{i,q}$ are defined to be the images of the $x^{i,q}$ under the corresponding Loewner evolution (as described just above).  Let $\Fh_t$ be the corresponding harmonic function in the statement of Theorem~\ref{thm::coupling_existence}.  Then $W_t$ and the $V_t^{i,q}$ can be coupled with a standard Brownian motion $B_t$ to describe an $\SLE_\kappa(\ul{\rho})$ process (up to the continuation threshold) if and only if $\Fh_t(z)$ evolves as a continuous local martingale in $t$ for each fixed $z \in \h$ until the time $z$ is absorbed by $K_t$.
\end{theorem}

The assumption that $\eta$ has a continuous Loewner driving function implies that $\eta$ is non-self-tracing and does not trace $\partial \h$ and that $\eta$ does not enter into the bounded components that it draws.  See Proposition~\ref{prop::cont_driving_function}.

\begin{proof}[Proof of Theorem~\ref{thm::martingale}]
That $\Fh_t$ evolves as a continuous local martingale if $W_t$ and $V^{i,q}_t$ correspond to an $\SLE_\kappa(\ul{\rho})$ can be seen by applying {\Ito}'s formula.  Thus, we need only prove the reverse implication.  We will assume that $\Fh_t$ is a continuous local martingale for each $z \in \h$ and verify the conditions of Definition~\ref{def::slekrdef} one at a time:

{\bf Proof of Condition~\ref{lbl::awayfromcollision}:}
Let $\wt{\Fh}_t$ be the harmonic conjugate of $\Fh_t$.  This is only defined a priori up to additive constant, but since the harmonic conjugate of $\arg(z)$ is $\log|z|$, we can fix the additive constant by writing $\Fh_t^*(z) := -\wt{\Fh}_t(z) + i \Fh_t(z)$ as follows:
\begin{equation}
\label{eqn::harmonic_form}
\begin{split}
&\ \ \ \ -\sum_{i=0}^k \rho^{i,L} \log(f_t(z) - f_t(x^{i,L})) - \log(f_t(z)) + \\
&\sum_{i=0}^\ell \rho^{i,R} \big(i\pi-\log(f_t(z) - f_t(x^{i,R})\big)+
 \big(i\pi - \log(f_t(z)) \big) - \frac{\pi \chi}{\lambda} \log f_t'(z).
 \end{split}
\end{equation}
One can show that $\wt{\Fh}_t(z) - \wt{\Fh}_t(y)$ is a local martingale for any fixed $y$ and $z$ by using the fact that this quantity is a linear function of $\Fh_t$ (representable as the integral of $\Fh_t$ times a test function) and applying Fubini's theorem.  Taking one of these points to infinity, we find that in fact $\wt{\Fh}_t$ and hence $\Fh_t^*$ is a local martingale.

Observe that $\Fh_t(z) = \im(\Fh_t^*(z))$ evolves as a continuous semi-martingale in the intervals in which the $f_t(x^{i,q})$ are not colliding with $W_t$.  Indeed, note that in the expression~\eqref{eqn::harmonic_form} above $f_t(z) - f_t(x^{i,L}) = g_t(z) - g_t(x^{i,L})$ and $f_t'(z) = g_t'(z)$ are both differentiable in~$t$.  Thus, the terms of the form $\log(f_t(z) - f_t(x^{i,L}))$ in~\eqref{eqn::harmonic_form} are semi-martingales (and likewise when $L$ is replaced by $R$ and for $\log f_t'(z)$).  Note also that $W_t$ appears only in the term $-2\log(f_t(z)) = -2\log(g_t(z) - W_t)$.  Since the other terms are semi-martingales, this term is a semi-martingale, as is its exponential, which implies that $W_t$ is a semi-martingale.

Write
\[ W_t = v_t + m_t \]
where, in the non-collision intervals, $v_t$ (resp.\ $m_t$) evolves as a process of bounded variation (resp.\ continuous local martingale).  We will next show that $m_t$ evolves as $\sqrt{\kappa}$ times a Brownian motion in the non-collision intervals by proving $d\langle m \rangle_t = \kappa dt$ for such~$t$ and invoking the L\'evy characterization of Brownian motion.  To see this, we compute the {\Ito} derivative of~\eqref{eqn::harmonic_form}.  Observe
\[ \frac{d}{dt} \log f_t'(z) =
   \frac{1}{f_t'(z)} \left(\frac{2}{f_t(z)} \right)' =
   -\frac{2}{f_t^2(z)}.\]
Consequently, the drift of the {\Ito} derivative of~\eqref{eqn::harmonic_form} takes the form
\begin{align*}
&\ \ \ \ \ \ \sum_{i=0}^k \frac{2\rho^{i,L} dt}{f_t(z) f_t(x^{i,L})} + \sum_{j=0}^\ell \frac{2 \rho^{i,R} dt}{f_t(z)f_t(x^{i,R})}\\
+&\frac{2\chi}{\lambda f_t^2(z)} dt -  \frac{4}{f_t^2(z)}dt + \frac{2}{f_t(z)} dv_t + \frac{1}{f_t^2(z)} d\langle m \rangle_t.
\end{align*}
This has to vanish since $\Fh_t^*$ is a local martingale.  Thus if we multiply through by $f_t^2(z)$ and evaluate at two different points (or simply consider points for which $f_t(z)$ is extremely close to zero), we see that we must have
\[ \left(\frac{2 \chi}{\lambda} - 4 \right) dt + d\langle m \rangle_t = 0.\]
This implies $d\langle m \rangle_t = \kappa dt$, as desired.  Inserting this back into the formula for the drift and solving for $v_t$ shows that $v_t$ takes on the desired form.

What we have shown so far implies that if $\tau$ is any stopping time for the driving process $(W,V^{i,q})$ which is almost surely not a collision time then we have that
\begin{equation}
\label{eqn::bm_non_collide}
W_{\tau + t} - W_\tau - \int_\tau^{\tau + t} \sum_{i,q} \frac{\rho^{i,q}}{W_s - V_s^{i,q}} ds = \sqrt{\kappa} B_t^{(\tau)}
\end{equation}
where $B^{(\tau)}$ is a standard Brownian motion, at least up until the first time $t \geq 0$ such that $t+\tau$ is a collision time of $(W,V^{i,q})$.  We will now argue that there is in fact a single Brownian motion $B$ such that 
\begin{equation}
\label{eqn::bm_stopping}
 W_{\tau + t} - W_\tau - \int_\tau^{\tau + t} \sum_{i,q} \frac{\rho^{i,q}}{W_s - V_s^{i,q}} ds = \sqrt{\kappa}(B_{\tau + t} - B_\tau)
\end{equation}
for all such stopping times $\tau$ up until the first $t$ so that $\tau + t$ is a collision time of $(W,V^{i,q})$.  This will complete the proof of Condition~\ref{lbl::awayfromcollision}.  To see this, we fix $\epsilon > 0$ and inductively define stopping times $\sigma_j,\tau_j$ as follows.  We let $\sigma_0$ be the infimum of times~$t$ that the distance between $W$ and the $V^{i,q}$ is at least $\epsilon$ and let $\tau_0$ be the first time after $\sigma_0$ that $W$ collides with one of the $V^{i,q}$.  Assuming that $\sigma_j,\tau_j$ have been defined for $0 \leq j \leq n$, we let $\sigma_{n+1}$ be the first time after $\tau_n$ that the distance between $W$ and the $V^{i,q}$ is at least $\epsilon$ and let $\tau_{n+1}$ be the first time after $\sigma_{n+1}$ that $W$ collides with one of the $V^{i,q}$.  We can define a Brownian motion $B^\epsilon$ by defining $B_{t+\sigma_j}^\epsilon - B_{\sigma_j}^\epsilon$ as in~\eqref{eqn::bm_non_collide} for $t \in [0,\tau_j - \sigma_j]$ for each $j$ and sampling the evolution of $B^\epsilon$ as a standard Brownian motion independently of everything else in the intervals of the form $[\tau_j,\sigma_{j+1}]$.  Note that the joint law of $B^\epsilon$ and $(W,V^{i,q})$ restricted to any compact time interval $[0,T]$ is tight as $\epsilon \to 0$ as the marginal laws of $B^\epsilon$ and $(W,V^{i,q})$ do not change with $\epsilon$.  Consequently, there exists a sequence $(\epsilon_k)$ of positive numbers decreasing to $0$ such that the joint law of $B^{\epsilon_k}$ and $(W,V^{i,q})$ converges weakly to a limit.  This gives us a coupling of a standard Brownian motion $B$ with $(W,V^{i,q})$.

We will now show that~\eqref{eqn::bm_stopping} holds (up until the first $t$ such that $\tau+t$ is a collision time) for the coupling of $B$ with $(W,V^{i,q})$.  In what follows, it will be helpful to introduce some extra notation.  For each $\epsilon > 0$, we will let $(W^\epsilon,V^{i,q,\epsilon})$ and $B^\epsilon$ have the same joint law as $(W,V^{i,q})$ and $B^\epsilon$ introduced just above.  We suppose that $\tau$ is a fixed stopping time as above and write $\tau^\epsilon$ for the corresponding stopping time for $(W^\epsilon,V^{i,q,\epsilon})$.  By compactness, there exists a subsequence $(\epsilon_{j_k})$ of $(\epsilon_k)$ such that the joint law of $(W^{\epsilon_{j_k}},V^{i,q,\epsilon_{j_k}})$, $\tau^{\epsilon_{j_k}}$, and $B^{\epsilon_{j_k}}$ converges as $k \to \infty$ to a triple $(W,V^{i,q})$, $\wt{\tau}$, and $B$.  Note that the joint law of $(W,V^{i,q})$ and $\wt{\tau}$ is the same as the joint law of $(W,V^{i,q})$ and $\tau$ because it is the same as the joint law of $(W^{\epsilon_{j_k}}, V^{i,q,\epsilon_{j_k}})$ and $\tau^{\epsilon_{j_k}}$ for all $k$.  This implies that $\wt{\tau}  = \tau$ is determined by $(W,V^{i,q})$ (as $\tau$ is a stopping time for $(W,V^{i,q})$).  Consequently, the limiting joint law of $(W,V^{i,q})$, $\wt{\tau} = \tau$, and $B$ does not depend on the choice of subsequence $(\epsilon_{j_k})$ of $(\epsilon_k)$ and therefore we have the weak convergence of the joint law of $(W^{\epsilon_k},V^{i,q,\epsilon_k})$, $\tau^{\epsilon_k}$, and $B^{\epsilon_k}$ to the joint law of $(W,V^{i,q})$, $\tau$, and $B$.  By the Skorohod representation theorem, we may couple the $(W^{\epsilon_k},V^{i,q,\epsilon_k})$, $\tau^{\epsilon_k}$, $B^{\epsilon_k}$ and $(W,V^{i,q})$, $\tau$, $B$ onto a common probability space so that the convergence is almost sure.

By the definition of $B^{\epsilon_k}$, we observe that~\eqref{eqn::bm_stopping} holds for $(W^{\epsilon_k},V^{i,q,\epsilon_k})$, $\tau^{\epsilon_k}$, and $B^{\epsilon_k}$ (up until the first $t$ such that $\tau^{\epsilon_k}+t$ is a collision time) on the event that the distance between $W_{\tau^{\epsilon_k}}^{\epsilon_k}$ and the $V_{\tau^{\epsilon_k}}^{i,q,\epsilon_k}$ is at least $\epsilon_k$.  Using that the joint law of $(W^{\epsilon_k}, V^{i,q,\epsilon_k})$ and $\tau^{\epsilon_k}$ does not depend on $k$, we have that the probability of this event tends to $1$ as $k \to \infty$.  We will now deduce from this that~\eqref{eqn::bm_stopping} holds for $(W,V^{i,q})$, $\tau$, and $B$ (up until the first $t$ such that $\tau+t$ is a collision time).  For each $t \geq 0$, letting
\begin{align*}
     \Delta_t^{\epsilon_k} &= W_{\tau^{\epsilon_k}+t}^{\epsilon_k} - W_{\tau^{\epsilon_k}}^{\epsilon_k} - \int_{\tau^{\epsilon_k}}^{\tau^{\epsilon_k}+t} \sum_{i,q} \frac{\rho^{i,q}}{W_s^{\epsilon_k} - V_s^{i,q,\epsilon_k}} ds - \sqrt{\kappa}(B_{\tau^{\epsilon_k}+t}^{\epsilon_k} - B_{\tau^{\epsilon_k}}^{\epsilon_k}) \quad\text{and}\\
     \Delta_t &= W_{\tau+t} - W_{\tau} - \int_{\tau}^{\tau+t} \sum_{i,q} \frac{\rho^{i,q}}{W_s - V_s^{i,q}} ds - \sqrt{\kappa}(B_{\tau+t} - B_{\tau}),
\end{align*}
and, for each $\delta,T > 0$, $\tau_{\delta,T}$ be the minimum of $\tau+T$ and the first time $t$ after $\tau$ that the distance between $W_t$ and the $V_t^{i,q}$ is at most $\delta$,
we will show this by arguing that
\begin{equation}
\label{eqn::delta_to_zero}
   \sup_{t \in [0,\tau_{\delta,T}-\tau]} |\Delta_t^{\epsilon_k} - \Delta_t| \to 0 \quad\text{as}\quad k \to \infty.
\end{equation}

Note that on the event that the distance between $W_\tau$ and the $V_\tau^{i,q}$ is equal to or smaller to $\delta$, we have $\tau = \tau_{\delta,T}$.  We define $\tau_{\delta,T}^{\epsilon_k}$ similarly with $(W^{\epsilon_k}, V^{i,q,\epsilon_k})$ in place of $(W,V^{i,q})$.  By the argument explained just above (with $\tau_{\delta,T}^{\epsilon_k}$ and $\tau_{\delta,T}$ in place of $\tau^{\epsilon_k}$ and $\tau$) and recoupling the laws if necessary using the Skorohod representation theorem, we almost surely have that $\tau_{\delta,T}^{\epsilon_k} \to \tau_{\delta,T}$ as $k \to \infty$.  The almost sure local uniform convergence of $W^{\epsilon_k},V^{i,q,\epsilon_k},B^{\epsilon_k}$ to $W,V^{i,q},B$, the convergence of $\tau^{\epsilon_k}$ to $\tau$, and $\tau_{\delta,T}^{\epsilon_k}$ to $\tau_{\delta,T}$ as $k \to \infty$ implies we have both
\begin{align*}
    \sup_{t \in [0,\tau_{\delta,T}-\tau]} | ( W_{\tau^{\epsilon_k}+t}^{\epsilon_k} - W_{\tau^{\epsilon_k}}^{\epsilon_k} ) - (W_{\tau+t} - W_\tau)| \to 0 \quad&\text{as}\quad k \to \infty\\
    \sup_{t \in [0,\tau_{\delta,T}-\tau]} | ( B_{\tau^{\epsilon_k}+t}^{\epsilon_k} - B_{\tau^{\epsilon_k}}^{\epsilon_k} ) - (B_{\tau+t} - B_\tau)| \to 0 \quad&\text{as}\quad k \to \infty
\end{align*}
Thus to finish proving~\eqref{eqn::delta_to_zero}, we need to prove the uniform convergence of the integral in $\Delta_t^{\epsilon_k}$ to the integral in $\Delta_t$.  Let $I_s^{\epsilon_k}$ (resp.\ $I_s$) denote the integrand in the integral in the definition of $\Delta_t^{\epsilon_k}$ (resp.\ $\Delta_t$).  Then we have that
\begin{align}
\label{eqn::integrals_converge}
   &\left|\int_{\tau^{\epsilon_k}}^{\tau^{\epsilon_k}+t} I_s^{\epsilon_k} ds - 
 \int_{\tau}^{\tau+t} I_s ds\right|
\leq \int_{\tau^{\epsilon_k} \wedge \tau}^{\tau^{\epsilon_k} \vee \tau} | I_s^{\epsilon_k}| + | I_{s+t}^{\epsilon_k}| ds +  \int_\tau^{\tau+t} | I_s^{\epsilon_k} - I_s| ds.
\end{align}
The convergence of $\tau^{\epsilon_k}$ to $\tau$ and the uniform convergence of $W^{\epsilon},V^{i,q,\epsilon_k}, B^{\epsilon_k}$ to $W,V^{i,q},B$ in $[0,\tau_{\delta,T}]$ implies that $I_s^{\epsilon_k}$ converges uniformly to $I_s$ in $[\tau,\tau_{\delta,T}]$ as $k \to \infty$.  Therefore second integral on the right hand side of~\eqref{eqn::integrals_converge} tends to $0$ as $k \to \infty$.  It is similarly not difficult to see that the first integral on the right side of~\eqref{eqn::integrals_converge} tends to $0$ as $k \to \infty$.

We have shown that~\eqref{eqn::bm_stopping} holds for $(W,V^{i,q})$, $\tau$, and $B$ (up until $\tau_{\delta,T} - \tau$).  Since $\delta,T > 0$ were arbitrary, we therefore have that~\eqref{eqn::bm_stopping} holds for $(W,V^{i,q})$, $\tau$, and $B$ (up until the first $t$ that $\tau+t$ is a collision time).

We note that at this point we have constructed a coupling of a standard Brownian motion $B$ with $(W,V^{i,q})$ so that Condition~\ref{lbl::awayfromcollision} holds.  We will now check that $B$ is a Brownian motion with respect to the filtration $\CF_t = \sigma( W_s, V_s^{i,q}, B_s : s \leq t)$.  To prove this, it suffices to show that $B|_{[\tau,\infty)} - B_\tau$ is a Brownian motion independently of $(W_\tau,V_\tau^{i,q})$ where $\tau$ is a stopping time as above.  In order to justify this, we will first argue that the conditional law of $B^{\epsilon}|_{[\tau^{\epsilon},\infty)} - B_{\tau^{\epsilon}}^{\epsilon}$ given $(W_{\tau^{\epsilon}}^{\epsilon},V_{\tau^{\epsilon}}^{i,q,\epsilon})$ is that of a standard Brownian motion.  Let $\sigma_j^\epsilon,\tau_j^\epsilon$ be as defined above for $(W^\epsilon,V^{i,q,\epsilon})$.  By definition, we can write
\begin{align}
\label{eqn::b_eps_rep}
   B_t^\epsilon = M_t^{1,\epsilon} + M_t^{2,\epsilon} :=  \sum_j \left( \wt{B}_{t \wedge \tau_j^\epsilon}^\epsilon - \wt{B}_{t \wedge \sigma_j^\epsilon}^\epsilon \right) + \sum_j \left( \wh{B}_{t \wedge \sigma_{j+1}^\epsilon} - \wh{B}_{t \wedge \tau_j^\epsilon}^\epsilon \right)
\end{align}
where $\wt{B}_t^\epsilon$ for $t \in [\sigma_j^\epsilon,\tau_j^\epsilon]$ is determined from $(W^\epsilon,V^{i,q,\epsilon})$ as in~\eqref{eqn::bm_non_collide} and $\wh{B}^\epsilon$ is a standard Brownian motion which is independent of $(W^\epsilon,V^{i,q,\epsilon})$.  From the representation~\eqref{eqn::b_eps_rep}, it is easy to see that $M_{\tau^\epsilon+t}^{1,\epsilon} - M_{\tau^\epsilon+t}^{1,\epsilon}$ and $M_{\tau^\epsilon+t}^{2,\epsilon} - M_{\tau^\epsilon}^{2,\epsilon}$ are both continuous martingales with respect to the filtration generated by $(W_s^\epsilon,V_s^{i,q,\epsilon})$ and $\wh{B}_s^\epsilon$ for $s \leq \tau^\epsilon+t$.  Consequently, $B_{\tau^\epsilon+t}^\epsilon - B_{\tau^\epsilon}^\epsilon$ is also a continuous martingale with respect to the same filtration.  Moreover, it is easy to see from~\eqref{eqn::b_eps_rep} that $\langle B_{\tau^\epsilon + \cdot}^\epsilon - B_{\tau^\epsilon}^\epsilon \rangle_t = t$ for all $t \geq 0$.  Therefore the L\'evy characterization of Brownian motion implies that $B_{[\tau^\epsilon,\infty)}^\epsilon - B_{\tau^\epsilon}^\epsilon$ has the law of a standard Brownian motion conditionally on $(W_{\tau^\epsilon}^\epsilon,V_{\tau^\epsilon}^{i,q,\epsilon})$.

Fix $T > 0$ and a bounded function $f$ on the Cartesian product of the space of continuous functions defined on the interval $[0,T]$ equipped with the uniform topology and $\R^n$ where $n$ is the number of elements of the vector $(W_{\tau^{\epsilon_k}}^{\epsilon_k},V_{\tau^{\epsilon_k}}^{i,q,\epsilon_k})$.  This implies that
\[ \lim_{k \to \infty} \E[ f( B^{\epsilon_k}(\tau^{\epsilon_k}+\cdot)|_{[0,T]} - B_{\tau^{\epsilon_k}}^{\epsilon_k} , W_{\tau^{\epsilon_k}}^{\epsilon_k},V_{\tau^{\epsilon_k}}^{i,q,\epsilon_k})] = \E[ f(\wt{B}, W_\tau,V_\tau^{i,q})]\]
where $\wt{B}$ is a standard Brownian motion on $[0,T]$ which is independent of $W_\tau,V_\tau^{i,q}$.  Since the joint law of $B^{\epsilon_k}|_{[\tau^{\epsilon_k},\tau^{\epsilon_k}+T]}$ and $(W_{\tau^{\epsilon_k}}^{\epsilon_k},V_{\tau^{\epsilon_k}}^{i,q,\epsilon_k})$ converges to the joint law of $B|_{[\tau,\tau+T]}$ and $W_\tau,V_\tau^{i,q}$ as $k \to \infty$ (as explained just above), we have that
\[ \E[ f(\wt{B}, W_\tau,V_\tau^{i,q})] = \E[ f(B(\tau+\cdot)|_{[0,T]} - B_\tau, W_\tau,V_\tau^{i,q})].\]
The claim thus follows since $f$ was an arbitrary bounded, continuous function.

{\bf Proof of Condition~\ref{lbl::instantaneousreflection}:}
To obtain instantaneous reflection, note that the set of times $t$ at which $W_t$ is equal to a force point is a subset of the set of times at which $\eta(t) \in \partial \h = \R$.  It turns out that this set must have Lebesgue measure zero for any continuous path $\eta$ with a continuous driving function.  This fact is stated and proved as Lemma~\ref{lem::zerocapacityonboundary} below.

{\bf Proof of Condition~\ref{lbl::nolocalpushes}:}  We know that $V_t^{1,L}$ is a non-increasing process which, by Condition~\ref{lbl::awayfromcollision}, evolves according to the Loewner flow driven by $W_t$ in those intervals of time in which $V_t^{1,L} \neq W_t$.  In particular, this implies that if we have any finite collection of disjoint open intervals $(a_j,b_j)$ for $1 \leq j \leq k$ such that $V_s^{1,L} \neq W_s$ for all $s \in (a_j,b_j)$ and $1 \leq j \leq k$ then for $t \geq b_k$ we have that
\begin{align*}
        V_t^{1,L}
 \leq \sum_{j=1}^k \big( V_{b_j}^{1,L} - V_{a_j}^{1,L} \big)
 = \sum_{j=1}^k \int_{a_j}^{b_j} \frac{2}{V_s^{1,L} - W_s} ds.
\end{align*}
Since this inequality holds for any finite collection of disjoint open intervals contained in $\{s \in [0,t] : W_s \neq V_s^{1,L}\}$ and $\{ s \in [0,t] : W_s = V_s^{1,L}\}$ has Lebesgue measure zero, by the monotone convergence theorem we have that
\begin{align*}
     V_t^{1,L} &\leq \int_0^t \frac{2}{V_s^{1,L} - W_s} ds.
\intertext{In particular, $(V_s^{1,L} - W_s)^{-1} \leq 0$ is integrable on bounded intervals of time in $[0,\infty)$.  The same argument implies that}
V_t^{1,L} - V_u^{1,L} &\leq \int_u^t \frac{2}{V_s^{1,L} - W_s} ds
\intertext{for all $0 \leq u < t$.  Therefore}
I_t &= \int_0^t \frac{2}{V_s^{1,L}-W_s}ds - V^{1,L}_t
\end{align*}
is non-decreasing.

We claim that $I_t$ is almost surely zero.  To prove this, consider a force point $V^{1,L}_t$ with which $W_t$ can collide.  We define an interval of time $(s_1,t_1)$ such that $s_1$ is the first time at which $V^{1,L}_t = W_t$ and $t_1$ is the first subsequent time at which $|V^{1,L}_t - W_t| = \epsilon$.  Inductively, we define $s_k$ to be the first time after $t_{k-1}$ at which $V^{1,L}_t = W_t$ and then take $t_k$ to be the first subsequent time at which $|V^{1,L}_t - W_t| = \epsilon$.  We consider how much the quantities $V^{1,L}_t$, $I_t$ and $\Fh_t(z)$ change during the intervals $(s_k,t_k)$ with $t_k \leq T$ where $T$ is a stopping time that a.s.\ occurs before $W_t$ gets within some fixed distance of $V_t^{1,L}$.  Fix some $z_0 \in \h$ and further assume that $T$ is a.s.\ bounded by some fixed constant, that $T$ a.s.\ occurs before $\Im(g_t(z_0))$ gets below some fixed positive value and also before $\Fh_t(z_0)$ a.s.\ changes by at most some fixed constant amount.

The sum of the changes to $\int_0^t 2(V^{1,L}_s - W_s)^{-1}ds$ during the intervals $(s_j,t_j)$ tends to zero as $\epsilon \to 0$ (simply because the integral is finite a.s.\ and the combined Lebesgue measure of the intervals tends to zero with $\epsilon$).  Thus the overall sum of changes to $V^{1,L}_t$ during these times tends to the $I_t$ change as $\epsilon \to 0$.  We will now argue that it is almost surely the case that the sum of the changes $V^{1,L}_t - W_t$ makes during the intervals $(s_j,t_j)$ tends to zero as $\epsilon \to 0$.  To see this, we note that the change in each interval is equal to $\epsilon$ by definition.  Thus controlling the total change is equivalent to controlling the number of such intervals before time $T$.  This, in turn follows, from the same argument that we used to show that $N_{\wt{\epsilon}} \epsilon \to 0$ as $\epsilon \to 0$ in the proof that Conditions~\ref{lbl::awayfromcollision}--\ref{lbl::nolocalpushes} imply that the integrated version of~\eqref{eqn::sle_kappa_rho_sde} given above (which we emphasize only uses the form of the evolution of the process at times when it is not colliding with a force point).  Thus, the overall sum of the changes to $W_t$ during these intervals must also tend to the $I_t$ change as $\epsilon \to 0$.  Since the Lebesgue measure of the union of the intervals tends to zero as $\epsilon \to 0$, it follows from Loewner evolution that the sum of the changes made to any force point $V^{i,q}$ other than $V^{1,L}_t$ tends to zero.

Recalling the definition of $\Fh_t$ (in terms of $W_t$ and the $V^{i,q}$) in Theorem~\ref{thm::coupling_existence} and the fact that $\Im g_t(z_0)$ is bounded below (recall that we fixed $z_0$ above), we find that if $\epsilon$ is sufficiently close to zero, the sum of the net changes to $\Fh_t(z_0)$ during the intervals is between constant non-zero-same-sign multiples of the $I_t$ change (due to the corresponding changes to the pair $W_t$ and $V^{1,L}_t$ during these intervals --- the effect from changes to other force points becomes negligible as $\epsilon \to 0$).  Thus to prove that the amount $I_t$ changes up until time $T$ is zero, it suffices to prove that the sum of the net changes made to $\Fh_t(z_0)$ during these intervals tends to zero as $\epsilon \to 0$.
 
To this end, note that the expected size of the total change of $\Fh_t(z_0)$ during these intervals is zero since $\Fh_t(z_0)$ is a local martingale that is bounded if stopped at time $T$ (hence a martingale if stopped at time $T$).  

We claim that Condition~\ref{lbl::awayfromcollision} and Condition~\ref{lbl::instantaneousreflection} together imply that the quadratic variation of $\Fh_t(z_0)$ that occurs during these intervals tends to zero as $\epsilon \to 0$.  To see this, we recall that in the non-collision intervals $(t_j,s_{j+1})$ we have that the evolution of $W_t - V_t^{1,L}$ is absolutely continuous with respect to the law of $\sqrt{\kappa}$ times a Bessel process with dimension in $(1,2)$ where in each such interval the Bessel process starts from $\epsilon$ and is then run until it first hits $0$.  In particular, we can weight the law of $(W,V^{i,q})$ by a Radon-Nikodym derivative $Z^\epsilon$ which has the same form as in~\eqref{eqn::rn_bessel_form} so that, under the weighted law, in the intervals of time of the form $(t_j,s_{j+1})$ we have that $W_t - V_t^{1,L}$ evolves as $\sqrt{\kappa}$ times a Bessel process up until the stopping time $T$ defined above.  That is, under the weighted law, we can construct a coupling between $W_t - V_t^{1,L}$ and $\sqrt{\kappa}$ times a Bessel process $X$ so that the two processes agree in the intervals of time in which they are making excursions from $\epsilon$ back to $0$ up until the stopping time $T$ defined just above.  Since we are considering $(W,V^{i,q})$ up until time $T$, just as in the case of~\eqref{eqn::rn_bessel_form} in our proof that Conditions~\ref{lbl::awayfromcollision}--\ref{lbl::nolocalpushes} imply that~\eqref{eqn::sle_kappa_rho_sde} is satisfied in integrated form, we have that this Radon-Nikodym derivative $Z^\epsilon$ has finite moments of all (positive and negative) orders, each of which is bounded uniformly in $\epsilon$ (recall the arguments just after~\eqref{eqn::rn_bessel_form}).  

For each $\epsilon > 0$, we let $(W^\epsilon,V^{i,q,\epsilon})$, $Z^\epsilon$ be an instance of the processes described just above and let $T^\epsilon$ be the stopping which corresponds to $T$.  Then there exists a sequence $(\epsilon_k)$ of positive numbers decreasing to $0$ so that the joint law of $(W^{\epsilon_k},V^{i,q,\epsilon_k})$, $Z^{\epsilon_k}$, $T^{\epsilon_k}$ converges weakly (with respect to the uniform topology on compact intervals) to a limit $(W,V^{i,q})$, $Z$, $T$ (where the joint law of $(W,V^{i,q})$ and $T$ is the same as just above; this follows from the argument given in the proof of Condition~\ref{lbl::awayfromcollision}).  By the construction, the law of the non-collision intervals of $W - V_t^{1,L}$ under the law weighted by $Z$ agree in law with those of $\sqrt{\kappa}$ times a Bessel process, up until the stopping time $T$.  That is, under the weighted law, we can construct a coupling of $W - V_t^{1,L}$ with $\sqrt{\kappa}$ times a Bessel process so that the excursions that each makes from $0$ are the same, up time time $T$.  Combining this with Condition~\ref{lbl::instantaneousreflection} and \cite[Proposition~3.3]{SHE_CLE}\footnote{\label{fn::prop3.3}We note that \cite[Proposition~3.3]{SHE_CLE} states that if $X$ is a continuous process which is adapted to the filtration of a Brownian motion $B$, solves the Bessel SDE of a fixed dimension $\delta \in (0,2)$ when it is not hitting $0$, and is instantaneously reflecting at $0$, then $X$ has the law of a Bessel process of dimension $\delta$.  This result holds under more generality, namely one may assume that for every stopping time $\tau$ for $X$ such that $X_\tau \neq 0$ almost surely with $\sigma = \inf\{t \geq \tau : X_t = 0\}$ we have that $X|_{[\tau,\sigma]}$ is adapted to the filtration $\CF_t = \sigma(X_\tau, B_s : s \leq t)$ and that $B|_{[\tau,\infty)}$ is a Brownian motion with respect to $\CF_t$ for $t \geq \tau$.  To see this, we note that under these hypotheses we have that for each $\epsilon > 0$ the law of the ordered collection of excursions that $X$ makes from $\epsilon$ to $0$ has the same law as the corresponding collection for a Bessel process $\wt{X}$ as the Bessel SDE has a unique strong solution.  Therefore we can couple the laws of $X$ and $\wt{X}$ together so that these excursions are equal.  By sending $\epsilon \to 0$, we get an asymptotic coupling between the laws of $X$ and $\wt{X}$ so that the ordered collection of excursions that $X$ makes from $0$ is equal to the corresponding collection for $\wt{X}$.  Fix $t > 0$ and let $A_t$ (resp.\ $\wt{A}_t$) be the amount of quadratic variation accumulated by $X$ in $\{ s \leq t : X_s \neq 0\}$ (resp.\ $\wt{X}$ in $\{s \leq t : \wt{X}_s \neq 0\}$).  Then $A_t$ (resp.\ $\wt{A}_t$) is equal to the amount of time that $X$ (resp.\ $\wt{X}$) has spent in $\{s : X_s \neq 0\}$ (resp.\ $\{ s : \wt{X}_s \neq 0\}$) by time $t$.  Since both $X$ and $\wt{X}$ are instantaneously reflecting at $0$, it follows that $A_t = \wt{A}_t = t$ almost surely.  On the other hand, by the construction of our coupling between $X$ and $\wt{X}$ we almost surely have for any fixed $t > 0$ that $X_{A_t} = \wt{X}_{\wt{A}_t}$.  Combining, this implies that $X_t = \wt{X}_t$ almost surely for each fixed $t > 0$.  Thus as both $X$ and $\wt{X}$ are continuous, we have that $X_t = \wt{X}_t$ for all $t > 0$ almost surely, as desired.}, this implies that $W_t - V_t^{1,L}$ is absolutely continuous with respect to $\sqrt{\kappa}$ times a Bessel process even in the intervals in which it is hitting~$0$.  Combining this with the argument used to prove Condition~1, we have that $(W,V^{i,q})$ is a semimartingale for all times and the quadratic variation of $W$ is given by $\kappa t$ for all $t$.  Consequently, we have that $d\langle \Fh_t(z_0) \rangle = \kappa \im(1/f_t(z_0))^2 dt$ for all $t$, which proves the claim.

Let~$T_\epsilon$ denote absolute value of the total (cumulative) change to $\Fh_t(z_0)$ that occurs during these intervals.  On the event that the quadratic variation is less than some value~$\delta$, the probability that $T_\epsilon$ exceeds a constant $a>0$ is bounded by the probability that the absolute value of a Brownian motion run for time~$\delta$ exceeds~$a$.

Now, one can use the Borel-Cantelli lemma to show  that we can take a sequence~$(\epsilon_k)$ of~$\epsilon$ values decreasing fast enough so that the event that the quadratic variation corresponding to~$\epsilon_k$ a.s.\ exceeds $2^{-k}$ for at most finitely many~$k$.  On this event, one can use the probability bound above and another application of Borel-Cantelli to show that~$T_{\epsilon_k}$ a.s.\ exceeds~$a$ for at most finitely many values of $k$.  Since this is true for any $a$, we conclude that a.s. $\lim_{k \to \infty} T_{\epsilon_k} = 0$. Hence the sum of the changes to $I_t$ during these intervals tends to zero, and since $I_t$ only changes when $W_t - V^{1,L}_t = 0$ we find that $I_t$ is almost surely constant up until time $T$.  Repeating this procedure iteratively (choosing new $z_0$ values as necessary) allows us to conclude that $I_t$ is almost surely constant for all time.

We have now shown that $(W,V^{i,q})$ satisfies Conditions~\ref{lbl::awayfromcollision}--\ref{lbl::nolocalpushes} in Definition~\ref{def::slekrdef}, which completes the proof.
\end{proof}

\begin{lemma}
\label{lem::zerocapacityonboundary}
Suppose that~$\eta$ is a continuous (non-random) curve on $\overline \h$ from~$0$ to~$\infty$ with a continuous Loewner driving function~$W_t$.  Then the set $\{t : \eta(t) \in \R \}$ has Lebesgue measure zero.
\end{lemma}
\begin{proof}
First we recall a few basic facts about the half-plane capacity (which we will denote by $\hcap$).  Let $A$ and $B$ and $C$ be bounded hulls, i.e., closed subsets of $\ol{\h}$ whose complements are simply connected.  (We interpret all bounded sets $X$ written below as hulls by including in $X$ the set of points that it disconnects from infinity and then taking the closure.)  We claim that the following hold:
\begin{enumerate}
\item $\hcap(A \cup B) \leq \hcap (A) + \hcap(B)$.
\item $\hcap(A \cup B \cup C) - \hcap (A \cup B) \leq \hcap (A \cup C) - \hcap (A)$.
\end{enumerate}
The first is seen by recalling one of the definitions of half-plane capacity: $$\hcap(A) := \lim_{s \to \infty} s \E[\Im B^{is}_\tau],$$ where $B^{is}$ is a Brownian motion started at $is$ and $\tau$ is the first time it hits $A \cup \R$.  For the second one, let $g_A: \h \setminus A \to \h$ be the conformal map normalized so that $\lim_{z \to \infty}|g_A(z) - z| = 0$.  Recalling the additivity of capacity under compositions of normalizing maps, the second claim above is equivalent to
\[\hcap\bigl(g_A( B \cup C)\bigr) - \hcap \bigl(g_A (B)\bigr) \leq \hcap \bigl(g_A (C) \bigr),\]
which follows by applying the first claim to $g_A(B)$ and $g_A(C)$.

Now, to prove the lemma, it suffices to show that for each $R$, the set of capacity times at which the tip of the path lies in $[-R,R]$ has Lebesgue measure zero.  The capacity of the closure $\ol{S}_\delta$ of the rectangle $S_\delta := \{ x+iy : -R < x < R, 0 < y < \delta \}$ tends to zero as $\delta$ tends to zero.  By the continuity of $\eta$, the set $\eta^{-1}(S_\delta)$ is a countable collection of disjoint intervals.  Let $I = \cup_j (s_j,t_j)$ be the union of those (disjoint) intervals in $\eta^{-1}(S_\delta)$ during which $\eta$ hits $\R$ at some point.  Now we claim that the total half-plane capacity time elapsed during $I$ is at most the capacity of $\ol{S}_\delta$.  Clearly, $\hcap(\eta([0,t] \cap \ol{S}_\delta))$ is non-decreasing and bounded by $\hcap(\ol{S}_\delta)$.  The change to $\hcap(\eta(I \cap [0,t]))$ during an interval $(s_j,t_j)$ is greater than or equal to the change to $\hcap(\eta([0,t]))$ during that time.  This follows from the second property above if we take $B = \eta([0,s_j])$ and $A = \eta(I \cap [0,s_j])$ and $C = \eta([s_j,t_j])$.  (We note that $A \cup \R$ is connected because $\eta((s_j,t_j)) \cup \R$ is connected for each $j$ as $\eta$ hits $\R$ at a time in $(s_j,t_j)$.  For the same reason, $C \cup \R$ is connected.)  Summing the changes over the $(s_j,t_j)$, we find that the total change to $\hcap(\eta([0,t]))$ during such intervals is at most the capacity of $\ol{S}_\delta$.  By taking $\delta \to 0$, we find that indeed the set of capacity times at which the tip of the path lies in $[-R,R]$ has Lebesgue measure zero.
\end{proof}

Consider an $\SLE_\kappa(\ul{\rho})$ process $(W,V^{i,q})$, let $(g_t)$ be the Loewner flow driven by $W$, and let $(K_t)$ be the corresponding growing family of hulls.  Then we can write $g_t(K_t) = [L_t,R_t]$.  It will be important for us in this article to know that $L$ and $R$ both solve the Loewner equation driven by $W$ (up until the continuation threshold is hit).  In particular, once either $L$ or $R$ have collided with one of the $V^{i,q}$, the two processes will continue to agree.  That this holds in the case of single-force-point $\SLE_\kappa(\rho)$ with $\rho > -2$ is remarked just before the statement of \cite[Lemma~8.3]{CONF_RES_CHORDAL}, however a proof of this fact is not given in \cite{CONF_RES_CHORDAL}.  We will now explain why this is true.  We first note that by absolute continuity (i.e., the Girsanov theorem), it suffices to explain why the result holds for single-force-point $\SLE_\kappa(\rho)$ with $\rho > -2$.  Indeed, this follows because $W$ can only interact with one force point at a time except at the times when force points merge and the set of such times is finite (at each merging time, the number of force points decreases by at least $1$).  The result in the setting of single-force-point $\SLE_\kappa(\rho)$ with $\rho > -2$ will follow from two observations.

Before making these two observations, we will first explain why the result is true for ordinary $\SLE$ (so that $\rho=0$).  Let $\eta$ be the $\SLE$ trace.  In this case, we know that $R_t$ evolves according to the Loewner flow in those time intervals in which $\eta$ is not hitting $\partial \h$.  This implies that $R_t - W_t$ evolves as $\sqrt{\kappa}$ times a Bessel process of dimension $1+\tfrac{4}{\kappa}$ in those time intervals in which it is not hitting $0$.  Moreover, by Lemma~\ref{lem::zerocapacityonboundary} we have that $R_t-W_t$ is instantaneously reflecting at $0$.  These two properties together imply that $R_t - W_t = \sqrt{\kappa} X_t$ for all $t \geq 0$ where $X_t$ is a Bessel process of dimension $1+\tfrac{4}{\kappa}$.  This proves the claim for ordinary $\SLE$.

Now we are going to generalize the result to the setting of $\SLE_\kappa(\rho)$ for $\rho > -2$.  Let $(W,V^{1,R})$ be the driving process.  First, by using absolute continuity to compare to the case of ordinary $\SLE_\kappa$ we have that $R_t$ evolves according to the Loewner flow up until time $\tau$, the first time $t$ that $R_t$ and $V_t^{1,R}$ collide.  We now claim that $R_t \geq V_t^{1,R}$ for all $t \geq \tau$.  To see this, we note that both $R_t$ and $V_t^{1,R}$ evolve according to the Loewner equation when they are not colliding with $W_t$.  The claim follows because we know that $V_t^{1,R}$ does not get any extraneous pushes when it is colliding with $W_t$ (since $V_t^{1,R} - W_t$ evolves as $\sqrt{\kappa}$ times a Bessel process of dimension $\delta > 1$), $R_t$ is monotone increasing (so extraneous pushes during collision times can only push~$R_t$ further to the right), and both $V_t^{1,R}$ and $R_t$ evolve according to the Loewner flow in those intervals when they are not colliding.  In particular, if there is an interval in which $R_t < V_t^{1,R}$ which starts after time $\tau$, then we get a contradiction because then $d(R_t - V_t^{1,R}) > 0$ in this interval.  Since $\int_0^t 2(V_s^{1,R}-W_s)^{-1} ds$ exists, that~$R_t \geq V_t^{1,R}$ after time $\tau$ implies that $\int_0^t 2(R_s - W_s)^{-1} ds$ exists.  We also have that $d(R_t - V_t^{1,R}) < 0$ whenever $R_t > V_t^{1,R}$ as both processes evolve according to the Loewner flow driven by $W$ at such times.  Therefore there cannot be such intervals, hence $V_t^{1,R} \geq R_t$ for $t \geq \tau$, and therefore $V_t^{1,R} = R_t$ for all $t \geq \tau$.  Thus since $V_t^{1,R}$ solves the Loewner equation driven by $W$ for $t \geq \tau$, so does $R_t$.

Although we will not use it in this paper, we remark (and sketch a proof) that it is possible to give an alternate martingale characterization in which one only requires $\Fh_t(z)$ to be a local martingale for a single point $z$, but one requires a particular form for its quadratic variation:

\begin{theorem} 
\label{thm::martingale2}
As in Theorem~\ref{thm::martingale}, suppose we are given a random continuous $\eta$ on $\overline \h$ from $0 \to \infty$ whose Loewner driving function $W_t$ is almost surely continuous.  Suppose further that $W_t$ describes the evolution of a random continuous path, and $V_t^{i,q}$ the image of force points under the corresponding Loewner evolution, and that for some fixed $z \in \h$ the correspondingly defined $\Fh_t(z)$ is given by $\sqrt \kappa$ times a Brownian motion when time is parameterized by minus the log of the conformal radius.  Then these processes describe an SLE$_\kappa(\ul{\rho})$ evolution at least up to the first time that $z$ is swallowed by the path.
\end{theorem}

\begin{proof}
The proof is similar to the proof of Theorem~\ref{thm::martingale}. Condition
\ref{lbl::awayfromcollision} follows from a straightforward calculation (see, e.g., \cite{SchrammShe10,SHE_WELD} for more explanation of the log conformal radius point of view) and the other two follow from the arguments in the proof of Theorem~\ref{thm::martingale}.
\end{proof}

\section{The Gaussian free field}
\label{sec::gff}

\subsection{Construction and basic properties}
\label{subsec::gff_construction}

We will now describe the construction of the two-dimensional GFF as well as some properties that will be important for us later.  The reader is referred to \cite{SHE06} for a more detailed introduction.  Let $D \subseteq \C$ be open with harmonically non-trivial boundary.  By this, we mean that a Brownian motion started at $z \in D$ almost surely hits $\partial D$.  Let $C_0^\infty(D)$ denote the set of $C^\infty$ functions compactly supported in $D$.

We let $H(D)$ be the Hilbert-space closure of $C_0^\infty(D)$ equipped with the {\em Dirichlet inner product}:
\[ (f,g)_\nabla = \frac{1}{2\pi} \int_D \nabla  f(x) \cdot \nabla g(x) dx \quad\text{for}\quad f,g \in C_0^\infty(D).\]
The GFF $h$ on $D$ can be expressed as a random linear combination of a $(\cdot,\cdot)_\nabla$-orthonormal basis $(\phi_n)$ of $H(D)$
\[ h = \sum_n \alpha_n \phi_n,\quad (\alpha_n) \quad\text{i.i.d.}\quad N(0,1).\]
Although this expansion of $h$ does not converge in $H(D)$, it does converge almost surely in the space of distributions or (when $D$ is bounded) in the fractional Sobolev space $H^{-\epsilon}(D)$ for each $\epsilon > 0$ (see \cite[Proposition~2.7]{SHE06} and the discussion thereafter).  Let $(\cdot,\cdot)$ denote the standard $L^2(D)$ inner product.  If $f,g \in C_0^\infty(D)$ then an integration by parts gives $(f,g)_\nabla = -(2\pi)^{-1} ( f,\Delta g)$.  Using this, we define
\begin{equation}
\label{eqn::ibp}
(h,f)_\nabla = -\frac{1}{2\pi} (h, \Delta f) \quad \text{for}\quad f \in C_0^\infty(D).
\end{equation}
Observe that $(h,f)_\nabla$ is a Gaussian random variable with mean zero and variance $(f,f)_\nabla$.  Hence $h$ induces a map $C_0^\infty(D) \to \CG$, $\CG$ a Gaussian Hilbert space, that preserves the Dirichlet inner product.  This map extends uniquely to $H(D)$ and allows us to make sense of $(h,f)_\nabla$ for all $f \in H(D)$ and, moreover,
\begin{equation}
\label{eqn::gff_cov}
\cov((h,f)_\nabla,(h,g)_\nabla) = (f,g)_\nabla \quad\text{for all}\quad f,g \in H(D).
\end{equation}

More generally, if $D \subseteq \C$ is not necessarily connected, then the GFF on $D$ is given by taking $h$ to be independently a GFF on each of the components of $D$.

Suppose that $U \subseteq D$ with $U \neq D$ is open.  There is a natural inclusion $\iota$ of $H(U)$ into $H(D)$ where
\[ \iota(f)(x) = \begin{cases} f(x) \quad \text{if}\quad x \in U,\\ 0\quad\text{otherwise.} \end{cases}\]

We define the following $\sigma$-algebras.  For $U \subseteq D$ open, we let $\CF_U^h$ be the $\sigma$-algebra generated by the restriction $h|_U$ of $h$ to $U$.  In other words, $\CF_U^h$ is the $\sigma$-algebra generated by $(h,f)$ for $f \in C_0^\infty(U)$.  For every closed set $K \subseteq D$, we let $\CF_{K^+}^h$ be the $\sigma$-algebra generated by the projection of $h$ onto $H^\perp(D \setminus K)$.  Throughout this article, we often consider conditional expectations where we condition on $h|_V$ for $V \subseteq D$ and by this we mean that we consider the conditional expectation given $\CF_V^h$ (resp.\ $\CF_{V^+}^h$) if $V$ is open (resp.\ closed) in $D$.

If $f \in C_0^\infty(U)$ and $g \in C_0^\infty(D)$, then as $(f,g)_\nabla = -(f,\Delta g)$ it is easy to see that $H(D)$ admits the $(\cdot,\cdot)_\nabla$-orthogonal decomposition $H(U) \oplus H^\perp(U)$ where $H^\perp(U)$ is the set of functions in $H(D)$ harmonic on $U$.  Thus we can write
\[ h = h_U + h_{U^c} = \sum_n \alpha_n^U \phi_n^U + \sum_n \alpha_n^{U^c} \phi_n^{U^c}\]
where $(\alpha_n^U),(\alpha_n^{U^c})$ are independent i.i.d.\ sequences of standard normal random variables and $(\phi_n^U)$, $(\phi_n^{U^c})$ are orthonormal bases of $H(U)$ and $H^\perp(U)$, respectively.  Observe that $h_U$ is a GFF on $U$.  The distribution $h_{U^c}$ is a random distribution that agrees with $h$ on $U^c$.  The restriction of $h_{U^c}$ to $U$ can be interpreted as the ``harmonic extension'' of $h|_{\partial U}$ to $U$.  Note that $h_U$ and $h_{U^c}$ are independent.  We arrive at the following proposition:

\begin{proposition}[Markov Property]
\label{gff::prop::markov}
We can write $h = h_U + h_{U^c}$ where $h_U,h_{U^c}$ are independent distributions on $D$, $h_U$ has the law of a zero boundary GFF on $D \setminus U$ and is zero on $U$, and $h_{U^c}$ is harmonic on $U^c$.  That is, the conditional law of $h|_U$ given $h |_{D \setminus U}$ is that of the GFF on $U$ plus the harmonic extension (in the sense described just above) of $h|_{\partial U}$ to $U$.
\end{proposition}

We emphasize that $h|_U$ in the statement of Proposition~\ref{gff::prop::markov} refers to the \emph{restriction} of $h$ to $U$ while $h_U$ above refers to the projection of $h$ to $H(U)$.  We remark that the orthogonality of $H(U)$ and the set of functions in $H(D)$ which are harmonic on $U$ is also proved in \cite[Theorem~2.17]{SHE06} and it is explained thereafter how this is related to the Markov property of the GFF.  The proposition allows us to make sense of the GFF with non-zero boundary conditions: if $f \colon \partial D \to \R$ is any function that is $L^1$ with respect to harmonic measure on $\partial D$ viewed from some point (hence every point) in $D$, and $F$ is its harmonic extension from $\partial D$ to $D$, then the law of the GFF on $D$ with boundary condition $f$ is given by the law of $F + h$ where $h$ is a zero boundary GFF on~$D$.

Using~\eqref{eqn::ibp} and~\eqref{eqn::gff_cov}, we can derive the covariance function for $(h,f), (h,g)$ for $f,g \in C_0^\infty(D)$.  Namely, we have that
\begin{align}
 \cov( (h,f), (h,g))
&= (2\pi)^2 \cov( (h, \Delta^{-1} f)_\nabla, (h, \Delta^{-1} g)_\nabla)
 = (2\pi)^2 ( \Delta^{-1} f, \Delta^{-1} g)_\nabla \notag\\
&= -2\pi ( f, \Delta^{-1} g)
  = \iint f(x) G(x,y) g(y) dx dy. \label{eqn::gff_l2_cov}
\end{align}
where $G$ is the Green's function for $\Delta$ on $D$ with Dirichlet boundary conditions.  That is, $G$ solves $\Delta G(x,y) = -2\pi \delta_x(y)$ (in the distributional sense) where $\delta_x$ denotes the Dirac mass at $x$ and $G(x,y) = 0$ if $x$ or $y$ is in $\partial D$.  We note that the Green's function is monotone in $D$ in the sense that if $D \subseteq \wt{D}$ and $G$ (resp.\ $\wt{G}$) is the Green's function on $D$ (resp.\ $\wt{D}$) then $G(x,y) \leq \wt{G}(x,y)$ for all $x,y \in D$.  This can be seen because $\wt{G} - G$ is harmonic in $D$ (viewed as a function of one of the variables with the other fixed) and has non-negative boundary conditions on $\partial D$.  We also note that the law of $h$ is determined by $(h,f)$ as $f$ ranges over $C_0^\infty(D)$.

We are now going to show that $\CF_{K^+}^h$ is given by the intersection of $\CF_U^h$ over all $U \subseteq D$ open containing $K$.

\begin{proposition}
\label{prop::gff_sigma_algebra_converges}
For any deterministic, closed set $K \subseteq D$ we have that $\CF_{K^+}^h = \cap_{U \supseteq K} \CF_U^h$ where the intersection is over all open sets $U \subseteq D$ containing $K$.
\end{proposition}
\begin{proof}
We will give the proof in the case that $D$ is bounded.  Upon establishing this, the result follows for general domains $D$ using the conformal invariance of the GFF.

Let $(\phi_n)$ be an orthonormal basis of $H^\perp(D \setminus K)$ and, for each $n$, let $\alpha_n = (h,\phi_n)_\nabla$.  Then we know that $\CF_{K^+}^h = \sigma(\alpha_n : n \in \N)$.  Let $\psi$ be a radially symmetric $C_0^\infty$ bump function supported in $\D$.  That is, $\psi(z) \geq 0$ for all $z \in \C$, $\psi(z)$ depends only on $|z|$, $\int \psi(z) dz = 1$, and $\psi(z) = 0$ for all $|z| \geq 1$.  For each~$n$ and~$m$, let
\[ \wt{\phi}_{n,m}(z) = \int \phi_n(y) m^2 \psi(m (z-y)) dy.\]
Note that $\wt{\phi}_{n,m} \in  C^\infty(D)$.  Let $\phi_{n,m}$ be given by $\wt{\phi}_{n,m}$ minus the harmonic extension of its values from $\partial D$ to $D$.  Since $\phi_n$ is harmonic in $D \setminus K$ and $\psi$ is radially symmetric, it follows that $\phi_{n,m}$ is harmonic on the set of $z \in D$ with $\dist(z,K \cup \partial D) \geq \tfrac{1}{m}$.  Moreover, note that $\phi_{n,m} \in H(D)$ as $\wt{\phi}_{n,m} \in C^\infty(D)$ and
\[ \| \phi_{n,m} \|_\nabla \leq \| \wt{\phi}_{n,m} \|_\nabla \leq \| \phi_n \|_\nabla \quad\text{for all}\quad n,m.\]
Since
\[ (\phi_{n,m},\phi)_\nabla \to (\phi_n,\phi)_\nabla \quad\text{as}\quad m \to \infty \quad\text{for all}\quad \phi \in C_0^\infty(D)\]
it therefore follows that $\phi_{n,m} \to \phi_n$ as $m \to \infty$ in $H(D)$ (recall that if $(w_m)$ is a sequence in a Hilbert space which converges weakly to a limit $w$ with the property that $\| w_m \| \leq \| w\|$ for all $m$ then $\| w_m  - w \| \to 0$ as $m \to \infty$).  Thus with $\alpha_{n,m} = (h,\phi_{n,m})_\nabla$, we have that
\[ \E[  (\alpha_n - \alpha_{n,m})^2] = \| \phi_n - \phi_{n,m} \|_\nabla^2 \to 0 \quad\text{as}\quad m \to \infty.\]
Suppose that $U \subseteq D$ is an open set which contains $K$ and a neighborhood in $D$ of $\partial D$.  Then there exists $m_0 \in \N$ depending only on $K$ and $U$ such that $\alpha_{n,m}$ is $\CF_U^h$-measurable for all $m \geq m_0$.  Indeed, this follows because
\[ \alpha_{n,m} = (h,\phi_{n,m})_\nabla = -2\pi (h,\Delta \phi_{n,m})\]
and $\Delta \phi_{n,m}(z)$ vanishes for $z \in D$ with $\dist(z,K \cup \partial D) \geq \tfrac{1}{m}$.  Therefore $\alpha_n$ is also $\CF_U^h$-measurable.  That is, $\CF_{K^+}^h \subseteq \CF_U^h$.

We note that the $\sigma$-algebra given by $\cap \CF_V^h$, where the intersection is over all $V \subseteq D$ open which contain a neighborhood of $\partial D$ in $D$, is trivial by Proposition~\ref{gff::prop::markov}.  Indeed, this follows because the variance of the integral of the projection of $h$ onto $H(D \setminus \ol{V})$ against a $C_0^\infty(D)$ test function $\phi$ increases to the variance of the integral of $h$ against $\phi$ as $\ol{V}$ decreases to $\partial D$ because the Green's function for $\Delta$ on $D \setminus \ol{V}$ converges to the Green's function on $D$.  This, in particular, implies that the variance of the integral of the projection of $h$ onto $H^\perp(D \setminus \ol{V})$ against $\phi$ decreases to $0$ as $V$ decreases (simply because the sum of the two variances must equal the variance of the integral of $h$ against $\phi$ by independence).  This implies the claim.  We therefore have that $\CF_{K^+}^h \subseteq \CF_U^h$ for all $U \subseteq D$ open which contain $K$.

Note that $\CF_U^h \subseteq \CF_{\ol{U}^+}^h$.  Thus to show that $\CF_{K^+}^h = \cap_{U \supseteq K} \CF_U^h$, it suffices to show that $\CF_{K^+}^h = \cap_{U \supseteq K} \CF_{\ol{U}^+}^h$.  For each $m \in \N$ we let $U_m$ be the $\tfrac{1}{m}$-neighborhood of $K$.  It suffices to show that for any fixed $\psi_1,\ldots,\psi_k \in C_0^\infty(D)$, we have that the joint law of $(h,\psi_j)$ for $1 \leq j \leq k$ given $\CF_{\ol{U}_m^+}^h$ converges to the joint law of $(h,\psi_j)$ for $1 \leq j \leq k$ given $\CF_{K^+}^h$ as $m \to \infty$.  Given $\CF_{\ol{U}_m^+}^h$, we know that the $(h,\psi_j)$ are jointly Gaussian with mean given by integrating the projection of $h$ onto $H^\perp(D \setminus \ol{U}_m)$ against $\psi_j$ and covariance given by the Green's function for $\Delta$ on $D \setminus \ol{U}_m$.  Their conditional law given $\CF_{K^+}^h$ admits an analogous description with $K$ in place of $\ol{U}_m$.  Since the Green's function for $\Delta$ on $D \setminus \ol{U}_m$ converges as $m \to \infty$ to the Green's function for $\Delta$ on $D \setminus K$, it follows that we have the desired convergence of the conditional covariance functions.  To see the convergence of the conditional means, we let $(\wt{\phi}_n)$ be an orthonormal basis of $H(D \setminus K)$ consisting of $C_0^\infty(D \setminus K)$ functions so that $(\phi_n)$ and $(\wt{\phi}_n)$ together form an orthonormal basis of $H(D)$.  Then we can write $h = \sum_n \alpha_n \phi_n + \sum_n \wt{\alpha}_n \wt{\phi}_n$ and we have that the $(\alpha_n)$ and $(\wt{\alpha}_n)$ are i.i.d.\ $N(0,1)$.  Thus
\begin{align}
\label{eqn::cond_exp_expand}
   \E\!\left[ (h,\psi_j) \giv \CF_{\ol{U}_m^+}^h \right] = \sum_n \alpha_n (\phi_n,\psi_j) + \sum_n \E\!\left[\wt{\alpha}_n \giv \CF_{\ol{U}_m^+}^h \right]  (\wt{\phi}_n,\psi_j).
\end{align}
The first summation on the right hand side is equal to $\E[ (h,\psi_j) \giv \CF_{K^+}^h]$.  We note that for any $n_0$ fixed we have that
\begin{align*}
\sum_{n=1}^{n_0} \E\!\left[\wt{\alpha}_n \giv \CF_{\ol{U}_m^+}^h \right]  (\wt{\phi}_n,\psi_j) \to 0 \quad\text{as}\quad m \to \infty
\end{align*}
since for each fixed $n$ there exists $m_0$ sufficiently large so that the support of $\wt{\phi}_n$ is disjoint from $\ol{U}_m$ for all $m \geq m_0$ and for such $m$ we have that $\E[ \wt{\alpha}_n \giv \CF_{\ol{U}_m^+}^h] = 0$.  Moreover, by Jensen's inequality we have that
\begin{align*}
   \E\!\left[ \left( \sum_{n > n_0} \E\!\left[\wt{\alpha}_n \giv \CF_{\ol{U}_m^+}^h \right]  (\wt{\phi}_n,\psi_j) \right)^2 \right]
&=  \E\!\left[ \left( \E\!\left[ \sum_{n > n_0} \wt{\alpha}_n (\wt{\phi}_n,\psi_j)  \giv \CF_{\ol{U}_m^+}^h \right] \right)^2 \right]\\
&\leq \sum_{n > n_0} (\wt{\phi}_n,\psi_j)^2 = (2\pi)^2 \sum_{n > n_0} (\wt{\phi}_n,\Delta^{-1} \psi_j)_\nabla^2.
\end{align*}
Since $\Delta^{-1} \psi_j$ is in $H(D)$, it follows that the summation on the right hand side tends to $0$ as $n_0 \to \infty$.  Therefore the second summation in~\eqref{eqn::cond_exp_expand} tends to $0$ in probability as $m \to \infty$ which implies that there exists a sequence $(m_k)$ tending to $\infty$ sufficiently quickly along which it tends to $0$ almost surely.  Since it is also an $L^2$-bounded martingale, the martingale convergence theorem implies that it converges almost surely.  Combining, this almost sure limit must be equal to $0$, and this implies the result.
\end{proof}

\begin{proposition}
\label{prop::gff_boundary_values}
Assume that $D$ is a non-trivial simply connected domain and let $K$ be a deterministic closed subset of $D$.  Let $h_1$ be the harmonic function on $D \setminus K$ which agrees with the projection of $h$ onto $H^\perp(D \setminus K)$, restricted to $D \setminus K$.  Then a.s.\
\[ \lim_{D \setminus K \ni z \to z_0} h_1(z) = 0 \quad\text{for all}\quad z_0 \in  \partial D \setminus K.\]
\end{proposition}
\begin{proof}
Let $U$ be a simply connected neighborhood of $z_0$ such that the distance from~$U$ to~$K$ is positive.  Then \cite[Lemma~3.2]{SchrammShe10} states that the projections of $h$ onto $H^\perp(D \setminus K)$ and $H^\perp(D \setminus U)$ are  {\em almost independent}, which means that their joint law is absolutely continuous with respect to the product of their marginal laws.  This implies that for almost all instances of the former, the conditional law of the latter is absolutely continuous with respect to its unconditioned law.   We note that the restriction to $U$ of the projection of $h$ onto $H^\perp(D \setminus U)$ is equal to $h|_U$.  Therefore the conditional law of $h|_U$ given the projection of $h$ onto $H^\perp(D \setminus K)$ is absolutely continuous with respect to its unconditioned law.  Moreover, $h|_U$ given the projection of $h$ onto $H^\perp(D \setminus K)$ can be written as the sum of a function $h_1$ which is harmonic in $D \setminus K$ and a zero-boundary GFF $\wt{h}$ in $D \setminus K$ restricted to $U$.

Assume that we have fixed an instance of $h_1$.  Let $(z_n)$ be a sequence in $D$ with $z_n \to z_0$ as $n \to \infty$.  For each $n \in \N$, we let $\varphi_n \colon U \to \D$ be the unique conformal map with $\varphi_n(z_n) = 0$ and $\varphi_n'(z_n) > 0$.  Let $\mu_n$ be the law of $h \circ \varphi_n^{-1}$ and let $\wt{\mu}_n$ (resp.\ $\wt{\mu}_n'$) be the law of $\wt{h} \circ \varphi_n^{-1}$ (resp.\ $(\wt{h}+h_1) \circ \varphi_n^{-1}$) with the instance of $h_1$ fixed.  Then $\mu_n$ and $\wt{\mu}_n'$ are almost surely (in the realization of $h_1$) mutually absolutely continuous for each $n$.  Moreover, the Radon-Nikodym derivative $Z_n$ of $\wt{\mu}_n'$ with respect to $\mu_n$ has the following property.  For every $\epsilon > 0$ there exists $\delta > 0$ (uniformly in $n$) such that if $A$ is any event with $\mu_n[A] < \delta$ then $\wt{\mu}_n'[A] = \int Z_n \one_A d\mu_n < \epsilon$.  Indeed, this follows because we can always express $A$ in terms of an event which involves $h|_U$ and then compute using the Radon-Nikodym derivative of the law of $(\wt{h} + h_1)|_U$ with respect to the law of $h|_U$.

For each $\zeta > 0$ we let $\phi_\zeta$ be a $C^\infty$ function which is non-negative, radially symmetric, has integral $1$, and which is supported in the annulus $\D \setminus \ol{B(0,1-\zeta)}$.  Then we have that both $(h \circ \varphi_n^{-1},\phi_\zeta)$ and $(\wt{h} \circ \varphi_n^{-1},\phi_\zeta)$ tend to $0$ in probability as $\zeta \to 0$ and $n \to \infty$.  Indeed, this can be seen because both are mean-zero Gaussian random variables for each $\zeta > 0$ and $n \in \N$ with variance tending to $0$ as $\zeta \to 0$ and $n \to \infty$.  That is, for each $\delta > 0$ there exists $\zeta_0 > 0$ and $n_0 \in \N$ such that $\zeta \in (0,\zeta_0)$ and $n \in \N$ with $n \geq n_0$ implies that the variance of both $(h \circ \varphi_n^{-1},\phi_\zeta)$ and $(\wt{h} \circ \varphi_n^{-1},\phi_\zeta)$ is at most $\delta$.  On the other hand, since $h_1$ is harmonic in $U$ we have that $h_1 \circ \varphi_n^{-1}$ is harmonic in $\D$, hence $(h_1 \circ \varphi_n^{-1},\phi_\zeta) = h_1 \circ \varphi_n^{-1}(0) =  h_1(z_n)$ for all $\zeta \in (0,1)$ and $n \in \N$.

Fix $\delta > 0$ so that if $\mu_n[A] < \delta$ then $\wt{\mu}_n'[A] < 1/4$.  Let $a > 0$ be arbitrary.  Assume that we have fixed $\zeta_0 > 0$ and $n_0 \in \N$ such that $\zeta \in (0,\zeta_0)$ and $n \geq n_0$ implies that $\mu_n[ |(h \circ \varphi_n^{-1},\phi_\zeta)| > a/2] < \delta$.  Then $\wt{\mu}_n'[|( (\wt{h} + h_1) \circ \varphi_n^{-1},\phi_\zeta)| > a/2] < 1/4$.  By possibly further decreasing the value of $\zeta_0 >0$ and increasing the value of $n_0$ we also have that $\zeta \in (0,\zeta_0)$ and $n \geq n_0$ implies that $\wt{\mu}_n[ |( \wt{h} \circ \varphi_n^{-1},\phi_\zeta)| > a/2] < 1/4$.  Using that the event $|h_1(z)| > a$ is the same as the event $|((\wt{h}+h_1) \circ \varphi_n^{-1}, \phi_\zeta) - (\wt{h} \circ \varphi_n^{-1},\phi_\zeta)| > a$, we have that (with a slight abuse of notation after the first equality)
\begin{align*}
\one_{\{ |h_1(z_n)| > a\}}
&=  \wt{\mu}_n[ |( (\wt{h} + h_1) \circ \varphi_n^{-1},\phi_\zeta) - (\wt{h} \circ \varphi_n^{-1}, \phi_\zeta)| > a ]\\
&\leq \wt{\mu}_n'[ |( (\wt{h} + h_1) \circ \varphi_n^{-1},\phi_\zeta)| > a/2] + \wt{\mu}_n[ |(\wt{h} \circ \varphi_n^{-1}, \phi_\zeta)| > a/2]\\
&\leq \frac{1}{2}.
\end{align*}
Since $\one_{\{ |h_1(z_n)| > a\}}$ takes values in $\{0,1\}$, we have that $\one_{\{ |h_1(z_n)| > a\}} \leq 1/2$ implies that $|h_1(z_n)| \leq a$.  Therefore $|h_1(z_n)| \leq a$ for $n \geq n_0$, hence $h_1(z_n) \to 0$ as $n \to \infty$, as desired.
\end{proof}

We finish this subsection with the following proposition, which describes the absolute continuity properties of the GFF.

\begin{proposition}[Absolute Continuity]
\label{prop::gff_abs_continuity}
Suppose that $D_1,D_2$ are simply connected domains with $D_1 \cap D_2 \neq \emptyset$. For $i=1,2$, let $h_i$ be a zero boundary GFF on $D_i$ and $F_i$ harmonic on $D_i$.  Fix a bounded simply connected open domain $U \subseteq D_1 \cap D_2$.
\begin{enumerate}[(i)]
\item \label{lem::gff_abs_continuity_interior} (Interior) If $\dist(U, \partial D_i) > 0$ for $i = 1,2$ then the laws of $(h_1+F_1)|_U$ and $(h_2+F_2)|_U$ are mutually absolutely continuous.
\item \label{lem::gff_abs_continuity_boundary} (Boundary) Suppose that there is a neighborhood $U'$ of the closure $\overline U$ of $U$ such that $\ol{D}_1 \cap U' = \ol{D}_2 \cap U'$, and that $F_1 - F_2$ tends to zero as one approaches points in the sets $\partial D_i \cap U'$.  Then the laws of $(h_1+F_1)|_U$ and $(h_2+F_2)|_U$ are mutually absolutely continuous.
\end{enumerate}
\end{proposition}

Although Proposition~\ref{prop::gff_abs_continuity} is stated in the case that $U$ is bounded, an analogous result holds when $U$ is not bounded.  We will in particular use this result in the following setting without reference: $D_1 = D_2 = D$ and there exists a conformal transformation $\varphi \colon D \to \wt{D}$ where $\wt{D}$ is bounded and $\wt{U} = \varphi(U)$ satisfies the hypotheses of Proposition~\ref{prop::gff_abs_continuity} part~\eqref{lem::gff_abs_continuity_interior} or part~\eqref{lem::gff_abs_continuity_boundary}.

We will often make use of Proposition~\ref{prop::gff_abs_continuity} in the following manner.  Theorem~\ref{thm::coupling_existence} gives the existence of a coupling of an $\SLE_\kappa(\ul{\rho})$ (a flow line) or $\SLE_{\kappa'}(\ul{\rho})$ process (a counterflow line) with the GFF (say, on $\h$) provided the boundary data of the GFF is piecewise constant and changes values a finite number of times.  For more general boundary data, one can still construct the flow and counterflow lines of the GFF provided the boundary data near the starting point is piecewise constant.  The reason for this is that Proposition~\ref{prop::gff_abs_continuity} implies that the law of the field near the starting point is absolutely continuous with respect to the law of a GFF that does have piecewise constant boundary data which changes a finite number of times.

More precisely, suppose that $h$ is a GFF on $\h$ with piecewise constant boundary data and that $\eta$ is its flow or counterflow line starting from $x \in \partial \h$.  If $\wt{h}$ is another GFF on $\h$ whose boundary data agrees with that of $h$ in a neighborhood of $x$ and $U$ is any bounded, open subset of $\ol{\h}$ which contains a neighborhood of $x$ and with positive distance from those boundary segments where the boundary data of $h$ and $\wt{h}$ differ, then we know that the laws of $h|_U$ and $\wt{h}|_U$ are mutually absolutely continuous.  Let~$Z_U$ denote the Radon-Nikodym derivative of the law of the latter with respect to the former.  Then weighting the law of $(h|_U,\eta_U)$, $\eta_U$ given by~$\eta$ stopped upon first exiting~$U$, by~$Z_U$ yields a coupling where the marginal law of the first element is the same as the law of~$\wt{h}|_U$, the marginal law of the second element is mutually absolutely continuous with respect to the law of $\eta_U$, and the pair satisfies the same Markov property described in the statement of Theorem~\ref{thm::coupling_existence}.  (We will justify the Markov property more carefully just after the proof of Lemma~\ref{lem::local_char}.)

Note that all of the almost sure properties of $\eta_U$ are preserved under this change of measure.  The caveat is that by combining Theorem~\ref{thm::coupling_existence} and Proposition~\ref{prop::gff_abs_continuity} in this way (with larger and larger domains $U$), the flow or counterflow line of $\wt{h}$ will only be defined up until it hits a boundary segment which does not have piecewise constant boundary data.

One can also apply this operation in reverse.  Namely, suppose that $h$ is a GFF on $\h$, $U \subseteq \h$ is bounded and open such that $h$ has piecewise constant boundary data in $\partial \h \cap \ol{U}$, and $\eta$ is a path coupled with $h$ which satisfies the Markov property described in the statement of Theorem~\ref{thm::coupling_existence}.  Then Proposition~\ref{prop::gff_abs_continuity} and Theorem~\ref{thm::martingale} together imply that the joint law of $(h|_U,\eta_U)$, with $\eta_U$ as defined above, is mutually absolutely continuous with respect to the joint law of a GFF with piecewise constant boundary data on all of $\partial \h$ and its flow or counterflow line starting from the origin, both restricted to $U$.  In particular, all of the almost sure properties of $\eta_U$ are the same under both laws.  This is useful because the Loewner driving function for $\eta_U$ may not be described by a simple SDE.

We also remark that it is possible to construct directly (without appealing to Proposition~\ref{prop::gff_abs_continuity}) a coupling as in Theorem~\ref{thm::coupling_existence} with more general types of boundary data.  We will not need this in the present article because we will not need to analyze the specific form of the driving diffusion in the more general setting.  See \cite[Theorem~4.5]{SHE_WELD} for a precise statement of this in the setting of a very closely related coupling of $\SLE$ with the GFF (the so-called ``reverse'' coupling).

\begin{proof}[Proof of Proposition~\ref{prop::gff_abs_continuity}]
Let $\wt{U} \subseteq D_1 \cap D_2$ be such that $\dist(\wt{U},\partial D_i) > 0$ for $i=1,2$, $U \subseteq \wt{U}$, and $\dist(U,\partial \wt{U}) > 0$.  By Proposition~\ref{gff::prop::markov}, we can write $(h_i+F_i)|_{\wt{U}} = h_i^{\wt{U}} + (h_i^{\wt{U}^c}+F_i)|_{\wt{U}}$ where $h_i^{\wt{U}}$ is a zero boundary GFF on $\wt{U}$ and $h_i^{\wt{U}^c}+F_i$ is harmonic on $\wt{U}$, both restricted to $\wt{U}$.  We assume that $h_1,h_2$ are coupled together on a common probability space so that $h_1^{\wt{U}} = h_2^{\wt{U}}$ and that $h_1^{\wt{U}^c}, h_2^{\wt{U}^c}$ are independent.  Let $\phi \in C_0^\infty(\wt{U})$ be such that $\phi|_U \equiv 1$ and let $g = \phi \left( (h_2^{\wt{U}^c} + F_2) - (h_1^{\wt{U}^c} + F_1) \right)$.  Then we have that $g \in H(\wt{U})$ and $(h_1+F_1 + g)|_U= (h_2+F_2)|_U$.  We note that $g$ is a measurable function of $h_1,h_2$ since the $h_i^{\wt{U}^c}$ for $i=1,2$ are measurable functions of $h_1,h_2$ as their series expansions can be determined by taking a $(\cdot,\cdot)_\nabla$-inner product of $h_1,h_2$ with respect to an appropriately chosen orthonormal basis.  Similarly, $(h_1^{\wt{U}},g)_\nabla$ is a measurable function of $h_1,h_2$ because we can represent it as $\sum_j a_j b_j$ where $(a_j)$, $(b_j)$ are the coordinates of $h_1^{\wt{U}}, g$ with respect to a $(\cdot,\cdot)_\nabla$-orthonormal basis.  We note that, although $g$ is a random function in $H(U)$, the series $\sum_j a_j b_j$ converges almost surely because $g$ is independent of $h_1^{\wt{U}}$.  (Here, we are using that this holds if $g$ is any fixed function in $H(U)$ hence also holds in the case that $g$ is independent of $h_1^{\wt{U}}$ by Fubini's theorem.)  If we weight the law of $h_1,h_2$ by $\CZ^{-1} \exp((h_1^{\wt{U}},g)_\nabla)$ where $\CZ = \exp(\tfrac{1}{2} \| g \|_\nabla^2)$ then the law of $h_1^{\wt{U}}$ under the weighted law is equal to the law of $h_1^{\wt{U}}+g$ under the unweighted law.  As $g$ is compactly supported in $\wt{U}$ and $h_1^{\wt{U}^c} + F_1$ is harmonic in $\wt{U}$, we have that $(h_1^{\wt{U}^c} + F_1,g)_\nabla = 0$, hence $(h_1^{\wt{U}},g)_\nabla = (h_1+F_1 ,g)_\nabla$.  Therefore if we weight the law of $h_1,h_2$ by $\CZ^{-1} \exp((h_1+F_1,g)_\nabla)$ where $\CZ = \exp( \tfrac{1}{2} \| g \|_\nabla^2)$ then the law of $(h_1+F_1)|_U$ under the weighted law is equal to the law of $(h_1+F_1+g)|_U = (h_2 + F_2)|_U$ under the unweighted law.  We conclude that, under the coupling that we have constructed of $h_1,h_2$, an event for $(h_1+F_1)|_U$ has zero probability if and only if it has zero probability for $(h_2+F_2)|_U$.  That is, the law of $(h_1+F_1)|_U$ is mutually absolutely continuous with respect to the law of $(h_2+F_2)|_U$.  This proves part~\eqref{lem::gff_abs_continuity_interior}.

For part~\eqref{lem::gff_abs_continuity_boundary}, we let $\wt{U} = U' \cap D_1 = U' \cap D_2$.  Then we have $(h_i+F_i)|_{\wt{U}} = h_i^{\wt{U}} + (h_i^{\wt{U}^c} + F_i)|_{\wt{U}}$ where $h_i^{\wt{U}}$ is a zero boundary GFF on $\wt{U}$ and $h_i^{\wt{U}^c}+F_i$ is harmonic in $\wt{U}$, both restricted to~$\wt{U}$.  We assume that $h_1,h_2$ are coupled on a common probability space so that $h_1^{\wt{U}} = h_2^{\wt{U}}$ and $h_1^{\wt{U}^c}, h_2^{\wt{U}^c}$ are independent.  Let $\phi$ be a $C^\infty$ function with $\phi|_U \equiv 1$ and with $\phi$ equal to $0$ on a neighborhood of $(U')^c$ and let $g = \phi \left( (h_2^{\wt{U}^c}+F_2) - (h_1^{\wt{U}^c}+F_1) \right)$.  Then we have that $(h_1 + F_1 + g)|_U = (h_2 + F_2)|_U$ where $g \in H(\wt{U})$.  The rest of the proof thus follows using the same argument used to prove the proof of part~\eqref{lem::gff_abs_continuity_interior}.
\end{proof}

We remark that a slight variant of Proposition~\ref{prop::gff_abs_continuity} is stated and proved in \cite[Lemma~3.2]{SchrammShe10} in the case that one considers the law of the GFF on disjoint closed subsets.

\begin{remark}
\label{rem::gff_abs_continuity} Suppose that $h$ is a GFF on a domain $D$, and $f \in C_0^\infty(D)$.  As in the proof of Proposition~\ref{prop::gff_abs_continuity}, we know that reweighing the law $\mu$ of $h$ by $\CZ^{-1}\exp((h,f)_\nabla)$ yields the law $\mu_f$ of $h+f$.  Consequently, the Cauchy-Schwarz inequality implies that $\mu(E) \leq \mu_f^{1/2}(E) \exp(\| f \|_\nabla^2)$ for every event $E$.  The same holds if we reverse the roles of $\mu$ and $\mu_f$.  More generally, if we apply H\"older's inequality, we have that $\mu(E) \leq \mu_f^{1/p}(E) \exp(\tfrac{q}{2} \| f \|_\nabla^2)$ where $\tfrac{1}{p} + \tfrac{1}{q} = 1$.  These simple facts allow us to give explicit bounds which relate the probabilities of events in the setting of parts~\eqref{lem::gff_abs_continuity_interior} and~\eqref{lem::gff_abs_continuity_boundary} of Proposition~\ref{prop::gff_abs_continuity}.  For example, suppose that $D = D_1 = D_2$ is a Jordan domain and $L$ is an interval in $\partial D$.  Suppose that $h_1$ and $h_2$ are GFFs on $D$ with $h_1|_{\partial D \setminus L} = h_2|_{\partial D \setminus L}$ and $|h_1|_L| \leq M$, $|h_2|_L| \leq M$ for some constant $M > 0$.  Let $F$ be the function which is harmonic in $D$ with $F|_{\partial D \setminus L} \equiv 0$ and $F|_L = h_2|_L - h_1|_L$.  Then $h_1+F \stackrel{d}{=} h_2$.  Let $U$ be the set of points in $D$ which have distance at least $\epsilon > 0$ from $L$ and let $G$ be the function which agrees with $F$ in $U$, is $0$ on the set of points with distance at most $\tfrac{1}{2} \epsilon$ from $\partial D$, and otherwise harmonic.  Then there exists $C > 0$ depending only on $M,D,\epsilon$ such that $\| G \|_\nabla^2 \leq C$.  Moreover, with $\mu_G$ the law of $h_1+G$ restricted to $U$, we have that $h_2|_U \sim \mu_G$.  Thus with $\mu$ the law of $h_1|_U$, for each $p > 1$ there exists a constant $C_p > 0$ depending only on $M,D,\epsilon$ such that $\mu(E) \leq C_p \mu_G^{1/p}(E)$ for all events $E$.
\end{remark}

\subsection{Local sets}
\label{subsec::local_sets}

The theory of local sets, developed in \cite{SchrammShe10}, extends the Markovian structure of the field (Proposition~\ref{gff::prop::markov}) to the setting of conditioning on the values it takes on a \emph{random set} $A \subseteq D$.  In this section, we will give an overview of local sets which is very closely based on the treatment given in \cite[Section~3.3]{SchrammShe10}.  We will cite some of the results in \cite{SchrammShe10} in the proofs of the statements below.

Throughout, we shall assume for convenience that $D \subseteq \C$ is a bounded domain and $h$ is a GFF on~$D$.  Let $\Gamma$ be the collection of all nonempty, closed subsets of $\ol{D}$ which contain $\partial D$.  We view $\Gamma$ as a metric space endowed with the {\bf Hausdorff} distance.  That is, the distance between sets $S_1, S_2 \in \Gamma$ is
\[d_{\rm HAUS}(S_1,S_2):=\max\Bigl \{ \sup_{x \in S_1} \dist(x, S_2), \sup_{y \in S_2} \dist(y,S_1)\Bigr \}.\]
It is well known (and the reader may easily verify) that $\Gamma$ is a compact metric space.  (In order to treat the case that $D$ is unbounded, we can conformally map to the sphere and use the induced Euclidean metric on the sphere.)  Let $\CG$ be the Borel $\sigma$-algebra on $\Gamma$ induced by this metric.

Suppose that $(A,h)$ is a coupling of a GFF $h$ on $D$ and a random variable $A$ taking values in $\Gamma$.  Then $A$ is said to be a {\bf local set} of $h$ if there exists a law on pairs $(A,h_1)$ where $h_1$ is a distribution on $D$ with $h_1|_{D \setminus A}$ harmonic is such that a sample with the law $(A,h)$ can be produced by
\begin{enumerate}
\item choosing the pair $(h_1,A)$,
\item then sampling an instance $h_2$ of the zero boundary GFF on $D \setminus A$ and setting $h=h_1+h_2$.
\end{enumerate}

Whenever we use the phrase ``$A$ is a local set of $h$,'' we will always assume that $\partial D \subseteq A$ even though we may not say this explicitly.

Deterministic closed sets are also obviously local and, by Theorem~\ref{thm::coupling_existence}, so is the hull $K_\tau$ of an $\SLE_\kappa(\ul{\rho})$ process stopped at time $\tau$ at or before the continuation threshold is reached.  These are the motivating examples for the theory.

Given $A  \in \Gamma$, let $A_\delta$ denote the closed set containing all points in $D$ whose distance from $A$ is at most $\delta$.  Let $\CA_\delta$ be the smallest $\sigma$-algebra in which $A$ and the restriction of $h$ (as a distribution) to the interior of $A_\delta$ are measurable. Let $\CA = \bigcap_{\delta > 0} \CA_\delta$. Intuitively, this is the smallest $\sigma$-algebra in which $A$ and the values of $h$ in an infinitesimal neighborhood of $A$ are measurable.  Suppose that $A$ is local for $h$.  We let $\CC_A$ be the conditional expectation of $h$ given~$\CA$.  We note that we can view $\CC_A$ as a random variable which takes values in the space of distributions and we can identify $\CC_A$ with the harmonic function $h_1$, and this is the perspective that we will take.  Indeed, as we will explain carefully just below, the $\sigma$-algebra generated by the values of $h_2$ in $A_\delta$ becomes trivial as $\delta \to 0$, i.e., $\CA = \sigma(h_1)$.  Hence, we can write $(\CC_A,\phi) = \E[ (h,\phi) \giv \CA] = \E[ (h_1,\phi) \giv \CA] = (h_1,\phi)$.

Let us now explain carefully why $\CA = \sigma(h_1)$.  We first note that $h_1$ is obviously $\CA$-measurable because it is $\CA_\delta$-measurable for every $\delta > 0$ as it is harmonic in $D \setminus A$.  To show that $\CA \subseteq \sigma(h_1)$, it suffices to show that the conditional law of $h_2$ given $\CA$ is the same as the conditional law of $h_2$ given $\sigma(h_1)$.  (Indeed, this will imply that the conditional law of $h$ given $\sigma(h_1)$ is the same as the conditional law of $h$ given $\CA$.)  That is, it is a zero boundary GFF on $D \setminus A$.  Let $B_\delta =\{ x \in D : \dist(x,A) > \delta\}$.  Let $\CB_\delta$ be the $\sigma$-algebra generated by $h_1$ and the projection $h_{21}^\delta$ of $h_2$ onto $H^\perp(B^\delta)$.  Note that $\CB_\delta$ contains $\CA_\delta$ for every $\delta > 0$ because if $\phi \in C_0^\infty(D)$, on the event that the support of $\phi$ is contained in $A_\delta$, we have $(h,\phi) = (h_1,\phi) + (h_{21}^\delta,\phi)$.  Therefore $\CB = \cap_{\delta > 0} \CB_\delta$ contains $\CA$.  It suffices to show that the conditional law of $h_2$ given $\CB$ is that of a zero-boundary GFF on $D \setminus A$ (because then the same will be true for $\CA$).  This follows from the backwards martingale convergence theorem.  Indeed, if $\phi \in C_0^\infty(D)$ and $S_\delta$ is the event that the support of $\phi$ is contained in $B_\delta$ and $S$ is the event that the support of $\phi$ is contained in $D \setminus A$, then with $h_{22}^\delta$ the projection of $h_2$ onto $H(B_\delta)$ we have that
\[ \E[ \exp(i \theta(h_2,\phi)_\nabla ) \giv \CB_\delta] \one_{S_\delta} = \E[ \exp(i \theta(h_{22}^\delta,\phi)_\nabla) \giv \CB_\delta] \one_{S_\delta} = \exp(-\theta^2 \| \phi \|_\nabla^2 / 2) \one_{S_\delta}\]
where the final equality follows by the Markov property of the GFF.  Taking a limit as $\delta \to 0$, the right hand side converges to $\exp(-\theta^2 \| \phi \|_\nabla^2/2) \one_S$ as $S_\delta$ increases to $S$ while the left hand side converges to
\[ \E[ \exp(i \theta(h_2,\phi)_\nabla) \giv \CB] \one_S\]
by the backwards martingale convergence theorem.

We note that in the case that $A$ is a deterministic, closed set then Proposition~\ref{prop::gff_sigma_algebra_converges} implies that $\CA$ is the same as the $\sigma$-algebra generated by the projection of $h$ onto $H^\perp(D \setminus A)$ so that $\CC_A$ is the same as the projection of $h$ onto $H^\perp(D \setminus A)$.  The proof of Lemma~\ref{lem::local_char} given just below implies that $\CA$ is equal to the $\sigma$-algebra generated by $(h_1,A)$ as in the definition of a local set given above.  

In many places in this article, we will consider conditional expectations where we condition on $A$ and $h|_A$ where $A$ is a local set for $h$.  By this, we mean that we consider the conditional expectation $\CC_A$ of $h$ given the $\sigma$-algebra $\CA$ defined just above.

There are several other characterizations of local sets which are given in the following restatement of \cite[Lemma~3.9]{SchrammShe10}.

\begin{lemma}
\label{lem::local_char}
Suppose that $(A, h)$ is a random variable which is a coupling of an instance $h$ of the GFF on $D$ with a random element $A$ of $\Gamma$.  Then the following are equivalent:
\begin{enumerate}[(i)]
\item \label{it::Ucond} For each deterministic open $U \subseteq D$, we have that {\em given} the
projection of $h$ onto $H^\perp(U)$, the event $A \cap U = \emptyset$
is independent of the projection of $h$ onto $H(U)$.  In other words,
the conditional probability that $A \cap U = \emptyset$ given $h$ is a measurable function of
the projection of $h$ onto $H^\perp(U)$.

\item \label{it::U2cond} For each deterministic open $U \subseteq D$, we let $S$ be the event that $A$ intersects $U$ and let
\[ \wt{A} = \begin{cases} A  \quad\text{on}\quad S^c, \\ \emptyset \quad\text{otherwise.} \end{cases}\]
Then we have that {\em given} the projection of $h$ onto $H^\perp(U)$, the pair $(S, \wt A)$ is independent of the projection of $h$ onto $H(U)$.

\item \label{it::Acond} Conditioned on $\CA$, (a regular version of) the conditional law of $h$ is that of $h_1 + h_2$ where $h_2$ is the GFF with zero boundary values on $D\setminus A$ (extended to all of $D$) and $h_1$ is an $\CA$-measurable random distribution (i.e., as a distribution-valued function on the space of distribution-set pairs $(A, h)$, $h_1$ is $\CA$-measurable) which is almost surely harmonic on $D \setminus A$.
\item \label{it::twostep} $A$ is a local set for $h$.  That is, a sample with the law of $(A, h)$ can be produced as follows.  First choose the pair $(A, h_1)$ according to some law where $h_1$ is almost surely harmonic on $D \setminus A$.  Then sample an instance $h_2$ of the GFF on $D \setminus A$ and set $h=h_1+h_2$.
\end{enumerate}
\end{lemma}
\begin{proof}
Trivially,~\eqref{it::U2cond} implies~\eqref{it::Ucond}.

Next, suppose $A$ satisfies~\eqref{it::Ucond}.  We will show that~\eqref{it::Acond} holds.  For each $\delta > 0$, let $\CD_\delta$ be the collection of sets which can be written as $\ol{D} \cap S$ where $S$ is a closed square in $\C$ of side length $\delta$ with corners in the grid $\delta \Z^2$ and let $\wh{A}_\delta = \cup\{ S \in \CD_\delta : S \cap A_\delta \neq \emptyset\}$.  We claim that $\wh{A}_\delta$ satisfies~\eqref{it::Ucond} for each $\delta > 0$.  To see this, fix $U \subseteq D$ open and let $\wh{U}_\delta$ be given by the interior of $\cup\{ S \in \CD_\delta : S \cap U \neq \emptyset\}$.  Then $\wh{A}_\delta$ intersects $U$ if and only if $A_\delta$ intersects the intersection of $D$ and the closure of~$\wh{U}_\delta$.  For each $\delta' > 0$, let $\wt{U}_{\delta,\delta'}$ be given by the $\delta'$-neighborhood of $\wh{U}_\delta$ in $D$.  Then $A_\delta$ intersects the intersection of $D$ and the closure of~$\wh{U}_\delta$ if and only if $A$ intersects the intersection of $D$ and the closure of $\wt{U}_{\delta,\delta}$.  Equivalently, $A_\delta$ intersects the intersection of $D$ and the closure of~$\wh{U}_\delta$ if and only if $A$ intersects $\cap_{\delta' > \delta} \wt{U}_{\delta,\delta'}$.  Since $A$ satisfies~\eqref{it::Ucond}, we know that the conditional probability of $\{A \cap \wt{U}_{\delta,\delta'} \neq \emptyset\}$ given $h$ is a measurable function of the projection of $h$ onto $H^\perp(\wt{U}_{\delta,\delta'})$.  This is in turn a measurable function of the projection of $h$ onto $H^\perp(U)$ (since $U \subseteq \wt{U}_{\delta,\delta'}$).  Therefore the event $\{ \wh{A}_\delta \cap U \neq \emptyset\} = \cap_{\delta' > \delta} \{ A \cap \wt{U}_{\delta,\delta'} \neq \emptyset\}$ is a measurable function of the projection of $h$ onto $H^\perp(U)$ (note that we can represent the intersection as a countable intersection so that the event in question is measurable).  Note that there are only finitely many possible choices for $\wh{A}_\delta$ since we assumed $D$ to be bounded.

We will now show that $\wh{A}_\delta$ satisfies~\eqref{it::Acond}.  Let $\CD_\delta^\cup$ be the collection of all sets which can be expressed as a finite union of elements in $\CD_\delta$.  Fix $C \in \CD_\delta^\cup$ and assume that $\p[\wh{A}_\delta \subseteq C] > 0$.  Since $\wh{A}_\delta$ satisfies~\eqref{it::Ucond}, we have that the conditional law of $h$ in $D \setminus C$ given both $\wh{A}_\delta \subseteq C$ and the projection of $h$ onto $H^\perp(D \setminus C)$ is the same as the conditional law of $h$ in $D \setminus C$ given just the projection of $h$ onto $H^\perp(D \setminus C)$.  That $\wh{A}_\delta$ satisfies~\eqref{it::Acond} thus follows from Proposition~\ref{gff::prop::markov}.

We are now going to use a limiting procedure to deduce that $A$ satisfies~\eqref{it::Acond} from the fact that $\wh{A}_\delta$ satisfies~\eqref{it::Acond}.  Assume that $\delta=2^{-j}$ for some $j \in \N$.  Let $\wh{\CA}_\delta$ be defined analogously to $\CA_\delta$ but with $\wh{A}_\delta$ in place of $A_\delta$ and let $\wt{\CA}_\delta$ be the smallest $\sigma$-algebra which contains $\wh{\CA}_\delta$ and with respect to which $A$ is measurable.  As we will see momentarily, the $\sigma$-algebras $\wt{\CA}_\delta$ will be useful to consider because they decrease as $\delta$ decreases.  We claim that the conditional law of $h$ given $\wh{\CA}_\delta$ is the same as the conditional law of $h$ given $\wt{\CA}_\delta$.  To see that this is the case, fix $k \geq j$, let $\delta' = 2^{-k}$, and let $\wt{\CA}_{\delta,\delta'}$ be the smallest $\sigma$-algebra which contains both $\wh{\CA}_\delta$ and $\wh{\CA}_{\delta'}$.  By the argument in the previous paragraph, it follows that the conditional law of $h$ given $\wh{\CA}_\delta$ is the same as the conditional law of $h$ given $\wt{\CA}_{\delta,\delta'}$.  The claim then follows because $\wt{\CA}_\delta = \cap_{k=j}^\infty \wt{\CA}_{\delta,2^{-k}}$.

Since the $\sigma$-algebras $\wt{\CA}_{2^{-j}}$ are decreasing, the conditional law of $h$ given $\wh{\CA}_{2^{-j}}$ is the same as the conditional law of $h$ given $\wt{\CA}_{2^{-j}}$, and $\cap_{j=1}^\infty \wt{\CA}_{2^{-j}} = \CA$, the reverse martingale convergence theorem implies the almost sure convergence $\CC_{\wh A_{2^{-j}}}$ as $j \to \infty$ in the weak sense, i.e., for each fixed $\phi$, we have a.s.\ that
\begin{equation}
\label{eqn::ce_converge}
(\CC_{\wh A_{2^{-j}}}, \phi) \to \E[ (h,\phi) \giv \CA] \quad\text{as}\quad j \to \infty.
\end{equation}
In order to finish the proof that~\eqref{it::Acond} holds for $A$, we need to show that $\E[ (h,\phi)\giv \CA]$ defines a distribution on $D$ and that the conditional law of $\phi \mapsto (h,\phi) - \E[ (h,\phi)\giv \CA]$ given $\CA$ is that of a GFF on $D \setminus A$ with zero boundary conditions.  Since~\eqref{it::Acond} holds for each $\wh{A}_\delta$, we have that
\begin{equation}
\label{eqn::char_form}
\E[ \exp(i \theta (h,\phi)) \giv \wh{\CA}_{2^{-j}} ] = \exp\!\left(i \theta (\CC_{\wh{A}_{2^{-j}}},\phi) - \theta^2 \sigma_{D \setminus \wh{A}_{2^{-j}}}^2(\phi)/2 \right)
\end{equation}
where $\sigma_{D \setminus \wh{A}_{2^{-j}}}^2(\phi) = \iint \phi(x) G_{D \setminus \wh{A}_{2^{-j}}}(x,y) \phi(y) dx dy$ and $G_{D \setminus \wh{A}_{2^{-j}}}$ is the Green's function for $\Delta$ on $D \setminus \wh{A}_{2^{-j}}$.  Thus sending $j \to \infty$, by combining~\eqref{eqn::ce_converge} and~\eqref{eqn::char_form} we have a.s.\ that 
\[ \E[ \exp(i \theta (h,\phi)) \giv \CA ] = \exp\!\left(i \theta \E[(h,\phi) \giv \CA] - \theta^2 \sigma_{D \setminus A}^2(\phi)/2\right)\]
where $\sigma_{D \setminus A}^2(\phi) = \iint \phi(x) G_{D \setminus A}(x,y) \phi(y) dx dy$ and $G_{D \setminus A}$ is the Green's function for $\Delta$ on $D \setminus A$.  That is, for each fixed $\phi$ we have that the conditional law of $(h,\phi) - \E[ (h,\phi) \giv \CA]$ given $\CA$ is equal to that of $(\wt{h},\phi)$ where, given $\CA$, $\wt{h}$ has the law of a zero-boundary GFF on $D \setminus A$.  This implies that we can find a coupling of $\wt{h}$ and $h,A$ so that for each $\phi \in C_0^\infty(D)$ we almost surely have that $(\wt{h},\phi) = (h,\phi) - \E[ (h,\phi) \giv \CA]$.  By letting $(\phi_j)$ be a suitably chosen family in $C_0^\infty(D)$, it is easy to see that the family of random variables $(h,\phi_j) - \E[ (h,\phi_j) \giv \CA]$ a.s.\ determines $\wt{h}$.  Thus using this, we set $(\CC_A,\phi) = (h,\phi) - (\wt{h},\phi)$ and see that this a.s.\ defines a distribution on $D$ since $h,\wt{h}$ are a.s.\ distributions on $D$.  Therefore~\eqref{it::Acond} holds for $A$.

Now~\eqref{it::twostep} is immediate from~\eqref{it::Acond} when we set $h_1 = \CC_A$.

To obtain~\eqref{it::U2cond} from~\eqref{it::twostep}, if suffices to show that given the projection of $h$ onto $H^\perp(U)$ {\em and} the pair $(S, \wt A)$, the conditional law of the projection of $h$ onto $H(U)$ is the same as its a priori law (or its law conditioned on only the projection of $h$ onto $H^\perp(U)$), namely the law of the zero boundary GFF on $U$.  Assuming~\eqref{it::twostep}, we can write $h = h_1 + h_2$ where, conditional on $\CA$, we have that $h_1$ is harmonic in $D \setminus A$ and $h_2$ has the law of a zero-boundary GFF in $D \setminus A$.  On the event that $A \subseteq D \setminus U$ (which we emphasize is $\CA$-measurable), we can apply the Markov property to the GFF $h_2$ so that, given the $\sigma$-algebra $\CF_1$ generated by $\CA$ and $\CF_{( (D \setminus A) \setminus U)^+}^{h_2}$ we have that
	\[ h_2 = h_{21} + h_{22}\]
where $h_{21}$ is harmonic on $U$ and $h_{22}$ has the law of a GFF on $U$ with zero boundary conditions.  We claim that $\CF_1$ is the same as the $\sigma$-algebra $\CF_2$ generated by $\CA$ and $\CF_{(D \setminus U)^+}^h$.  (It is intuitively obvious that this should be true because both $\sigma$-algebras depend on the values of the GFF in the same way, we can build $h$ from $h_1$ and $h_2$, and we can also build $h_2$ from $h_1$ and $h$.  As we will see below, this intuition is what leads to the proof.)  Upon showing this, we will have (on $A \subseteq D \setminus U$) that the conditional law of $h$ given $\CA$ and $\CF_{(D \setminus U)^+}^h$ can be written as
\[ h = h_1 + h_{21} + h_{22}\]
where $h_1$ is a distribution on $D$ which is harmonic on $D \setminus A$, $h_{21}$ is a distribution on $D \setminus A$ which is harmonic on $U$, and $h_{22}$ is a zero-boundary GFF on $U$ and $h_1,h_{21}$ are determined by $\CA$ and $\CF_{(D \setminus U)^+}^h$.  This, in particular will imply that the conditional law of the projection of $h$ onto $H(U)$ given $\CA$ and $\CF_{(D \setminus U)^+}^h$, on the event that $A \subseteq D \setminus U$, is that of a zero-boundary GFF on $U$.

To see the claim, one simply has to note that, on $A \subseteq D \setminus U$, conditioning on either $\CF_1$ or $\CF_2$ is equivalent to conditioning on $\sigma(h_1,h_{21})$.  Indeed, we note that this is obviously true in the case of $\CF_1$ from how it is defined.  To see that conditioning on $\CF_2$ is equivalent to conditioning on $\sigma(h_1,h_{21})$ on $A \subseteq D \setminus U$, we first note that $h_1$ is obviously $\CF_2$-measurable.  Let $\CG_{2,\delta}$ be the $\sigma$-algebra generated by $\CA$ and the values of $h_2$ in $((D \setminus A) \setminus U)_\delta$.  That is, $\CG_{2,\delta}$ is the $\sigma$-algebra generated by $h_1$ and, with $S_\delta(\phi)$ the event that the support of $\phi$ is contained in $((D \setminus A) \setminus U)_\delta$, the random variables $(h_2,\phi) \one_{S_\delta(\phi)}$ for $\phi \in C_0^\infty(D)$.  Then we have that $\CG_{2,\delta}$ is contained in the $\sigma$-algebra generated by $h_1$ and $\CF_{(D \setminus U)_\delta}^h$ because $(D \setminus U)_\delta$ contains $((D \setminus A) \setminus U)_\delta$ and, for any $\phi \in C^\infty$ with support contained in $((D \setminus A) \setminus U)_\delta$ we have that $(h_2,\phi)$ is determined by $(h,\phi)$ and $(h_1,\phi)$.  Thus, by taking an intersection over $\delta > 0$, by Proposition~\ref{prop::gff_sigma_algebra_converges} applied to the conditional law of $h$ given $\CA = \sigma(h_1)$, this implies that (on $A \subseteq D \setminus U$) $\sigma(h_1,h_{21}) \subseteq \CF_2$.  For the reverse inclusion, we have that $h_1$ is $\sigma(h_1,h_{21})$-measurable.  Moreover, $(D \setminus U)_\delta \subseteq A_\delta \cup ((D \setminus A) \setminus U)_\delta$.  Therefore $\CF_{(D \setminus U)_\delta}^h$ is contained in the $\sigma$-algebra generated by $\CA$ and $(h_2,\phi) \one_{T_\delta(\phi)} = ((h,\phi)-(h_1,\phi)) \one_{T_\delta(\phi)}$ for $\phi \in C_0^\infty(D)$ where $T_\delta(\phi)$ is the event that the support of $\phi$ is contained either in $A_\delta$ or in $((D \setminus A) \setminus U)_\delta$.  Recall that the conditional law of $h_2$ given $\CA$ is that of a GFF on $D \setminus A$ with zero boundary conditions.  Therefore, as $\delta \to 0$, the $\sigma$-algebra generated by $\CA$ and the values of $h_2$ in $A_\delta$ (defined analogously to the case of $\CG_{2,\delta}$) decreases to the $\sigma$-algebra generated by $\CA$.  Moreover, it follows from Proposition~\ref{prop::gff_sigma_algebra_converges} that $\CG_{2,\delta}$ decreases to the $\sigma$-algebra generated by $\CA$ and $\sigma(h_{21})$.  By sending $\delta \to 0$, we thus see that $\CF_2 \subseteq \sigma(h_1,h_{21})$.

Now, we know from the Markov property (Proposition~\ref{gff::prop::markov}) for the GFF $h$ by itself that we can also write
	\[ h = \wh{h}_1 + \wh{h}_2\]
where $\wh{h}_1$ is a distribution which is harmonic on $U$ and the conditional law of $\wh{h}_2$ given $\CF_{(D \setminus U)^+}$ is that of a GFF on $U$.  Combining our two expressions tells us that \emph{further conditioning on $\CA$ tells us nothing about the projection of $h$ onto $H(U)$ on the event that $A \subseteq D \setminus U$}.  In other words, the conditional law of the projection of $h$ onto $H(U)$ given $\CF_{(D \setminus U)^+}$, $\CA$, and $A \subseteq D \setminus U$ is \emph{equal} to the conditional law of the projection of $h$ onto $H(U)$ given $\CF_{(D \setminus U)^+}$.  This implies~\eqref{it::U2cond}.
\end{proof}

\begin{figure} [h!]
\begin{center}
\includegraphics[scale=0.85]{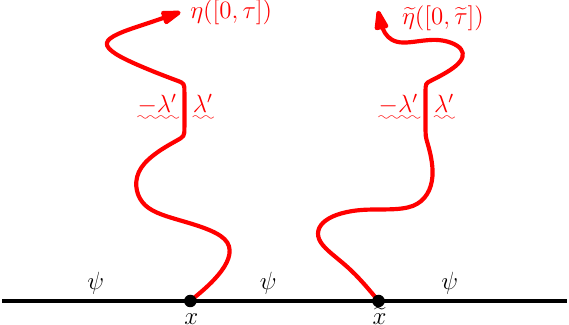}
\end{center}
\caption{\label{fig::multiple_sle}Proposition~\ref{gff::prop::cond_union_local} allows us to extend Theorem~\ref{thm::coupling_existence} to the setting of multiple $\SLE$-related paths by taking their conditionally independent union.  The illustration depicts a zero-boundary GFF plus an unspecified harmonic function $\psi$ coupled with two $\SLE$ processes $\eta, \wt{\eta}$ (possibly with different $\kappa$ values) emanating from distinct points $x,\wt{x} \in \partial \h$.  If $\tau,\wt{\tau}$ are $\eta,\wt{\eta}$ stopping times, respectively, then the field given $\eta([0,\tau]), \wt{\eta}([0,\wt{\tau}])$ is the sum of a zero boundary GFF and a harmonic function, whose boundary values are given above.  One can also change the angle of one or both paths by adding a constant to the boundary conditions along that path.  It is not obvious {\em a priori} what the boundary conditions should be at locations where the paths intersect (we will treat this issue systematically in Section~\ref{sec::interacting}).}
\end{figure}

Since $K_\tau$ is local for $h$, we can write $h = h_1 + h_2$ where $h_1$ is harmonic in $\h \setminus K_\tau$ and the conditional law of $h_2$ given $h_1$ is that of a zero-boundary GFF on $\h \setminus K_\tau$.  Equivalently, we can write $h = \wt{h} \circ f_\tau + h_1$ where $h_1$ is harmonic in $\h \setminus K_\tau$ and $\wt{h}$ is a zero-boundary GFF on $\h$ independently of $f_\tau$ and $h_1$.  Theorem~\ref{thm::coupling_existence} also implies that on $\h \setminus K_\tau$, we have that $h = \wt{h} \circ f_\tau + \Fh_\tau$ where $\wt{h}$ is a zero-boundary GFF given $f_\tau$.  Combining, it is not difficult to see that $h_1 = \Fh_\tau$ on $\h \setminus K_\tau$.  (We will see later in Lemma~\ref{lem::local_restriction_determine} that under certain weak hypotheses which are satisfied by $\SLE_\kappa$ for $\kappa \in (0,8)$ \cite{RS05} we have that $h_1$ is determined by its restriction to $\h \setminus K_\tau$.  That is, $h_1$ is determined by $\Fh_\tau$.)

Let us now elaborate further on the remark made just after the proof of Proposition~\ref{prop::gff_abs_continuity} regarding the Markov property when one performs a change of measure to the GFF which corresponds to changing its boundary data away from the starting point of the $\SLE$.  We will use the same notation introduced just after Proposition~\ref{prop::gff_abs_continuity} and write $K_t$ for the hull of $\eta([0,t])$.

Suppose that $\phi \in H(\h)$.  Then the Radon-Nikodym derivative of the law of $h+\phi$ with respect to the law of $h$ is given by a normalizing constant times $\exp( (h,\phi)_\nabla)$ (this is the infinite dimensional analog of the fact that if $Z \sim N(0,1)$ and $\mu \in \R$ then the Radon-Nikodym derivative of the law of $Z + \mu$ with respect to the law of $Z$ is given by a normalizing constant times $e^{\mu x}$).  Let $\tau$ be any stopping time for~$\eta_U$.  Then we have that $h = h_1+h_2$ where $h_2 = \wt{h} \circ f_\tau$ and $\wt{h}$ has the law of a zero-boundary GFF on~$\h$ which is independent of the pair $h_1,f_\tau$ and $h_1$ is a distribution defined on all of $\h$ which is harmonic in $\h \setminus K_\tau$.  As the Dirichlet inner product is conformally invariant, we have that
\[ \| \phi \circ f_\tau^{-1}\|_\nabla = \| \phi |_{\h \setminus K_\tau} \|_\nabla \leq \| \phi \|_\nabla < \infty.\]
This implies that it makes sense to take the $(\cdot,\cdot)_\nabla$-inner product of the partial sums in the series expansion of $\wt{h}$ with $\phi \circ f_\tau^{-1}$ and that these partial sums converge as $n \to \infty$.  Moreover, using the conformal invariance of the Dirichlet inner product, we also have that
$(\wt{h} \circ f_\tau,\phi)_\nabla = (\wt{h},\phi \circ f_\tau^{-1})_\nabla$ as this equality holds when one replaces $\wt{h}$ with the partial sums in the series expansion for $\wt{h}$.  If $\phi \in C_0^\infty(\h)$, then we can define $(h_1,\phi)_\nabla$ by taking it to be equal to $-\tfrac{1}{2\pi} (h_1,\Delta \phi)$, using that this integral is defined as $h_1$ is a distribution on all of $\h$.  With this definition, we have that $(h_1,\phi)_\nabla = (h - \wt{h} \circ f_\tau,\phi)_\nabla$ for all $\phi \in C_0^\infty(\h)$.  If $\phi \in H(\h)$, then we can find a sequence $(\phi_n)$ in $C_0^\infty(\h)$ which converges to $\phi$ in $H(\h)$.  Since $(h - \wt{h} \circ f_\tau,\phi_n)_\nabla \to (h - \wt{h} \circ f_\tau,\phi)_\nabla$ almost surely (at least along a subsequence) we can define $(h_1,\phi)_\nabla$ by taking it to be equal to $(h - \wt{h} \circ f_\tau,\phi)_\nabla$.  This gives a definition of $(h_1,\phi)_\nabla$ for all $\phi \in H(\h)$ (which, as in the case of the GFF, is defined for each $\phi$ up to a set of measure zero which \emph{a priori} depends on $\phi$).   We can write
\begin{align*}
      \exp( (h,\phi)_\nabla)
&= \exp( (\wt{h} \circ f_\tau + h_1,\phi)_\nabla)\\
&= \exp( (\wt{h},\phi \circ f_\tau^{-1})_\nabla) \times \exp( (h_1,\phi)_\nabla).
\end{align*}
Note that $(\wt{h},\phi \circ f_\tau^{-1})_\nabla$ is equal to the $(\cdot,\cdot)_\nabla$-inner product of $\wt{h}$ and the $(\cdot,\cdot)_\nabla$-orthogonal projection $\wt{\phi}_\tau$ of $\phi \circ f_\tau^{-1}$ onto $H(\h)$ as the series expansion for $\wt{h}$ is given in terms of an orthonormal basis of $H(\h)$.  Equivalently, $\wt{\phi}_\tau$ is given by subtracting from $\phi \circ f_{\tau}^{-1}$ the harmonic extension of its values from $\partial \h$ to $\h$.  For $x \in \partial \h$, we note that $\phi \circ f_{\tau}^{-1}(x)$ can only be non-zero provided $x \in f_\tau(\eta_U([0,\tau]))$ as $\phi \in H(\h)$.  Therefore the conditional law of $\wt{h}$ given $f_\tau,h_1$ under the weighted law is equal to that of a zero-boundary GFF on~$\h$ plus $\wt{\phi}_\tau$.  This implies that under the weighted law we have that
\begin{equation}
\label{eqn::weighted_markov}
h = \wh{h} \circ f_\tau + \wt{\phi}_\tau \circ f_\tau + h_1
\end{equation}
where $\wh{h}$ has the law of a zero-boundary GFF on~$\h$ given $f_\tau$, $\wt{\phi}_\tau$, and $h_1$.

Suppose that $\Fh$ is a function which is harmonic in $\h$ with zero boundary values on a neighborhood of $\partial U \cap \partial \h$.  We then take $\wh{\phi} \in C^\infty(\h)$ so that $\wh{\phi} \equiv 1$ on $U$ and $\wh{\phi} \equiv 0$ on a small enough neighborhood of $U$ so that $\phi = \Fh \wh{\phi}$ vanishes on $\partial \h$.  Thus $\phi \in H(\h)$ so that we may apply the above for this choice of $\phi$.  As mentioned earlier, we have that $h_1 = \Fh_\tau$ in $\h \setminus K_\tau$.  Observe that $\wt{\phi}_\tau \circ f_\tau + \Fh_\tau$ (for this choice of $\phi$) restricted to $U$ is equal to the restriction to $U$ of the function $\wh{\Fh}_\tau$ which is harmonic in $\h \setminus K_\tau$ with boundary values given by those of $\Fh_\tau$ on $\h \cap \partial K_\tau$ and those of $\Fh_\tau + \Fh$ on $\partial \h \setminus K_\tau$.  That is, we have that
\begin{equation}
\label{eqn::weighted_markov2}
h = \wh{h} \circ f_\tau + \wh{\Fh}_\tau
\end{equation}
in $\h \setminus K_\tau$.  This proves the claimed Markov property because we can always transform the law of the field restricted to~$U$ to the law of the field restricted to $U$ with boundary values which differ outside of~$U$ by adding such a function to the field.

One important property of local sets is that given local sets $A_1$ and $A_2$, the {\em conditionally independent union} $A_1 \wt \cup A_2$ (defined in the proposition statement below) is also local.  This result is contained in the following restatement of \cite[Lemma~3.10]{SchrammShe10}; see also Figure~\ref{fig::multiple_sle}.

\begin{proposition}
\label{gff::prop::cond_union_local}
Suppose $h$ is a GFF on $D$, $A_1,A_2$ are random variables taking values in $\Gamma$, and that $(A_1,h)$ and $(A_2,h)$ are couplings for which $A_1$ and $A_2$ are local.  Let $A = A_1 \wt{\cup} A_2$ denote the random variable taking values in $\Gamma$ which is given by first sampling $h$, then sampling $A_1, A_2$ independently from their conditional laws given $h$, and then taking the union of $A_1$ and $A_2$.  Then $A$ is also a local set of $h$.  Moreover, given $\CA$ and the pair $(A_1,A_2)$, the conditional law of $h$ is given by the sum of $\CC_A$ plus an instance of the GFF on $D \setminus A$.
\end{proposition}
\begin{proof}
We use characterization~\eqref{it::U2cond} for locality as given in Lemma~\ref{lem::local_char} and observe that \cite[Lemma~3.5]{SchrammShe10} implies the analogous result holds for the quadruple $(S_1, \wt A_1, S_2, \wt A_2)$---namely, that for each deterministic open $U \subseteq D$, we have that {\em given} the projection of $h$ onto $H^\perp(U)$ and the quadruple $(S_1, \wt A_1, S_2, \wt A_2)$, the conditional law of the projection of $h$ onto $H(U)$ is the law of the GFF on $U$.

The proof that this analog of~\eqref{it::U2cond} implies the corresponding analog of~\eqref{it::Acond} in the statement of Proposition~\ref{gff::prop::cond_union_local} is essentially the same as the proof of the equivalence of~\eqref{it::U2cond} and~\eqref{it::Acond}.
\end{proof}

\noindent We say that a local set $A$ of $h$ is \emph{almost surely determined} by $h$ if there exists a modification of $A$ which is $\sigma(h)$-measurable.  Many of the local sets we will work with in this article will be almost surely determined by the corresponding GFF, in which case the conditionally independent union is almost surely the same as an ordinary union.

The following proposition (see \cite[Lemma~3.11]{SchrammShe10}) allows us to estimate $\CC_{A_1 \wt{\cup} A_2}$ near connected components of $A_1 \setminus A_2$ and $A_1 \cap A_2$ which consist of more than a single point in terms of $\CC_{A_2}$.

\begin{proposition}
\label{gff::prop::cond_union_mean}
Assume that $D$ is a bounded, simply connected domain.  Let $A_1,A_2$ be connected local sets.  Then $\CC_{A_1 \wt{\cup} A_2} - \CC_{A_2}$ is almost surely a harmonic function in $D \setminus (A_1 \wt{\cup} A_2)$ that tends to zero on all sequences of points in $D \setminus (A_1 \wt{\cup} A_2)$ that tend to a limit in either:
\begin{enumerate}[(i)]
\item \label{it::a2_minus_a1} a connected component of $A_2 \setminus A_1$ (consisting of more than a single point) or 
\item \label{it::a1_and_a2} a connected component of $A_1 \cap A_2$ (consisting of more than a single point) at a point that has positive distance from either $A_1 \setminus A_2$ or $A_2 \setminus A_1$.
\end{enumerate}
\end{proposition}
\begin{proof}
We are going to give the argument for~\eqref{it::a2_minus_a1}; the argument for~\eqref{it::a1_and_a2} is analogous.

By Proposition~\ref{gff::prop::cond_union_local}, the union $A_1 \wt{\cup} A_2$ is itself a local set, so $\CC_{A_1 \wt{\cup} A_2}$ is well defined.  Now, conditioned on $\CA_1$ the law of the field in $D \setminus A_1$ is given by a GFF in $D \setminus A_1$ plus $\CC_{A_1}$.  We next claim that $\ol{A_2 \setminus A_1}$ is a local subset of $D \setminus A_1$, with respect to this GFF on $D \setminus A_1$.  To see this, note that characterization~\eqref{it::Acond} for locality from Lemma~\ref{lem::local_char} follows from the latter statement in Proposition~\ref{gff::prop::cond_union_local}.

By replacing $D$ with $D \setminus A_1$ and subtracting $\CC_{A_1}$, we may thus reduce to the case that $A_1$ is deterministically empty and $\CC_{A_1} = 0$.  What remains to show is that if $A$ is any local set on $D$ then $\CC_A$ (when viewed as a harmonic function on $D \setminus A$) tends to zero almost surely along all sequences of points in $D \setminus A$ that approach a point $x$ that lies in a connected component of $\partial D \setminus \wt{A}$, $\wt{A}$ given by the closure of $A \cap D$, that consists of more than a single point.

If we fix a neighborhood $U_1$ of $x$ and another neighborhood $U_2$ whose distance from $U_1$ is positive, then the fact that the statement holds on the event $A \cap D \subseteq U_2$ is immediate from Proposition~\ref{prop::gff_boundary_values}. 
\end{proof}

In the setting of Proposition~\ref{gff::prop::cond_union_mean}, we note that a given point in $A_2 \setminus A_1$ (corresponding to part~\eqref{it::a2_minus_a1}) or in $A_1 \cap A_2$ (corresponding to part~\eqref{it::a1_and_a2}) may be associated with multiple prime ends.  Proposition~\ref{gff::prop::cond_union_mean} implies that~$\CC_{A_1 \wt{\cup} A_2}$ has the same boundary behavior as~$\CC_{A_2}$ near each such prime end because we can always choose the sequence so that it approaches the given prime end.

Proposition~\ref{gff::prop::cond_union_local} and Proposition~\ref{gff::prop::cond_union_mean} allow us to extend Theorem~\ref{thm::coupling_existence} to the setting of coupling multiple $\SLE$s with the free field by taking their conditionally independent union (once Theorem~\ref{thm::coupling_uniqueness} is established, we can replace the conditionally independent union with a usual union).  Figure~\ref{fig::multiple_sle} contains an illustration of this result in the case of two (counter)flow lines of the same field emanating from different points.  See \cite[Lemma~6.1 and Theorem~6.4]{DUB_PART} for another approach to constructing couplings with multiple $\SLE$s.

The argument described after the proof of Lemma~\ref{lem::local_char} implies that if we have a path $\eta$ coupled with a GFF $h$ on $\h$ as a flow line (or counterflow line), then weighting the law of the field/path pair $h,\eta$ restricted to a subdomain $U$ in a way that produces a GFF with boundary conditions which agree with those of $h$ on $\partial \h \cap \partial U$ (but possibly differ elsewhere) yields a flow line (or counterflow line) of the new GFF which is defined at least up until when it first exits~$U$.  Moreover, Theorem~\ref{thm::martingale} allows one to identify the law of this flow line.  We will now explain a variant of this in which the domains on which the GFFs are defined are different.  Specifically, we suppose that $h$ is a GFF on $\h$ with piecewise constant boundary data which changes only finitely many times and that~$\eta$ is a flow line of $h$ starting from $0$.  Suppose that $D \subseteq \h$ is a simply connected domain whose boundary contains a segment of $\partial \h$ which is a neighborhood of $0$.  Fix a simply connected domain $U \subseteq D$ whose boundary also contains a segment of $\partial \h$ which is a neighborhood of $0$ and satisfies $\dist(\partial U \setminus \partial \h, \partial D \setminus \partial \h) > 0$.  Let $\tau_U$ be the first time that $\eta$ exits $U$ and let $\tau$ be a stopping time for $\eta$ such that $\p[ \tau \leq \tau_U] = 1$.  Note that both $\partial D$ and $\eta([0,\tau])$ are local for $h$.  Since the former is deterministic, we therefore have by Proposition~\ref{gff::prop::cond_union_local} that $\partial D \cup \eta([0,\tau])$ is local for $h$ and by Proposition~\ref{gff::prop::cond_union_mean} that the conditional law of $h$ given $\eta|_{[0,\tau]}$ and its values on $\partial D$ is given by that of the sum of a zero-boundary GFF $\wt{h}$ on $D \setminus \eta([0,\tau])$ plus the function $\Fg_\tau$ which is harmonic in $D \setminus \eta([0,\tau])$ with boundary values agreeing with those of $h$ on $\partial D$ and with $\Fh_\tau$ (as in Theorem~\ref{thm::coupling_existence}) on $\eta([0,\tau])$.

We will now argue that the boundary data for $\Fg_\tau$ in the conditional law of $h$ just above is correct for $\eta$ to be a flow line.  To this end, we fix $x \in \partial D \setminus \partial \h$ and let $\varphi \colon D \to \h$ be a conformal transformation which fixes $0$ and sends $x$ to $\infty$.  Let $\wt{\Fh}_\tau$ be as in the statement of Theorem~\ref{thm::coupling_existence} for $\varphi(\eta([0,\tau]))$  (in what follows, it will not matter how we choose the boundary data for $\wt{\Fh}_\tau$ on $\partial \h$).  Then it suffices to show that the boundary values of $\Fh_\tau$ along $\eta([0,\tau])$ agree with those of $\wt{\Fh}_\tau \circ \varphi - \chi \arg\varphi'$.  To see this, we let $f_\tau$ (resp.\ $\wt{f}_\tau$) be the centered Loewner map associated with $\eta$ (resp.\ $\varphi(\eta)$) at time $\tau$ so that $f_\tau$ (resp.\ $\wt{f}_\tau$) takes $\eta(\tau)$ (resp.\ $\varphi(\eta(\tau))$) to $0$.  Let $m$ be the function which is harmonic in $\h$ with boundary values given by $-\lambda$ (resp.\ $\lambda$) on $\R_-$ (resp.\ $\R_+$).  By definition, $\Fh_\tau$ has the same boundary behavior along $\eta([0,\tau])$ as the harmonic function
\[ m \circ f_\tau - \chi \arg f_\tau'\]
and likewise, by definition, $\wt{\Fh}_\tau$ has the same boundary behavior along $\varphi(\eta([0,\tau]))$ as the harmonic function
\[ m \circ \wt{f}_\tau - \chi \arg \wt{f}_\tau'.\]
Note that the function $m \circ \wt{f}_\tau \circ \varphi$ is harmonic in $\h \setminus \eta([0,\tau])$ with boundary values along the left (resp.\ right) side of $\eta([0,\tau])$ and $\R_-$ (resp.\ $\R_+$) given by $-\lambda$ (resp.\ $\lambda$).  These are the same as the boundary values of $m \circ f_\tau$ and therefore $m \circ f_\tau = m \circ \wt{f}_\tau \circ \varphi$.  We will next show that 
\begin{equation}
\label{eqn::arg_equiv0}
\lim_{\substack{z \to z_0\\z \in D \setminus \eta([0,\tau])}} \left( \arg f_\tau'(z) - \left(  \arg ( \wt{f}_\tau' \circ \varphi(z) ) + \arg \varphi'(z) \right) \right) = 0 \quad\text{for}\quad z_0 \in \eta([0,\tau]).
\end{equation}
Let $I_\tau = f_\tau(\eta([0,\tau])) \subseteq \partial \h$ and $D_\tau = f_\tau(D \setminus \eta([0,\tau]))$.  By precomposing both sides of~\eqref{eqn::arg_equiv0} with $f_\tau^{-1}$ it suffices to show that
\begin{equation}
\label{eqn::arg_equiv1}
\begin{split}
\lim_{\substack{z \to z_0\\ z \in D_\tau}} \left( \arg ( f_\tau' \circ f_\tau^{-1}(z) ) - \left( \arg ( \wt{f}_\tau' \circ \varphi \circ f_\tau^{-1}(z) ) + \arg \big( \varphi' \circ f_\tau^{-1}(z) \big) \right)\right) = 0\\ \quad\text{for}\quad z_0 \in I_\tau.
\end{split}
\end{equation}
By adding and subtracting $\arg (f_\tau^{-1})'$ to the left side of~\eqref{eqn::arg_equiv1} and using that
\begin{align*}
   &\arg ( f_\tau' \circ f_\tau^{-1} ) + \arg (f_\tau^{-1})' = \arg (f_\tau \circ f_\tau^{-1})' = 0 \quad\text{and}\\
    & \arg ( \wt{f}_\tau' \circ \varphi \circ f_\tau^{-1} ) + \arg \big( \varphi' \circ f_\tau^{-1} \big) + \arg( f_\tau^{-1})' = \arg (\wt{f}_\tau \circ \varphi \circ f_\tau^{-1})',
\end{align*}
we see that~\eqref{eqn::arg_equiv1} is equivalent to showing that
\begin{equation}
\label{eqn::arg_equiv2}
\lim_{\substack{z \to z_0 \\ z \in D_\tau}} \arg (\wt{f}_\tau \circ \varphi \circ f_\tau^{-1})'(z) = 0 \quad\text{for}\quad z_0 \in I_\tau.
\end{equation}
We conclude that~\eqref{eqn::arg_equiv2} holds because $\wt{f}_\tau \circ \varphi \circ f_\tau^{-1}$ is conformal map from $D_\tau$ into $\h$ which takes the interval $I_\tau$ of $\partial \h$ corresponding to $f_\tau(\eta([0,\tau]))$ to an interval of $\h$ and is orientation preserving.  This proves~\eqref{eqn::arg_equiv0}, which completes the proof of the claim.

In summary, we have argued that the conditional law of $(h|_D,\eta|_{[0,\tau_U]})$ given the values of $h$ on $\partial D$ is given by a pair consisting of:
\begin{itemize}
\item a GFF on $D$ whose boundary values agree with those of $h$ along $\partial D$ and
\item a path in $D$ with the property that if $\tau$ is any stopping time which a.s.\ occurs before $\tau_U$ we have that the conditional law of the field given $\eta|_{[0,\tau]}$ is that of a GFF on $D \setminus \eta([0,\tau])$ with the same boundary values as $h$ on $\partial D$ and the correct boundary values along $\eta([0,\tau])$ so that it is a flow line.
\end{itemize}
In other words, conditionally on the values of $h$ on $\partial D$, we have that the coupling $(h|_D,\eta|_{[0,\tau_U]})$ is that of a GFF on $D$ and a flow line of $h|_D$ stopped upon exiting $U$.

Suppose that $\wt{h}$ is a GFF on $D$ with boundary values which agree with those of $h$ on a neighborhood of $0$ in $\partial \h$.  Then, conditionally on the values of $h$ on $\partial D$, we have that both $h|_D$ and $\wt{h}$ are GFFs on \emph{the same domain} $D$ with boundary data that agrees on a neighborhood of $0$ in $\partial \h$.  Since we can transform from the conditional law of $h|_D$ given its values on $\partial D$ to the law of $\wt{h}$ by adding to the former a function which is harmonic in $D$ with zero boundary values on $\partial D \cap \partial \h$, it follows from the procedure described just after Lemma~\ref{lem::local_char} that the result of weighting the law of the field/path pair $h,\eta$ in such way so that the restriction of the field to $U$ agrees with the law of $\wt{h}|_U$ yields a flow line of $\wt{h}$ which is defined up until the first time it exits $U$.

\begin{proposition}
\label{gff::prop::local_independence}
Suppose that $D$ is a bounded, simply connected domain and that $A_1,A_2$ are connected local sets which are conditionally independent given $h$.  Suppose that $C$ is a $\sigma(A_1)$-measurable open subset of $D \setminus A_1$ which can be written as a union of components of $D \setminus A_1$ such that $C \cap A_2 = \emptyset$ almost surely.  Then $\CC_{A_1 \cup A_2}|_{C} = \CC_{A_1}|_{C}$ almost surely.  In particular, $h|_C$ is independent of the pair $(h|_{D \setminus C},A_2)$ given~$\CA_1$.
\end{proposition}
\begin{proof}
This follows from the argument used to prove Proposition~\ref{gff::prop::cond_union_mean}.
\end{proof}

A simple example of the type of application we have in mind for Proposition~\ref{gff::prop::local_independence} is the following.  Suppose that $h$ is a GFF on $\h$ and $\eta_1,\eta_2$ are flow lines of $h$ starting from $0$ which intersect $\partial \h$ only at $\{0\}$.  Suppose further that~$\eta_2$ almost surely lies to the left of~$\eta_1$.  Then Proposition~\ref{gff::prop::local_independence} implies that the restriction of~$h$ to the left side of $\h \setminus \eta_2$ is independent of the pair consisting of the restriction of $h$ to the right side of $\h \setminus \eta_2$ and $\eta_1$, conditionally on $\eta_2$.

We end this subsection with a lemma which gives a simple condition under which $h$ is determined by its restriction to $D \setminus A$ for a local set $A$.  When this condition holds, it will in particular imply that $h$ is determined by the restriction of $\CC_A$ to $D \setminus A$ and the projection of $h$ onto $H(D \setminus A)$.  Informally, this means that there is no ``extra information'' which is only contained in $\CC_A$ on $A$ itself.

\begin{lemma}
\label{lem::local_restriction_determine}
Suppose that $A$ is a local set for $h$ such that for every compact set $K \subseteq D$ there exists a sequence $(\delta_k)$ of positive numbers with $\delta_k \to 0$ as $k \to \infty$ such that we almost surely have that the number of squares with corners in $\delta_k \Z^2$ required to cover $A \cap K$ is $o( \delta_k^{-2} (\log \delta_k^{-1})^{-1})$ as $k \to \infty$.  Then $h$ is almost surely determined by the restriction $\wt{h}$ of $h$ to $D \setminus A$.
\end{lemma}
\begin{proof}
Fix a compact square $K \subseteq D$ and $\phi \in C_0^\infty(D)$ with $\supp(\phi) \subseteq K$.  For each $\delta > 0$, we let $\CD_\delta$ be the collection of half-open squares $(a,b] \times (c,d]$ of side length $\delta$ contained in $K$ with corners in $\delta \Z^2$.  We also let $\phi_\delta$ be the function on $K$ whose common value on each $S \subseteq \CD_\delta$ is given by the average of $\phi$ on $S$.  By bounding the variance of $(h,\phi-\phi_\delta)$, it is easy to see that we almost surely have
\begin{equation}
\label{eqn::phi_delta_approx}
(h,\phi_\delta) \to (h,\phi) \quad\text{as}\quad \delta \to 0,
\end{equation}
at least if the limit is taken along a subsequence of $(\delta_k)$ which tends to $0$ sufficiently quickly.  Let $A_\delta$ be the union of the set of squares in $\CD_\delta$ which intersect $A$.  Then it suffices to show that we almost surely have
\begin{equation}
\label{eqn::phi_delta_a_delta_c_approx}
\left|(\wt{h},\phi_\delta \one_{A_\delta^c}) - (h,\phi_\delta)\right| = \left|(h,\phi_\delta \one_{A_\delta^c}) - (h,\phi_\delta)\right| \to 0 \quad\text{as}\quad \delta \to 0,
\end{equation}
at least if the limit is taken along a subsequence of $(\delta_k)$ which tends to $0$ sufficiently quickly.  An argument analogous to the proof of \cite[Proposition~3.2]{DS08} implies that the variance of $(h,\delta^{-2} \one_S)$ for $S \in \CD_\delta$ is $O( \log \delta^{-1})$ where the implicit constant is uniform in $\delta > 0$.  It therefore follows from the Gaussian tail bound
\[ \p[ Z \geq \lambda] \sim \frac{1}{\sqrt{2\pi} \lambda} e^{-\lambda^2/2} \quad\text{as}\quad \lambda \to \infty\]
for $Z \sim N(0,1)$ (see, e.g., \cite[Lemma~A.4]{HMP10}) and the Borel-Cantelli lemma that there exists a constant $c > 0$ such that the average of $h$ on each $S \in \CD_{\delta}$ is at most $c \log \delta^{-1}$, at least along a subsequence of $(\delta_k)$ tending to $0$ sufficiently quickly.  Consequently, with $N_{\delta}$ equal to the number of squares in $\CD_{\delta}$ which $A$ intersects we almost surely have that
\begin{equation}
\label{eqn::bad_squares_do_not_contribute}
\left| (h,\phi_\delta\one_{A_\delta^c}) - (h,\phi_\delta) \right| \leq \| \phi \|_\infty \times N_\delta \times \delta^2 \times c \log \delta^{-1} = o(1) \quad\text{as}\quad \delta \to 0,
\end{equation}
at least along a subsequence of $(\delta_k)$ tending to $0$ sufficiently quickly.  The equality in~\eqref{eqn::bad_squares_do_not_contribute} follows because we have assumed that we almost surely have that $N_{\delta_k}$ is $o( \delta_k^{-2} (\log \delta_k^{-1})^{-1})$ as $k \to \infty$.  The result thus follows by combining~\eqref{eqn::phi_delta_approx}, \eqref{eqn::phi_delta_a_delta_c_approx}, and~\eqref{eqn::bad_squares_do_not_contribute}.
\end{proof}

\subsection{Proof of Theorem~\ref{thm::coupling_existence}}
\label{subsec::existence_proof}

The purpose of this section is to prove Theorem~\ref{thm::coupling_existence}.  As we mentioned before, many of the steps in the proof given below are slight generalizations of those from \cite[Section~4]{SHE_WELD}.  Let $W$ and $V^{i,q}$ be a solution to the $\SLE_\kappa(\ul{\rho})$ SDE as in Definition~\ref{def::slekrdef} stopped upon hitting the continuation threshold, let $(g_t)$ be the chordal Loewner evolution driven by $W$, and let $f_t = g_t - W_t$.  We let $(\CF_t)$ be the filtration generated by $(W,V^{i,q})$ and $B$ as in Definition~\ref{def::slekrdef}.  Then $(\CF_t)$ is right-continuous, i.e.\ $\CF_t = \cap_{s > t} \CF_s$ for each $t \geq 0$.  This property will be important for us later on.  We also let $(K_t)$ be the corresponding family of hulls.  In what follows, it is important that the SDEs for our $\SLE_\kappa(\ul{\rho})$ driving process make sense in integrated form (as defined in Section~\ref{sec::sle}) so that all of the SDEs written below make sense in integrated form.  We begin by writing down the \Ito\ derivatives of the four processes $f_t(z)$, $\log f_t(z)$, $f_t'(z)$, and $\log f_t'(z)$.  Here, $f_t'(z)$ denotes the spatial derivative $\frac{\partial}{\partial z} f_t$.  For $z \in \h$ and $t < \tau(z) = \sup\{t \geq 0: \im(g_t(z)) > 0\}$, we have that
\begin{align*}
df_t(z) &= \left(\frac{2}{f_t(z)} - \sum_{q \in \{L,R\}} \sum_i \frac{\rho^{i,q}}{W_t - V_t^{i,q}} \right)dt- \sqrt{\kappa} dB_t,\\
d \log f_t(z) &= \left(\frac{(4-\kappa)}{2f_t(z)^2}  - \sum_{q \in \{L,R\}} \sum_i \frac{\rho^{i,q}}{f_t(z)(W_t - V_t^{i,q})} \right)dt - \frac{\sqrt{\kappa}}{f_t(z)}dB_t,\\
d f_t'(z) &= \frac{-2f_t'(z)}{f_t(z)^2}dt, \quad\text{and}\\
d \log f_t'(z) &= \frac{-2}{f_t(z)^2}dt.
\end{align*}

We next define the martingale $\Fh_t$ and compute its stochastic derivatives.  The calculations below will show that it is a local martingale (and the fact it is a martingale will be proved later).  Let $\chi = 2/\sqrt{\kappa} - \sqrt{\kappa}/2$ as in the statement of Theorem~\ref{thm::coupling_existence}.  We also let $\Fh_t^*(z)$ be given by $1/\sqrt{\kappa}$ times the expression in~\eqref{eqn::harmonic_form} where $f_t(x^{i,q}) = V_t^{i,q} - W_t$.  Then it is not hard to see that
\begin{align}
d\Fh_t^*(z) =  \frac{2}{f_t(z)} dB_t,\quad 
&\Fh_t(z) := \im\left(\Fh_t^*(z)\right),\quad \text{and} \notag\\
d \Fh_t(z) = &\im\left(\frac{2}{f_t(z)}\right) dB_t. \label{eqn::dfht}
\end{align}
See \cite[Remark~4.1 and Remark~4.2]{SHE_WELD} for some additional interpretation.

Since $\Fh_t(z)$ is a continuous local martingale for each fixed $z$, it is thus a Brownian motion under the quadratic variation parameterization, which we can give explicitly:
\begin{align*}
d\langle \Fh_t(z), \Fh_t(z) \rangle = -d C_t(z) \quad\text{and}\quad C_t(z) := \log \im\left( f_t(z)\right) - \re\left(\log f_t'(z)\right).
\end{align*}
If $z$ is a point in a simply connected domain $D$, and $\phi$ conformally maps the unit disk to $D$, with $\phi(0) = z$, then we refer to the quantity $|\phi'(0)|$ as the {\em conformal radius} of $D$ viewed from $z$.  If, in the above definition of conformal radius, we replaced the unit disk with $\h$ and $0$ with $i$, this would only change the definition by a multiplicative constant.  Thus, $C_t(z)$ is (up to an additive constant) the log of the conformal radius of $\h \setminus K_t$ viewed from $z$.  If the time parameter $-C_t(z)$ (which is increasing as a function of $t$) then $\Fh_t(z)$ is a Brownian motion.  The fact that $d\langle \Fh_t(z), \Fh_t(z) \rangle = -d C_t(z)$ may be computed directly via \Ito's formula but it is also easy to see by taking $y \to z$ in the formulas for $\langle \Fh_t(y), \Fh_t(z)\rangle$ and the formulas we will give below.

We will now show that weighted averages of $\Fh_t$ over multiple points in $\h$ are also continuous local martingales (and hence Brownian motions when properly parameterized).
The calculation will make use of the function
\begin{equation*}
G(y,z) := \log|y-\bar z| - \log |y- z|,
\end{equation*}
which is the Dirichlet Green's function for $\Delta$ on $\h$.  That is, $G$ is the distributional solution of $\Delta G(y,\cdot) = -2\pi \delta_y(\cdot)$ with zero boundary conditions where $\delta_y$ denotes the Dirac mass at $y$.

We let $G_t(y,z) = G(f_t(y), f_t(z))$ when $y$ and $z$ are both
in the infinite component of $\h \setminus K_t$; it is otherwise equal to zero.  Observe for $t < \tau(y) \wedge \tau(z)$ that
\begin{align}
dG_t(y,z) &= -\im\left(\frac{2}{f_t(y)}\right) \im\left(\frac{2}{f_t(z)}\right) dt \quad\text{and} \label{eqn::dgt}\\
d\langle \Fh_t(y), \Fh_t(z)\rangle &= -dG_t(y,z). \label{eqn::fh_cv}
\end{align}
The details of~\eqref{eqn::dgt} in the case of $\SLE_\kappa$ (with no $\rho$ values) is worked out explicitly in \cite[Section~4]{SHE_WELD}.  The case with $\rho$ values follows from a similar calculation.  It is then immediate from~\eqref{eqn::dfht} that~\eqref{eqn::fh_cv} holds.

Now we let $\phi$ be a smooth compactly supported function on $\h$, fix $U \subseteq \h$ open which contains the support $\supp(\phi)$ of $\phi$, and let
\begin{equation}
\label{eqn::tau_u}
\tau_U = \inf\{t \geq 0 : K_t \cap U \neq \emptyset\}.
\end{equation}
We let
\begin{align*}
E_t(\phi) &:= \iint_{\h \times \h} \phi(y) G_t(y,z) \phi(z) dydz
\end{align*}
and we will show that
\begin{align}
d\langle (\Fh_t,\phi), (\Fh_t,\phi)\rangle &= -dE_t(\phi) \quad\text{for}\quad t < \tau_U. \label{eqn::dfh_int}
\end{align}
A Fubini type calculation gives~\eqref{eqn::dfh_int} but it requires some justification. First, we claim that the $(\Fh_t, \phi)$ is a continuous martingale.  We first note that $\langle (\Fh_t, \phi) \rangle|_{[0,\tau_U]}$ is characterized by the fact that
\begin{align}
 (\Fh_t, \phi)^2 - \langle (\Fh_t, \phi) \rangle \quad&\text{for}\quad t < \tau_U \notag\\
 \intertext{is a continuous local martingale.  Thus it suffices to show that}
(\Fh_t, \phi)^2 +  E_t(\phi) \quad&\text{for}\quad t < \tau_U \label{eqn::intqvmartingale}
\intertext{is a continuous local martingale.  We know from the above calculations that}
\Fh_t(y)\Fh_t(z) + G_t(y,z) \quad&\text{for}\quad t < \tau_U \notag
\end{align}
is a continuous local martingale for fixed $y,z \in \supp(\phi)$.  For each $M < \infty$, we let 
\[\tau_{U,\phi,M} =\inf\left\{t \geq 0 : \sup_{z \in \supp(\phi)} |\Fh_t(z)| \geq M \right\} \wedge \tau_U.\]
Since $G_t(y,z)$ is non-increasing and $\sup_{z \in \supp(\phi)} |\Fh_t(z)| \leq M$ for $t \leq \tau_{U,\phi,M}$, we can use Fubini's theorem to conclude that~\eqref{eqn::intqvmartingale} is a continuous martingale up to time $\tau_{U,\phi,M}$.  Since $\tau_{U,\phi,M} \to \tau_U$ as $M \to \infty$, it follows that~\eqref{eqn::intqvmartingale} is a continuous local martingale.

We will now combine the above calculations to establish the following intermediate step in the proof of Theorem~\ref{thm::coupling_existence}.

\begin{lemma}
\label{lem::coupling_up_to_u}
Fix $U \subseteq \h$ open, let $\tau_U$ be as in~\eqref{eqn::tau_u}, and let $\tau$ be any $\CF_t$-stopping time such that $\p[\tau \leq \tau_U] = 1$.  Consider the random field $h_{U,\tau}$ on $U$ which is generated by:
\begin{enumerate}
\item Sampling $K_{\tau}$,
\item Sampling a zero boundary GFF $h^0$ on $\h \setminus K_\tau$ and adding it to $\Fh_{\tau}$, and then
\item Restricting the sum to $U$.
\end{enumerate}
Then $h_{U,\tau} \stackrel{d}{=} h|_U$ where $h = \wt{h} + \Fh_0$ and $\wt{h}$ is a zero-boundary GFF on $\h$.
\end{lemma}
\begin{proof}
Since $h|_U$ is a Gaussian field, it suffices to show that $(h_{U,\tau},\phi) \stackrel{d}{=} (h,\phi)$ for each $\phi \in C_0^\infty(\h)$ with $\supp(\phi) \subseteq U$.  Note that $\Fh_{\tau}$ is measurable with respect to $\CF_{\tau}$ and that $\var( (h^0,\phi)) = E_\tau(\phi)$.  Consequently, for each $\theta \in \R$, we have that
\begin{align*}
      \E\!\left[ \exp\big( i\theta ( h_{U,\tau},\phi) \big) \right]
&= \E\!\left[ \E\!\left[ \exp\big( i\theta (h^0,\phi \big) \ \big|\  \CF_{\tau} \right] \exp\big(i \theta (\Fh_{\tau},\phi) \big) \right]\\
&= \E\!\left[ \exp\!\left( i\theta (\Fh_{\tau},\phi) -\frac{\theta^2 E_{\tau}(\phi)}{2} \right) \right]\\
&= \E\!\left[ \exp\!\left( i \theta (\Fh_0,\phi) -\frac{\theta^2 E_0(\phi)}{2} \right) \right] \quad\text{(by~\eqref{eqn::dfh_int})}\\
&= \E\!\left[ \exp\big( i \theta (h,\phi) \big) \right].
\end{align*}
This proves the lemma.
\end{proof}

In order to work towards completing the proof of Theorem~\ref{thm::coupling_existence}, we will explain how to generalize Lemma~\ref{lem::coupling_up_to_u} to the setting in which it holds for multiple stopping times and open sets.  We begin with the setting of a finite number of stopping times.

\begin{lemma}
\label{lem::coupling_finite_stopping_times}
Assume that we have the setup of Lemma~\ref{lem::coupling_up_to_u} for a given open set $U \subseteq \h$.  Assume that $n \in \N$ and $\tau_1,\ldots,\tau_n$ are $\CF_t$-stopping times with $\p[ \tau_i \leq \tau_U] =1$ for each $1 \leq i \leq n$.  For each $1 \leq i \leq n$, let $h_{U,\tau_i}$ have the law as defined in Lemma~\ref{lem::coupling_up_to_u}.  There exists a coupling of the laws $h_{U,\tau_i}$ for $1 \leq i \leq n$ so that they are all generated using the same instance $(W,V^{i,q})$ of the $\SLE_\kappa(\ul{\rho})$ driving process such that for $h = \wt{h} + \Fh_0$, $\wt{h}$ a zero-boundary GFF on $\h$, we almost surely have that $h_{U,\tau_i} = h|_U$.  Moreover, we have that the conditional law of the projection of $h$ onto $H(U)$ given $\CF_{\tau_i}$ and the projection of $h$ onto $H^\perp(U)$ is that of a zero-boundary GFF on $H(U)$ for each $1 \leq i \leq n$.
\end{lemma}
\begin{proof}
For simplicity, we will first explain the proof in the case that $n=2$ and then later explain how to generalize it to arbitrary values of $n \in \N$.  Lemma~\ref{lem::coupling_up_to_u} implies that there exists a coupling of the law of the pair $(h_{U,\tau}, (W_{t \wedge \tau_U},V_{t \wedge \tau_U}^{i,q}))$ (where the latter element is the driving $\SLE_\kappa(\ul{\rho})$ process stopped at time $\tau_U$) with $h$ such that $h_{U,\tau} = h|_U$ almost surely.

Suppose that $\tau_1,\tau_2$ are $\CF_t$-stopping times with $\p[\tau_i \leq \tau_U] = 1$ for $i=1,2$.  By applying the above first with the stopping time $\sigma_1 = \tau_1 \wedge \tau_2$ and then with $\sigma_2 = \tau_1 \vee \tau_2$ (using the conformal Markov property of $\SLE_\kappa(\ul{\rho})$), it is easy to see that we can construct a coupling of the law of the triple $(h_{U,\sigma_1}, h_{U,\sigma_2}, (W_{t \wedge \tau_U},V_{t \wedge \tau_U}^{i,q}))$ with $h$ such that $h_{U,\sigma_i} = h|_U$ almost surely for $i=1,2$.  In particular, on the event that $\tau_1 = \sigma_1$, we have almost surely that
\begin{equation}
\label{eqn::tau_1_sigma_1}
h_{U,\tau_1} = h_{U,\sigma_1} = h|_U.
\end{equation}
On the event that $\tau_1 = \sigma_2$, we almost surely have that
\begin{equation}
\label{eqn::tau_1_sigma_2}
h_{U,\tau_1} = h_{U,\sigma_2} = h|_U.
\end{equation}
Since one of the two events $\tau_1 = \sigma_1$  or $\tau_1 = \sigma_2$ must hold, we conclude from~\eqref{eqn::tau_1_sigma_1} and~\eqref{eqn::tau_1_sigma_2} that $h_{U,\tau_1} = h|_U$ almost surely.  Similarly, we have that $h_{U,\tau_2} = h|_U$ almost surely.

By construction, we have that
\[ h_{U,\sigma_i} = (h^{\sigma_i} \circ f_{\sigma_i} + \Fh_{\sigma_i})|_U\]
where $h^{\sigma_i}$ for $i=1,2$ has the law of a zero-boundary GFF on $\h$ given $\CF_{\sigma_i}$.  We let
\[ h^{\tau_i} = h^{\sigma_1} \one_{\{\sigma_1 = \tau_i\}} + h^{\sigma_2} \one_{\{\sigma_2 = \tau_i\}}.\]
Then we have that
\[ h_{U,\tau_i} = (h^{\tau_i} \circ f_{\tau_i} + \Fh_{\tau_i})|_U\]
To finish the proof of the first assertion of the lemma for $n=2$, we need to show that $h^{\tau_i}$ for $i=1,2$ has the law of a zero-boundary GFF on~$\h$ given~$\CF_{\tau_i}$.  We will explain the proof for $i=1$ (as the proof for $i=2$ is analogous).  Fix a test function $\phi \in C_0^\infty(\h)$ and $\theta \in \R$.  Using that both of the events $\{ \sigma_1 = \tau_1\}$ and $\{\sigma_2 = \tau_1\}$ are $\CF_{\tau_1}$-measurable, we have that
\begin{align*}
  &    \E\!\left[ \exp(i \theta (h^{\tau_1},\phi)) \giv \CF_{\tau_1} \right]\\
=& \E\!\left[ \exp(i \theta (h^{\tau_1},\phi)) \giv \CF_{\tau_1} \right] \one_{\{\sigma_1 = \tau_1\}} +
     \E\!\left[ \exp(i \theta (h^{\tau_1},\phi)) \giv \CF_{\tau_1} \right] \one_{\{\sigma_2 = \tau_1\}}\\
=& \E\!\left[ \exp(i \theta (h^{\sigma_1},\phi)) \giv \CF_{\sigma_1} \right] \one_{\{\sigma_1 = \tau_1\}} +
     \E\!\left[ \exp(i \theta (h^{\sigma_2},\phi)) \giv \CF_{\sigma_2} \right] \one_{\{\sigma_2 = \tau_1\}}\\
=& \exp\!\left( -\frac{\theta^2 E_0(\phi)}{2} \right)
\end{align*}
where the last equality follows because we know that $h^{\sigma_i}$ has the law of a zero-boundary GFF on $\h$ given $\CF_{\sigma_i}$ for $i=1,2$.

Suppose now that $n \in \N$.  We can iterate the same argument above with a finite collection of stopping times $\tau_1,\ldots,\tau_n$ with $\p[\tau_i \leq \tau_U] = 1$ for all $i=1,\ldots,n$ to obtain a coupling such that $h_{U,\tau_i} = h|_U$ almost surely for all $i=1,\ldots,n$.  Indeed, to do so for each $1 \leq j \leq n$ we let $\sigma_j$ be the first time $t$ that at least $j$ of the stopping times $\tau_i$ have occurred at or before time $t$.  Then each $\sigma_j$ is a stopping time and we have that $\sigma_1 \leq \sigma_2 \leq \cdots \leq \sigma_n$.  Thus we can construct the coupling iteratively as described above in the case that $n=2$.  Under this coupling, we have that $h_{U,\sigma_j} = h|_U$ almost surely for each $1 \leq j \leq n$.  Since for each $1 \leq i \leq n$ there exists $1 \leq j \leq n$ such that $\tau_i = \sigma_j$, we then have that $h_{U,\tau_i} = h|_U$ almost surely for each $1 \leq i \leq n$.  The result follows because a simple elaboration of the argument given just above implies that we can write $h_{U,\tau_i}$ as $(h^{\tau_i} \circ f_{\tau_i} + \Fh_{\tau_i})|_U$ where $h^{\tau_i}$ has the law of a zero-boundary GFF on $\h$ given $\CF_{\tau_i}$.

In the construction of the coupling given above, it is not immediately clear that the conditional law of the projection of $h$ onto $H(U)$ given \emph{both} its projection onto $H^\perp(U)$ and $\CF_{\tau_i}$ is that of a zero-boundary GFF in $U$ for each $1 \leq i \leq n$ even though it is immediate from the construction that this holds if we condition on either the projection onto $H^\perp(U)$ or $\CF_{\tau_i}$.  We can modify the coupling so that it has this property by, given both $H^\perp(U)$ and $\CF_{\tau_U}$, resampling the projection of $h$ onto $H(U)$ from the law of a zero-boundary GFF on $U$.  Note that performing this resampling operation leaves the marginal law of $h$ unchanged (since the law of the projection of $h$ onto $H^\perp(U)$ is unchanged as is the conditional law of the projection of $h$ onto $H(U)$ given its projection onto $H^\perp(U)$).  This resampling operation also preserves the conditional law of $h|_U$ given $\CF_{\tau_i}$ for each $1 \leq i \leq n$.  Indeed, this follows because this resampling operation preserves the restriction of the projection of $h$ onto $H^\perp(U)$ to $U$.  Moreover, by the Markov property for $h_{U,\tau_i}$ given $\CF_{\tau_i}$ (which is the restriction of a GFF on $\h \setminus K_{\tau_i}$ to $U$), we know that the conditional law of the projection of $h$ onto $H(U)$ given both $\CF_{\tau_i}$ and the restriction of the projection of $h$ onto $H^\perp(U)$ to $U$ is that of a zero boundary GFF on $U$.  This proves the claim.
\end{proof}

We are now going to extend Lemma~\ref{lem::coupling_up_to_u} to the setting of both a finite number of stopping times and open sets.

\begin{lemma}
\label{lem::coupling_open_sets}
Suppose that $n \in \N$, $U_1,\ldots,U_n \subseteq \h$ are open sets, and that $\tau_1,\ldots,\tau_n$ are $\CF_t$-stopping times such that $\p[\tau_i \leq \tau_{U_i}] = 1$ for each $1 \leq i \leq n$.  For each $1 \leq i \leq n$, let $h_{U_i,\tau_i}$ have the law as defined in Lemma~\ref{lem::coupling_up_to_u} with $U = U_i$.  There exists a coupling of the laws $h_{U_i,\tau_i}$ for $1 \leq i \leq n$ so that they are all generated using the same instance $(W,V^{i,q})$ of the $\SLE_\kappa(\ul{\rho})$ driving process such that for $h = \wt{h} + \Fh_0$, $\wt{h}$ a zero-boundary GFF on $\h$, we almost surely have that $h_{U_i,\tau_i} = h|_{U_i}$.  Moreover, we have that the conditional law of the projection of $h$ onto $H(U_i)$ given $\CF_{\tau_i}$ and the projection of $h$ onto $H^\perp(U_i)$ is that of a zero-boundary GFF on $H(U_i)$ for each $1 \leq i \leq n$.
\end{lemma}
\begin{proof}
As in the proof of Lemma~\ref{lem::coupling_finite_stopping_times}, we begin with the argument in the case of two open sets $U_1,U_2 \subseteq \h$; define $\tau_{U_1}$, $\tau_{U_2}$ accordingly and suppose that $\p[\tau_i \leq \tau_{U_i}]  = 1$ for $i=1,2$.  We let $\sigma_1 = \tau_1 \wedge \tau_2$ and let $\sigma_2 = \tau_1 \vee \tau_2$.  Let $V_1 = U_1 \cup U_2$.  Then we can find a coupling of the laws of $(h_{V_1,\sigma_1}, (W_{t \wedge \sigma_1}, V_{t \wedge \sigma_1}^{i,q}))$ and $h$ such that $h_{V_1,\sigma_1} = h|_{V_1}$ almost surely and such that the conditional law of the projection of $h$ onto $H(V_1)$ given both its projection onto $H^\perp(V_1)$ and $\CF_{\sigma_1}$ is that of a zero-boundary GFF on $V_1$.  Conditionally on $(W,V^{i,q})$ up to time $\sigma_1$, we know whether $\tau_1$ or $\tau_2$ occurred first.  On the event that $\tau_1$ occurred first, it is easy to see by the Markov property of the GFF that we can take our coupling so that $h|_{U_1} = h_{U_1,\tau_1}$ almost surely.  We can also repeat the same argument (using the conformal Markov property of $\SLE_\kappa(\ul{\rho})$) so that we also have that $h|_{U_2} = h_{U_2,\tau_2}$ almost surely.  Similarly, on the event that $\tau_2$ occurred first, it is easy to see by the Markov property of the GFF that we can take our coupling so that $h|_{U_2} = h_{U_2,\tau_2}$ almost surely.  We can also repeat the same argument (using the conformal Markov property of $\SLE_\kappa(\ul{\rho})$) so that we also have that $h|_{U_1} = h_{U_1,\tau_1}$ almost surely.

In summary, at this point we have a coupling such that the following hold:
\begin{itemize}
\item $h|_{U_j} = h_{U_j,\tau_j}$ for $j=1,2$ and $h|_{V_1} = h_{V_1,\sigma_1}$ almost surely.
\item The conditional law of the projection of $h$ onto $H(V_1)$ given its projection onto $H^\perp(V_1)$ and $\CF_{\sigma_1}$ is that of a zero boundary GFF on $V_1$.
\end{itemize}
We will now explain how to modify $h$ so that it satisfies the above properties and so that the conditional law of the projection of $h$ onto $H(U_j)$ given its projection onto $H^\perp(U_j)$ and $\CF_{\tau_j}$ for $j=1,2$ is that of a GFF on $U_j$ with zero boundary conditions.  To this end, we let $h_{U_j^c}$ for $j=1,2$ be the projection of $h$ onto $H^\perp(U_j)$.  We then take
\[ \wt{h} = (\wt{h}_{U_2} +  h_{U_2^c}) \one_{\{\tau_1 < \tau_2\}} + (\wt{h}_{U_1} + h_{U_1^c}) \one_{\{\tau_2 \leq \tau_1\}}\]
where $\wt{h}_{U_1}$ (resp.\ $\wt{h}_{U_2}$) has the law of a GFF on $U_1$ (resp.\ $U_2$) with zero boundary conditions where we take $\wt{h}_{U_1}, \wt{h}_{U_2}$ to be independent of each other and independent of everything else.  The Markov property of the GFF implies that the coupling of $\wt{h}$ with $(W,V^{i,q})$  satisfies all of the desired properties.  In particular, the conditional law of the projection of $\wt{h}$ onto $H(U_j)$ given its projection onto $H^\perp(U_j)$ and $\CF_{\tau_j}$ is that of a zero boundary GFF on $U_j$ for $j=1,2$.

We can iterate this argument with any finite collection of stopping times $\tau_1,\ldots,\tau_n$ and open sets $U_1,\ldots,U_n$ with $\p[\tau_i \leq \tau_{U_i}] = 1$ for each $i=1,\ldots,n$ to obtain a coupling such that $h_{U_i,\tau_i} = h|_{U_i}$ for each $i=1,\ldots,n$ almost surely and so that the conditional law of the projection of $h$ onto $H(U_i)$ given its projection onto $H^\perp(U_i)$ and $\CF_{\tau_i}$ is that of a zero-boundary GFF on $U_i$.  To do so, we let $\sigma_j$ be the first time $t$ that at least $j$ of the $\tau_i$ have occurred at or before time~$t$.  Then we know that we can construct a coupling so that with $V_1 = \cup_{i=1}^n U_i$ we have that $h_{V_1,\sigma_1} = h|_{V_1}$ almost surely and the conditional law of the projection of $h$ onto $H(V_1)$ given its projection onto $H^\perp(V_1)$ and $\CF_{\sigma_1}$ is that of a zero-boundary GFF on $V_1$.  At time $\sigma_1$, we know which of the $\tau_i$ have occurred.  Assume for simplicity that the $\tau_i$ are almost surely distinct and that $\tau_{i_1}$ is the first to occur.  Then we can repeat the same argument (using the conformal Markov property of $\SLE_\kappa(\ul{\rho})$) with $V_2 = \cup_{i \neq i_1} U_i$ and $\sigma_2$ to obtain a coupling such that we have $h_{V_j,\sigma_j} = h|_{V_j}$ almost surely and the conditional law of the projection of $h$ onto $H(V_j)$ given its projection onto $H^\perp(V_j)$ and $\CF_{\sigma_j}$ is that of a zero-boundary GFF on $V_j$  for $j=1,2$.  After iterating this $n$ times, we get a coupling with $h_{V_j,\sigma_j} = h|_{V_j}$ almost surely and such that the conditional law of the projection of $h$ onto $H(V_j)$ given its projection onto $H^\perp(V_j)$ and $\CF_{\sigma_j}$ is that of a zero-boundary GFF on $V_j$ for $j=1,\ldots,n$ .  If we let $i_j$ be such that $\tau_{i_j} = \sigma_j$, then we have that $U_{i_j} \subseteq V_j$.  Therefore $h_{U_{i_j},\tau_{i_j}} = h|_{U_{i_j}}$ almost surely for each $j$ hence $h_{U_i,\tau_i} = h|_{U_i}$ for each $i$.  Likewise, the conditional law of the projection of $h$ onto $H(U_i)$ for each $i$ given its projection onto $H^\perp(U_i)$ and $\CF_{\tau_i}$ is that of a zero-boundary GFF on $U_i$.
\end{proof}

We now turn to extend Lemma~\ref{lem::coupling_up_to_u} to the setting of a countable collection of stopping times and open sets.

\begin{lemma}
\label{lem::coupling_sequence_stopping_times}
Suppose that $(U_i)$ and $(\tau_i)$ are respectively sequences of open sets in~$\h$ and $\CF_t$-stopping times such that $\p[\tau_i \leq \tau_{U_i}] = 1$ for all $i \in \N$.  For each $i \in \N$, let $h_{U_i,\tau_i}$ have the law as defined in Lemma~\ref{lem::coupling_up_to_u} with $U = U_i$.  There exists a coupling of the laws $h_{U,\tau_i}$ for $i \in \N$ so that they are all generated with the same instance $(W,V^{i,q})$ of the $\SLE_\kappa(\ul{\rho})$ driving process such that for $h = \wt{h} + \Fh_0$, $\wt{h}$ a zero-boundary GFF on $\h$, we almost surely have that $h_{U_i,\tau_i} = h|_{U_i}$.  Moreover, we have that the conditional law of the projection of $h$ onto $H(U_i)$ given $\CF_{\tau_i}$ and the projection of $h$ onto $H^\perp(U_i)$ is that of a zero-boundary GFF on $H(U_i)$.
\end{lemma}
\begin{proof}
Fix $n \in \N$.  Lemma~\ref{lem::coupling_open_sets} implies that there exists a coupling of $h_{U_1,\tau_1},\ldots,h_{U_n,\tau_n}$ and $h,(W,V^{i,q}))$ which satisfies the property that $h_{U_i,\tau_i} = h|_{U_i}$ and that the conditional law of the projection of $h$ onto $H(U_i)$ given the projection of $h$ onto $H^\perp(U_i)$ and $\CF_{\tau_i} $ is that of a zero-boundary GFF on $U_i$ for each $1 \leq i \leq n$.  Let $\mu_n$ denote the law of this coupling.  Let $k$ be the number of force points for the $\SLE_\kappa(\ul{\rho})$ driving process.  We view $\mu_n$ as a law on $(\R^\N)^{n+1} \times C([0,\infty))^{k+1}$ by expressing each of the $h_{U_i,\tau_i}$ as well as $h$ in terms of coordinates by integrating each against a countable subset of $C_0^\infty(\h)$ which is dense in $L^2(\h)$, say.  We note that for $m \geq n$ we have the $\mu_m$ marginal law of $(h_{U_1,\tau_1},\ldots,h_{U_n,\tau_n},h,(W,V^{i,q}))$ is tight in $m$ since the marginal of each of the coordinates does not depend on $m$.  It therefore follows that for each $n \in \N$ there exists a subsequence $(\mu_{n_k^m})$ of $(\mu_n)$ such that the marginal law of $(h_{U_1,\tau_1},\ldots,h_{U_m,\tau_m},h,(W,V^{i,q}))$ under $\mu_n$ converges weakly to a limiting law as $k \to \infty$.  By passing to a further diagonal subsequence if necessary, we can find a subsequence $(\mu_{n_k})$ of $(\mu_n)$ such that for every $m \in \N$ the marginal law of $(h_{U_1,\tau_1},\ldots,h_{U_m,\tau_m},h,(W,V^{i,q}))$ under $\mu_{n_k}$ converges weakly to a limiting law as $k \to \infty$.  Letting $\mu^m$ denote this law for a given value of $m$, we observe that the $\mu^m$ are consistent in the sense that for every $j \leq n,m$ we have that the $\mu^m$ marginal law of $(h_{U_1,\tau_1},\ldots,h_{U_j,\tau_j},h,(W,V^{i,q}))$ is the same as the $\mu^n$ marginal law of the same vector.  Therefore by the Kolmogorov extension theorem, there exists a law $\mu$ on  $(\R^\N)^\N \times C([0,\infty))^{k+1}$ whose marginals agree with the $\mu^m$.  By construction, this law $\mu$ clearly satisfies the property that $h_{U_i,\tau_i} = h|_{U_i}$ almost surely for every $i \in \N$ and that the conditional law of the projection of $h$ onto $H(U_i)$ given the projection of $h$ onto $H^\perp(U_i)$ and $\CF_{\tau_i}$ is that of a zero-boundary GFF on $U_i$.
\end{proof}

We now complete the proof of Theorem~\ref{thm::coupling_existence}.

\begin{proof}[Proof of Theorem~\ref{thm::coupling_existence}]
To complete the proof of the theorem, we need to extend Lemma~\ref{lem::coupling_sequence_stopping_times} so that it holds simultaneously for \emph{all} $\CF_t$-stopping times.  We can choose our $\tau_i$ and $U_i$ so that every pair of the form $(U, r \wedge \tau_U)$ where $U$ is a finite union of balls which are centered at a point with rational coordinates and with rational radii and $r \in \Q_+$ appears in the sequence $(U_i,\tau_i)$.

Suppose that $\tau$ is any $\CF_t$-stopping time.  Given $\CF_\tau$, we fix a test function $\phi \in C_0^\infty(\h)$ whose support is disjoint from $K_\tau$.  We then let $\sigma_n = \tau_{j_n}$ be the smallest element of $\tau_1,\ldots,\tau_n$ which is larger than $\tau$ such that the support of $\phi$ is contained in $U_{j_n}$ (on the event that there is no such $\tau_j$ we take $\sigma_n=\infty$).  Fix $\theta \in \R$.  Note that for each $1 \leq j \leq n$, the event $\{\sigma_n = \tau_j\}$ is $\CF_{\tau_j}$-measurable.  On the event that $\{\sigma_n < \infty\}$, we have that
\begin{align*}
      \E\!\left[ \exp( i \theta (h,\phi)) \giv \CF_{\sigma_n} \right]
&= \sum_{j=1}^n \E\!\left[ \exp(i \theta (h,\phi)) \giv \CF_{\sigma_n} \right] \one_{\{\sigma_n = \tau_j \}}\\
&= \sum_{j=1}^n \E\!\left[ \exp(i \theta (h,\phi)) \giv \CF_{\tau_j} \right] \one_{\{\sigma_n = \tau_j \}}\\
&= \sum_{j=1}^n \exp\!\left( i \theta (\Fh_{\tau_j},\phi)-\frac{\theta^2 E_{\tau_j}(\phi)}{2} \right) \one_{\{\sigma_n = \tau_j \}}\\
&=  \exp\!\left( i \theta (\Fh_{\sigma_n},\phi)-\frac{\theta^2 E_{\sigma_n}(\phi)}{2}\right) .
\end{align*}
This proves that the conditional law of $(h,\phi)$ given $\CF_{\sigma_n}$ is equal to the law of $(h^{\sigma_n} \circ f_{\sigma_n} + \Fh_{\sigma_n},\phi)$ where $h^{\sigma_n}$ given $\CF_{\sigma_n}$ has the law of a zero-boundary GFF on $\h$.  Note that the sequence of stopping times $(\sigma_n)$ almost surely decreases to $\tau$ as $n \to \infty$.  The backwards martingale convergence theorem and the right-continuity of the filtration $(\CF_t)$ together imply that we almost surely have
\begin{align}
  \lim_{n \to \infty} \E\!\left[ \exp( i \theta (h,\phi)) \giv \CF_{\sigma_n} \right] = 
    \E\!\left[ \exp( i \theta (h,\phi)) \giv \CF_\tau \right]. \label{eqn::sigma_n_decrease}
\end{align}
Moreover, the continuity of the $\SLE_\kappa(\ul{\rho})$ driving process implies that we almost surely have
\begin{align}
     \lim_{n \to \infty} \exp\!\left( i \theta (\Fh_{\sigma_n},\phi)-\frac{\theta^2 E_{\sigma_n}(\phi)}{2}\right) = \exp\!\left( i \theta (\Fh_\tau,\phi)-\frac{\theta^2 E_\tau(\phi)}{2} \right). \label{eqn::sigma_n_char_converge}
\end{align}
Combining~\eqref{eqn::sigma_n_decrease} and~\eqref{eqn::sigma_n_char_converge} implies that the conditional law of $(h,\phi)$ given $\CF_\tau$ is equal to the law of $(h^\tau \circ f_\tau + \Fh_\tau,\phi)$ where $h^\tau$ given $\CF_\tau$ has the law of a zero-boundary GFF on $\h$.

Fix $U \subseteq \h$ open.  A similar argument using the backwards martingale convergence theorem implies that the conditional law of the projection of $h$ onto $H(U)$ given its projection onto $H^\perp(U)$ and $\CF_\tau$ is that of a zero-boundary GFF on $U$ on the event that $\tau \leq \tau_U$.  Indeed, by our choice of $\tau_i,U_i$ we can find a sequence $(j_n)$ such that $U_{j_n} \subseteq U$, $U_{j_n} \subseteq U_{j_{n+1}}$, and $\tau_{j_n} = \tau_{U_{j_n}}$ for every $n$ and such that $\cup_n U_{j_n} = U$.  We then take $\sigma_n = \tau_{j_n}$ for each $n$.  Fix $\phi \in C_0^\infty(U)$.  Then we have that $\supp(\phi) \subseteq U_{j_n}$ for all $n$ large enough.  For each $n$, we let $h_{U_{j_n}}$ (resp.\ $h_{U_{j_n}^c}$) be the projection of $h$ onto $H(U_{j_n})$ (resp.\ $H^\perp(U_{j_n})$).  We similarly let $h_U$ (resp.\ $h_{U^c}$) be the projection of $h$ onto $H(U)$ (resp.\ $H^\perp(U)$).  Proposition~\ref{prop::gff_sigma_algebra_converges} and the backwards martingale convergence theorem together imply that
\[ (h_{U_{j_n}^c},\phi) = \E[ (h,\phi) \giv \CF_{ (\h \setminus U_{j_n})^+}^h] \to  \E[ (h,\phi) \giv \CF_{(\h \setminus U)^+}^h] = (h_{U^c},\phi)\]
almost surely as $n \to \infty$.  Therefore
\begin{equation}
\label{eqn::zero_boundary_converges}
(h_{U_{j_n}},\phi) = (h,\phi) - (h_{U_{j_n}^c},\phi) \to (h,\phi) - (h_{U^c},\phi) = (h_U,\phi)
\end{equation}
almost surely as $n \to \infty$.  Let $\sigma_{U_{j_n}}^2(\phi) = \iint \phi(x) G_{U_{j_n}}(x,y) \phi(y) dx dy$ where $G_{U_{j_n}}$ is the Green's function for $\Delta$ on $U_{j_n}$.  We define $\sigma_U^2(\phi)$ similarly and note that $\sigma_{U_{j_n}}^2(\phi) \to \sigma_U^2(\phi)$ as $n \to \infty$.  By the construction of the coupling, for each $n \in \N$ we have that
\begin{align}
\label{eqn::characteristic_zero_boundary_form}
   \E[ \exp(i \theta (h_{U_{j_n}},\phi)) \giv \CF_{\sigma_n}, \CF_{(\h \setminus U_{j_n})^+}^h] = \exp\left( - \frac{\theta^2 \sigma_{U_{j_n}}^2(\phi)}{2} \right).
\end{align}
As $n \to \infty$, the right hand side of~\eqref{eqn::characteristic_zero_boundary_form} converges to $\exp(-\theta^2 \sigma_U^2(\phi)/2)$ while, by~\eqref{eqn::zero_boundary_converges} and the backwards martingale convergence theorem, the left hand side of~\eqref{eqn::characteristic_zero_boundary_form} converges to $\E[ \exp(i \theta (h_U,\phi)) \giv \CF_{\tau_U}, \CF_{(\h \setminus U)^+}^h]$.  Since $\phi \in C_0^\infty(U)$ was arbitrary, it therefore follows that the conditional law of $h_U$ given $h_{U^c}$ and $\CF_{\tau_U}$ is that of a GFF on $U$ with zero boundary conditions.  It therefore follows that the conditional law of $h_U$ given $h_{U^c}$ and $\CF_\tau$ on $\tau \leq \tau_U$ is that of a GFF on $U$ with zero boundary conditions.  That is, characterization~\eqref{it::Ucond} of Lemma~\ref{lem::local_char} implies that $K_\tau$ is local for $h$.
\end{proof}

\section{Dub\'edat's argument}
\label{sec::dubedat}

This section will present the argument from \cite{DUB_DUAL, DUB_PART} to establish Theorem~\ref{thm::coupling_uniqueness} for $\kappa < 4$ with some particular boundary conditions.  One of its nice features is that it simultaneously establishes a particular case of so-called {\em Duplantier duality}: that the outer boundary of a certain $\SLE_{16/\kappa}(\ul{\rho})$ process is equal in law to a certain $\SLE_{\kappa}(\ul{\rho})$ process.  Indeed, we find that the left and right boundaries of a counterflow line corresponding to a given $h$ are almost surely flow lines for the same $h$.  Our exposition (the interpretations, illustrations, and geometric point of view) is rather different from what appears in \cite{DUB_DUAL, DUB_PART}, but the basic argument is the same.  We will also explain how duality implies the transience (continuity upon exiting) of certain non-boundary intersecting flow lines (see \cite{RS05} for an alternative approach to proving the transience of $\SLE$).

When following the illustrations, it will be useful to keep in mind a few definitions and identities:
\begin{equation}
\label{eqn::deflist}
\lambda := \frac{\pi}{\sqrt \kappa},\quad \kappa' := \frac{16}{\kappa},\quad \lambda' := \frac{\pi}{\sqrt{\kappa'}} = \frac{\pi \sqrt{\kappa}}{4} = \frac{\kappa}{4} \lambda < \lambda,\quad \chi := \frac{2}{\sqrt \kappa} - \frac{\sqrt \kappa}{2}
\end{equation}
\begin{equation}
\label{eqn::fullrevolution}
2 \pi \chi = 4(\lambda-\lambda'), \quad \lambda' = \lambda - \frac{\pi}{2} \chi
\end{equation}
\begin{equation}
\label{eqn::fullrevolutionrho}
2 \pi \chi = (4-\kappa)\lambda = (\kappa'-4)\lambda'.
\end{equation}

Throughout, we shall assume that $\kappa \in (0,4)$ so that $\kappa' \in (4,\infty)$.  Each of $\lambda$, $\lambda'$, and $\chi$ is determined by $\kappa$ through~\eqref{eqn::deflist} and is positive.  We will frequently go back and forth between these four values using the identities above.  To interpret
\eqref{eqn::fullrevolution}, recall that an angle change of $\theta$ corresponds to a change of $\theta \chi$ in the value of the field.  Thus~\eqref{eqn::fullrevolution} says that the difference $\lambda - \lambda'$ corresponds to ``a ninety-degree turn to the left'' in the imaginary geometry, which will frequently be useful.  (This fact was already illustrated in Figure~\ref{fig::flowlineheights}.)  Recall also from Theorem~\ref{thm::coupling_existence} that a force point of weight $\rho$ corresponds to a jump of size $\rho \lambda$ in the boundary values.  Thus~\eqref{eqn::fullrevolutionrho} implies that a gap of size $2 \pi \chi$ (a ``full revolution'' gap) corresponds to a $\rho$ value of $(4-\kappa)$.  More generally, a $\theta \chi$ gap in boundary values corresponds to $\rho = \frac{\theta}{2\pi}(4-\kappa)$.

\subsection{Critical heights for boundary intersection, path continuation}
\label{subsec::intersection}

We begin by observing some basic facts about $\SLE_{\kappa}(\ul{\rho})$ processes, which tell us what kinds of boundary segments the flow and counterflow lines of Theorem~\ref{thm::coupling_existence} can intersect.  Throughout, we let $\strip = \R \times (0,1) \subseteq \C$ be an infinite horizontal strip and we decompose $\partial \strip$ into its lower and upper boundaries $\stripbot = (-\infty,\infty)$ and $\striptop = (-\infty,\infty) + i$, respectively.

\begin{figure}[h!]
\begin{center}
\includegraphics[scale=0.85]{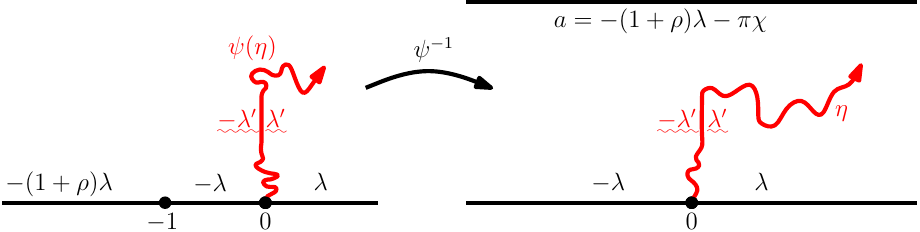}
\caption{\label{fig::criticalforintersection} Suppose that $h$ is a GFF on $\strip$ whose boundary data is as depicted on the right side.  Then the flow line $\eta$ of $h$ shown has the law of an $\SLE_{\kappa}(\rho)$ process in $\strip$ from $0$ to $+ \infty$ where the force point is located at $-\infty$.  To see this, let $\psi \colon \strip \to \h$ be the conformal map which fixes $0$, sends $-\infty$ to $-1$, and $+\infty$ to $\infty$.  Then $h \circ \psi^{-1} - \chi \arg(\psi^{-1})'$ is a GFF on $\h$ whose boundary data is depicted on the left side.  Given the path, the expectation of the field (in the Theorem~\ref{thm::coupling_existence} coupling) is the harmonic extension of the given values on $\R$ and $\pm\lambda' + \chi \cdot \mathrm{winding}$ on the curve, as in Figure~\ref{fig::flowlineheights}.   The critical $\rho$ for $\psi(\eta)$ to be able to intersect $(-\infty,-1)$ before reaching $\infty$ is $\rho_0 = \kappa/2-2$.  This implies that $\eta$ accumulates in the upper boundary $\striptop$ or at $-\infty$ before $+\infty$ if and only if $\rho < \kappa/2-2$, i.e., $a > -(1+\rho_0)\lambda - \pi \chi = -(\kappa/2-1)\lambda - \pi \chi = -\lambda$, and otherwise it accumulates at $+\infty$ without hitting $\striptop$.  (Recall that $2\pi\chi = (4-\kappa)\lambda$.)  Symmetrically, it accumulates in $\striptop$ or at $+\infty$ before $-\infty$ if and only if $a < \lambda$, and otherwise it accumulates at $-\infty$ without hitting $\striptop$.  The same result also holds when the boundary data on $\stripbot$ is piecewise constant, changes values a finite number of times, and is at most $-\lambda+\pi \chi$ to the left of $0$ and at least $\lambda-\pi\chi$ to the right of~$0$ (see Remark~\ref{rem::boundary_data}).  Furthermore, the analogous statement holds when $\eta$ is replaced by an $\SLE_{\kappa'}$ counterflow line, $\lambda$ is replaced by $\lambda'$, and $\chi$ is replaced by $-\chi$.}
\end{center}
\end{figure}

\begin{remark}
\label{rem::boundary_data}
Throughout this and the next subsection, we will often make the assumption that the boundary data of $h$ on $\stripbot$ is at most $-\lambda+\pi\chi$ to the left of $0$ and at least $\lambda-\pi \chi$ to the right of $0$.  The significance of this assumption is that, by Remark~\ref{rem::continuity_non_boundary} and absolute continuity (the Girsanov theorem), it implies that the flow line $\eta$ of $h$ is almost surely a continuous path, at least until it accumulates in $\striptop$.  Moreover, by \cite[Lemma~15]{DUB_DUAL}, it implies that $\eta$ must accumulate either in $\striptop$ or at $\pm \infty$ before accumulating in $\stripbot$ after time $0$ (and upon proving Theorem~\ref{thm::continuity}, we will later be able to show that it actually never accumulates in $\stripbot$ after time $0$).  This will be important for the results explained in this section because \cite[Lemma~15]{DUB_DUAL} only gives us information regarding the \emph{first} force point disconnected by $\eta$.
\end{remark}

\begin{remark}
\label{rem::counterflow_lines_boundary_data}  All of the results of this subsection are applicable both to $\eta \sim \SLE_\kappa(\ul{\rho})$ and $\eta' \sim \SLE_{\kappa'}(\ul{\rho})$.  In the case of the latter, the boundary data is given as in the statement of Theorem~\ref{thm::coupling_existence} (the boundary data is $-\lambda'$ to the left of $\eta'$ and $\lambda'$ to its right) to make the lemma statements consistent for $\SLE_\kappa$ and $\SLE_{\kappa'}$ processes.  In this case, the condition that implies that $\eta'$ does not hit $\stripbot$ is that the boundary data for $h$ should be at most $-\lambda'-\pi \chi$ to the left of $0$ and at least $\lambda' + \pi\chi$ to the right of $0$.  In Section~\ref{subsec::uniqueness_non_boundary_intersecting}, we will apply these results to $\eta' \sim \SLE_{\kappa'}$ where $\eta'$ is coupled with $-h$ as in the statement of Theorem~\ref{thm::coupling_existence}, so the boundary data is reversed.
\end{remark}

\begin{figure}[h!]
\begin{center}
\includegraphics[scale=0.85]{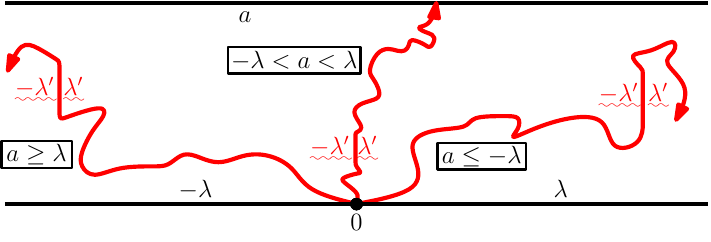}
\caption{\label{fig::hittingrange}  In the setting of Figure~\ref{fig::criticalforintersection}, the flow line behavior depends on $a$.  Curves shown represent a.s.\ behaviors corresponding to the three different regimes of $a$ (indicated by the closed boxes).  From Figure~\ref{fig::criticalforintersection}, the path hits the upper boundary of the strip a.s.\ if and only if $a \in (-\lambda, \lambda)$.  When $a \geq \lambda$, it tends to $-\infty$ (left end of the strip) and when $a \leq - \lambda$ it tends to $\infty$ (right end of the strip) without hitting the upper boundary.  These facts also hold whenever the boundary data of $h$ on $\stripbot$ is piecewise constant, changes only a finite number of times, and is at most $-\lambda+\pi \chi$ to the left of $0$ and at least $\lambda-\pi\chi$ to the right of $0$ (see Remark~\ref{rem::boundary_data}).  The same statement holds when $\eta$ is replaced by an $\SLE_{\kappa'}$ process, $\lambda$ is replaced by $\lambda'$ and $\chi$ is replaced by $-\chi$ (see Remark~\ref{rem::counterflow_lines_boundary_data}).}
\end{center}
\end{figure}

\begin{lemma}
\label{lem::hitting_and_bouncing}
Suppose that $h$ is a GFF on the strip $\strip$ whose boundary data is as depicted in Figures~\ref{fig::criticalforintersection}-\ref{fig::bouncingrange} and let $\eta$ be the flow line of $h$ starting at $0$.  If $a \geq \lambda$, then $\eta$ almost surely accumulates at $-\infty$ and if $a \leq -\lambda$, then $\eta$ almost surely accumulates at $+\infty$.  In both cases, $\eta$ almost surely does not hit $\striptop$.  If $a \in (-\lambda,\lambda)$, then $\eta$ almost surely accumulates in $\striptop$.  If, moreover,  $a > -\lambda+\chi\pi$ (resp.\ $a < \lambda -\pi \chi$), then $\eta$ can be continued when it is targeted toward $-\infty$ (resp.\ $+\infty$) --- i.e., the continuation threshold is not reached in this case when $\eta$ first accumulates in $\striptop$.  This holds more generally when the boundary data on $\stripbot$ is piecewise constant, changes only a finite number of times, and is at most $- \lambda + \pi\chi$ to the left of $0$ and is at least $\lambda - \pi \chi$ to the right of $0$ (see Remark~\ref{rem::boundary_data}).  Furthermore, the analogous statement holds for $\SLE_{\kappa'}$ processes when $\lambda$ is replaced by $\lambda'$ and $\chi$ is replaced by $-\chi$ (see Remark~\ref{rem::counterflow_lines_boundary_data}).
\end{lemma}

\begin{remark}
\label{rem::branching}
If it happens that $-\lambda + \pi \chi < a < \lambda - \pi \chi$, then, after the path first hits $\striptop$, it is possible to {\em branch} the path and continue it in both directions (both toward $-\infty$ and toward $+\infty$).  Since 
\[ -\lambda + \chi \pi = -\lambda + (4-\kappa)\frac{\lambda}{2} = \lambda\left(1 - \frac{\kappa}{2}\right),\]
there exists $a$ values for which this is possible whenever $\kappa > 2$.
\end{remark}

\begin{proof}[Proof of Lemma~\ref{lem::hitting_and_bouncing}]
Let $\rho \in \R$ be such that $a = -(1+\rho)\lambda - \pi \chi$ and let $\psi \colon \strip \to \h$ be the conformal transformation which sends $0$ to $0$, $-\infty$ to $-1$, and $+\infty$ to $\infty$.  Then the transformation rule~\eqref{eqn::ac_eq_rel} implies that $\psi(\eta)$ is an $\SLE_\kappa(\rho)$ process in $\h$ from $0$ to $\infty$ where the force point of weight $\rho$ is located at $-1$.  We will prove the lemma by checking the criterion of \cite[Lemma~15]{DUB_DUAL}.

We first suppose that $a \geq \lambda$.  This holds if and only if $\rho \leq \tfrac{\kappa}{2} - 4$.  Consequently, \cite[Lemma~15]{DUB_DUAL} implies that $\psi(\eta)$ almost surely hits $\partial \h$ for the first time (after its initial point) at $-1$, which translates into $\eta$ tending to $-\infty$ without hitting $\striptop$.  On the other hand, if $a \leq -\lambda$ then $\rho \geq \tfrac{\kappa}{2} - 2$.  Applying \cite[Lemma~15]{DUB_DUAL} analogously implies $\psi(\eta)$ tends to $\infty$ without exiting $\h$ so that $\eta$ tends to $+\infty$ without hitting $\striptop$.  The case where $a \in (-\lambda,\lambda)$ so that $\rho \in (\tfrac{\kappa}{2}-4,\tfrac{\kappa}{2}-2)$ is similar.  Finally, we note that $a < \lambda - \chi \pi$ (resp.\ $a > -\lambda + \chi \pi$) translates into $\rho > -2$ which means that the continuation threshold for the path targeted at $+\infty$ (resp.\ $-\infty$) is not reached when it hits $\striptop$.  This observation implies that we have the behavior described in the statement of the lemma for these ranges of $a$ values.
\end{proof}

\begin{figure}[h!]
\begin{center}
\includegraphics[scale=0.85]{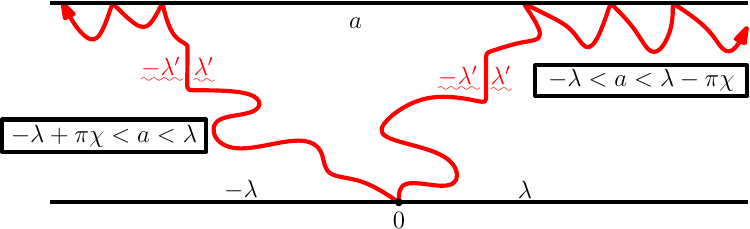}
\caption{\label{fig::bouncingrange} Take $a \in (-\lambda, \lambda)$.  Consider Figure~\ref{fig::criticalforintersection} with $\rho=-2$, the critical value of $\rho$ for the path being able to continue after the force point is absorbed.  Conformally mapping to the strip $\strip$, we find that the path may be continued to the left when $a \in (-\lambda + \pi \chi, \lambda)$ and to the right when $a \in (-\lambda, \lambda - \pi \chi)$.  In the extreme case $a = \lambda - \pi \chi$ ($\rho = -2$), the path on the right ``merges into'' the upper line.  (When the path turns right and runs parallel to $\R$, the height on its lower side is $\lambda - \pi \chi$.  The merging phenomenon will be developed in Section~\ref{subsec::monotonicity_merging_crossing}.)  A similar statement holds in the setting of $\SLE_{\kappa'}$ processes when $\lambda$ is replaced by $\lambda'$ and $\chi$ is replaced by $-\chi$ (see Remark~\ref{rem::counterflow_lines_boundary_data}).}
\end{center}
\end{figure}

\begin{figure}[h!]
\begin{center}
\includegraphics[scale=0.85]{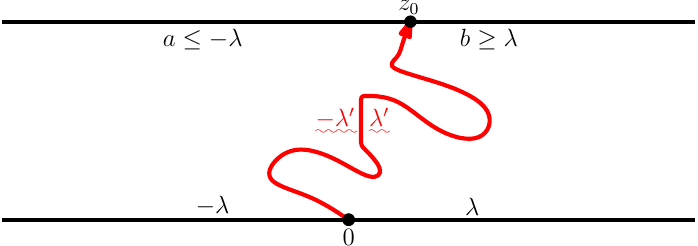}
\caption{\label{fig::hittingsinglepoint} If the left interval has height $a \leq -\lambda$ and the right has height $b \geq \lambda$, then the first accumulating point of the path on the upper line will be at $z_0$.  The same result holds if the boundary data is piecewise constant and changes at most a finite number of times, is at most $-\lambda$ to the left of $z_0$, at least $\lambda$ to the right of $z_1$, at most $-\lambda+\pi\chi$ to the left of $0$, and at least $\lambda-\pi\chi$ to the right of $0$ (see Remark~\ref{rem::boundary_data}).  The same statement holds for $\SLE_{\kappa'}$ processes provided $\lambda$ is replaced by $\lambda'$ and $\chi$ is replaced by $-\chi$ (see Remark~\ref{rem::counterflow_lines_boundary_data}).}
\end{center}
\end{figure}

\begin{lemma}
\label{lem::hitting_single_point}
Suppose that $h$ is a GFF on $\strip$ whose boundary data is as depicted in Figure~\ref{fig::hittingsinglepoint}.  Let $z_0$ be the point of $\striptop$ which separates $\striptop$ into the segments where the boundary data changes from $a$ to $b$.  Then the flow line $\eta$ of $h$ starting at $0$ almost surely exits $\strip$ at $z_0$ without otherwise hitting $\striptop$.  This result holds more generally when the boundary data of $h$ on $\striptop$ is piecewise constant, changes a finite number of times, and is at most $-\lambda$ to the left of $z_0$, at least $\lambda$ to the right of $z_0$, and on $\stripbot$ is piecewise constant, changes a finite number of times, and is at most $-\lambda + \pi \chi$ to the left of $0$ and at least $\lambda-\pi \chi$ to the right of $0$ (see Remark~\ref{rem::boundary_data}).  Moreover, this result holds in the setting of $\SLE_{\kappa'}$ processes when $\lambda$ is replaced by $\lambda'$ and $\chi$ is replaced by $-\chi$ (see Remark~\ref{rem::counterflow_lines_boundary_data}).
\end{lemma}
\begin{proof}
Let $\psi \colon \strip \to \h$ be the conformal transformation which sends $0$ to $0$, $-\infty$ to $-1$, and $z_0$ to $-2$.  Applying the transformation rule (and analogously to Figure~\ref{fig::criticalforintersection}), we see that $\psi(\eta) \sim \SLE_\kappa(\rho_1,\rho_2)$ where the weights $\rho_1,\rho_2$ are located at the force points $-1,-2$, respectively, and are determined from $a,b$ by the equations
\begin{align*}
a  = -(1+\rho_1)\lambda - \chi\pi,\quad
b  = -(1+\rho_1+\rho_2)\lambda - \chi\pi.
\end{align*}
The condition $a \leq -\lambda$ implies $\rho_1 \geq \tfrac{\kappa}{2}-2$ and that $b \geq \lambda$ gives $\rho_1+\rho_2 \leq \tfrac{\kappa}{2}-4$.  Consequently, it follows from \cite[Lemma~15]{DUB_DUAL} that $\psi(\eta)$ first exits $\h$ at $-2$, which is to say that $\eta$ first exits $\strip$ at $z_0$.
\end{proof}

\begin{figure}[h!]
\begin{center}
\includegraphics[scale=0.85]{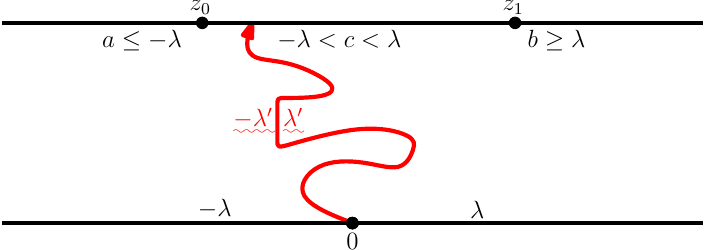}
\caption{\label{fig::hittinginterval} In the same setting as Figure~\ref{fig::hittingsinglepoint} but with an additional interval of height $c \in (-\lambda, \lambda)$, the path first hits the upper boundary in this middle interval.  The same result holds when the constants $a,b$ are replaced by piecewise constant functions which each take on a finite number of values which do not exceed $-\lambda$ and are at least $\lambda$ in their respective intervals and the boundary data on $\stripbot$ is piecewise constant, changes a finite number of times, and is at most $ - \lambda + \pi \chi$ to the left of $0$ and is at least $\lambda - \pi \chi$ to the right of $0$ (see Remark~\ref{rem::boundary_data}).  The same statement holds in the setting of counterflow lines provided $\lambda$ is replaced by $\lambda'$ and $\chi$ is replaced by $-\chi$ (see Remark~\ref{rem::counterflow_lines_boundary_data}).}
\end{center}
\end{figure}

\begin{lemma}
\label{lem::hitting_interval}
Suppose that $h$ is a GFF on $\strip$ whose boundary data is as depicted in Figure~\ref{fig::hittinginterval}.  Let $z_0,z_1$ be the points which separate $\striptop$ into the segments where the boundary data changes from $a$ to $c$ and $c$ to $b$, respectively.  Then the flow line $\eta$ of $h$ almost surely exits $\strip$ in $[z_0,z_1]$ without otherwise hitting $\striptop$.  This holds more generally when the boundary data of $h$ is piecewise constant, changes a finite number of times, is at most $-\lambda$ to the left of $z_0$, at least $\lambda$ to the right of $z_0$, at most $-\lambda+\pi \chi$ to the left of $0$, and at least $\lambda-\pi \chi$ to the right of $0$ (see Remark~\ref{rem::boundary_data}).  Moreover, this result holds when $\eta$ is replaced by a counterflow line provided $\lambda$ is replaced by $\lambda'$ and $\chi$ is replaced by $-\chi$ (see Remark~\ref{rem::counterflow_lines_boundary_data}).
\end{lemma}
\begin{proof}
Let $\psi \colon \strip \to \h$ be the conformal transformation which sends $0$ to $0$, $-\infty$ to $-1$, $z_0$ to $-2$, and let $y_1 = \psi(z_1) < -2$.  Applying the transformation rule (analogously to Figure~\ref{fig::criticalforintersection}), we see that $\psi(\eta) \sim \SLE_\kappa(\rho_1,\rho_2,\rho_3)$ where the weights $\rho_1,\rho_2,\rho_3$ are located at the force points $-1,-2,y_1$, respectively, and are determined from $a,b,c$ by the equations
\begin{align*}
a  = -(1+\rho_1)&\lambda - \chi\pi,\quad
b  = -(1+\rho_1+\rho_2)\lambda - \chi\pi,\\
c  =& -(1+\rho_1+\rho_2+\rho_3)\lambda - \chi\pi.
\end{align*}
As in the proof of the previous lemma, the hypothesis on $a$ implies $\rho_1 \geq \tfrac{\kappa}{2}-2$, and that $b \geq \lambda$ gives $\rho_1+\rho_2+\rho_3 \leq \tfrac{\kappa}{2}-4$.  Finally, that $c \in (-\lambda,\lambda)$ gives $\rho_1+\rho_2 \in (\tfrac{\kappa}{2}-4,\tfrac{\kappa}{2}-2)$.  Consequently, it is easy to see from \cite[Lemma~15]{DUB_DUAL} that $\psi(\eta)$ first exits $\h$ in $[y_1,-2]$, which is to say that $\eta$ first exits $\strip$ in $[z_0,z_1]$.
\end{proof}

\subsection{A special case of Theorem~\ref{thm::coupling_uniqueness}}
\label{subsec::uniqueness_non_boundary_intersecting}

In this subsection, we will prove Theorem~\ref{thm::coupling_uniqueness} for a certain class of boundary data in which the flow line is non-boundary intersecting.

\begin{figure}[h!]
\begin{center}
\subfigure[The boundary data for a flow line and counterflow line.]{\label{subfig::dubedat1}
\includegraphics[scale=0.85,page=1]{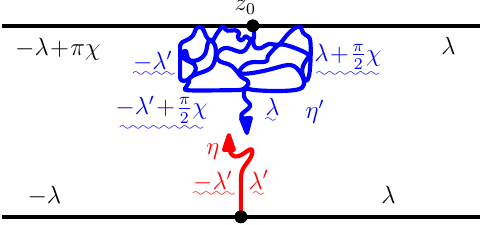}}
\hspace{0.01\textwidth}
\subfigure[Subtracting $\frac{\pi}{2}\chi$ from the boundary data on the left makes the upper boundary conditions symmetric.]{\label{subfig::dubedat2}
\includegraphics[scale=0.85,page=3]{figures/dubedat.pdf}}
\caption{\label{fig::dubedat} The path $\eta$ represents a flow line and $\eta'$ a counterflow line in the reverse direction, once we subtract $\frac{\pi}{2} \chi$ from the field.   The boundary data on the vertical segment of $\eta$ matches that of the left of the vertical segment of $\eta'$, suggesting that $\eta$ could ``merge into'' the left boundary of $\eta'$.
Since $-\lambda' + \frac{\pi}{2} \chi \in (-\lambda, \lambda)$, Figure~\ref{fig::hittingrange} shows $\eta$ must hit the upper boundary of the strip somewhere on the left (not the right) semi-infinite interval.  Using Figure~\ref{fig::hittingsinglepoint}, we see that $\eta'$ must first hit the lower boundary of the strip at the lower dot.  This is because (subtracting $\frac{\pi}{2} \chi$ from everything) we find that since $-\lambda - \frac{\pi}{2} \chi \leq  -\lambda'$ (left side) and $\lambda - \frac{\pi}{2} \chi = \lambda'$ (right side).  If we grow $\eta$ and $\eta'$ (up to some time before they intersect) then after a coordinate change, we can map the complement of these paths back to the original strip so that the boundary conditions are the same as they originally were (though the locations of the dots may be translated).  Thus $\eta'$ a.s.\ hits each of a countable dense set of points along $\eta$ (up until the first time it hits the upper interval) and $\eta$ always first hits to the left of the tip of $\eta'([0,t])$ (for some countable dense set of $t$ values).}
\end{center}
\end{figure}

\begin{lemma}
\label{lem::hit_in_order}
Suppose we are in the setting of Figure~\ref{fig::dubedat}.  That is, we fix $\kappa \in (0,4)$, let $h$ be a GFF on $\strip$ whose boundary data is as depicted in Figure~\ref{subfig::dubedat1}, $\eta$ the flow line of $h$ starting at $0$, $\eta'$ the counterflow line of $h-\tfrac{\pi}{2}\chi$ starting at $z_0$, and assume that $\eta,\eta',h$ are coupled together so that $\eta$ and $\eta'$ are conditionally independent given $h$.  Let $\tau$ be any stopping time for $\eta$.  Then $\eta'$ almost surely first hits  $\stripbot \cup \eta([0,\tau])$ at $\eta(\tau)$.  In particular, $\eta'$ contains $\eta$ and hits the points of $\eta$ in reverse chronological order: if $s < t$ then $\eta'$ hits $\eta(t)$ before $\eta(s)$.
\end{lemma}
\begin{proof}
Conditional on the realization of $\eta|_{[0,\tau]}$, the field is equal in distribution to a GFF whose boundary data is $-\lambda'+\chi \cdot {\rm winding}$ on the left side of $\eta([0,\tau])$ and $\lambda'+\chi \cdot {\rm winding}$ on the right side of $\eta([0,\tau])$ (see Figure~\ref{fig::conditional_boundary_data}).  We note that $\eta'$ viewed as a path in $\strip \setminus \eta([0,\tau])$ has a continuous Loewner driving function up until it first hits $\eta([0,\tau])$ (see, for example Proposition~\ref{prop::cont_driving_function}).  On the event that $t$ occurs before the first time that $\eta'$ accumulates in $\stripbot \cup \eta([0,\tau])$, Proposition~\ref{gff::prop::cond_union_local} and Proposition~\ref{gff::prop::cond_union_mean} together tell us that the boundary data for the conditional law of $h$ given both $\eta|_{[0,\tau]}$ and $\eta'|_{[0,t]}$ in the unbounded component of $\strip \setminus (\eta([0,\tau]) \cup \eta'([0,t]))$ is that of a GFF whose boundary conditions along $\stripbot \cup \eta([0,\tau])$ (resp.\ $\striptop \cup \eta'([0,t])$) agree with those of the conditional law of $h$ given just $\eta|_{[0,\tau]}$ (resp.\ $\eta'|_{[0,t]}$) alone.  This allows us to use Theorem~\ref{thm::martingale} to compute the conditional law of $\eta'$ given $\eta([0,\tau])$.  In particular, we see that conformally mapping $\strip \setminus \eta([0,\tau])$ back to $\strip$ with $\eta(\tau)$ sent to $0$ and $\pm \infty$ fixed leaves us in the setting of the lemma with $\tau = 0$.  Thus we just need to argue that $\eta'$ first hits $\stripbot$ at $0$.  The boundary data for $\eta'$ on $\stripbot$ is (as described in Figure~\ref{subfig::dubedat2})
\[-\lambda-\frac{\pi}{2} \chi \leq -\lambda' \quad\text{on}\quad  (-\infty,0) \quad\text{and}\quad \lambda-\tfrac{\pi}{2} \chi = \lambda' \quad\text{on}\quad (0,\infty);\]
recall~\eqref{eqn::fullrevolution}.
Consequently, Lemma~\ref{lem::hitting_single_point} implies $\eta'$ almost surely first hits $\stripbot$ at $0$, as desired.  Since $\eta$ up until the first time it hits $\striptop$ is almost surely continuous (recall Remark~\ref{rem::continuity_non_boundary}), it follows that $\eta'$ almost surely contains $\eta$ by applying this argument to a countable dense set of stopping times (e.g., the positive rationals).  It is also easy to see from this that $\eta'$ hits the points of $\eta$ in reverse chronological order (recall Figure~\ref{fig::dubedat}).
\end{proof}

\begin{remark}
\label{rem::cfl_hitting_general_bd}
The proof of Lemma~\ref{lem::hit_in_order} has two inputs:
\begin{enumerate}[(a)]
\item\label{lst::cont} the continuity of $\eta$ up until it first accumulates in $\striptop$ and that
\item\label{lst::hitting_single_point} for every $\eta$ stopping time $\tau$, $\eta'$ almost surely first exits $\strip \setminus \eta([0,\tau])$ at $\eta(\tau)$.
\end{enumerate}
Condition~\eqref{lst::cont} holds more generally when the boundary data on $\stripbot$ is piecewise constant, changes a finite number of times, and is at most $-\lambda+\pi \chi$ to the left of $0$ and at least $\lambda - \pi \chi$ to the right of zero.  This ensures that $\eta$ almost surely does not hit $\stripbot$ after starting, so is almost surely continuous since its law is mutually absolutely continuous with respect to $\SLE_\kappa$ ($\rho \equiv 0$) by the Girsanov theorem (recall Remark~\ref{rem::continuity_non_boundary}).  Condition~\eqref{lst::hitting_single_point} holds when the boundary data of $h-\tfrac{\pi}{2} \chi$ on $\striptop$ is piecewise constant, changes a finite number of times, and is at most $-\lambda'-\pi\chi$ to the left of $z_0$ and at least $\lambda'+\pi\chi$ to the right of $z_0$ (or we have $-\lambda',\lambda'$ boundary data as in the statement of Lemma~\ref{lem::hit_in_order}; note that the reason for the sign is that $\chi=\chi(\kappa) =-\chi(\kappa')$).  This also implies the continuity of $\eta'$ up until it hits $\stripbot$, arguing using absolute continuity as before (this will be important for a more general version of Proposition~\ref{prop::duality}).  Additionally, we need that the boundary data of $h-\tfrac{\pi}{2}\chi$ on $\stripbot$ is piecewise constant, changes a finite number of times, is not more than $-\lambda'$ to the left of $0$, and is at least $\lambda'$ to the right of $0$.  In short, Lemma~\ref{lem::hit_in_order} holds when the boundary data of $h$ is ``large and negative'' to the left of $0$ and $z_0$ and ``large and positive'' to the right of $0$ and $z_0$.
\end{remark}

\begin{proposition}
\label{prop::duality}
Suppose we are in the setting of Figure~\ref{fig::dubedat}.  That is, we fix $\kappa \in (0,4)$, let $h$ be a GFF on $\strip$ whose boundary data is as depicted in Figure~\ref{subfig::dubedat1}, let $\eta$ be the flow line of $h$ starting at $0$, $\eta'$ the counterflow line of $h-\tfrac{\pi}{2}\chi$ starting at $z_0$, and assume that $\eta,\eta',h$ are coupled together so that $\eta,\eta'$ are conditionally independent given $h$.  Almost surely, $\eta$ is equal to the left boundary $\eta'$.
\end{proposition}
\begin{proof}
Lemma~\ref{lem::hit_in_order} implies that the range of $\eta'$ contains $\eta$.
To complete the proof, we just need to show that $\eta$ is to the left of $\eta'$.  To see this, fix any stopping time $\tau'$ for $\eta'$ such that $\eta'$ almost surely has not hit $\stripbot$.  Then $\eta$ almost surely exits $\strip \setminus \eta'([0,\tau'])$ on the left side of $\eta'([0,\tau'])$ or to the left of $z_0$ on $\striptop$.  Indeed, arguing as in the proof of Lemma~\ref{lem::hit_in_order}, we can conformally map the picture back to $\strip$ to see that it suffices to show that $\eta$ almost surely first hits $\striptop$ to the left of $z_0$.  This, in turn, is a consequence of Lemma~\ref{lem::hitting_interval}.

\begin{figure}[h!]
\begin{center}
\includegraphics[scale=0.85]{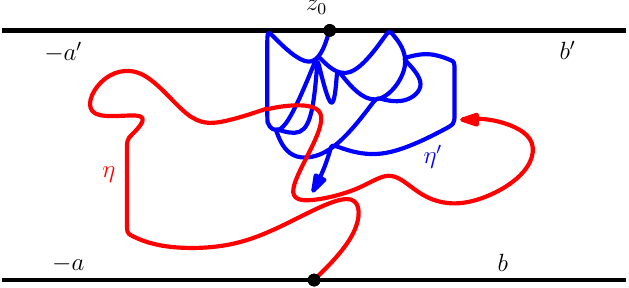}
\caption{\label{fig::dub_wrap_around} Suppose that we have the same setup as Figure~\ref{fig::dubedat}.  Fix an $\eta'$ stopping time $\tau'$.  Then it must be that $\eta'(\tau')$ is to the right of $\eta$.  Indeed, if this were not the case, then after hitting the left side of $\eta'([0,\tau'])$, say at time $\tau$, $\eta$ would have to wrap around $\eta'(\tau')$ and then hit the right side of $\eta'([0,\tau'])$, say at time $\sigma$.  This is a contradiction since $\eta'$ hits all of the points in $\eta$ in reverse chronological order.  In particular, we have that $\eta(\tau),\eta(\sigma) \in \eta'([0,\tau'])$ while there exists $s \in (\tau,\sigma)$ such that $\eta(s) \notin \eta'([0,\tau'])$.  We conclude that $\eta'$ lies to the right of $\eta$ applying this result to a dense collection of stopping times $\tau'$ (e.g., the positive rationals) and using the continuity of $\eta'$.
}
\end{center}
\end{figure}

We will now argue that $\eta'(\tau')$ is to the right of $\eta$.  Let $\tau$ be the first time $\eta$ exits $\strip \setminus \eta'([0,\tau'])$.  There are two possibilities: $\eta(\tau)$ is either in $\striptop$ or in $\eta'([0,\tau'])$.  In the former case we are done, so we shall assume that we are in the latter.  Now, the only way that $\eta'(\tau')$ could be strictly to the left of $\eta$ is if after $\tau$, $\eta$ wraps around $\eta'([0,\tau'])$ and hits its right side.  This implies the existence of times $t_1 < t_2 < t_3$ such that $\eta(t_1),\eta(t_3) \in \eta'([0,\tau'])$ but $\eta(t_2) \notin \eta'([0,\tau'])$.  This is a contradiction since $\eta'$ absorbs the points of $\eta$ in reverse chronological order by Lemma~\ref{lem::hit_in_order}.  The result now follows by taking a countable dense collection of stopping times $\tau'$ (e.g., the positive rationals) and invoking the almost sure continuity of $\eta'$ up until when it first hits $\stripbot$.
\end{proof}

\begin{remark}
\label{rem::duality_general_bc}
In addition to the hypotheses described in Remark~\ref{rem::cfl_hitting_general_bd} (which imply the almost sure continuity of $\eta'$ up until when it first accumulates in $\stripbot$), Proposition~\ref{prop::duality} requires the boundary data of $h$ to be such that $\eta$ almost surely first accumulates in $\striptop$ to the left of $z_0$.  By Lemma~\ref{lem::hitting_single_point}, this means that the boundary data for $h$ on $\striptop$ should be piecewise constant, change only a finite number of times, and be less than $\lambda$ to the left of $z_0$ and be at least $\lambda$ to the right of $z_0$.
\end{remark}

By combining Remark~\ref{rem::cfl_hitting_general_bd} and Remark~\ref{rem::duality_general_bc}, we obtain the following extension of Proposition~\ref{prop::duality}.
\begin{proposition}
\label{prop::duality_many_force_points}
Suppose we are in the setting of Figure~\ref{fig::dubedat} where the boundary data on $\stripbot$ is replaced by any piecewise constant function which changes a finite number of times, does not exceed $-\lambda+\pi \chi$ to the left of $0$, and is at least $\lambda$ to the right of $0$.  Assume furthermore that the boundary data of $h$ on $\striptop$ is piecewise constant, changes a finite number of times, and does not exceed $-\lambda$ to the left of $z_0$ and is at least $\lambda+\pi \chi$ to the right of $z_0$.  Almost surely, the flow line $\eta$ starting from $0$ is equal to the left boundary of the counterflow line $\eta'$ of the field minus $\tfrac{\pi}{2} \chi$ starting from $z_0$.  Here, we assume that $\eta,\eta',h$ are coupled together so that $\eta$ is conditionally independent of $\eta'$ given $h$.
\end{proposition}

Using Proposition~\ref{prop::duality_many_force_points}, we obtain Theorem~\ref{thm::coupling_uniqueness} in the special case we have only boundary force points with weights $\ul{\rho} = (\ul{\rho}^L,\ul{\rho}^R)$ with $|\ul{\rho}^L| = k$ and $|\ul{\rho}^R| = \ell$ satisfying:
\begin{equation}
\label{eqn::non_boundary_intersecting_rho}
 \begin{split}
 \sum_{i=1}^j \rho^{i,L} \geq \tfrac{\kappa}{2} - 2 \quad&\text{for all}\quad 1 \leq j \leq k \quad\text{and}\\
   \sum_{i=1}^j \rho^{i,R} \geq 0 \quad&\text{for all}\quad 1 \leq j \leq \ell.
   \end{split}
\end{equation}

\begin{proof}[Proof of Theorem~\ref{thm::coupling_uniqueness} for $\kappa \in (0,4)$ assuming~\eqref{eqn::non_boundary_intersecting_rho}]

We suppose that we are in the setting of Figure~\ref{fig::dubedat} but with boundary data satisfying Remark~\ref{rem::cfl_hitting_general_bd} and Remark~\ref{rem::duality_general_bc}.  That is, we consider the Gaussian free field $h$ on the strip $\strip$ with piecewise constant boundary data on $\stripbot$ which changes a finite number of times, is at most $-\lambda+\pi \chi$ on $(-\infty,0)$ and at least $\lambda$ on $(0,\infty)$.  Moreover, we assume that the boundary data of $h$ on $\striptop$ is at most $-\lambda$ to the left of $z_0$ and at least $\lambda+\pi \chi$ to the right of $z_0$.  Then Proposition~\ref{prop::duality_many_force_points} implies that the flow line $\eta$ of $h$ starting at $0$ is equal to the left boundary of the counterflow line $\eta'$ starting at $z_0$ of $h-\tfrac{\pi}{2} \chi$, where $\eta,\eta',h$ are coupled together so that $\eta$ is conditionally independent of $\eta'$ given $h$.  Since $\eta$ and $\eta'$ are conditionally independent given $h$, it follows that $\eta$ is almost surely determined by~$h$.  The result follows since by adjusting the boundary data of $h$, we can arrange so that $\eta \sim \SLE_\kappa(\ul{\rho})$ with any choice of weights $\ul{\rho}$ satisfying~\eqref{eqn::non_boundary_intersecting_rho}.
\end{proof}

At this point in the article, it follows that the flow lines thus considered (almost surely non-boundary intersecting) are deterministic functions of the GFF.  We have not yet shown that the counterflow lines in any setting are deterministic functions of the GFF; this will be shown in Section~\ref{sec::non_boundary_intersecting}.

We finish this section with the following proposition, which gives an alternative proof of the transience of certain non-boundary intersecting $\SLE_\kappa(\ul{\rho})$ processes using duality (see \cite{RS05} for another approach for ordinary $\SLE$).

\begin{proposition}
\label{prop::transience}
Suppose that $\eta \sim \SLE_\kappa(\ul{\rho})$, $\kappa \in (0,4)$, from $0$ to $\infty$ in $\h$ with the weights $\ul{\rho}$ satisfying~\eqref{eqn::non_boundary_intersecting_rho}.  Then $\lim_{t \to \infty} \eta(t) = \infty$ almost surely.
\end{proposition}
\begin{proof}
This is similar to the proof of Theorem~\ref{thm::coupling_uniqueness} for $\kappa \in (0,4)$ assuming that the weights $\ul{\rho}$ satisfy~\eqref{eqn::non_boundary_intersecting_rho}.  Indeed, we can choose the boundary data of a GFF $h$ on $\strip$ so that the flow line $\eta$ of $h$ starting at $0$ is an $\SLE_\kappa(\ul{\rho})$ process for any choice of weights $\ul{\rho}$ satisfying~\eqref{eqn::non_boundary_intersecting_rho}.  Let $\eta'$ be the counterflow line of $h - \tfrac{\pi}{2} \chi$ starting at $z_0$, taken to be conditionally independent of $\eta$ given $h$.  Let $\tau$ be any positive stopping time for $\eta$ such that $\eta$ almost surely has not hit $z_0$ by time $\tau$.  We know from the proof of Lemma~\ref{lem::hit_in_order} that $\eta'$ almost surely first exits $\strip \setminus \eta([0,\tau])$ at $\eta(\tau)$, say at time $\tau'$.  Moreover, it follows from Proposition~\ref{prop::duality_many_force_points} that $\eta|_{[\tau,\infty)}$ is equal to the left boundary of $\eta'([0,\tau'])$ (we know that $\eta$ lies to the left of $\eta'([0,\tau'])$ and that $\eta'$ hits the points of $\eta$ in reverse chronological order).  By Remark~\ref{rem::continuity_non_boundary}, we know that $\eta'|_{[0,\tau']}$ is almost surely continuous which implies that the left boundary of $\eta'([0,\tau'])$ is locally connected.  Therefore the range of $\eta$ is locally connected (we know that $\eta|_{[0,\tau)}$ is locally connected by Remark~\ref{rem::continuity_non_boundary}), hence $\eta$ is continuous even when it hits $z_0$.  Applying a conformal map $\psi \colon \strip \to \h$ completes the proof of the proposition.
\end{proof}

Part of Theorem~\ref{thm::continuity}, which will be proved in Section~\ref{subsec::many_boundary_force_points} for $\kappa \in (0,4)$ and in Section~\ref{subsec::counterflow} for $\kappa' \in (4,\infty)$, is the continuity of general $\SLE_\kappa(\ul{\rho})$ and $\SLE_{\kappa'}(\ul{\rho})$ processes upon hitting their terminal point.  This is equivalent to transience when the terminal point is $\infty$.

\section{The non-boundary-intersecting regime}
\label{sec::non_boundary_intersecting}

This section contains two main results.  First, we will show that the flow lines of the GFF $h$ enjoy the same monotonicity properties as if $h$ were a smooth function: namely, if $\theta_1 < \theta_2$ and $\eta_{\theta_i}$ is the flow line of $h$ with angle $\theta_i$ for $i=1,2$ started at a given boundary point then $\eta_{\theta_1}$ almost surely lies to the right of $\eta_{\theta_2}$ (Proposition~\ref{prop::monotonicity_non_boundary}).  Second, we will show that $\SLE_{16/\kappa}$ for $\kappa \in (0,4)$ can be realized as a so-called ``light cone'' of angle-varying $\SLE_{\kappa}$ flow lines (Proposition~\ref{prop::light_cone_construction}).  The proofs of this section will apply to a certain class of boundary data in which the flow lines are non-boundary-intersecting.  In Section~\ref{sec::uniqueness}, we will extend these results to the setting of general piecewise constant boundary data, in particular in the setting in which the paths may hit the boundary.

\subsection{Monotonicity of flow and counterflow lines}
\label{subsec::monotonicity}

In order to prove our first version of the monotonicity result, it will be more convenient for us to work on the strip $\strip = \R \times (0,1)$.  Throughout, we let $\striptop$ and $\stripbot$ denote the upper and lower boundaries of $\CS$, respectively.  This puts us into a setting in which we can make use of the $\SLE$ duality theory from the previous section.  Assume that $\theta_1 < \theta_2$.  Then we know from Proposition~\ref{prop::duality_many_force_points} that $\eta_{\theta_1}$ is almost surely the left boundary of the counterflow line $\eta_{\theta_1}'$ of $h+(\theta_1-\tfrac{\pi}{2}) \chi$ emanating from the upper boundary $\striptop$ of $\strip$, provided the boundary data of $h$ is chosen appropriately (we will spell out the restrictions on the boundary data in the statements of the results below).  Thus to show that $\eta_{\theta_2}$ passes to the left of $\eta_{\theta_1}$ it suffices to show that $\eta_{\theta_2}$ passes to the left of $\eta_{\theta_1}'$.  This in turn is a consequence of the following proposition, which gives us the range of angles in which a flow line passes either to the left or to the right of a counterflow line:

\begin{figure}[h]
\begin{center}
\subfigure[On the left: $\theta > \tfrac{1}{\chi}(\lambda-\lambda') = \tfrac{\pi}{2}$.]{
\includegraphics[scale=0.85,page=1]{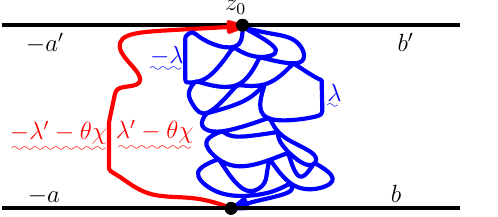}}
\hspace{-0.05\textwidth}
\subfigure[On the right: $\theta < \tfrac{1}{\chi}(\lambda' - \lambda) = -\tfrac{\pi}{2}$.]{
\includegraphics[scale=0.85,page=2]{figures/dual_left.pdf}}
\caption{\label{fig::flow_counterflow_left_right} Consider a GFF on the infinite strip $\strip$ with boundary data as described in the illustrations above.  Let $\eta_\theta$ be the flow line starting at $0$ with initial angle $\theta$ (which corresponds to adding $\theta \chi$ to the boundary data) and $\eta'$ the counterflow line starting at $z_0$.  We assume that $a,b$ are chosen sufficiently large so that both $\eta'$ and $\eta_\theta$ almost surely do not intersect the lower boundary $\stripbot$ of $\strip$ except at $0$ and $a',b'$ are sufficiently large so that $\eta'$ almost surely does not intersect the upper boundary $\striptop$ of $\strip$ except at $z_0$.  We will show in Proposition~\ref{prop::flow_counterflow_left_right} that if $\theta > \tfrac{1}{\chi}(\lambda-\lambda') = \tfrac{\pi}{2}$, then $\eta_\theta$ almost surely passes to the left of $\eta'$ and if $\theta < \tfrac{1}{\chi}(\lambda'-\lambda) = -\tfrac{\pi}{2}$, then $\eta_\theta$ almost surely passes to the right of $\eta'$.  By $\SLE$ duality, this implies the monotonicity of $\eta_\theta$ in $\theta$.}
\end{center}
\end{figure}

\begin{proposition}
\label{prop::flow_counterflow_left_right}
Suppose that $h$ is a GFF on $\strip$ with boundary data as described in Figure~\ref{fig::flow_counterflow_left_right}.  Assume $a',b' \geq \lambda'+\pi \chi$ and $a,b \geq \lambda'$.  Let $\eta'$ be the counterflow line of $h$ starting at $z_0$.  Fix $\theta$ such that
\begin{equation}
\label{eqn::angle_assump_flow_counterflow}
 \frac{\lambda-\pi \chi-b}{\chi} \leq \theta \leq \frac{a+\pi \chi-\lambda}{\chi}
\end{equation}
and let $\eta_\theta$ be the flow line of $h+\theta \chi$ starting from $0$.  If $\theta > \tfrac{1}{\chi}(\lambda-\lambda') = \tfrac{\pi}{2}$, then $\eta_\theta$ almost surely passes to the left of $\eta'$ and if $\theta < \tfrac{1}{\chi}(\lambda'-\lambda) = -\tfrac{\pi}{2}$ then $\eta_\theta$ almost surely passes to the right of $\eta'$.
\end{proposition}

The hypothesis~\eqref{eqn::angle_assump_flow_counterflow} implies that $\eta_\theta$ almost surely accumulates in $\striptop$ or tends to $\pm \infty$ before hitting $\stripbot$ (see Remark~\ref{rem::boundary_data}; upon proving Theorem~\ref{thm::continuity}, we in fact will be able to prove that $\eta_\theta$ almost surely does not hit $\stripbot$ after time $0$).  The reason we assume $a',b' \geq \lambda'+\pi \chi$ and $a,b \geq \lambda'$ is that the former implies that $\eta'$ almost surely does not intersect $\striptop$ except at $z_0$ and the latter implies that $\eta'$ intersects $\stripbot$ only when it terminates at $0$.  Consequently, $\eta'$ is almost surely continuous (recall Remark~\ref{rem::continuity_non_boundary}).  Moreover, it implies that if $\theta < \tfrac{1}{\chi}(\lambda'-\lambda)=-\tfrac{\pi}{2}$ then $-a'+\theta \chi < -\lambda$ so that $\eta_\theta$ does not hit the side of $\striptop$ which is to the left of $z_0$ and if $\theta > \tfrac{1}{\chi}(\lambda-\lambda')=\tfrac{\pi}{2}$ then $b' + \theta \chi > \lambda$ so that $\eta_\theta$ does not hit the side of $\striptop$ which is to the right of $z_0$ (see Figure~\ref{fig::hittinginterval}).

We are now going to give an overview of the proof of Proposition~\ref{prop::flow_counterflow_left_right}, which is based on an extension of the proof of Proposition~\ref{prop::duality}.  We assume without loss of generality that $\theta > \tfrac{1}{\chi}(\lambda-\lambda') = \tfrac{\pi}{2}$.  We begin by fixing an $\eta'$-stopping time $\tau'$ and then show that $\eta'(\tau')$ almost surely lies to the right of $\eta_\theta$.  As in the proof of Proposition~\ref{prop::duality}, we will first show that $\eta_\theta$ almost surely first exits $\strip \setminus \eta'([0,\tau'])$ on either the left side of $\eta'([0,\tau'])$ or the side of $\striptop$ which is to the left of $z_0$ (see Figure~\ref{fig::flow_counterflow_first_exit}, which also contains an explanation of the critical angles).  If we are in the latter situation, then we have the desired monotonicity.  If not, we suppose for contradiction that $\eta'(\tau')$ is to the left of $\eta_\theta$.  Then after $\eta_\theta$ hits $\eta'([0,\tau'])$, it must be that $\eta_\theta$ wraps around $\eta'(\tau')$ and hits the right side of $\eta'([0,\tau'])$ since $\eta_\theta$ cannot exit $\strip$ on the part of $\striptop$ which lies to the right of $z_0$.  This, however, cannot happen because the evolution of $\eta_\theta$ near the right side of $\eta'([0,\tau'])$ looks like (up to conformal transformation and a mutually absolutely continuous change of measures) an $\SLE$ process on the strip $\strip$ starting from $0$ which cannot hit its upper boundary $\striptop$ (see Figure~\ref{fig::criticalforintersection}).  We begin by proving the following extension of Lemma~\ref{lem::hitting_and_bouncing}, which serves to make this last point precise.

\begin{figure}[h]
\begin{center}
\subfigure{
\includegraphics[scale=0.85,page=2]{figures/cannot_hit_intervals2.pdf}}
\subfigure{
\includegraphics[scale=0.85,page=1]{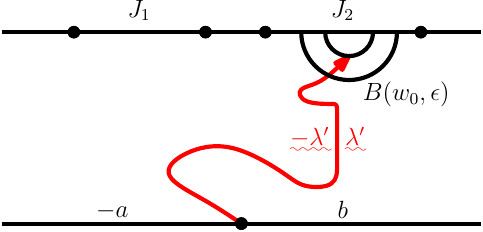}}
\caption{\label{fig::cannot_hit_intervals}  Suppose that $h$ is a GFF on the strip $\strip$ and let $J \subseteq \striptop$ be open.  Write $J = \cup_k J_k$ where the $J_k$ are disjoint open intervals and assume that $h|_{J_k} \equiv c_k$ for given constants $c_k \notin (-\lambda,\lambda)$.  Assume that $0 < T_0 < \infty$ is an $\eta$-stopping time such that $\eta|_{[0,T_0]}$ is almost surely continuous.  Then $\eta$ cannot hit $J$ by time $T_0$.  To see this, fix $w_0$ in the interior of any $J_k$ and pick $\epsilon > 0$ such that $\partial B(w_0,\epsilon) \cap \striptop$ is contained in the interior of $J_k$.  Fix $\epsilon' \in (0,\epsilon)$.  The evolution of $\eta$ after it hits $\partial B(w_0,\epsilon')$ for the first time up until when it subsequently exits $B(w_0,\epsilon)$ is (after applying a conformal transformation) mutually absolutely continuous with respect to the setup described in Figure~\ref{fig::criticalforintersection} by Proposition~\ref{prop::gff_abs_continuity}, consequently $\eta$ almost surely exits $B(w_0,\epsilon)$ before hitting $B(w_0,\epsilon) \cap J_k$.  The result follows by considering the collection of balls centered at some countable dense set of points in $J_k$ with rational radii and invoking the continuity of $\eta|_{[0,T_0]}$.}
\end{center}
\end{figure}

\begin{lemma}
\label{lem::flow_cannot_hit} Suppose that $h$ is a GFF on $\strip$ whose boundary data is as described in Figure~\ref{fig::cannot_hit_intervals} and let $\eta$ be the flow line of $h$ starting at $0$.  If $\eta|_{[0,T_0]}$ is almost surely continuous for some $\eta$-stopping time $0 < T_0 < \infty$, then $\eta([0,T_0]) \cap J = \emptyset$ almost surely.
\end{lemma}
In order for Lemma~\ref{lem::flow_cannot_hit} to be meaningful, there must exist a non-zero $\eta$-stopping time $T_0$ such that $\eta|_{[0,T_0]}$ is continuous.  This will be clear for the application we have in mind in this section.  Upon establishing Theorem~\ref{thm::continuity} in Section~\ref{sec::uniqueness}, we will be able to apply Lemma~\ref{lem::flow_cannot_hit} to the flow lines of a GFF with piecewise constant boundary conditions.
\begin{proof}[Proof of Lemma~\ref{lem::flow_cannot_hit}]
Write $J = \cup_k J_k$ where the $J_k$ are pairwise disjoint intervals in $\striptop$.  Fix $k \in \N$ and $w_0$ in the interior of $J_k$.  Let $\epsilon > 0$ be such that $\partial B(w_0,\epsilon) \cap \striptop$, where $\striptop$ is the upper boundary of $\strip$, is contained in the interior of $J_k$ and fix $\epsilon' \in (0,\epsilon)$.  Let $\tau$ be the first time $t$ that $\eta$ hits $\partial B(w_0,\epsilon')$ and $\sigma$ the first time $t$ after $\tau$ that $\eta$ is not in $B(w_0,\epsilon)$.  Assume that $\eta(\tau) \notin \striptop$.  Let $D$ be the connected component of $\strip \setminus \eta([0,\tau])$ which contains $w_0$.  By Proposition~\ref{prop::gff_abs_continuity}, the law of $h|_{B(w_0,\epsilon)}$ conditional on the realization of $\eta|_{[0,\tau]}$ is mutually absolutely continuous with respect to the law of the restriction of a GFF $\wt{h}$ on $D$ to $B(w_0,\epsilon)$ whose boundary data is $\lambda' - \chi \cdot {\rm winding}$ (resp.\ $-\lambda' - \chi \cdot {\rm winding}$) on the right (resp.\ left) side of $\eta([0,\tau])$, the same as $h$ on $\stripbot$, the lower boundary of $\strip$, and identically $c_k$ on all of $\striptop$.  Let $\wt{\eta}$ be the flow line of $\wt{h}$ starting at $\eta(\tau)$.  Then Lemma~\ref{lem::hitting_and_bouncing} implies that $\wt{\eta}$ almost surely does not hit $\striptop$ and, in particular, exits $B(w_0,\epsilon)$ before hitting $\striptop$.  Applying Proposition~\ref{prop::gff_abs_continuity} a second time implies that with $E(w_0,\epsilon',\epsilon)$ the event $\{\eta([\tau,\sigma]) \cap \striptop \neq \emptyset,\ \ \sigma \leq T_0\}$ we have $\p[E(w_0,\epsilon',\epsilon) \giv \eta(\tau) \notin \striptop] = 0$.

Let $(w_j)$ be a countable dense set in $J$ and let $\CD$ be the set of all triples of the form $(w_j,r',r)$ where $0 < r' < r$ are rational.  If $\eta$ hits $\striptop$ with positive probability before time $T_0$, the almost sure continuity of $\eta|_{[0,T_0]}$ implies there exists $(w_j,r',r) \in \CD$ such that $\p[E(w_j,r',r) \giv \eta(\tau) \notin \striptop ] > 0$ where $\tau \leq T_0$ is the first time $\eta$ hits $\partial B(w_j,r')$.  This is a contradiction, which proves the lemma.
\end{proof}

\begin{remark}
\label{rem::flow_cannot_hit}
By changing coordinates from $\strip$ to $\h$, Lemma~\ref{lem::flow_cannot_hit} implies the following.  Suppose that $\eta$ is an $\SLE_\kappa(\ul{\rho})$ process on $\h$ with weights $\ul{\rho} = (\ul{\rho}^L;\ul{\rho}^R)$ placed at force points $(\ul{x}^L;\ul{x}^R)$.  Assume that $\eta|_{[0,T_0]}$ is almost surely continuous for some $\eta$ stopping time $0 < T_0 < \infty$.  Then $\eta([0,T_0])$ almost surely does not intersect any interval $(x^{i+1,L}, x^{i,L})$ (resp.\ $(x^{i,R},x^{i+1,R})$) such that $\sum_{s=1}^i \rho^{s,L} \geq \tfrac{\kappa}{2}-2$ or $\sum_{s=1}^i \rho^{s,L} \leq \tfrac{\kappa}{2}-4$ (resp.\ $\sum_{s=1}^i \rho^{s,R} \geq \tfrac{\kappa}{2}-2$ or $\sum_{s=1}^i \rho^{s,R} \leq \tfrac{\kappa}{2}-4$).
\end{remark}

\begin{figure}[h]
\begin{center}
\subfigure[The evolution of $\eta'$ up to time $\tau'$ and $\eta_\theta$ up until the first time it hits $\eta'(\text{[}0,\tau'\text{]})$.]{
\includegraphics[scale=0.85,page=1]{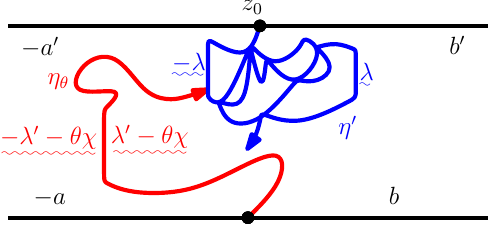}}
\subfigure[\label{fig::flow_counterflow_first_exit_conf_image}The boundary data for $\eta_\theta$ after conformally mapping the picture on the left hand side back to the strip with $\pm \infty$ and $0$ fixed, $w_0$ the image of $z_0$.]{
\includegraphics[scale=0.85,page=2]{figures/dual_first_exit_left.pdf}}
\caption{\label{fig::flow_counterflow_first_exit}Suppose that we have the same setup as in Figure~\ref{fig::flow_counterflow_left_right} and let $\tau'$ be a stopping time for $\eta'$.  We assume $\theta$ is chosen so that $\eta_\theta$ almost surely does not hit $\stripbot$ after time $0$.  If $\theta > \tfrac{1}{\chi}(\lambda-\lambda') = \tfrac{\pi}{2}$ so that $\lambda' + \theta \chi > \lambda$ and $b'+\theta \chi > \lambda$ then by Figure~\ref{fig::hittinginterval}, $\eta_\theta$ exits $\strip \setminus \eta'([0,\tau'])$ on either the left side of $\eta'([0,\tau'])$ or on the left side of $\striptop$, almost surely.  If $\theta < \tfrac{1}{\chi}(\lambda'-\lambda) = -\tfrac{\pi}{2}$ so that $-\lambda'+\theta \chi < -\lambda$ and $-a'+\theta \chi < -\lambda$, $\eta_\theta$ first exits on the right side of $\eta'([0,\tau'])$ or the right side of $\striptop$, almost surely.}
\end{center}
\end{figure}

\begin{lemma}
\label{lem::flow_counterflow_not_hit_right}
Suppose that we have the same hypotheses as Proposition~\ref{prop::flow_counterflow_left_right} with $\theta > \tfrac{1}{\chi}(\lambda-\lambda') = \tfrac{\pi}{2}$ fixed.  Let $\tau'$ be an $\eta'$ stopping time such that, almost surely, $\eta'$ has not hit $0$ by time $\tau'$.  For each $t > 0$, let $K_t'$ be the hull of $\eta'([0,t])$, i.e.\ the complement of the unbounded connected component of $\strip \setminus \eta'([0,t])$.  Let $\tau$ be any stopping time for the filtration $\CF_t = \sigma(\eta(s) : s \leq t,\ \eta'([0,\tau']))$ and let $E = \{ \dist(\eta(\tau),K_{\tau'}') > 0\}$.  Then on $E$, $\eta|_{[\tau,\infty)}$ intersects neither the right side of $\eta'([0,\tau'])$ nor the part of $\striptop$ which lies to the right of $z_0$ before hitting either the left side of $\eta'([0,\tau'])$ or the part of $\striptop$ which lies to the left of $z_0$.
\end{lemma}
\begin{proof}
If $\eta([0,\tau]) \cap K_{\tau'}' = \emptyset$, then the result is immediate from the argument described in Figure~\ref{fig::flow_counterflow_first_exit}.  Thus for the rest of the proof, we shall assume that $\eta([0,\tau]) \cap K_{\tau'}' \neq \emptyset$.  Since $\eta$ is almost surely a simple path, there exists a unique connected component $D$ of $\strip \setminus (\eta([0,\tau]) \cup K_{\tau'}')$ such that for some $\epsilon_0 > 0$, $\eta(\tau+\epsilon) \in D$ for all $\epsilon \in (0,\epsilon_0)$.

We consider two cases.  First, suppose that $\partial D$ has non-empty intersection with the part of $\striptop$ which is to the left of $z_0$.  Then there is nothing to prove since a simple topological argument implies that $\eta|_{[\tau,\infty)}$ can only exit $D$ either on the left side of $\eta'([0,\tau'])$ or on the part of $\striptop$ which lies to the left of $z_0$ (since $\eta([0,\tau])$ must have an intersection with either the left side of $\eta'([0,\tau'])$ or the part of $\striptop$ which is to the left of $z_0$; see Figure~\ref{fig::hittinginterval}).  Second, suppose that $\partial D$ has non-empty intersection with the part of $\striptop$ which is to the right of $z_0$ or the right side of $\eta'([0,\tau'])$.  Observe that $D$ is simply connected.  Let $\tau_0$ be the largest time $t \leq \tau$ such that $\eta(t) \in K_{\tau'}'$.  Let $\psi \colon D \to \strip$ be the conformal transformation which sends $\eta(\tau)$ to $0$, and the left and right boundaries of $\eta([\tau_0,\tau])$ to $(-\infty,0)$ and $(0,\infty)$, respectively.  Since $\eta|_{[0,\tau]}$ and $\eta'|_{[0,\tau']}$ are continuous paths, it follows that $\psi$ extends as a homeomorphism to $\ol{D}$.  Let $\wt{h} = h \circ \psi^{-1} - \chi \arg (\psi^{-1})'$.  Then $\wt{h}$ is a GFF on $\strip$ by Proposition~\ref{gff::prop::cond_union_local} since $K_{\tau'}'$ is local for $h$ and $\eta([0,\tau])$ is local for $h$ given $K_{\tau'}'$.  By Proposition~\ref{gff::prop::cond_union_mean}, we know the boundary data of $\wt{h}$ on $\stripbot$ as well as the parts of $\striptop$ whose preimage under $\psi$ lies in either $\eta([0,\tau])$ or in $\eta'([0,\tau'])$ but (at this point) we cannot determine the boundary behavior of $h$ near points in $\eta([0,\tau]) \cap \eta'([0,\tau'])$.  This is indicated in the right panel of Figure~\ref{fig::wrap_around}.  The result now follows from Lemma~\ref{lem::flow_cannot_hit} (see Figure~\ref{fig::wrap_around}).
\end{proof}

\begin{figure}[h!]
\begin{center}
\subfigure[In order for $\eta'(\tau')$ to be to the left of $\eta_\theta$, $\eta_\theta$ must wrap around $\eta'(\tau')$ after hitting on the left side of $\eta'(\text{[}0,\tau'\text{]})$.]{
\includegraphics[scale=0.85,page=1]{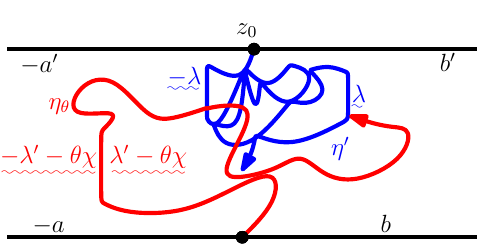}}
\hspace{0.0\textwidth}
\subfigure[\label{fig::wrap_around_right_side} The boundary data for the flow line $\eta_\theta$ as it approaches the right side of $\eta'$, conformally mapped back to the strip.]{
\includegraphics[scale=0.85,page=2]{figures/dual_wrap_around.pdf}}\caption{\label{fig::wrap_around}Suppose we have the same setup as Figure~\ref{fig::flow_counterflow_first_exit} with $\theta > \tfrac{1}{\chi}(\lambda-\lambda') = \tfrac{\pi}{2}$ so that $\eta_\theta$ first exits $\strip \setminus \eta'([0,\tau'])$ on the left side of $\eta'([0,\tau'])$.  The only way that $\eta'(\tau')$ can be to the left of $\eta_\theta$ is if, after hitting the left side of $\eta'([0,\tau'])$, $\eta_\theta$ wraps around $\eta'(\tau')$ and then hits the right side of $\eta'([0,\tau'])$.  Let $K_t'$ be the hull of $\eta'([0,t])$.  This implies that with $\tau_\delta$ the first time $t$ after $\tau$ for which $\eta(t)$ is in the right connected component of $\strip \setminus (\eta'([0,\tau']) \cup \eta([0,\tau]))$ and $\dist(\eta(t),K_{\tau'}') \geq \delta$ we have that $\p[\tau_\delta < \infty] > 0$ provided $\delta > 0$ is small enough.  Let $D$ be the connected component of $\strip \setminus (\eta([0,\tau_\delta] \cup K_{\tau'}')$ which, for some $\epsilon > 0$, contains $\eta([\tau_\delta,\tau_\delta+\epsilon])$.  Let $\psi \colon D \to \strip$ be a conformal transformation sending the left side of $\eta$ from the largest time $t < \tau_\delta$ that $\eta(t) \in \partial (\strip \setminus K_{\tau'}')$ to $\tau_\delta$ to $(-\infty,0)$ and the corresponding right side of $\eta$ to $(0,\infty)$.  Then the boundary data for the coordinate change of the GFF $h|_D + \theta \chi$ by $\psi$ in $\strip$ is shown on the right hand side.  The intervals $J_k$ are the images of the segments in $\partial D$ which also lie on the right side of $\eta'([0,\tau'])$.  By Figure~\ref{fig::cannot_hit_intervals}, we know in this case $\eta$ cannot hit any of the $J_k$ in finite time.  This leads to a contradiction.}
\end{center}
\end{figure}

\begin{proof}[Proof of Proposition~\ref{prop::flow_counterflow_left_right}]
We assume $\theta > \tfrac{1}{\chi}(\lambda-\lambda') = \tfrac{\pi}{2}$; the argument for the other case is the same.  Let $\tau'$ be any stopping time for $\eta'$ such that $\eta'$ has almost surely not yet hit $0$ by time $\tau'$.  Conditioning $h$ on $\eta'([0,\tau'])$ and conformally mapping $\strip \setminus K_{\tau'}'$ (recall that $K_t'$ is the hull of $\eta'([0,t])$) back to $\strip$, the boundary data of the corresponding field plus $\theta \chi$ is given in Figure~\ref{fig::flow_counterflow_first_exit_conf_image}.  That $\theta > \tfrac{1}{\chi}(\lambda-\lambda') = \tfrac{\pi}{2}$ implies $-\lambda' + \theta \chi > -\lambda$, $\lambda' + \theta \chi > \lambda$, and $b' + \theta \chi > \lambda$.  Consequently, it follows from that $\eta_\theta$ almost surely exits $\strip \setminus \eta'([0,\tau'])$ on the left side of $\eta'([0,\tau'])$ or on the part of $\striptop$ which lies to the left of $z_0$, say at time $\tau$, or does not hit $\striptop$ (Figure~\ref{fig::hittingrange} and Figure~\ref{fig::hittinginterval}).

We will now argue that $\eta'(\tau')$ is almost surely to the right of $\eta_\theta$.  If this is not the case, then after time $\tau$, $\eta$ must wrap around (but not hit) $\eta'(\tau')$ and then hit the right side of $\eta'([0,\tau'])$ (recall that $\eta$ almost surely does not hit the side of $\striptop$ which lies to the right of $z_0$).  Let $\tau_\delta$ be the first time $t$ after $\tau$ that $\eta(t)$ is in the right connected component of $\strip \setminus (\eta([0,\tau]) \cup \eta'([0,\tau']))$ and $\dist(\eta(\tau_\delta),K_{\tau'}') \geq \delta$ (we take $\tau_\delta = \infty$ if this never happens).  Then it must be that $\lim_{\delta \to 0^+} \p[\tau_\delta < \infty] > 0$, for otherwise $\eta'(\tau')$ is contained in the range of $\eta$ since we have assumed that $\eta'(\tau')$ is to the left of $\eta$.  This leads to a contradiction since Lemma~\ref{lem::flow_counterflow_not_hit_right} implies that $\eta$ cannot hit the right side of $\eta'([0,\tau'])$ or the part of $\striptop$ which lies to the right of $z_0$ before it hits the left side of $\eta'([0,\tau'])$, the part of $\striptop$ which lies to the left of $z_0$, or tends to $\infty$ after time $\tau_\delta$, any $\delta > 0$.
\end{proof}

\begin{proposition}
\label{prop::monotonicity_non_boundary}
Suppose that $h$ is a GFF on $\strip$ whose boundary data is as in Figure~\ref{fig::flow_counterflow_left_right}.  Fix $\theta_1,\theta_2$ such that
\begin{equation}
\label{eqn::theta_assumption}
 \frac{\lambda-b}{\chi} \leq \theta_1 < \theta_2 \leq \frac{a+\pi \chi-\lambda}{\chi}.
\end{equation}
Let $\eta_{\theta_i}$, for $i=1,2$, be the flow line of $h+\theta_i \chi$ starting at $0$ and let $\tau_i$ be the first time that $\eta_{\theta_i}$ accumulates in $\striptop$.  Then $\eta_{\theta_2}|_{[0,\tau_2]}$ almost surely lies to the left of $\eta_{\theta_1}|_{[0,\tau_1]}$.  The same result holds if $\eta_{\theta_1},\eta_{\theta_2}$ are flow lines of a GFF $h$ on $\h$ from $0$ to $\infty$ with boundary data $-a$ on $(-\infty,0)$ and $b$ on $(0,\infty)$.
\end{proposition}
The reason for the asymmetry in the hypothesis~\eqref{eqn::theta_assumption} is to allow for ``enough space'' so that we can fit a counterflow line $\eta_{\theta_1}'$ whose left boundary is $\eta_{\theta_1}$ which does not intersect $\stripbot$.  We note that it is not necessary to make any hypotheses about the boundary data of $h$ on $\striptop$.  The reason is that the lemma is only applicable for the paths up until when they first accumulate in $\striptop$.  This means that we can prove the result with a convenient choice and then use Proposition~\ref{prop::gff_abs_continuity}.
\begin{proof}[Proof of Proposition~\ref{prop::monotonicity_non_boundary}]
Fix $\epsilon > 0$.  For $i=1,2$, let $\tau_i^\epsilon$ be the first time $t$ that $\eta_{\theta_i}$ gets within distance $\epsilon$ of $\striptop$.  It suffices to show that $\eta_{\theta_2}^{x_2}|_{[0,\tau_2^\epsilon]}$ almost surely lies to the left of $\eta_{\theta_1}^{x_1}|_{[0,\tau_1^\epsilon]}$ for every $\epsilon > 0$.  We assume that $a' \geq \lambda' + (\theta_1 + \tfrac{\pi}{2}) \chi$ and $b' \geq \lambda' + (\tfrac{3}{2} \pi-\theta_1) \chi$ (by Proposition~\ref{prop::gff_abs_continuity}, it suffices to prove this result with any choice of $a',b'$).  This implies that the boundary data of $h+(\theta_1-\tfrac{\pi}{2})\chi$ on $\striptop$ which lies to the left of $z_0$ is at most $-\lambda'-\pi \chi$ and to the right of $z_0$ is at least $\lambda'+\pi\chi$.  Moreover, the hypothesis~\eqref{eqn::theta_assumption} implies that the boundary data of $h+(\theta_1-\tfrac{\pi}{2}) \chi$ on $\stripbot$ which lies to the left of $0$ is at most $-\lambda'$ and to the right of $0$ is at least $\lambda'$.  Consequently, Proposition~\ref{prop::flow_counterflow_left_right} is applicable to the counterflow line $\eta_{\theta_1}'$ of $h+(\theta_1-\tfrac{\pi}{2})\chi$.  Since $\eta_{\theta_2}|_{[0,\tau_2^\epsilon]}$ is the flow line of $h+(\theta_1-\tfrac{\pi}{2})\chi$ with angle $\theta_2-\theta_1+\tfrac{\pi}{2} > \tfrac{\pi}{2}$, Proposition~\ref{prop::flow_counterflow_left_right} implies that $\eta_{\theta_2}|_{[0,\tau_2^\epsilon]}$ is to the left of $\eta_{\theta_1}'$.  The result then follows since Proposition~\ref{prop::duality_many_force_points} implies that $\eta_{\theta_1}|_{[0,\tau_1^\epsilon]}$ is contained in the left boundary of $\eta_{\theta_1}'$.  The result when the $\eta_{\theta_i}$, $i=1,2$, are flow lines of a GFF on $\h$ follows from the result on $\strip$ and Proposition~\ref{prop::gff_abs_continuity}.
\end{proof}

\subsection{Light cone construction of counterflow lines}
\label{subsec::light_cone}

In this section, we will prove Proposition~\ref{prop::light_cone_construction}, our first version of Theorem~\ref{thm::lightconeroughstatement}.  Along the way, we will explain the inputs we need in order to prove the result in its full generality (the technical ingredients for the general version will be developed in Section~\ref{sec::uniqueness}).  Suppose that $h$ is a GFF on $\strip$ with boundary data as depicted in Figure~\ref{fig::light_cone_construction}.  Throughout, we will make the same hypotheses on the boundary data of $h$ as in Proposition~\ref{prop::flow_counterflow_left_right}.  That is, we shall assume that $a,b \geq \lambda - \tfrac{\pi}{2} \chi = \lambda'$; the reason for this choice is that it implies that the counterflow line $\eta'$ of $h$ starting at $z_0$ almost surely hits $\stripbot$, the bottom of $\partial \strip$, only when it exits at $0$.  We also assume that $a',b' \geq \lambda'+\pi \chi$ so that $\eta'$ almost surely does not intersect $\striptop$, the top of $\partial \strip$, except at $z_0$ (recall Figure~\ref{fig::hittingsinglepoint}).

Fix angles $\theta_1,\ldots,\theta_\ell$.  Let $\eta_{\theta_1}$ be the flow line of $h$ starting at $0$ with angle $\theta_1$, let $\tau_1$ be an $\eta_{\theta_1}$ stopping time, and let $\eta_{\theta_1}^{\tau_1} = \eta_{\theta_1}|_{[0,\tau_1]}$ (i.e., $\eta_{\theta_1}$ stopped at time $\tau_1$).  For each $2 \leq j \leq \ell$, we inductively let $\eta_{\theta_1 \cdots \theta_j}^{\tau_1 \cdots \tau_{j-1}}$ be the flow line of $h$ conditional on $\eta_{\theta_1 \cdots \theta_{j-1}}^{\tau_1 \cdots \tau_{j-1}}|_{[0,\tau_{j-1}]}$ starting at $\eta_{\theta_1 \cdots \theta_{j-1}}^{\tau_1 \cdots \tau_{j-1}}(\tau_{j-1})$ with angle $\theta_j$ and let $\tau_j$ be an $\eta_{\theta_1 \cdots \theta_j}$ stopping time, as depicted in Figure~\ref{fig::light_cone_construction}.  We call $\eta_{\theta_1 \cdots \theta_\ell}^{\tau_1 \cdots \tau_j} = \eta_{\theta_1 \cdots \theta_\ell}^{\tau_1 \cdots \tau_{j-1}}|_{[0,\tau_j]}$ an {\bf angle-varying flow line} with angles $\theta_1,\ldots,\theta_\ell$ with respect to the stopping times $\tau_1,\ldots,\tau_\ell$.  Note that
\[ \tau_1 \leq \tau_2 \leq \cdots \leq \tau_\ell \quad\text{and}\quad \eta_{\theta_1}^{\tau_1} \subseteq \eta_{\theta_1 \theta_2}^{\tau_1 \tau_2} \subseteq \cdots \subseteq \eta_{\theta_1 \cdots \theta_\ell}^{\tau_1 \cdots \tau_\ell}.\]
We emphasize that $\eta_{\theta_1 \cdots \theta_j}^{\tau_1 \cdots \tau_j}$ is defined on $[0,\tau_j]$ and $\eta_{\theta_1 \cdots \theta_j}^{\tau_1 \cdots \tau_j}|_{[0,\tau_{j-1}]} = \eta_{\theta_1 \cdots \theta_{j-1}}^{\tau_1 \cdots \tau_{j-1}}$.  The {\bf light cone} $\lightcone$ of $h$ starting at $0$ is the closure of the set of points accessible by angle-varying flow lines with rational angles $\theta$ restricted by
\begin{equation}
\label{eqn::angle_assumption}
 -\frac{\pi}{2} = \frac{1}{\chi}\left(\lambda' - \lambda \right) \leq
   \theta \leq
   \frac{1}{\chi}\left(\lambda-\lambda' \right) = \frac{\pi}{2},
\end{equation}
i.e.\ always pointing in a northerly direction, and with positive rational angle change times.  More generally, if $\eta_{\phi_1 \cdots \phi_k}^{\sigma_1 \cdots \sigma_k}$ is any angle-varying flow line and $\sigma$ is an $\eta_{\phi_1 \cdots \phi_k}^{\sigma_1 \cdots \sigma_k}$ stopping time, the light cone $\lightcone(\eta_{\phi_1 \cdots \phi_k},\sigma)$ starting at $\eta_{\phi_1 \cdots \phi_k}^{\sigma_1 \cdots \sigma_k}(\sigma)$ is the closure of the set of points accessible by angle-varying flow lines starting at $\eta_{\phi_1 \cdots \phi_k}^{\sigma_1 \cdots \sigma_k}(\sigma)$ with rational angles restricted by~\eqref{eqn::angle_assumption} and with positive rational angle change times.  The main result of this subsection (Proposition~\ref{prop::light_cone_construction}) states that the range of $\eta'$ stopped when it hits the tip $\eta_{\phi_1 \cdots \phi_k}^{\sigma_1 \cdots \sigma_k}(\sigma)$ of $\eta_{\phi_1 \cdots \phi_k}^{\sigma_1 \cdots \sigma_k}|_{[0,\sigma]}$ is almost surely equal to $\lightcone(\eta_{\phi_1 \cdots \phi_k}^{\sigma_1 \cdots \sigma_k},\sigma)$.  The first step in its proof is Lemma~\ref{lem::light_cone_contains_av}, which states that any angle-varying flow line $\eta_{\theta_1 \cdots \theta_\ell}^{\tau_1 \cdots \tau_\ell}$ whose angles are restricted by~\eqref{eqn::angle_assumption} is almost surely contained in the range of $\eta'$ and that $\eta'$ hits the points of $\eta_{\theta_1 \cdots \theta_\ell}^{\tau_1 \cdots \tau_\ell}$ in reverse chronological order.  Before we prove Lemma~\ref{lem::light_cone_contains_av}, we record the following technical result which gives that non-boundary intersecting angle-varying flow lines with relative angles which are not larger than $\pi$ in magnitude are almost surely simple and determined by $h$.

\begin{figure}[h!]
\begin{center}
\includegraphics[scale=0.85]{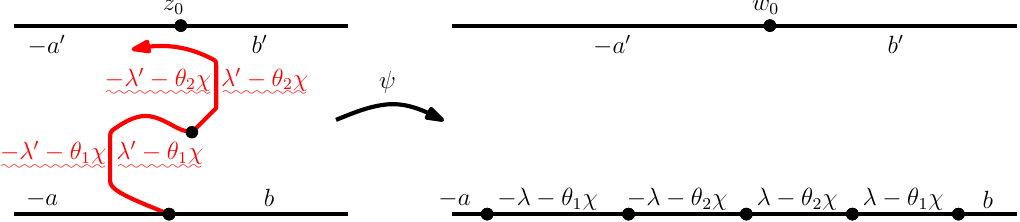}
\caption{\label{fig::light_cone_construction} Suppose that $h$ is a GFF on $\strip$ with the boundary data depicted above.  We assume that $a,b \geq \lambda - \tfrac{\pi}{2} \chi = \lambda'$ and $a',b' \geq \lambda'+\pi \chi$.  The reason for these choices is that the former implies that the counterflow line $\eta'$ of $h$ starting at $z_0$ almost surely first exits $\stripbot$ at $0$ and the latter implies that $\eta'$ intersects $\striptop$ only at $z_0$.  Fix angles $\theta_1,\ldots,\theta_\ell$.  Let $\eta_{\theta_1}$ be the flow line of $h+\theta_1 \chi$, let $\tau_1$ be an $\eta_{\theta_1}$ stopping time, and let $\eta_{\theta_1}^{\tau_1} = \eta_{\theta_1}|_{[0,\tau_1]}$.  For $j \geq 2$, inductively let $\eta_{\theta_1 \cdots \theta_j}^{\tau_1 \cdots \tau_{j-1}}$ be the flow line of $h+\theta_j \chi$ conditional on $\eta_{\theta_1 \cdots \theta_{j-1}}^{\tau_1 \cdots \tau_{j-1}}|_{[0,\tau_{j-1}]}$ starting at $\eta_{\theta_1 \cdots \theta_{j-1}}^{\tau_1 \cdots \tau_{j-1}}(\tau_{j-1})$, let $\tau_j$ be any $\eta_{\theta_1 \cdots \theta_j}^{\tau_1 \cdots \tau_{j-1}}$ stopping time, and let $\eta_{\theta_1 \cdots \theta_j}^{\tau_1 \cdots \tau_j} = \eta_{\theta_1 \cdots \theta_j}^{\tau_1 \cdots \tau_{j-1}}|_{[0,\tau_{j-1}]}$.  The random subset of $\strip$ one obtains by taking the closure of the union of the ranges of $\eta_{\theta_1 \cdots \theta_\ell}$ where the stopping times $\tau_i$ and angles $\theta_i$ with $-\tfrac{\pi}{2} = \tfrac{1}{\chi}(\lambda'-\lambda) \leq \theta_j \leq \tfrac{1}{\chi}(\lambda-\lambda') = \tfrac{\pi}{2}$ for all $1 \leq j \leq k$ vary among appropriate countable dense collections is equal in distribution to the range of the counterflow line starting at $z_0$.  The picture on the right side shows the boundary data of the GFF ($\ell=2$) after applying a conformal transformation $\psi \colon \strip \setminus \eta_{\theta_1 \cdots \theta_k}^{\tau_1 \cdots \tau_k}([0,\tau_k]) \to \strip$ which sends $\eta_{\theta_1 \cdots \theta_k}^{\tau_1 \cdots \tau_k}(\tau_k)$ to $0$ and fixes $\pm \infty$; $w_0 = \psi(z_0) \in \striptop$.
}
\end{center}
\end{figure}

\begin{lemma}
\label{lem::av_simple_determined}
Let $\eta_{\theta_1 \cdots \theta_\ell}^{\tau_1 \cdots \tau_\ell}$ be an angle-varying flow line of $h$ with angles $\theta_i$, $1 \leq i \leq \ell$, with $|\theta_i - \theta_j| \leq \pi$ for all pairs $1 \leq i,j \leq \ell$.  Assume, moreover, that $\eta_{\theta_1 \cdots \theta_\ell}^{\tau_1 \cdots \tau_\ell}$ is non-boundary-intersecting.  Then $\eta_{\theta_1 \cdots \theta_\ell}^{\tau_1 \cdots \tau_\ell}$ is almost surely simple and continuous.  If we assume further that the boundary data for $h+\theta_1 \chi$ is at least $\lambda$ on $(0,\infty)$ and at most $-\lambda+\pi \chi$ on $(-\infty,0)$, then $\eta_{\theta_1 \cdots \theta_\ell}^{\tau_1 \cdots \tau_\ell}$ is almost surely determined by $h$.  The same likewise holds if the boundary data for $h+\theta_1 \chi$ is at least $\lambda-\pi \chi$ on $(0,\infty)$ and at most $-\lambda$ on $(-\infty,0)$.
\end{lemma}

The exact conditions on the $\theta_j$ for $\eta_{\theta_1 \cdots \theta_\ell}^{\tau_1 \cdots \tau_\ell}$ to be non-boundary-intersecting are as follows.  In order for the path not to $\stripbot$, we need both
\begin{equation}
\label{eqn::not_hit_bottom}
\frac{a - \theta_j \chi}{\lambda} - 1 \geq \frac{\kappa}{2}-2 \quad\text{and}\quad \frac{b+ \theta_j \chi}{\lambda} -1 \geq  \frac{\kappa}{2}-2.
\end{equation}
In order for the path not to hit $\striptop$ other than at $z_0$, we need both
\begin{equation}
\label{eqn::not_hit_top}
\frac{a' - (\theta_j+\pi) \chi}{\lambda} - 1 \geq \frac{\kappa}{2}-2 \quad\text{and}\quad \frac{b+ (\theta_j-\pi) \chi}{\lambda} -1 \geq  \frac{\kappa}{2}-2.
\end{equation}
Indeed, these two conditions together with the condition that $|\theta_i-\theta_j| \leq \pi$ for all pairs $1 \leq i,j \leq \ell$ imply that all of the partial sums of the force points which are to the left and right of the driving function exceed $\tfrac{\kappa}{2}-2$.  Note that for any fixed choice of $\theta_j$ for $1 \leq j \leq \ell$, we can always pick $a,b,a',b'$ large enough so that~\eqref{eqn::not_hit_bottom} and~\eqref{eqn::not_hit_top} hold.

\begin{proof}[Proof of Lemma~\ref{lem::av_simple_determined}]
The proof is by induction on $\ell \geq 1$.  Suppose $\ell=1$.  That the path is simple and continuous in this case follows from the mutual absolute continuity of the path to usual $\SLE_\kappa$, $\kappa \in (0,4)$ (see Section~\ref{sec::sle} and Remark~\ref{rem::continuity_non_boundary}).  Suppose the result holds for angle-varying paths with $\ell-1 \geq 0$ angle changes.  To see it holds when there are $\ell$ angle changes, we condition on the realization of the path until the $(\ell-1)$st angle change and conformally map back to $\strip$ as in Figure~\ref{fig::light_cone_construction}.  Due to the restriction $|\theta_i - \theta_j| \leq \pi$, it follows that the image of the path after the $(\ell-1)$st angle change is a non-boundary intersecting $\SLE_\kappa(\ul{\rho})$, so the induction step clearly follows.  The final claim of the lemma for $\ell=1$ follows from the special case of Theorem~\ref{thm::coupling_uniqueness} we proved in Section~\ref{subsec::uniqueness_non_boundary_intersecting}.  That it holds for larger $\ell$ also follows by induction and Proposition~\ref{prop::gff_abs_continuity}.
\end{proof}

\begin{lemma}
\label{lem::light_cone_contains_av}
Let $\eta_{\theta_1 \cdots \theta_\ell}^{\tau_1 \cdots \tau_\ell}$ be an angle-varying flow line of $h$ with angles $\theta_i$, $1 \leq i \leq \ell$, satisfying~\eqref{eqn::angle_assumption}.  The counterflow line $\eta'$ of $h$ starting at $z_0$ almost surely contains $\eta_{\theta_1\cdots \theta_\ell}^{\tau_1 \cdots \tau_\ell}$ and, moreover, hits the points of $\eta_{\theta_1 \cdots \theta_\ell}^{\tau_1 \cdots \tau_\ell}$ in reverse chronological order.
\end{lemma}
\begin{proof}
We first note that $\eta_{\theta_1 \cdots \theta_\ell}^{\tau_1 \cdots \tau_\ell}$ is almost surely continuous by Lemma~\ref{lem::av_simple_determined}.  Fix an $\eta_{\theta_1 \cdots \theta_\ell}^{\tau_1 \cdots \tau_\ell}$ stopping time $\sigma$ such that $\eta_{\theta_1 \cdots \theta_\ell}^{\tau_1 \cdots \tau_\ell}([0,\sigma])$ almost surely does not contain $z_0$.  We apply a conformal transformation $\psi \colon \strip \setminus \eta_{\theta_1 \cdots \theta_\ell}^{\tau_1 \cdots \tau_\ell}([0,\sigma]) \to \strip$ which fixes $\pm \infty$ and sends $\eta_{\theta_1 \cdots \theta_\ell}^{\tau_1 \cdots \tau_\ell}(\sigma)$ to $0$; the boundary data for the GFF which describes the evolution of $\psi(\eta')$ is as in the right panel of Figure~\ref{fig::light_cone_construction} (with the obvious generalization from $\ell=2$ to other values of $\ell$).  Let $w_0 = \psi(z_0)$.  Our hypotheses on $\theta_i$ imply
\[ -\lambda - \theta_i \chi \leq -\lambda' \quad\text{and}\quad \lambda -  \theta_i \chi \geq \lambda' \quad\text{for all}\quad 1 \leq i \leq \ell.\]
Lemma~\ref{lem::hitting_single_point} (see also Figure~\ref{fig::hittingsinglepoint}) thus implies that $\psi(\eta')$ does not hit $\stripbot \setminus \{0\}$ and exits $\strip$ at $0$.  This implies that $\eta'$ almost surely exits $\strip \setminus \eta_{\theta_1 \cdots \theta_\ell}^{\tau_1 \cdots \tau_\ell}([0,\sigma])$ at $\eta_{\theta_1 \cdots \theta_\ell}^{\tau_1 \cdots \tau_\ell}(\sigma)$.  By choosing a dense collection of stopping times (e.g., the positive rationals) and using the continuity of $\eta'$ and $\eta_{\theta_1 \cdots \theta_\ell}^{\tau_1 \cdots \tau_\ell}$ we conclude that the range of $\eta'$ contains $\eta_{\theta_1 \cdots \theta_\ell}^{\tau_1 \cdots \tau_\ell}$.  Moreover, the proof clearly also implies that the points of $\eta_{\theta_1 \cdots \theta_\ell}^{\tau_1 \cdots \tau_\ell}$ are hit by $\eta'$ in reverse chronological order.
\end{proof}

\begin{remark}
\label{rem::light_cone_contains_av_general}
The proof of Lemma~\ref{lem::light_cone_contains_av} requires two inputs:
\begin{enumerate}[(i)]
\item $\eta'$ almost surely exits $\strip$ at $0$ and
\item $\eta_{\theta_1 \cdots \theta_\ell}^{\tau_1 \cdots \tau_\ell}$ and $\eta'$ are almost surely continuous paths.
\end{enumerate}
The reason that we chose our boundary data so that $\eta'$ does not intersect $\stripbot \setminus \{0\}$ was to ensure the continuity of $\eta'$.  Upon proving Theorem~\ref{thm::continuity} in Section~\ref{subsec::many_boundary_force_points} and Section~\ref{subsec::counterflow}, we will have shown that (i) and (ii) hold whenever $\eta'$ and $\eta_{\theta_1 \cdots \theta_\ell}^{\tau_1 \cdots \tau_\ell}$ make sense (the SDE for the driving functions of these processes has a solution), so that Lemma~\ref{lem::light_cone_contains_av} also holds in the same generality.
\end{remark}

\begin{figure}[h!]
\begin{center}
\includegraphics[scale=0.85]{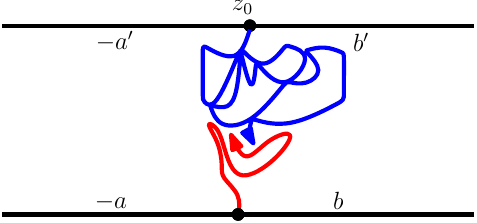}
\caption{\label{fig::light_cone_back_and_forth} Consider the GFF with the boundary data depicted in the illustration above with $a,b \geq \lambda - \tfrac{\pi}{2} \chi = \lambda'$ and $a',b' \geq \lambda' + \pi \chi$ and let $\eta'$ be the counterflow line starting at $z_0$.  Fix any $\eta'$ stopping time $\tau'$.  To show that the random set described in Figure~\ref{fig::light_cone_construction} almost surely contains $\eta'(\tau')$, we consider flow lines of the form $\eta_{\theta_1 \cdots \theta_j}^{\tau_1 \cdots \tau_j}$ where $\theta_j = (-1)^{j+1} \tfrac{1}{\chi}(\lambda-\lambda') = (-1)^{j+1} \tfrac{\pi}{2}$.  With this choice, $\eta_{\theta_1 \cdots \theta_j}^{\tau_1 \cdots \tau_j}$ can hit the left but not the right side of $\eta'$ when $j$ is odd and vice-versa when $j$ is even.  We choose the stopping times $\tau_j$ so that $\eta_{\theta_1 \cdots \theta_j}^{\tau_1 \cdots \tau_j}$ gets progressively closer to the left side if $j$ is odd and to the right side if $j$ is even.  Taking a limit as the number of angles tends to $\infty$, the corresponding curves almost surely accumulate at $\eta'(\tau')$.
}
\end{center}
\end{figure}

\begin{figure}[h!]
\begin{center}
\subfigure{
\includegraphics[scale=0.85,page=1]{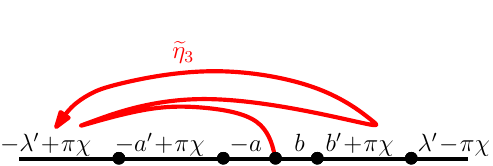}}
\subfigure{
\includegraphics[scale=0.85,page=2]{figures/light_cone_back_and_forth_proof2.pdf}}
\caption{\label{fig::light_cone_back_and_forth_proof} In order to prove that the construction in Figure~\ref{fig::light_cone_back_and_forth} accumulates at $\eta'(\tau')$, we apply a conformal map $\psi$ which takes the unbounded connected component of  $\strip \setminus \eta'([0,\tau'])$ to $\h$ and which sends $\eta'(\tau')$ to $\infty$ and fixes $0$.  Then it suffices to show that $\wt{\eta}_j := \psi(\eta_{\theta_1 \cdots \theta_j}^{\tau_1 \cdots \tau_j})$ is almost surely unbounded as $j \to \infty$.  To prove this, it suffices so show that the amount of capacity time it takes $\wt{\eta}_j$ to traverse from left to right (resp.\ right to left) if $j$ is even (resp.\ odd) is stochastically bounded from below by a non-negative random variable whose law has positive mean.}
\end{center}
\end{figure}

Lemma~\ref{lem::light_cone_contains_av} implies that $\eta'$ almost surely contains $\lightcone(\eta_{\phi_1 \cdots \phi_k}^{\sigma_1 \cdots \sigma_k},\sigma)$ for any $\phi_1,\ldots,\phi_k$ satisfying~\eqref{eqn::angle_assumption}.  We now turn to prove the reverse inclusion.

\begin{proposition}
\label{prop::light_cone_construction}
Let $\eta_{\phi_1 \cdots \phi_k}^{\sigma_1 \cdots \sigma_k}$ be an angle-varying flow line of $h$ with angles $\phi_i$ satisfying~\eqref{eqn::angle_assumption} and let $\sigma$ be any $\eta_{\phi_1 \cdots \phi_k}^{\sigma_1 \cdots \sigma_k}$ stopping time.  The random set $\lightcone(\eta_{\phi_1 \cdots \phi_k}^{\sigma_1 \cdots \sigma_k},\sigma)$ obtained by taking the closure of the union of $\eta_{\phi_1 \cdots \phi_k \theta_1 \cdots \theta_\ell}^{\sigma_1 \cdots \sigma_k \tau_1 \cdots \tau_\ell}([0,\tau_\ell])$ as $\ell$ ranges over $\N$ and $\theta_1,\ldots,\theta_\ell$ range over any countable dense set of angles satisfying~\eqref{eqn::angle_assumption}, and $\sigma < \tau_1 < \cdots < \tau_\ell$ range over any dense subset of $(\sigma,\infty)$ is almost surely equal to the range of the counterflow line $\eta'$ of $h$ starting at $z_0$ stopped upon first hitting $\eta_{\phi_1 \cdots \phi_k}^{\sigma_1 \cdots \sigma_k}(\sigma)$ (which is the same time that it hits $\eta_{\phi_1 \cdots \phi_k}^{\sigma_1 \cdots \sigma_k}([0,\sigma])$).
\end{proposition}

We pause to give an overview of the proof of Proposition~\ref{prop::light_cone_construction}.  Fix an $\eta'$ stopping time $\tau'$.  We will prove for every $\epsilon > 0$ there exists an angle-varying flow line with angles restricted by~\eqref{eqn::angle_assumption} whose range comes within distance $\epsilon$ of $\eta'(\tau')$ (see Figure~\ref{fig::light_cone_back_and_forth} for an illustration).  By conformally mapping the unbounded connected component of $\strip \setminus \eta'([0,\tau'])$ to $\h$ with $\eta'(\tau')$ mapped to $\infty$ and $0$ fixed, it suffices to show for every $R > 0$ there exists an angle-varying flow line with angles restricted by~\eqref{eqn::angle_assumption} of the corresponding field on $\h$ whose diameter is at least $R$.  By rescaling, we may assume that the images of the left and right sides of $z_0$ are contained in $\R \setminus (-3,3)$.  Consider the path $\wt{\eta}_1$ which begins by flowing at angle $\tfrac{\pi}{2}$, i.e.\ maximally to the left.  By the choice of boundary data, $\wt{\eta}_1$ first hits $\R$ after time $0$ in $(-\infty,-3]$.  We let $\wt{\eta}_2$ be the path which starts at the tip of $\wt{\eta}_1$ when it gets close to $(-\infty,-3]$ and then flows at angle $-\tfrac{\pi}{2}$, i.e.\ maximally to the right.  Again by the choice of boundary data, $\wt{\eta}_2$ first hits $\R$ in $[3,\infty)$.  We inductively let $\wt{\eta}_j$ be the path which flows at angle $(-1)^{j+1} \tfrac{\pi}{2}$ starting at the tip of $\wt{\eta}_{j-1}$ when it gets close to $(-1)^j[3,\infty)$.  This is depicted in Figure~\ref{fig::light_cone_back_and_forth_proof}.  To show that the paths $\wt{\eta}_j$ are unbounded as $j \to \infty$, it suffices to show that the amount of capacity time it takes for $\wt{\eta}_j$ to traverse from left to right (resp.\ right to left) if $j$ is even (resp.\ odd) is stochastically bounded from below by a non-negative random variable with positive mean (Lemma~\ref{lem::sle_kappa_rho_crossing} serves to make this point precise).  The challenge in showing this is that each change of angle leads to the creation of two additional force points.  The amount of force applied to the driving function, nevertheless, remains bounded because the force alternates in sign but has the same magnitude with each angle change.

\begin{lemma}
\label{lem::sle_kappa_rho_crossing}
Suppose that $(W_t,V_t^{i,q})$ is an $\SLE_\kappa(\ul{\rho}^L;\ul{\rho}^R)$ process with $W_0 = 0$ and force points located $(\ul{x}^L;\ul{x}^R)$.  Let $\tau$ be the first time that $W$ exits the interval $[-1,1]$.  Assume there exists $C > 0$ such that
\[ \left|\sum_{i=1}^k \frac{\rho^{i,L}}{W_t - V_t^{i,L}}\right| + \left|\sum_{i=2}^\ell \frac{\rho^{i,R}}{W_t - V_t^{i,R}}\right| \leq C \quad\text{for all}\quad t \in [0,\tau]\]
where $k = |\ul{\rho}^L|$ and $\ell = |\ul{\rho}^R|$.  Then $\p[\tau \geq 1] \geq \rho_0 > 0$ where $\rho_0$ depends only on $C$, $\kappa$, and $\rho^{1,R}$.
\end{lemma}
\begin{proof}
Let $\wh{\p}_x$ be the law under which $(W_t,V_t^{1,R})$ evolves as an $\SLE_\kappa(\rho^{1,R})$ process with a single force point of weight $\rho^{1,R}$ with initial position $x = V_0^{1,R}$.  Let $\zeta = \tau \wedge 1$, $A = \{ \tau \geq 1\}$, and
\begin{align*}
U_t &= \sum_{i=1}^k \frac{\rho^{i,L}}{W_t -  V_t^{i,L}} + \sum_{i=2}^\ell \frac{\rho^{i,R}}{W_t -  V_t^{i,R}} \quad\text{and}\quad M_t = \frac{1}{\sqrt{\kappa}}  \int_{0}^t U_s dB_s
\end{align*}
where $B$ is the standard Brownian motion driving $W$.  We denote by $\langle M \rangle_t = \frac{1}{\kappa} \int_0^t U_s^2 ds$ the quadratic variation of $M$ at time $t$.  By the Girsanov theorem \cite{RY04, KS98} and the Cauchy-Schwarz inequality, we have that
\begin{align}
      \wh{\p}_x[A]
=& \E\left[ \one_A \exp\left( M_{\zeta} - \tfrac{1}{2} \langle M \rangle_{\zeta} \right)  \right]
\leq \left( \p[A] \E[ \exp(2M_{\zeta} - \langle M \rangle_{\zeta} ) ] \right)^{1/2}. \label{eqn::cs_bound}
\end{align}
The optional stopping theorem implies that $\E[ \exp(a M_{\zeta} - \tfrac{1}{2}a^2 \langle M \rangle_\zeta)] = 1$ for all $a \in \R$.  Since $\langle M \rangle_\zeta \leq C^2/\kappa$, it consequently follows that
\[  \E[ \exp(a M_\zeta)] \leq \exp(\tfrac{1}{2 \kappa} C^2 a^2) \quad\text{for}\quad a \in \R.\]
Hence, $\E[ \exp(2M_\zeta - \langle M \rangle_\zeta)]  \leq C_0$ for some constant~$C_0$ depending only on~$C$ and~$\kappa$.  Thus rearranging~\eqref{eqn::cs_bound}, we obtain $\p[A] \geq C_0^{-1} (\wh{\p}_x[A])^2$.  To finish the proof, we just have to argue that $\wh{\p}_x[A]$ is uniformly positive in $x$.  This is clear when $x \geq 2$ since then the total amount of force $V_t^{1,R}$ applies to $W$ in the time interval $[0,\zeta]$ is bounded by $|\rho^{1,R}|$.  The result then follows since $\wh{\p}_x[A]$ is both continuous and positive for $x \in [0,2]$.
\end{proof}

\begin{proof}[Proof of Proposition~\ref{prop::light_cone_construction}]
We have already shown in Lemma~\ref{lem::light_cone_contains_av} that $\eta'$ stopped at the first time $\sigma'$ it hits $\eta_{\phi_1 \cdots \phi_k}^{\sigma_1 \cdots \sigma_k}(\sigma)$ almost surely contains $\lightcone(\eta_{\phi_1 \cdots \phi_k}^{\sigma_1 \cdots \sigma_k},\sigma)$, so we need to show that $\lightcone(\eta_{\phi_1 \cdots \phi_k}^{\sigma_1 \cdots \sigma_k},\sigma)$ almost surely contains $\eta'([0,\sigma'])$.  To this end, fix any $\eta'$ stopping time $\tau'$ such that $\tau' < \sigma'$ almost surely; we will show that $\p[\eta'(\tau') \in \lightcone(\eta_{\phi_1 \cdots \phi_k}^{\sigma_1 \cdots \sigma_k},\sigma)] = 1$.  This completes the proof since, by the continuity of $\eta'$, it suffices to have this result for a countable collection of stopping times $(\tau_j')$ such that $\{\eta'(\tau_j')\}$ is dense in $\eta'([0,\sigma'])$.

Let $K_t'$ denote the hull of $\eta'$ at time $t$, i.e.\ the complement of the unbounded connected component of $\strip \setminus \eta'([0,t])$.  We begin by applying a conformal map $\psi \colon \strip \setminus (\eta_{\phi_1 \cdots \phi_k}^{\sigma_1 \cdots \sigma_k}([0,\sigma]) \cup K_{\tau'}') \to \h$ which takes $\eta_{\phi_1 \cdots \phi_k}^{\sigma_1 \cdots \sigma_k}(\sigma)$ to $0$ and $\eta'(\tau')$ to $\infty$.  The boundary data for the corresponding GFF $\wt{h} := h \circ \psi^{-1} - \chi \arg (\psi^{-1})'$ on $\h$ is shown in Figure~\ref{fig::light_cone_back_and_forth_proof} in the special case $k=0$.  Let $z_-$ and $z_+$ be the images of the left and right sides of $K_{\tau'}' \cap \striptop$, respectively.  We may assume without loss of generality that $\tau' > 0$ almost surely, which in turn implies that $z_- < 0 < z_+$.  We have only specified $\psi$ up to rescaling since we have only fixed the image of the two boundary points $\eta_{\phi_1 \cdots \phi_k}^{\sigma_1 \cdots \sigma_k}(\sigma)$ and $\eta'(\tau')$.  Thus by choosing the scaling factor to be large enough, we may assume without loss of generality that $z_- \leq -3$ and $z_+ \geq 3$.

It suffices to exhibit angles $\theta_j$ satisfying~\eqref{eqn::angle_assumption} and stopping times $(\wt{\tau}_j)$ such that the flow line $\wt{\eta}$ of the GFF $\wt{h}$ with angles $\theta_1,\ldots,\theta_\ell$ and angle-change times $\tau_1,\ldots,\tau_{\ell-1}$ is unbounded as $\ell \to \infty$.  To this end, we let
\[ \theta_j = (-1)^{j+1} \frac{1}{\chi}\left(\lambda-\lambda'\right) = (-1)^{j+1} \frac{\pi}{2}.\]
Note that this particular choice may not lie in our fixed dense set of angles satisfying~\eqref{eqn::angle_assumption}.  It will be clear from the proof, however, that we can achieve the same effect by approximating the $\theta_j$ by elements of this set.  Indeed, to see this, we just have to explain why an angle-varying flow line with angles contained in $\{-\tfrac{\pi}{2},\tfrac{\pi}{2}\}$ can be approximated arbitrarily well by an angle-varying flow line with angles in a dense subset of $[-\tfrac{\pi}{2},\tfrac{\pi}{2}]$.  It suffices to explain why if $(\wt{\theta}_n)$ is a sequence of angles in $[-\tfrac{\pi}{2},\tfrac{\pi}{2}]$ which decrease to $-\tfrac{\pi}{2}$ as $n \to \infty$ then the sequence of flow lines with angles $\wt{\theta}_n$ converge to the flow line of angle $-\tfrac{\pi}{2}$.  Note that the SDE which describes the driving function for a flow line of angle $\wt{\theta}_n$ is that of an $\SLE_\kappa(\ul{\rho}_n)$ process which, for large $n$, has weights which are close to those which correspond to the driving function of the flow line with angle $-\tfrac{\pi}{2}$.  It is thus a simple matter using the Girsanov theorem to see that the laws of the driving functions of the flow lines with $\wt{\theta}_n$ converge to that of the driving function of the $-\tfrac{\pi}{2}$ angle flow line as $n \to \infty$.  Moreover, by monotonicity, the flow lines themselves each drawn up to a fixed time $t$ convergence almost surely (say, with respect to the Hausdorff topology after conformally mapping to a bounded domain) to a limiting hull.  This gives a coupling of two growing families of hulls whose Loewner driving functions have the same law but with one almost surely to the left of the other.  This can only be the case if the two processes are equal.  Our stopping times will also not necessarily lie in our fixed dense subset of $(0,\infty)$, however, it will also be clear from the proof that the path we construct will still be contained in $\lightcone(\eta_{\phi_1 \cdots \phi_k}^{\sigma_1 \cdots \sigma_k},\sigma)$ by the continuity of our angle varying trajectories.

We now turn to the construction of $\wt{\eta}$.  Let $\wt{\eta}_1$ be the flow line of $\wt{h} + \theta_1 \chi$ which starts at $0$ and let $\wt{W}^1$ be its Loewner driving function.  By Lemma~\ref{lem::hitting_interval}, it follows that $\wt{\eta}_1$ must hit $\R$ to the left of $z_- \leq -3$ in finite time.  Let $\wt{\tau}_1$ be the first time $t$ that $\wt{W}_t^1 \leq -2$.  Inductively let $\wt{\eta}_j$ be the curve which traces along $\wt{\eta}_{j-1}$ and then flows at angle $\theta_j$ starting at $\wt{\eta}_{j-1}(\wt{\tau}_{j-1})$ and $\wt{W}^j$ its Loewner driving function.  Let $\wt{\tau}_j$ be the first time $t$ after $\wt{\tau}_{j-1}$ that $\wt{W}_t^j$ enters $(-\infty,-2]$ (resp.\ $[2,\infty)$) if $j$ is odd (resp.\ $j$ is even).  Then $\wt{W}^j|_{t \geq \wt{\tau}_{j-1}}$ is a solution to the $\SLE_\kappa(\ul{\rho})$ SDE with $N+1$, $N := k+j+2$, force points on each side of $0$ (this can be seen by reading off the boundary data for the GFF after mapping back and applying the change of coordinates formula).  The first $k+j$ arise from the $k+j$ angle changes and the last $3$ come from the initial choice of boundary data.  Let $\wt{V}_t^{i,q,j}$, $t \geq \wt{\tau}_{j-1}$ for $q \in \{L,R\}$ and $1 \leq i \leq N+1$, denote the time evolution of the force points associated with $\wt{W}^j|_{t \geq \wt{\tau}_{j-1}}$.

To complete the proof, it suffices to show there exists i.i.d.\ non-negative random variables $(Z_j)$ with $\E[Z_1] > 0$ such that, almost surely, $\wt{\tau}_{2j} - \wt{\tau}_{2j-1} \geq Z_j$ for all $j$.  Indeed, the strong law of large numbers then implies
\[ \wt{\tau}_{2j} \geq \sum_{i=1}^j Z_i \to \infty \quad\text{almost surely as}\quad j \to \infty,\]
hence the result follows from the diameter-capacity lower bound
\[ \diam(\wt{\eta}_j([0,\wt{\tau}_j]) \geq c\sqrt{\hcap(\wt{\eta}_j([0,\wt{\tau}_j]))} = c\sqrt{\wt{\tau}_j}\]
(the inequality is \cite[Equation 3.8]{LAW05} and the equality follows since our curve is parameterized by capacity).

\begin{figure}[h]
\begin{center}
\includegraphics[scale=0.85]{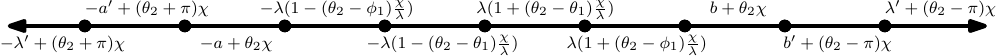}
\end{center}
\caption{\label{fig::light_cone_bd} The boundary data for $\wt{g}_{\tau_{j-1}}(\wt{\eta}_j)$ in the special case $\ell=2$ and $k=1$.}
\end{figure}

Note that $\wt{V}_{\wt{\tau}_{j-1}}^{1,L,j} = [(-1)^{j+1} 2]^-$, $\wt{V}_{\wt{\tau}_{j-1}}^{1,R,j} = [(-1)^{j+1}2]^+$ (the left and right sides of $(-1)^{j+1} 2$, respectively), and $\wt{V}_{\wt{\tau}_{j-1}}^{i,q,j} = \wt{V}_{\wt{\tau}_{j-1}}^{i-1,q,j-1}$ for $q \in \{L,R\}$.  Assume that $j$ is even.  By the monotonicity of the evolution of the $\wt{V}_t^{i,q,j}$ ($\wt{V}_t^{i,L,j}$ is decreasing in $t$ and $\wt{V}_t^{i,R,j}$ is increasing in $t$), it follows that $\wt{V}_t^{i,L,j} \leq -2$ for all $i$ and $t \geq \wt{\tau}_{j-1}$ and $\wt{V}_t^{i,R,j} \geq 2$ for all $i \geq 2$ and $t \geq \wt{\tau}_{j-1}$.  In particular, $\wt{V}_t^{1,R,j}$ is the only force point which can bounce off of $\wt{W}_t^j$ for $t \geq \wt{\tau}_{j-1}$ when $\wt{W}_t^j$ is in $[-1,1]$ and all other force points have distance at least $1$ to $\wt{W}_t^j$.  Let $\xi_j$ be the first time $t$ after $\wt{\tau}_{j-1}$ that $\wt{W}_t^j = 0$ and
\[ \zeta_j = \inf\{ t \geq \xi_j : \wt{W}_t^j = 1\} \wedge (\xi_j+1).\]
We are now going to check that $\wt{W}_t^j$, $t \geq \xi_j$, satisfies the hypotheses of Lemma~\ref{lem::sle_kappa_rho_crossing}.  Let $\ul{\rho} = (\ul{\rho}^L;\ul{\rho}^R)$ denote the weights of the force points of $\wt{W}_t^j$, $t \geq \wt{\tau}_{j-1}$.  For $1 \leq i \leq j-1$, note that
\[ \rho^{i,R} = (\theta_{j-i+1} - \theta_{j-i}) \frac{\chi}{\lambda} \quad\text{and}\quad \rho^{i,L} = -\rho^{i,R}\]
(see Figure~\ref{fig::light_cone_bd} for an example of this when $\ell=2$ and $k=1$).  In particular, by our choice of $\theta_i$, the weights $\rho^{1,q},\ldots,\rho^{j-1,q}$, $q \in \{L,R\}$, only alternate in sign but have the same magnitude.  Let $C_0 = |\rho^{1,L}|$.  We have $|\wt{W}_t^j - \wt{V}_t^{i,q,j}|^{-1} \leq 1$ for $t \geq \wt{\tau}_{j-1}$ when $\wt{W}_t^j \in [-1,1]$ for all $1 \leq i \leq j-1$ for $q = L$ and all $2 \leq i \leq j$ for $q = R$.  Thus since $|\wt{W}_t^j - \wt{V}_t^{i,q,j}|^{-1}$ is decreasing in $i$ for $t,j$ fixed, when $\wt{W}_t^j \in [-1,1]$ we have
\begin{equation}
\label{eqn::angle_varying_rho_bounds}
 \left|\sum_{i=1}^{j-1} \frac{\rho^{i,L}}{\wt{W}_t^j - \wt{V}_t^{i,L,j}}\right| \leq C_0 \quad\text{and}\quad
  \left|\sum_{i=2}^{j-1} \frac{\rho^{i,R}}{\wt{W}_t^j - \wt{V}_t^{i,R,j}} \right| \leq C_0 \quad\text{for}\quad t \geq \wt{\tau}_{j-1}.
\end{equation}
Moreover, there exists $C_1 > 0$ depending only on $k,\kappa,a,b,a',b'$, in particular not growing with $j$, such that
\begin{equation}
\label{eqn::angle_varying_rho_bounds2}
 \left| \sum_{i=j}^{N+1} \frac{\rho^{i,L}}{\wt{W}_t^j - \wt{V}_t^{i,L,j}}\right| + \left| \sum_{i=j}^{N+1} \frac{\rho^{i,R}}{\wt{W}_t^j - \wt{V}_t^{i,R,j}} \right| \leq C_1 \quad\text{for}\quad t \geq \wt{\tau}_{j-1}
\end{equation}
when $\wt{W}_t^j \in [-1,1]$.  Let $\CF_t^j = \sigma(\wt{\eta}_j(s) : s \leq t)$.  Applying Lemma~\ref{lem::sle_kappa_rho_crossing}, we see that $\p[\zeta_j - \xi_j \geq 1 \giv \CF_{\xi_j}^j] \geq \rho_0 > 0$ for $\rho_0$ depending only on $C_0, C_1, \rho^{1,R},\kappa$ which implies $\p[ \wt{\tau}_{2j} - \wt{\tau}_{2j-1} \geq 1 \giv \CF_{\xi_j}^j] \geq \rho_0 > 0$.  This completes the proof.
\end{proof}

\begin{remark}
\label{rem::light_cone_construction_general}
The proof of Proposition~\ref{prop::light_cone_construction} requires two inputs:
\begin{enumerate}[(i)]
\item $\eta'$ is almost surely a continuous path,
\item $\eta_{\phi_1 \cdots \phi_k \theta_1 \cdots \theta_\ell}^{\sigma_1 \cdots \sigma_k \tau_1 \cdots \tau_\ell}$ almost surely hits the left side of $\eta'([0,\tau'])$ or the side of $\striptop$ to the left of $z_0$ when $\theta_\ell = \tfrac{\pi}{2}$ and almost surely hits the right side of $\eta'([0,\tau'])$ or the side of $\striptop$ to the right of $z_0$ when $\theta_\ell = -\tfrac{\pi}{2}$.
\end{enumerate}
We will show in Section~\ref{sec::uniqueness} that (i)--(ii) hold whenever the counterflow line $\eta'$ and angle-varying flow line $\eta_{\phi_1 \cdots \phi_k}^{\sigma_1 \cdots \sigma_k}$ make sense (the SDEs for the Loewner driving processes have solutions).  Combining this with Remark~\ref{rem::light_cone_contains_av_general} implies that the light cone construction holds whenever $\eta'$ makes sense (i.e.,\ is an $\SLE_\kappa(\ul{\rho}^L;\ul{\rho}^R)$ process with $\sum_{i=1}^j \rho^{i,q} > -2$ for all $1 \leq j \leq |\ul{\rho}^q|$ and $q \in \{L,R\}$).
\end{remark}

Taking $\sigma \equiv 0$ in Proposition~\ref{prop::light_cone_construction} allows us to construct the range of the entire counterflow line.  We consider the result sufficiently important that we restate it as the following corollary.

\begin{corollary}
\label{cor::entire_light_cone_construction}
The random set obtained by taking the closure of $\eta_{\theta_1 \cdots \theta_\ell}^{\tau_1 \cdots \tau_\ell}([0,\tau_\ell])$ as $\theta_1,\ldots,\theta_\ell$ range over any countable dense subset of the interval~\eqref{eqn::angle_assumption} and $\tau_1,\ldots,\tau_\ell$ over any countable dense subset of $(0,\infty)$ is almost surely equal to the range of the counterflow line $\eta'$ starting at $z_0$.
\end{corollary}

We will now show that the entire path of the counterflow line $\eta'$ is determined by the light cone, not just its range.

\begin{proposition}
\label{prop::counterflow_determined}
Almost surely, $\eta'$ is determined by $h$.
\end{proposition}
\begin{proof}
The proof of Proposition~\ref{prop::light_cone_construction} implies that if $\eta$ is any angle varying flow line with angles satisfying~\eqref{eqn::angle_assumption} and $\sigma$ is an $\eta$ stopping time then $\eta(\sigma)$ is almost surely contained in the range of $\eta'$.  Moreover, we can realize the entire range of $\eta'$ by taking the closure of $\Gamma' := \{\eta(\sigma) : (\eta,\sigma) \in \CD \times \Q_+\}$ where $\CD$ is the set of angle varying flow lines with rational angles satisfying~\eqref{eqn::angle_assumption} and rational angle change times and $\Q_+ = \Q \cap (0,\infty)$.  Proposition~\ref{prop::light_cone_construction} also allows us to order $\Gamma'$ according to the order in which $\eta(\sigma)$, $(\eta,\sigma) \in \CD \times \Q_+$, is traced by $\eta'$.  Indeed, we can associate with the pair $(\eta,\sigma) \in \CD \times \Q_+$ the light cone $\lightcone(\eta,\sigma)$ of angle varying flow lines starting at $\eta(\sigma)$ with angles restricted by~\eqref{eqn::angle_assumption}.  We know that $\lightcone(\eta,\sigma)$ is equal to the range of $\eta'$ stopped when it first hits $\eta([0,\sigma])$ (the point of intersection is always $\eta(\sigma)$).  This implies that if $(\eta,\sigma), (\wt{\eta},\wt{\sigma}) \in \CD \times \Q_+$, then we almost surely have that either $\lightcone(\eta,\sigma) \subseteq \lightcone(\wt{\eta},\wt{\sigma})$ or $\lightcone(\wt{\eta},\wt{\sigma}) \subseteq \lightcone(\eta,\sigma)$.  That is, the sets $\lightcone(\eta,\sigma)$ are ordered by inclusion --- which in turn orders $\eta'$.
\end{proof}

We finish this section by making two remarks.  Our first remark is that we only needed to make use of maximal angles in the proof of Proposition~\ref{prop::light_cone_construction}, though the range of $\eta'$ still of course contains those angle varying flow lines in which the angles take on intermediate values.  If $\eta_{\theta_1 \cdots \theta_k}^{\tau_1 \cdots \tau_k}$ is an angle varying flow line with angle change times $\tau_1,\ldots,\tau_{k-1}$ where the $\theta_i$ satisfy~\eqref{eqn::angle_assumption} and $\theta_i = \tfrac{1}{\chi}(\lambda-\lambda') = \tfrac{\pi}{2}$, then $\eta_{\theta_1 \cdots \theta_k}^{\tau_1 \cdots \tau_k}|_{[\tau_{i-1},\tau_i]}$ traces the left boundary of $\eta'$ stopped at the time $\sigma'$ it first hits $\eta_{\theta_1 \cdots \theta_k}^{\tau_1 \cdots \tau_k}([0,\tau_{i-1}])$ starting from $\eta'(\sigma')$.  The same is also true if $\theta_i = \tfrac{1}{\chi}(\lambda'-\lambda) = - \tfrac{\pi}{2}$ but with left replaced by right.  Thus we can view the light cone construction of $\eta'$ as a refinement of $\SLE$ duality, as described in Section~\ref{sec::dubedat}. (We remark that the simulations in Figures~\ref{fig::sle6_lightcone}~{--}~\ref{fig::sle6_decomposition} were generated using maximal angle changes).  A flow line with an intermediate angle $\theta$ can be thought of as an angle-varying flow line which only takes maximal angles, but oscillates between going maximally left and right at infinitesimal scales, with the rate of oscillation depending on the particular value of $\theta$.

Second, suppose that $\ul{\theta} \leq \ol{\theta}$ are fixed angles.  Then we define $\lightcone(\ul{\theta},\ol{\theta})$ to be the closure of the set of points which are accessible by angle-restricted trajectories with rational angles which lie in the interval $[\ul{\theta},\ol{\theta}]$ and with positive rational angle change times.  Lemma~\ref{lem::light_cone_contains_av} implies that if $\ul{\theta} \geq -\tfrac{\pi}{2}$ and $\ol{\theta} \leq \tfrac{\pi}{2}$, then $\lightcone(\ul{\theta},\ol{\theta})$ is almost surely contained in the counterflow line $\eta'$.  We are now going to argue that, in this case, $\lightcone(\ul{\theta},\ol{\theta})$ is actually almost surely determined by $\eta'$ when $\kappa \in (2,4)$.  In particular, if $\theta \in [-\tfrac{\pi}{2},\tfrac{\pi}{2}]$, then the flow line $\eta_\theta$ with angle $\theta$ is almost surely determined by~$\eta'$.  To see this, we first note that since $\kappa \in (2,4)$ it follows that $\kappa' = 16/\kappa \in (4,8)$.  This implies that the range of $\eta'$ almost surely satisfies the hypothesis of Lemma~\ref{lem::local_restriction_determine} (see \cite[Theorem~8.1]{RS05}), which in turn implies that the law of $h$ given $\eta'$ is the same as the law of $h$ given both $\eta'$ and $h|_{\eta'}$.  (We emphasize here that we are conditioning $h$ on $\eta'$ as a curve as opposed to the range of $\eta'$ as a set.)  Since $\lightcone(\ul{\theta},\ol{\theta})$ is almost surely contained in $\eta'$, Proposition~\ref{gff::prop::local_independence} implies that the law of $h$ given $\eta'$ is equal to the law of $h$ given \emph{both} $\eta'$ and $\lightcone(\ul{\theta},\ol{\theta})$.  This implies that $h$ and $\lightcone(\ul{\theta},\ol{\theta})$ are conditionally independent given $\eta'$.  Since $\lightcone(\ul{\theta},\ol{\theta})$ is almost surely determined by $h$ (Lemma~\ref{lem::av_simple_determined}), this, in turn, implies that $\lightcone(\ul{\theta},\ol{\theta})$ is almost surely determined by $\eta'$. (If $h$ and a function of $h$ are conditionally independent given $\eta'$, then that function of $h$ must be determined by~$\eta'$.)

\begin{proposition}
\label{prop::light_cone_determined_by_cf}
Suppose that $\kappa \in (2,4)$ so that $\kappa' \in (4,8)$.  For any $-\tfrac{\pi}{2} \leq \ul{\theta} \leq \ol{\theta} \leq \tfrac{\pi}{2}$, $\eta'$ almost surely determines $\lightcone(\ul{\theta},\ol{\theta})$.  In particular, $\eta'$ almost surely determines $\eta_\theta$ for any $-\tfrac{\pi}{2} \leq \theta \leq \tfrac{\pi}{2}$ and $\eta'$ almost surely determines $\fan$ where $\fan$ is the closure of $\cup_{\theta} \eta_\theta$ where the union is over any fixed, countable dense subset of $[-\tfrac{\pi}{2},\tfrac{\pi}{2}]$.
\end{proposition}

\section{Interacting flow lines}
\label{sec::interacting}

Suppose that $h$ is a GFF on $\h$ with boundary data as in Figure~\ref{fig::interacting_mean}.  For each $\theta$, let $\eta_\theta$ be the flow line of $h$ with angle $\theta$, i.e.\ the flow line of $h+\theta \chi$ from $0$ to $\infty$.  Fix $\theta_1 < \theta_2$ and assume $a,b > 0$ are large enough so that Proposition~\ref{prop::monotonicity_non_boundary} is applicable to $\eta_{\theta_1}$ and $\eta_{\theta_2}$ (the exact values of $a,b$ will not be important for the applications we have in mind in Section~\ref{sec::uniqueness}).  We know from Proposition~\ref{prop::monotonicity_non_boundary} that $\eta_{\theta_1}$ almost surely is to the right of $\eta_{\theta_2}$.  It may be that $\eta_{\theta_1}$ intersects $\eta_{\theta_2}$, depending on the choice of $\theta_1,\theta_2$.  The purpose of this section is to rule out pathological behavior in the conditional mean of $h$ given $\eta_{\theta_1},\eta_{\theta_2}$ (Section~\ref{subsec::interacting_conditional_mean}), in particular when they intersect, and to show that the Loewner driving function of $\eta_{\theta_1}$ viewed as a path in the right connected component of $\h \setminus \eta_{\theta_2}$ exists and is continuous (Section~\ref{subsec::interacting_loewner_driving}) and likewise when the roles of $\eta_{\theta_1}$ and $\eta_{\theta_2}$ are swapped.  We will also explain how similar results can be obtained in the setting of multiple flow lines as well as counterflow lines.  We will use these results in Section~\ref{sec::uniqueness} to compute the conditional law of one path given the realization of a configuration of other paths, even if they intersect.

We emphasize that, throughout this section, the results we will state and prove will be for paths which do not intersect the boundary.  The reason for this restriction is that this is the class of boundary data for which we have the almost sure continuity of flow and counterflow lines at this point in the article (recall Remark~\ref{rem::continuity_non_boundary} and \cite{RS05}) as well as the results of Section~\ref{sec::dubedat} and Section~\ref{sec::non_boundary_intersecting}.  Upon establishing Theorems~\ref{thm::coupling_uniqueness}-\ref{thm::monotonicity_crossing_merging} in Section~\ref{sec::uniqueness}, the arguments we present here will imply that the conditional mean of the field given any configuration of flow and counterflow lines does not exhibit pathological behavior and that the Loewner driving function of one path given the realization of a configuration of other paths exists and is continuous, at least until there is a crossing.

\subsection{The conditional mean}
\label{subsec::interacting_conditional_mean}

\begin{figure}[h!]
\begin{center}
\subfigure[The boundary data for the conditional mean of $h$ given flow lines $\eta_{\theta_1},\eta_{\theta_2}$.]{
\includegraphics[scale=0.85,page=1]{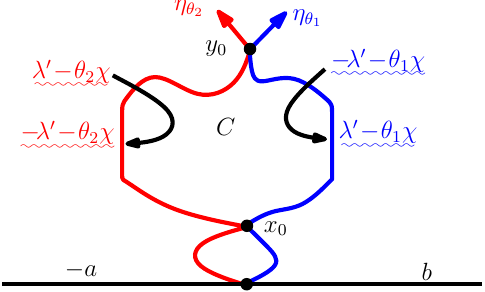}}
\hspace{0.01\textwidth}
\subfigure[The boundary data for the conditional mean of $h$ given flow lines $\eta_{\theta_1},\eta := \eta_0,\eta_{\theta_2}$ for $\theta_1 < 0 < \theta_2$.]{\includegraphics[scale=0.85,page=2]{figures/interacting_mean_three_flow_line.pdf}}
\caption{\label{fig::interacting_mean}
Suppose that $\theta_1 < \theta_2$ and $\eta_{\theta_i}$ is the flow line of a GFF on $\h$ with the boundary data depicted above with angle $\theta_i$.  Assume $\theta_1,\theta_2$ and the boundary data of $h$ are chosen so that Proposition~\ref{prop::monotonicity_non_boundary} applies to $\eta_{\theta_1},\eta_{\theta_2}$.  Since $\eta_{\theta_1} \cup \eta_{\theta_2}$ is a local set for $h$, Proposition~\ref{gff::prop::local_independence} implies that the boundary behavior of $\CC_{\eta_{\theta_1} \cup \eta_{\theta_2}}$ agrees with that of $\CC_{\eta_{\theta_2}}$ to the left of $\eta_{\theta_2}$ and with that of  $\CC_{\eta_{\theta_1}}$ to the right of $\eta_{\theta_1}$.  Proposition~\ref{gff::prop::cond_union_mean} implies the same in the connected components of $\h \setminus (\eta_{\theta_1} \cup \eta_{\theta_2})$ which lie between $\eta_{\theta_1}$ and $\eta_{\theta_2}$, except at those points where $\eta_{\theta_1}$ intersects $\eta_{\theta_2}$.  (There can only be two such points on the boundary of such a component, as illustrated.)  In Proposition~\ref{prop::cond_mean_height}, we will show that these intersection points do not introduce pathological behavior into the conditional mean.  The same result also holds when we consider multiple flow lines (see the right panel for the case $\theta_1 < 0 < \theta_2$ and $\eta := \eta_0$), as explained in Remark~\ref{rem::cond_mean_height}.}
\end{center}
\end{figure}

By Proposition~\ref{gff::prop::cond_union_local}, we know that $\eta_{\theta_1} \cup \eta_{\theta_2}$ (as a slight abuse of notation, throughout we will use $\eta_{\theta_i}$ to denote both the path and its range) is a local set for $h$.  Indeed, as a consequence of the special case of Theorem~\ref{thm::coupling_uniqueness} we proved in Section~\ref{subsec::uniqueness_non_boundary_intersecting}, the conditionally independent union of $\eta_{\theta_1}$ and $\eta_{\theta_2}$ given $h$ is almost surely the same as the usual union since $\eta_{\theta_1}$ and $\eta_{\theta_2}$ are both almost surely determined by $h$.  Recall that if $A$ is a local set for $h$, then $\CC_A$ is equal to the conditional mean of $h$ given $\CA$ (see Section~\ref{subsec::local_sets}).  By the continuity of $\eta_{\theta_1}$ and $\eta_{\theta_2}$ (recall Remark~\ref{rem::continuity_non_boundary}), we know that all of the connected components of $\eta_{\theta_1} \setminus \eta_{\theta_2}$ and $\eta_{\theta_2} \setminus \eta_{\theta_1}$ are larger than a single point.  Consequently, by Proposition~\ref{gff::prop::cond_union_mean}, we know the boundary behavior of $\CC_{\eta_{\theta_1} \cup \eta_{\theta_2}}$ away from $\eta_{\theta_1} \cap \eta_{\theta_2}$: if $z \in \eta_{\theta_1} \setminus \eta_{\theta_2}$ and $(z_k)$ is a sequence in $\h \setminus (\eta_{\theta_1} \cup \eta_{\theta_2})$ converging to $z$ then
\[ \CC_{\eta_{\theta_1} \cup \eta_{\theta_2}}(z_k) - \CC_{\eta_{\theta_1}}(z_k) \to 0 \quad\text{as}\quad k \to \infty \quad \text{almost surely}\]
and vice-versa when the roles of $\eta_{\theta_1}$ and $\eta_{\theta_2}$ are swapped.  Proposition~\ref{gff::prop::cond_union_mean} implies the same is true at those points $z$ which are at a positive distance from either $\eta_{\theta_1} \setminus \eta_{\theta_2}$ or $\eta_{\theta_2} \setminus \eta_{\theta_1}$ and are contained in a connected component of $\eta_{\theta_1} \cap \eta_{\theta_2}$ which consists of more than one point.  Also, Proposition~\ref{gff::prop::local_independence} gives us that the boundary behavior of $\CC_{\eta_{\theta_1} \cup \eta_{\theta_2}}$ agrees with that of $\CC_{\eta_{\theta_1}}$ (resp.\ $\CC_{\eta_{\theta_2}}$) to the right (resp.\ left) of $\eta_{\theta_1}$ (resp.\ $\eta_{\theta_2}$).  This leaves us to determine the boundary behavior of $\CC_{\eta_{\theta_1} \cup \eta_{\theta_2}}$ in between $\eta_{\theta_1}$ and $\eta_{\theta_2}$ at intersection points of $\eta_{\theta_1}$ and $\eta_{\theta_2}$.  This is the purpose of the next proposition.  Throughout, we let
\begin{equation}
\label{eqn::angles_filtration}
A(t) = \eta_{\theta_1}([0,t]) \cup \eta_{\theta_2}
\quad \text{and}\quad
\CF_t = \sigma(\eta_{\theta_1}(s) : s \leq t,\ \ \eta_{\theta_2}).
\end{equation}

\begin{proposition}
\label{prop::cond_mean_height}
Fix an $\CF_t$-stopping time $\tau$.  Let $\wt{h}$ be distributed according to the conditional law of $h$ given $A(\tau)$ and let $C$ be any connected component of $\h \setminus A(\tau)$ which is to the right of $\eta_{\theta_2}$.  Let $\partial C_{i,L}$ (resp.\ $\partial C_{i,R}$) be the part of $\partial C$ which is contained in the left (resp.\ right) side of $\eta_{\theta_i}$.  Let $x_0$ (resp.\ $y_0$) be the point on $\partial C$ which is visited first (resp.\ last) by $\eta_{\theta_2}$ and let $\varphi \colon C \to \h$ be a conformal transformation which takes $x_0$ (resp.\ $y_0$) to $0$ (resp.\ $\infty$).  Let $\Fg_C$ be the function which is harmonic in $\h$ with boundary values
\begin{align*}
   -\lambda - \theta_i \chi \quad\text{on}\quad \varphi(\partial C_{i,R}),\quad \lambda + \theta_i \chi \quad\text{on}\quad \varphi(\partial C_{i,L}),\quad\text{and}\quad
   b \quad\text{on}\quad \varphi((0,\infty))
\end{align*}
and let $\Fh_C = \Fg_C \circ \varphi - \chi \arg \varphi'$ (where the branch of $\arg \varphi'$ is chosen so that the boundary values of $\Fh_C$ agree with those of the conditional law of $h$ given either $\eta_{\theta_1}$ or $\eta_{\theta_2}$ on a segment of $\partial C$ which agrees with either $\eta_{\theta_1}$ or $\eta_{\theta_2}$).  Then the law of $\wt{h}|_{C}$ is equal to that of the sum of a zero boundary GFF in $C$ plus $\Fh_C$.  In particular, there is no singular contribution to $\Fh_C$ coming from the intersection points of the paths.
\end{proposition}

We note that the choice of the branch of the argument of $\arg \varphi'$ in the statement of Proposition~\ref{prop::cond_mean_height} is well-defined because Proposition~\ref{gff::prop::cond_union_mean} implies that the boundary behavior of $\wt{h}|_C$ agrees with that of $h$ given $\eta_{\theta_1}$ (resp.\ $\eta_{\theta_2}$) along the part of $\partial C$ which agrees with $\eta_{\theta_1}$ (resp.\ $\eta_{\theta_2}$), except possibly at two exceptional points.  (In fact, the proof of Proposition~\ref{prop::cond_mean_height} will in fact rule out such exceptional behavior.)

See Figure~\ref{fig::interacting_mean} for an illustration of the boundary data described in the statement of Proposition~\ref{prop::cond_mean_height}.  The main step in the proof is to show for $z \in \h \setminus \eta_{\theta_2}$ that $\CC_{A(t)}(z)$ has a modification which is continuous in $t$ up until the first time $\eta_{\theta_1}$ hits $z$.  Roughly, this suffices since pathological behavior in $\CC_{\eta_{\theta_1} \cup \eta_{\theta_2}}$ at a point $z_0 = \eta_{\theta_1}(t_0)$ for a time $t_0$ when $\eta_{\theta_1}$ intersects $\eta_{\theta_2}$ would correspond to a discontinuity in $\CC_{A(t)}(z)$ at $t_0$.  We begin by proving the following lemma, which implies that $A(\tau)$ is a local set for $h$ for every $\CF_t$-stopping time $\tau$.

\begin{lemma}
\label{lem::stopping_local_set}
Suppose that $\eta_1,\ldots,\eta_k$ are continuous paths such that for each $1 \leq i \leq k$, we have that
\begin{enumerate}
\item $\eta_i([0,\tau])$ is a local set for $h$ for every $\eta_i$-stopping time $\tau$ and
\item $\eta_i$ is almost surely determined by $h$.
\end{enumerate}
Suppose that $\tau_1$ is a stopping time for $\eta_1$ and, for each $2 \leq j \leq k$, inductively let $\tau_j$ be a stopping time for the filtration $\CF_t^j$ generated by $\eta_1|_{[0,\tau_1]},\ldots,\eta_{j-1}|_{[0,\tau_{j-1}]}$ and $\eta_j(s)$ for $s \leq t$.  Then $\cup_{i=1}^k \eta_i([0,\tau_i])$ is a local set for $h$.
\end{lemma}
\begin{proof}
Let $A_j = \cup_{i=1}^j \eta_i([0,\tau_i])$.  Fix $U \subseteq \h$ open.  We are going to prove that $A_j \cap U = \emptyset$ is almost surely determined by the projection $h_{U^c}$ of $h$ onto $H^\perp(U)$ and that, on the event $A_j \cap U = \emptyset$, $A_j$ is itself almost surely determined by $h_{U^c}$.  This suffices by characterization~\eqref{it::Ucond} of local sets given in Lemma~\ref{lem::local_char}, which we will in turn check by induction on the number of paths.  The hypotheses of the lemma imply this is true for $j=1$.  Suppose the result holds for $j-1$ paths for $j \geq 2$ fixed.  We will now show that it holds for $j$ paths.  Let $\tau_j^U$ be the infimum of times in which $\eta_j$ is in $U$.  We claim that the hypotheses of the lemma imply that $\eta_j([0,\tau_j^U])$ is almost surely determined by $h_{U^c}$.  Indeed, this follows because $\eta_j([0,\tau_j^U])$ is both almost surely determined by $h$ and local for $h$.  In particular, since the projections $h_U$ and $h_{U^c}$ of $h$ onto $H(U)$ and $H^\perp(U)$ together determine $h$, it follows that $h_U,h_{U^c}$ together almost surely determine $\eta_j([0,\tau_j^U])$.  On the other hand, since $\eta_j([0,\tau_j^U])$ is local for $h$ and almost surely does not intersect $U$, it follows that the conditional law of $h_U$ given $h_{U^c}$ is equal to the conditional law of $h_U$ given both $h_{U^c}$ and $\eta_j([0,\tau_j^U])$.  That is, $h_U$ and $\eta_j([0,\tau_j^U])$ are conditionally independent given $h_{U^c}$.  Combining these two observations proves the claim.

Observe
\[ \{A_j \cap U = \emptyset\} = \{A_{j-1} \cap U = \emptyset\} \cap \{ \tau_j \leq \tau_j^U\}\]
and that $\{ \tau_j \leq \tau_j^U\}$ is almost surely determined by $A_{j-1}$ and $h_{U^c}$.  Thus on the event $\{A_{j-1} \cap U = \emptyset\}$, we have that $\{\tau_j \leq \tau_j^U\}$ is almost surely determined by $h_{U^c}$ so that the event $\{A_j \cap U = \emptyset\}$ is almost surely determined by $h_{U^c}$.  Moreover, on $\{A_j  \cap U = \emptyset\}$, we have that $A_j$ is almost surely determined by $h_{U^c}$.  This completes the proof of the induction step.
\end{proof}

We next prove the following simple lemma, which says that if $X,Y$ are independent, $Y$ is Gaussian, and $X+Y$ is Gaussian, then $X$ is Gaussian as well.

\begin{lemma}
\label{lem::independent_gaussian_sum}
Suppose that $X,Y$ are independent random variables such that $Y \sim N(\mu_Y,\sigma_Y^2)$ and $Z: = X+Y \sim N(\mu_Z,\sigma_Z^2)$.  Then $X \sim N(\mu_X,\sigma_X^2)$ where $\mu_X = \mu_Z - \mu_Y$ and $\sigma_X^2 = \sigma_Z^2 - \sigma_Y^2$.
\end{lemma}
\begin{proof}
The proof follows from the calculus of characteristic functions.  Indeed, we know that
\[ \E[e^{i \lambda Y}] = e^{i \lambda \mu_Y- \lambda^2 \sigma_Y^2/2} \quad\text{and}\quad \E[e^{i\lambda Z}] = e^{i \lambda \mu_Z - \lambda^2 \sigma_Z^2/2}.\]
Since $X$ is independent of $Y$, we have
\[ \E[e^{i \lambda Z}] = \E[e^{i \lambda X}] \E[e^{i \lambda Y}],\]
which allows us to solve for $\E[e^{ i\lambda X}]$ to see that
\[  \E[e^{i \lambda X}] = e^{i \lambda (\mu_Z - \mu_Y) - \lambda^2(\sigma_Z^2 - \sigma_Y^2)/2}.\]
\end{proof}

Suppose that $D \subseteq \C$ is a non-trivial simply connected domain and that $z \in D$ is fixed.  Recall that the conformal radius $C(z;D)$ is the quantity $|\varphi'(0)|$ where $\varphi$ is a conformal transformation which takes the unit disk $\D$ to $D$ with $\varphi(0) = z$.  By the Koebe-$1/4$ Theorem~\cite[Theorem~3.16]{LAW05}, the ratio between the conformal radius of $z$ and the distance of $z$ to $\partial D$ is contained in $[\tfrac{1}{4},4]$.  Suppose that $h$ is a GFF on $D$ (with boundary conditions which are not necessarily equal to $0$) and let $\CC$ be the function which is harmonic in $D$ and has the same boundary data as $h$.  The following lemma gives us the law of $\CC_A(z)$ for a local set $A$ in terms of $C(z;D)$, $C(z;D \setminus A)$, and $\CC(z)$.

\begin{lemma}
\label{lem::cond_mean_height_distribution}
Suppose that $D \subseteq \C$ is a non-trivial, simply connected domain.  Let $h$ be a GFF on $D$ and fix $z \in D$.  Suppose that $A$ is a local set for $h$ such that $D \setminus A$ is simply connected and $C(z;D \setminus A)$ is almost surely constant and positive.  Then $\CC_A(z)$ is distributed as a Gaussian random variable with mean $\CC(z)$ and variance $\log C(z;D) - \log C(z;D \setminus A)$.
\end{lemma}
\begin{proof}
The Koebe $1/4$ theorem \cite[Theorem~3.16]{LAW05} implies there exists non-random $\epsilon > 0$ such that, almost surely, $B(z,2\epsilon) \subseteq D \setminus A$.  Let $h_\epsilon$ denote the average of $h$ on $\partial B(z,\epsilon)$ (the construction and properties of the circle average process are explained in detail in \cite[Section~3]{DS08}).  By \cite[Proposition~3.2]{DS08}, we know that $h_\epsilon(z) \sim N(\CC(z), -\log \epsilon + \log C(z;D))$.  Since $A$ is a local set for $h$, we can write $h = h_1 + h_2$ where $h_1$ is harmonic on $D \setminus A$ and the conditional law of $h_2$ given $h_1$ is that of a zero-boundary GFF on $D \setminus A$.  Since $h_1$ is harmonic in $D \setminus A$, we note that $h_1(z)$ is equal to the average of its values on $\partial B(z,\epsilon)$.  Moreover, we have that $\E[ h_\epsilon(z) \giv \CA] = h_1(z)$.  Consequently, we have that $h_\epsilon(z) - \E[h_\epsilon(z) \giv \CA]$ is equal to the average of $h_2$ on $\partial B(z,\epsilon)$.  Therefore it follows that $h_\epsilon(z) - \E[h_\epsilon(z) \giv \CA]$ given $\CA$ follows the $N(0, -\log \epsilon + \log C(z;D \setminus A))$ distribution.  This implies that
\[ h_\epsilon(z) - \E[h_\epsilon(z) \giv \CA] \sim N(0,-\log \epsilon + \log C(z;D \setminus A))\]
(no longer conditioning on $\CA$; recall that $C(z;D \setminus A)$ is almost surely constant).  Since $\E[h_\epsilon(z) \giv \CA ]$ is independent of $h_\epsilon(z) - \E[h_\epsilon(z) \giv \CA]$ (as the former is determined by $\CA$ and the latter is independent of $\CA$), Lemma~\ref{lem::independent_gaussian_sum} implies that
\[ \E[h_\epsilon(z) \giv \CA ] \sim N(\CC(z), \log C(z;D) - \log C(z;D \setminus A)).\]
Since $\CC_A(z)$ is harmonic in $D \setminus A$, we know that $\CC_A(z) = \E[h_\epsilon(z) \giv \CA]$ almost surely, which proves the lemma.
\end{proof}

From Lemma~\ref{lem::cond_mean_height_distribution} we obtain the following, which roughly says that the conditional mean at a fixed point given an increasing family of local sets evolves as a Brownian motion when it is parameterized by the log conformal radius:

\begin{proposition}
\label{prop::cond_mean_continuous}
Suppose that $D \subseteq \C$ is a non-trivial simply connected domain.  Let $h$ be a GFF on $D$ and suppose that $(Z(t) : t \geq 0)$ is an increasing family of closed sets such that
\begin{enumerate}[(i)]
\item $D \setminus Z(t)$ is simply connected for each $t \geq 0$ and
\item $Z(\tau)$ is local for $h$ for every $Z$-stopping time $\tau$.
\end{enumerate}
Suppose that $z \in D$ is such that $C(z;D \setminus Z(t))$ is almost surely continuous and strictly decreasing in $t$.  Then $\CC_{Z(t)}(z) - \CC_{Z(0)}(z)$ has a modification which is a Brownian motion when parameterized by $\log C(z;D \setminus Z(0))-\log C(z;D \setminus Z(t))$ up until the first time $\tau(z)$ that $Z(t)$ accumulates at $z$.  Moreover, with $S = \{ (t,z) : C(z; D \setminus Z(t)) > 0\}$ we have that the map $(t,z) \mapsto \CC_{Z(t)}(z)$ has a modification which is almost surely continuous.
\end{proposition}
\begin{proof}
For each $s > 0$ let
\[ \tau_s(z) = \inf\{t \geq 0 : \log C(z;D \setminus Z(0)) - \log C(z;D \setminus Z(t)) = s\}.\]
Fix $s_1 < s_2$.  Then Lemma~\ref{lem::cond_mean_height_distribution} implies that
\[ \CC_{Z(\tau_{s_2}(z))}(z) - \CC_{Z(\tau_{s_1}(z))}(z) \sim N(0, s_2 - s_1).\]
Since $\CC_{Z(\tau_{s_2}(z))}(z) - \CC_{Z(\tau_{s_1}(z))}(z)$ is independent of $\CC_{Z(\tau_{s_1}(z))}(z)$, the first part of the proposition follows since $\CC_{Z(\tau_s(z))}(z)$ has the same finite dimensional distributions as a standard Brownian motion.  This, in particular, implies that $\CC_{Z(t)}(z)$ has a modification which is continuous in $t$.

That $\CC_{Z(t)}(z)$ has a modification which is jointly continuous in $t$ and $z$ is a consequence of the proof given in \cite[Section~3]{DS08} that the circle average process $h_\epsilon(z)$ has a modification which is jointly continuous in $\epsilon$ and $z$.  Fix $T > 0$ and $w \in B(z, \tfrac{1}{16} e^{-T})$ and $\epsilon \in (0,\tfrac{1}{16} e^{-T})$.  Then for $s,t \in [0,T]$ and $p \geq 2$, we have for some constants $c_1,c_2 > 0$ that
\begin{align*}
    &  \E[ |\CC_{Z(\tau_s(z))}(z) - \CC_{Z(\tau_t(z))}(w)|^p]\\
\leq&  c_1\left( \E[ | \CC_{Z(\tau_s(z))}(z) - \CC_{Z(\tau_t(z))}(z) |^p ] + \E[  | \CC_{Z(\tau_t(z))}(z) - \CC_{Z(\tau_t(z))}(w)|^p] \right)\\
\leq& c_2 \left( |t-s|^{p/2} +  \E\left[  \left| \E[ h_\epsilon(z) - h_\epsilon(w) \giv \CA_t] \right|^p \right] \right)
\end{align*}
where $\CA_t$ is as in Section~\ref{subsec::local_sets} for the local set $Z(\tau_t(z))$.  By Jensen's inequality, the second term is bounded from above by $c_2 \E[ | h_\epsilon(z) - h_\epsilon(w)|^p]$.  The moments of this type are bounded in the proof of \cite[Proposition~3.1]{DS08} (see also \cite[Proposition~3.2]{DS08}).  The final claim of the proposition then follows from the Kolmogorov-\u{C}entsov theorem.
\end{proof}

\begin{figure}[h!]
\begin{center}
\subfigure[Pocket closing time.]{\includegraphics[scale=0.85,page=1]{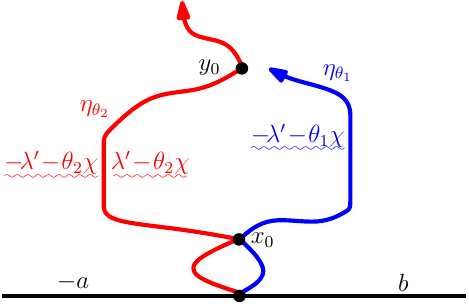}}
\hspace{0.01\textwidth}
\subfigure[Pocket opening time; this is the same as a pocket closing time for a counterflow line $\eta_{\theta_1}'$ whose left boundary is $\eta_{\theta_1}$.]{\includegraphics[scale=0.85,page=2]{figures/interacting_mean_almost_hit.pdf}}
\caption{\label{fig::interacting_mean_continuity} Suppose that we have the same setup as in Figure~\ref{fig::interacting_mean} and that $C$ is a connected component of $\h \setminus (\eta_{\theta_1} \cup \eta_{\theta_2})$ which lies between $\eta_{\theta_1}$ and $\eta_{\theta_2}$.  By the discussion in Figure~\ref{fig::interacting_mean}, we know the behavior of $\CC_{\eta_{\theta_1} \cup \eta_{\theta_2}}$ in $C$ away from $x_0$ and $y_0$, the first and last points of $\partial C$ traced by $\eta_{\theta_1}$, respectively.  To rule out pathological behavior near $x_0$, $y_0$, with $A(t) = \eta_{\theta_1}([0,t]) \cup \eta_{\theta_2}$, we first prove that $\CC_{A(t)}$ has a modification which is continuous in $t$.  We argue that the conditional mean does not behave pathologically at $y_0$ (illustrated in the left panel above) by taking a limit as $t$ increases to the time $\eta_{\theta_1}$ closes the pocket $C$ opened at $x_0$ and invoking the continuity of $\CC_{A(t)}$.  At the pocket opening point $x_0$, we can rule out pathological behavior by using that $\eta_{\theta_1}$ is almost surely the left boundary of a counterflow line $\eta_{\theta_1}'$ (see the right panel) so that $x_0$ is a pocket closing point of $\eta_{\theta_1}'$ and an analogous continuity argument.
}
\end{center}
\end{figure}

We now have all of the ingredients to complete the proof of Proposition~\ref{prop::cond_mean_height}.

\begin{proof}[Proof of Proposition~\ref{prop::cond_mean_height}]
We first assume we are in the setting of Proposition~\ref{prop::monotonicity_non_boundary}, so that $h$ is a GFF on the strip $\strip$ rather than $\h$.  The result for the GFF on $\h$ follows from absolute continuity (Proposition~\ref{prop::gff_abs_continuity}).  Note that $A(\tau)$ (recall the definition of $A(t)$ from~\eqref{eqn::angles_filtration}) is a local set for $h$ by Lemma~\ref{lem::stopping_local_set} applied for the case $k=2$.

We start by working in the special case $\tau=\infty$ and we let $A = A(\infty) = \eta_{\theta_1} \cup \eta_{\theta_2}$.  First of all, since $\eta_{\theta_1}$ and $\eta_{\theta_2}$ are almost surely continuous, the connected components of $\eta_{\theta_1} \setminus \eta_{\theta_2}$ and $\eta_{\theta_2} \setminus \eta_{\theta_1}$ consist of more than a single point.  Consequently, by Proposition~\ref{gff::prop::cond_union_local} and Proposition~\ref{gff::prop::cond_union_mean}, we know that $\CC_A - \CC_{\eta_{\theta_1}}$ tends to zero along any sequence of points $(z_k)$ which converges to a point that is contained in $\eta_{\theta_1} \setminus \eta_{\theta_2}$ or to a point in a connected component of $\eta_{\theta_1} \cap \eta_{\theta_2}$ which contains more than a single point and is at a positive distance from either $\eta_{\theta_1} \setminus \eta_{\theta_2}$ or $\eta_{\theta_2} \setminus \eta_{\theta_1}$.  Proposition~\ref{gff::prop::local_independence} implies the same if $(z_k)$ converges to a point on the right side of $\eta_{\theta_1}|_{[0,\tau]}$.  Likewise, $\CC_A - \CC_{\eta_{\theta_2}}$ converges to zero along any sequence of points $(z_k)$ which converges to a point in either $\eta_{\theta_2} \setminus \eta_{\theta_1}$ or to the left side of $\eta_{\theta_2}$.  Fix a component $C$ of $\h \setminus A$.  Then this implies that $\CC_A$ agrees with $\Fh_{C}$ if $C$ is either the unbounded connected component which lies to the right of $\eta_{\theta_1}$ or to the left of $\eta_{\theta_2}$.

Suppose that $C$ is a bounded and connected component of $\strip \setminus A$.  Then $\partial C$ has two special points, say $x_0,y_0$ which are contained in both the ranges of $\eta_{\theta_1}$ and $\eta_{\theta_2}$.  To complete the proof for $\tau = \infty$ in this case, we just need to show that the boundary behavior of $\CC_A$ agrees with $\Fh_{C}$ at $x_0$ and $y_0$.  Assume $x_0 = \eta_{\theta_1}(s_0)$ and $y_0 = \eta_{\theta_1}(t_0)$ for $s_0 < t_0$.  Proposition~\ref{prop::cond_mean_continuous} implies $\CC_{A(t)}(z)$ has a modification which is continuous in both~$t$ and~$z$ since we know that $A(\sigma)$ is local for any $\CF_t$-stopping time $\sigma$ and $\eta_{\theta_1}$ is continuous (recall Remark~\ref{rem::continuity_non_boundary}).

We can express $\CC_{A(t)}(z)$ explicitly as follows.  For each $t$, we let $C_t$ be the component of $\h \setminus A(t)$ which contains $C$.  We note that $x_0$ corresponds to two prime ends on $\partial C_t$.  We shall abuse notation in what follows and write $x_0$ for the prime end on $\partial C_t$ which corresponds to the boundary point of~$\partial C$ (as opposed to the prime end which corresponds to a boundary point of the component of $\h \setminus \eta_{\theta_1}$ which is to the right of~$\eta_{\theta_1}$).  Let $\varphi_t$ be the conformal map from $C_t$ to $\h$ which takes $x_0$ to $0$, $y_0$ to $\infty$, and a given point $w_0$ on $\partial C \cap \eta_{\theta_2}$ which is distinct from $x_0,y_0$ to $-1$.  We assume that $w_0$ does not change with $t$.  Let $\Fg_t$ be the function which is harmonic in $\h$ with boundary values given by $\lambda - \theta_2 \chi$ on $\R_-$, $-\lambda - \theta_1 \chi$ on the image of the left side of $\eta_{\theta_1}([0,t])$ under $\varphi_t$, $\lambda - \theta_1 \chi$ on the image of the right side of $\eta_{\theta_1}([0,t])$ under $\varphi_t$, and $b$ on $\varphi_t(\R_+)$.  Then $\Fg_t \circ \varphi_t - \chi \arg \varphi_t'$ is harmonic in $C_t$.

We claim that $\Fg_t \circ \varphi_t - \chi \arg \varphi_t'$ has the same  boundary behavior as $\CC_{A(t)}$ except possibly at $x_0$ (as we have not yet ruled out pathological boundary behavior at $x_0$).  That is, $\CC_{A(t)} - (\Fg_t \circ \varphi_t - \chi \arg \varphi_t')$ is harmonic in $C_t$ with zero boundary values on $\partial C_t \setminus \{x_0\}$.  To see this, we note that Proposition~\ref{gff::prop::cond_union_mean} implies that this is the case at points on $\partial C_t \setminus \{x_0\}$ which are also contained in the ranges of $\eta_{\theta_1}$ and $\eta_{\theta_2}$ after the paths have visited $x_0$.  Moreover, Proposition~\ref{gff::prop::cond_union_mean} implies that $\CC_{A(t)}$ has the same boundary behavior as $\CC_A$ at points on the right side of $\eta_{\theta_1}([0,t])$ and Proposition~\ref{gff::prop::local_independence} implies that $\CC_A$ restricted to the component of $\h \setminus \eta_{\theta_1}$ which is to the right of $\eta_{\theta_1}$ is equal to~$\CC_{\eta_{\theta_1}}$.  Combining implies the claim.

As $t \uparrow t_0$, we note that $\Fg_t$ converges locally uniformly to the function which is harmonic in $\h$ with boundary values given by $\lambda - \theta_2 \chi$ on $\R_-$ and $-\lambda - \theta_1 \chi$ on $\R_+$.  Moreover, $\varphi_t$ converges locally uniformly to the unique conformal transformation $C \to \h$ which takes $x_0$ to $0$, $y_0$ to $\infty$, and $w_0$ to $-1$.  Therefore $\Fg_t \circ \varphi_t - \chi \arg \varphi_t'$ converges locally uniformly to $\Fh_C$ as $t \uparrow t_0$.  Combining, we see that, as $t \uparrow t_0$, $\CC_{A(t)} - \Fh_C$ converges to a function which is harmonic in $C$ whose boundary values on $\partial C \setminus \{x_0\}$ are equal to $0$.  By the continuity of $\CC_{A(t)}$ in $t$ and $z$, this implies that $\CC_A - \Fh_C$ is harmonic in~$C$ with boundary values on $\partial C \setminus \{x_0\}$ are equal to~$0$.  This leaves us to deal with the boundary behavior near~$x_0$.

Let $\eta_{\theta_1}'$ be the counterflow line as in the proof of Proposition~\ref{prop::monotonicity_non_boundary} whose left boundary is almost surely $\eta_{\theta_1}$.  Note that $C$ is a bounded and connected component of $\strip \setminus A$ if and only if it is a bounded connected component of $\strip \setminus (\eta_{\theta_1}' \cup \eta_{\theta_2})$ whose boundary contains arcs from both $\eta_{\theta_1}'$ and $\eta_{\theta_2}$.  Since $\eta_{\theta_1}$ is almost surely the left boundary of $\eta_{\theta_1}'$, it follows from Proposition~\ref{gff::prop::local_independence} that $\CC_{\eta_{\theta_1}' \cup \eta_{\theta_2}}(z) = \CC_A(z)$ for all $z \in C$.  An analogous continuity argument implies $\CC_{\eta_{\theta_1}' \cup \eta_{\theta_2}}$ has the same boundary behavior as $\Fh_{C}$ near $x_0$.    Consequently, $\CC_A$ also has the same boundary behavior as $\Fh_{C}$ near $x_0$.  This completes the proof for $\tau = \infty$.

The case $\tau < \infty$ follows from the $\tau=\infty$ case.  Indeed, we know that $A(\tau)$ and $A$ are both local, so we can apply Proposition~\ref{gff::prop::cond_union_mean} to get that $\CC_{A(\tau)}$ has the same boundary behavior as $\Fh_{C}$ near $x_0,y_0$ since $\CC_{A(\tau)}$ has the same boundary behavior as $\CC_A$ near $x_0,y_0$.
\end{proof}

There are a number of other situations in which statements very similar to Proposition~\ref{prop::cond_mean_height} also hold.  We will describe these informally in the following remarks.  In each case, the justification is nearly the same as the proof of Proposition~\ref{prop::cond_mean_height} and we are careful to point out any differences in the proof.  Roughly speaking, the content is that the conditional mean of $h$ does not exhibit pathological behavior, even when many different types of flow and counterflow lines interact with each other.  The rest of this subsection may be skipped on a first reading.

\begin{remark}
\label{rem::cond_mean_general_bd}
All of the results that we state and prove here will be restricted to the regime of boundary data in which the flow and counterflow lines do not intersect the boundary.  The reason for this is that, at this point in the article, this is the setting in which we have the continuity of these curves (recall Remark~\ref{rem::continuity_non_boundary} as well as \cite{RS05}) and that they are determined by the field in the coupling of Theorem~\ref{thm::coupling_existence} (Section~\ref{subsec::uniqueness_non_boundary_intersecting}).  Moreover, our results thus far regarding the interaction of flow and counterflow lines are also restricted to this setting (Section~\ref{sec::non_boundary_intersecting}).  In Section~\ref{sec::uniqueness}, we will complete the proof of Theorems~\ref{thm::coupling_uniqueness}-\ref{thm::monotonicity_crossing_merging}, which together provide the missing ingredients to extend the arguments from this section to the setting of general piecewise constant boundary data without further modification.  
\end{remark}

\begin{remark}[Three flow lines]
\label{rem::cond_mean_height}
Suppose that $\theta_1 < 0 < \theta_2$ and $\eta := \eta_0$.  Assume that the boundary data of $h$ is chosen so that Proposition~\ref{prop::monotonicity_non_boundary} applies to $\eta_{\theta_1},\eta,\eta_{\theta_2}$.  Then we know that $\eta_{\theta_1}$ lies to the right of $\eta$ which in turn lies to the right of $\eta_{\theta_2}$, as in Figure~\ref{fig::interacting_mean}.  Let $C$ be any connected component of $\h \setminus (\eta_{\theta_1} \cup \eta_{\theta_2})$.  A statement analogous to Proposition~\ref{prop::cond_mean_height} also holds for the conditional law of $h$ given $\eta_{\theta_1}, \eta_{\theta_2}$, and $\eta|_{[0,\tau]}$ where $\tau$ is any stopping time for the filtration $\CF_t = \sigma(\eta(s) : s \leq t,\ \ \eta_{\theta_1},\eta_{\theta_2})$.  This is depicted in the right panel of Figure~\ref{fig::interacting_mean}.  Let $A(\tau) = \eta_{\theta_1} \cup \eta([0,\tau]) \cup \eta_{\theta_2}$.  Just as in the proof of Proposition~\ref{prop::monotonicity_non_boundary}, the general theory of local sets allows us to determine the boundary behavior of $\CC_{A(\tau)}$ at all points with the exception of those points where some pair of $\eta|_{[0,\tau]},\eta_{\theta_1},\eta_{\theta_2}$ intersect.  We can, however, reduce the three flow line case to the two flow line case as follows.  Proposition~\ref{gff::prop::cond_union_mean} allows us to compare $\CC_{A(\tau)}$ with $\CC_A$ where $A = \eta_{\theta_1} \cup \eta \cup \eta_{\theta_2}$.  The latter does not exhibit pathological behavior at intersection points because Proposition~\ref{prop::cond_mean_height} applies to both $\CC_{\eta \cup \eta_{\theta_1}}$ and $\CC_{\eta_{\theta_2} \cup \eta}$ and Proposition~\ref{gff::prop::cond_union_mean} and Proposition~\ref{gff::prop::local_independence} together imply that $\CC_A$ has the same boundary behavior as $\CC_{\eta \cup \eta_{\theta_1}}$ to the right of $\eta$ and the same as $\CC_{\eta_{\theta_2} \cup \eta}$ to the left of $\eta$.
\end{remark}

\begin{figure}[h!]
\begin{center}
\includegraphics[scale=0.85]{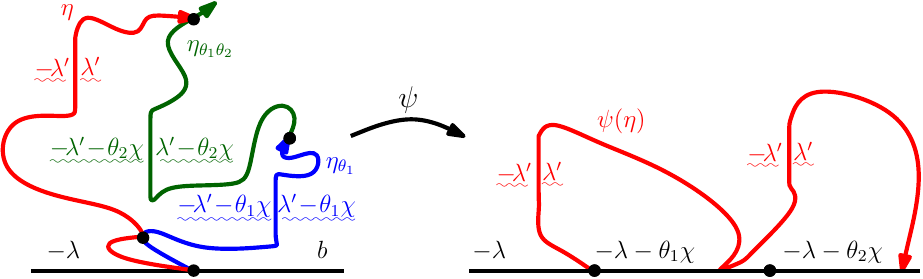}
\end{center}
\caption{\label{fig::multiple_force_points_angle_change} Suppose that $h$ is a GFF on $\h$ with boundary data as on the left side; assume $a,b > 0$ are large.  Let $\eta$ be the zero angle flow line of $h$.  Fix angles $\theta_1,\theta_2$ and let $\eta_{\theta_1 \theta_2}^{\tau_1 \tau_2}$ be an angle-varying flow line with angles $\theta_1,\theta_2$.  Assume $|\theta_1-\theta_2| \leq \pi$ so that $\eta_{\theta_1 \theta_2}^{\tau_1 \tau_2}$ is simple (Lemma~\ref{lem::av_simple_determined}; we will relax this to $|\theta_1-\theta_2| < 2\lambda/\chi$ in Section~\ref{sec::uniqueness}) and that $\eta_{\theta_1 \theta_2}^{\tau_1 \tau_2}$ passes to the right of $\eta$ (we will show in Proposition~\ref{prop::angle_varying_monotonicity} that $\theta_1,\theta_2 < 0$ is a sufficient condition for this to be true).  The boundary data for $\CC_{A(\tau)}$ for $A(t) = \eta([0,t]) \cup \eta_{\theta_1 \theta_2}^{\tau_1 \tau_2}$ and $\tau$ an $\CF_t = \sigma(\eta(s) : s \leq t,\ \ \eta_{\theta_1 \theta_2})$-stopping time is depicted in the left panel.  Changing the coordinates of the left side of $\h \setminus \eta_{\theta_1 \theta_2}^{\tau_1 \tau_2}$ to $\h$ by a conformal map $\psi$ which preserves $0$ and $\infty$ yields a GFF whose boundary data is as on the right side.}
\end{figure}

\begin{remark}[One flow line and one angle varying flow line]
\label{rem::angle_varying_mean_height}
The next analog of Proposition~\ref{prop::cond_mean_height} which we will describe is when we have a flow line $\eta$ and an angle varying flow line $\eta_{\theta_1 \cdots \theta_k}^{\tau_1 \cdots \tau_k}$ with angles $\theta_1,\ldots,\theta_k$ and with respect to the stopping times $\tau_1,\ldots,\tau_k$.  We assume that $|\theta_i - \theta_j| \leq \pi$ for all pairs $i,j$.  This implies that $\eta_{\theta_1 \cdots \theta_k}^{\tau_1 \cdots \tau_k}$ is simple, almost surely determined by $h$, and continuous (Lemma~\ref{lem::av_simple_determined}; in Section~\ref{sec::uniqueness}, we will be able to relax this to the case that $|\theta_i - \theta_j| < 2\lambda/\chi$, which is the condition which implies that $\eta_{\theta_1 \cdots \theta_k}^{\tau_1 \cdots \tau_k}$ does not cross itself).  We assume that $\eta_{\theta_1 \cdots \theta_k}^{\tau_1 \cdots \tau_k}$ almost surely stays to the right of $\eta$ (in Proposition~\ref{prop::angle_varying_monotonicity} we will prove that $\theta_1,\ldots,\theta_k < 0$ is a sufficient condition for this to be true).  Let $A(t) = \eta([0,t]) \cup \eta_{\theta_1 \cdots \theta_k}^{\tau_1 \cdots \tau_k}$ and $\CF_t = \sigma(\eta(s) : s \leq t,\ \ \eta_{\theta_1 \cdots \theta_k}^{\tau_1 \cdots \tau_k})$.  Lemma~\ref{lem::stopping_local_set} implies that $A(\tau)$ is a local set of $h$ for every $\CF_t$-stopping time $\tau$.  The boundary data for $\CC_{A(\tau)}$ is described in Figure~\ref{fig::multiple_force_points_angle_change} in the special case $k=2$.  The proof of this result is exactly the same as for the non-angle varying case: we rule out pathological behavior at pocket opening times using the continuity of $\CC_{A(t)}$ in $t$ and at pocket closing times by using the analogous quantity with $\eta$ replaced by the counterflow line $\eta'$ whose right boundary is the range of $\eta$.
\end{remark}

\begin{figure}[h!]
\begin{center}
\includegraphics[scale=0.85]{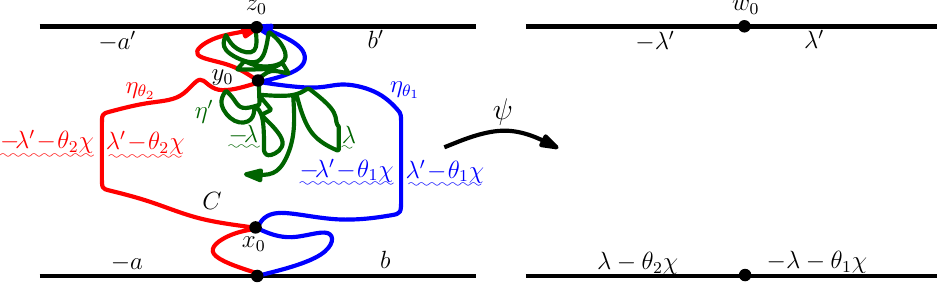}
\end{center}
\caption{\label{fig::cond_mean_non_crossing} Assume that $h$ is a GFF on the strip $\strip$ whose boundary data is depicted in the left panel.  Let $\theta_1 < \tfrac{1}{\chi}(\lambda'-\lambda) = -\tfrac{\pi}{2}$ and $\theta_2 > \tfrac{1}{\chi}(\lambda-\lambda') = \tfrac{\pi}{2}$.  Assume that the boundary data of $h$ is such that Proposition~\ref{prop::flow_counterflow_left_right} applies.  Then we know that $\eta_{\theta_1}$ and $\eta_{\theta_2}$ pass to the right and left, respectively, of the counterflow line $\eta'$ of $h$ starting at $z_0$.  Let $C$ be any connected component of $\strip \setminus (\eta_{\theta_1} \cup \eta_{\theta_2})$ which lies between $\eta_{\theta_1}$ and $\eta_{\theta_2}$ and let $\tau'$ be a stopping time for the filtration $\CF_t = \sigma(\eta'(s) : s \leq t,\ \eta_{\theta_1},\eta_{\theta_2})$ such that $\eta'(\tau') \in C$ almost surely.  Then $\CC_{\eta_{\theta_1} \cup \eta'([0,\tau']) \cup \eta_{\theta_2}}$ in $C$ has the boundary behavior depicted on the left side.  Let $\psi$ be the conformal map which takes the connected component $C_0(\tau')$ of $C \setminus \eta'([0,\tau'])$ which contains $x_0$ to $\strip$ where $x_0$ is sent to $0$ and the left and right sides of $C_0(\tau')$ which are contained in $\eta_{\theta_2}$ and $\eta_{\theta_1}$, respectively, are sent to $(-\infty,0)$ and $(0,\infty)$, respectively; $w_0 = \psi(\eta'(\tau'))$.  The boundary data for the GFF $h \circ \psi^{-1} - \chi \arg(\psi^{-1})'$ is depicted in the right panel.}
\end{figure}

\begin{remark}[Counterflow line between two flow lines]
\label{rem::cond_mean_height_cf}
We will now describe a version of Proposition~\ref{prop::cond_mean_height} which holds for counterflow lines (see Figure~\ref{fig::cond_mean_non_crossing}).  This result is easier to describe on the strip $\strip$. Assume $h$ is a GFF on $\strip$ whose boundary data is as in the left side of Figure~\ref{fig::cond_mean_non_crossing} where the constants $a,b,a',b'$ are chosen so that Proposition~\ref{prop::flow_counterflow_left_right} applies.  Let $\eta_{\theta_1},\eta_{\theta_2}$ be the flow lines of $h$ emanating from $0$ with angles
\[ \theta_1 < -\frac{\pi}{2} = \frac{1}{\chi}(\lambda'-\lambda) < \frac{1}{\chi}(\lambda-\lambda') = \frac{\pi}{2} < \theta_2,\] respectively, and let $\eta'$ be the counterflow line emanating from $z_0$.  For any stopping time $\tau'$ for the filtration $\CF_t = \sigma(\eta'(s) : s \leq t,\ \eta_{\theta_1},\eta_{\theta_2})$, we let $A(\tau') = \eta_{\theta_1} \cup \eta'([0,\tau']) \cup \eta_{\theta_2}$, which is local by Lemma~\ref{lem::stopping_local_set}, and $A = \eta_{\theta_1} \cup \eta' \cup \eta_{\theta_2}$.  Suppose that $C$ is any connected component of $\strip \setminus (\eta_{\theta_1} \cup \eta_{\theta_2})$ which lies between $\eta_{\theta_1}$ and $\eta_{\theta_2}$.  Then $C \setminus \eta'$ consists of three different types of connected components: those whose boundary does not intersect the outer boundary of $\eta'$, those whose boundary intersects $\eta_{\theta_1}$, and those whose boundary intersects $\eta_{\theta_2}$.  Proposition~\ref{gff::prop::local_independence} implies that $\CC_A$ in $C$ has the same boundary behavior as $\CC_{\eta'}$ in the former case.  The connected components which intersect $\eta_{\theta_1}$ are the same as the connected components of $\strip \setminus (\eta_{\theta_1} \cup \eta_{R})$ which intersect $C$ and $\eta_{\theta_1}$, where $\eta_R$ is the flow line of angle $\tfrac{1}{\chi}(\lambda'-\lambda) = - \tfrac{\pi}{2}$, since $\eta_R$ is the right boundary of $\eta'$.  Proposition~\ref{gff::prop::cond_union_mean} and Proposition~\ref{gff::prop::local_independence} thus imply $\CC_A$ agrees with $\CC_{\eta_R \cup \eta_{\theta_1}}$ in these connected components, so we have the desired boundary behavior here (Proposition~\ref{prop::cond_mean_height}).  The same is likewise true for those which intersect $\eta_{\theta_2}$.  The case where $\tau' < \infty$ follows from the $\tau' = \infty$ case by Proposition~\ref{gff::prop::cond_union_mean} using the same argument as in the proof of Proposition~\ref{prop::cond_mean_height}.
\end{remark}

\begin{figure}[h!]
\begin{center}
\includegraphics[scale=0.85]{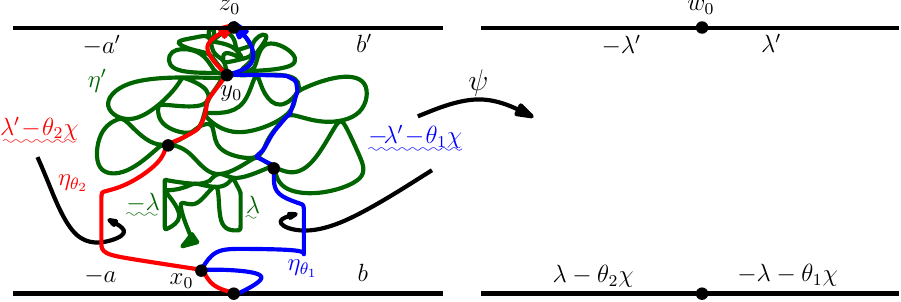}
\end{center}
\caption{\label{fig::cond_mean_contains} Assume that $h$ is a GFF on the strip $\strip$ whose boundary data is depicted in the left panel where the constants are chosen so that Proposition~\ref{prop::flow_counterflow_left_right} applies.  Let $\tfrac{1}{\chi}(\lambda'-\lambda) = -\tfrac{\pi}{2} \leq \theta_1 < \theta_2 \leq \tfrac{\pi}{2} =\tfrac{1}{\chi}(\lambda-\lambda')$.  Then $\eta_{\theta_1}$ and $\eta_{\theta_2}$ are both contained in the counterflow line $\eta'$ starting at $z_0$.  Let $C$ be any connected component of $\strip \setminus (\eta_{\theta_1} \cup \eta_{\theta_2})$ which lies between $\eta_{\theta_1}$ and $\eta_{\theta_2}$ and let $\tau'$ be a stopping time for the filtration $\CF_t = \sigma(\eta'(s) : s \leq t,\ \eta_{\theta_1},\eta_{\theta_2})$  such that $\eta'(\tau') \in C$ almost surely.  Then $\CC_{\eta_{\theta_1} \cup \eta'([0,\tau']) \cup \eta_{\theta_2}}$ in $C$ has the boundary behavior depicted on the left side.  Let $C_0(\tau')$ be the connected component of $C \setminus \eta'([0,\tau'])$ which contains $x_0$ and let $\psi \colon C_0(\tau') \to \strip$ be the conformal map which sends $x_0$ to $0$ and the left and right sides of $C_0(\tau')$ which are contained in $\eta_{\theta_2}$ and $\eta_{\theta_1}$, respectively, to $(-\infty,0)$ and $(0,\infty)$; $w_0 = \psi(\eta'(\tau'))$.  The boundary data for the GFF $h \circ \psi^{-1} - \chi \arg (\psi^{-1})'$ on $\strip$ is depicted in the right panel.  A similar result also holds when only one of the $\eta_{\theta_i}$ is contained in $\eta'$.}
\end{figure}

\begin{remark}[Counterflow line which contains flow lines.]
\label{rem::cond_mean_height_cf_contained}
Assume that we have the same setup as Remark~\ref{rem::cond_mean_height_cf}.  Let $I := [-\tfrac{\pi}{2},\tfrac{\pi}{2}]$.  If $\theta_i \in I$, then by Lemma~\ref{lem::light_cone_contains_av} we know that $\eta_{\theta_i}$ is almost surely contained in the range of $\eta'$.  Results analogous to those described in Remark~\ref{rem::cond_mean_height_cf} also hold in the case that one or both of $\theta_1 < \theta_2$ are contained in $I$.  We will describe this in a bit more detail in the case $\theta_1,\theta_2 \in I$.  Fix any connected component $C$ of $\strip \setminus (\eta_{\theta_1} \cup \eta_{\theta_2})$ which lies between $\eta_{\theta_1},\eta_{\theta_2}$.  Since $\eta'$ hits the points of $\eta_{\theta_1},\eta_{\theta_2}$ in reverse chronological order (Lemma~\ref{lem::light_cone_contains_av}), the connected components of $C \setminus \eta'$ are all completely surrounded by $\eta'$.  Thus the boundary behavior of $\CC_A$ in the connected components of $C \setminus \eta'$ agrees with that of $\CC_{\eta'}$ by Proposition~\ref{gff::prop::cond_union_mean} and Proposition~\ref{gff::prop::local_independence}.  Fix a stopping time $\tau' < \infty$ for the filtration $\CF_t = \sigma(\eta'(s) : s \leq t,\ \ \eta_{\theta_1},\eta_{\theta_2})$ and assume that $C$ is a connected component of $\strip \setminus (\eta_{\theta_1} \cup \eta_{\theta_2})$ such that $\eta'(\tau') \in C$ almost surely.  Proposition~\ref{gff::prop::local_independence} implies that $\CC_{A(\tau')}$ has the same boundary behavior as $\CC_{\eta'}$ in the connected components of $C \setminus \eta'([0,\tau'])$ which are surrounded by $\eta'([0,\tau'])$.  This leaves us to deal with the boundary behavior of $\CC_{A(\tau')}$ in the connected component $C_0(\tau')$ of $C \setminus \eta'([0,\tau'])$ which contains $x_0$, the first point in $\partial C$ traced by both $\eta_{\theta_1}$ and $\eta_{\theta_2}$.  This is depicted in Figure~\ref{fig::cond_mean_contains}.

We can rule out pathological behavior at points where $\eta'$ intersects either $\eta_{\theta_1}$ or $\eta_{\theta_2}$ in $C_0(\tau')$ as follows.  First, we assume that $\tau'$ is a rational time.  Then we can sample $C_0(\tau')$ by first picking $\eta'|_{[0,\tau']}$ then, conditional on $\eta'([0,\tau'])$, sample $\eta_{\theta_1}$ and $\eta_{\theta_2}$ up until the first time they hit $\eta'([0,\tau'])$.  It is easy to see by the continuity of the conditional mean (Proposition~\ref{prop::cond_mean_continuous}) that $\CC_{A(\tau')}$ has the desired boundary behavior where $\eta'([0,\tau'])$ intersects $\eta_{\theta_1},\eta_{\theta_2}$.  We now generalize to $\CF_t$-stopping times $\tau'$ which are not necessarily rational.  First assume that $\eta'(\tau') \notin \eta_{\theta_1} \cup \eta_{\theta_2}$.  Suppose that $r$ is any rational time.  Then on the event that neither $A(\tau') \setminus A(r)$ nor $A(r) \setminus A(\tau')$ intersects $\eta_{\theta_1} \cup \eta_{\theta_2}$, we know that the points in $\partial C_0(\tau')$ where $\eta'([0,\tau'])$ intersects $\eta_{\theta_1} \cup \eta_{\theta_2}$ are the same as those in $\partial C_0(r)$.  Proposition~\ref{gff::prop::cond_union_mean} thus implies that $\CC_{A(r)}$ has the same boundary behavior as $\CC_{A(\tau')}$ near these points.  This covers the case that $\eta'(\tau')$ is not in $\eta_{\theta_1} \cup \eta_{\theta_2}$ because the continuity of $\eta'$ implies there almost surely always exists such a rational.  If, on the other hand, $\eta'(\tau')$ is in $\eta_{\theta_1} \cup \eta_{\theta_2}$, then the desired result follows from the continuity of the conditional mean (Proposition~\ref{prop::cond_mean_continuous}) by first sampling $\eta_{\theta_1} \cup \eta_{\theta_2}$ and taking a limit of $\CC_{A(t)}$ as $t \uparrow \tau'$.
\end{remark}

\begin{figure}[h!]
\begin{center}
\includegraphics[scale=0.85]{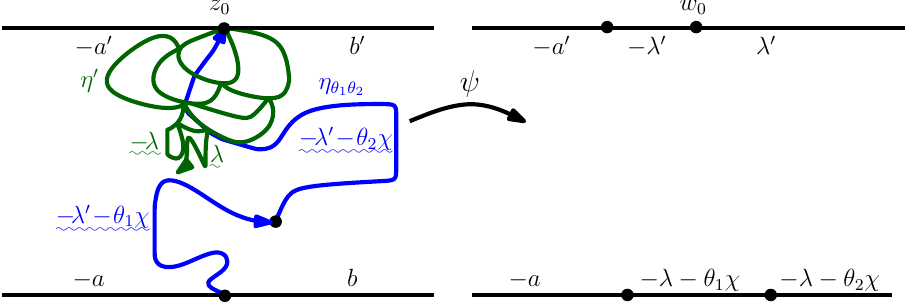}
\end{center}
\caption{\label{fig::multiple_force_points_angle_change_cf} Assume that $h$ is a GFF on the strip $\strip$ whose boundary data is depicted in the left panel.  Let $\eta_{\theta_1 \theta_2}^{\tau_1 \tau_2}$ be an angle varying flow line of $h$ with angles $\theta_1,\theta_2$ satisfying $|\theta_1-\theta_2| \leq \pi$ so that $\eta_{\theta_1 \theta_2}^{\tau_1 \tau_2}$ is simple (Lemma~\ref{lem::av_simple_determined}; we will relax this to $|\theta_1 - \theta_2| < 2\lambda/\chi$ in Section~\ref{sec::uniqueness}).  Let $\eta'$ be the counterflow line of $h$ starting at $z_0$ and assume that $\eta_{\theta_1 \theta_2}^{\tau_1 \tau_2}$ almost surely lies to the right of the left boundary of $\eta'$ (we will prove in Proposition~\ref{prop::angle_varying_monotonicity} that $\theta_1,\theta_2 < \tfrac{\pi}{2}$ is a sufficient condition for this to hold).  With $A(t) = \eta'([0,t]) \cup \eta_{\theta_1 \theta_2}^{\tau_1 \tau_2}$ and $\tau'$ any $\CF_t = \sigma(\eta'(s) : s \leq t,\ \ \eta_{\theta_1 \theta_2}^{\tau_1 \tau_2})$-stopping time, we know that $A(\tau')$ is a local set for $h$.  The boundary data for $\CC_{A(\tau')}$ is depicted on the left hand side above in the special case $\eta'$ contains part of $\eta_{\theta_1 \theta_2}^{\tau_1 \tau_2}$.  Let $C$ be the connected component of $\strip \setminus \eta_{\theta_1 \theta_2}^{\tau_1 \tau_2}$ which lies to the left of $\eta_{\theta_1 \theta_2}^{\tau_1 \tau_2}$ and let $C_0(\tau')$ be the connected component of $C \setminus \eta'([0,\tau'])$ which contains $0$.  With $\psi \colon C_0(\tau') \to \strip$ the conformal map sending $0$ to $0$, the part of $\partial C_0(\tau')$ contained in $\eta_{\theta_1 \theta_2}^{\tau_1 \tau_2}$ to $(0,\infty)$ and $(-\infty,0)$ to $(-\infty,0)$, the boundary data for the field $h \circ \psi^{-1} - \chi \arg (\psi^{-1})'$ is depicted on the right side ($w_0 = \psi(\eta'(\tau'))$).  Analogous results hold when $\eta_{\theta_1 \theta_2}^{\tau_1 \tau_2}$ does not intersect or enter the hull of $\eta'$ or we have more angles $\theta_1,\ldots,\theta_k$.}
\end{figure}

\begin{remark}[Counterflow line and an angle varying flow line]
\label{rem::cond_mean_height_cf_angle_varying}
Suppose that we have the same setup as in Remark~\ref{rem::cond_mean_height_cf} and Remark~\ref{rem::cond_mean_height_cf_contained}, except now we replace the flow lines $\eta_{\theta_1},\eta_{\theta_2}$ by a single angle varying flow line with angles $\theta_1,\theta_2$ which satisfy $|\theta_1 - \theta_2| \leq \pi$ (so that $\eta_{\theta_1 \theta_2}^{\tau_1 \tau_2}$ is simple, almost surely determined by $h$, and continuous (Lemma~\ref{lem::av_simple_determined}); we will relax this to $|\theta_1 -\theta_2| < 2\lambda/\chi$ in Section~\ref{sec::uniqueness}) and such that $\eta_{\theta_1 \theta_2}^{\tau_1 \tau_2}$ almost surely does not intersect $\partial \strip$ except at its starting point.  We further assume that $\eta_{\theta_1 \theta_2}^{\tau_1 \tau_2}$ almost surely lies to the right of the left boundary of $\eta'$ (we will prove in Proposition~\ref{prop::angle_varying_monotonicity} that $\theta_1,\theta_2 < \tfrac{\pi}{2}$ is a sufficient condition for this to hold).  Let $A(t) = \eta'([0,t]) \cup \eta_{\theta_1 \theta_2}^{\tau_1 \tau_2}$ and $\CF_t = \sigma(\eta'(s) : s \leq t,\ \ \eta_{\theta_1 \theta_2}^{\tau_1 \tau_2})$.  By Lemma~\ref{lem::stopping_local_set}, know that $A(\tau)$ is a local set for $h$ for every $\CF_t$-stopping time $\tau$.  The boundary data for $\CC_{A(\tau)}$ is described in the left panel of Figure~\ref{fig::multiple_force_points_angle_change_cf} in the special case that $\eta'$ contains part of $\eta_{\theta_1 \theta_2}^{\tau_1 \tau_2}$.  The justification of this follows from exactly the same argument as in Remark~\ref{rem::cond_mean_height_cf_contained}.  The case when $\eta_{\theta_1 \theta_2}^{\tau_1 \tau_2}$ is disjoint from $\eta'$ is analogous and follows from the argument in Remark~\ref{rem::cond_mean_height_cf}.
\end{remark}

One other case that will be especially important for us is when we have two angle varying flow lines $\eta_1,\eta_2$ which actually cross each other along with a counterflow line $\eta'$.  Since we have not yet discussed crossing of flow lines, we will defer the discussion of the case until Section~\ref{sec::uniqueness}.  (As we will explain later, flows lines of fixed angle can cross each other only when they start at distinct points $x_1 < x_2$ with respective angles $\theta_1 < \theta_2$.)

\subsection{Existence and continuity of Loewner driving functions}
\label{subsec::interacting_loewner_driving}

The purpose of this subsection is to establish the existence and continuity of the Loewner driving function of $\eta_{\theta_1}$, viewed as a path in the right connected component of $\h \setminus \eta_{\theta_2}$.  We will also describe related results which hold in the setting of multiple flow lines and counterflow lines.  We begin with the following proposition, which gives criteria which imply that a continuous curve $\eta$ in $\ol{\h}$ starting from $0$ has a continuous Loewner driving function.

\begin{proposition}
\label{prop::cont_driving_function}
Suppose that $T \in (0,\infty]$.  Let $\eta \colon [0,T) \to \ol{\h}$ be a continuous, non-crossing curve with $\eta(0) = 0$.  Assume $\eta$ satisfies the following: for every $t \in (0,T)$,
\begin{enumerate}[(a)]
\item $\eta((t,T))$ is contained in the closure of the unbounded connected component of $\h \setminus \eta((0,t))$ and
\item\label{cont_driving_connected} $\eta^{-1}(\eta([0,t]) \cup \R)$ has empty interior in $(t,T)$.
\end{enumerate}
For each $t > 0$, let $g_t$ be the conformal map which takes the unbounded connected component of $\h \setminus \eta([0,t])$ to $\h$ with $\lim_{z \to \infty} |g_t(z)-z| = 0$.  After reparameterization, $(g_t)$ solves the Loewner equation
\[ \partial_t g_t(z) = \frac{2}{g_t(z) - U_t},\quad \ g_0(z) = 0,\]
with continuous driving function $U_t$.
\end{proposition}
Roughly speaking, the first hypothesis states that $\eta$ never enters a loop consisting of either its own range or part of $\partial \h$ upon closing it.  The second condition intuitively means that $\eta$ traces neither $\partial \h$ nor itself.
\begin{proof}
The proof is essentially the same as that of \cite[Proposition~4.3]{LAW05}, though the statement of \cite[Proposition~4.3]{LAW05} contains the stronger hypothesis that $\eta$ is simple.
\end{proof}

The main step in our proof that $\eta_{\theta_1}$ admits a continuous Loewner driving function as a continuous path in the right connected component of $\h \setminus \eta_{\theta_2}$ is that the set of times $t$ that $\eta_{\theta_1}(t)$ is contained in the range of $\eta_{\theta_2}$ is nowhere dense in $[0,\infty)$.  The reason this holds is explained in Figure~\ref{fig::interval_overlap} and is proved rigorously in Lemma~\ref{lem::cont_loewner_driving_function}.

\begin{figure}[h!]
\begin{center}
\includegraphics[scale=0.85]{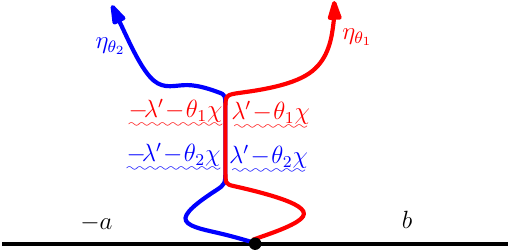}
\caption{\label{fig::interval_overlap} The set of times $I$ that $\eta_{\theta_1}(t)$ is contained in $\eta_{\theta_2}$ almost surely cannot contain an open interval.  Indeed, the contrary would imply that $(\CC_{\eta_{\theta_1} \cup \eta_{\theta_2}} - \chi \cdot {\rm winding})$ converges to both $\lambda' - \theta_1 \chi$ and $\lambda ' - \theta_2 \chi$ as $z_k$ converges to a point in an interval of intersection on the right side of $\eta_{\theta_1} \cap \eta_{\theta_2}$.  This is the key step to showing that $\eta_{\theta_1}$ has a continuous Loewner driving function viewed as a path in the right connected component of $\h \setminus \eta_{\theta_2}$.}
\end{center}
\end{figure}

\begin{lemma}
\label{lem::cont_loewner_driving_function}
Let $\psi$ be a conformal map which takes the right connected component of $\h \setminus \eta_{\theta_2}$  to $\h$ with $\psi(0) = 0$ and $\psi(\infty) = \infty$.  Then $\psi(\eta_{\theta_1})$ has a continuous Loewner driving function viewed as a path in $\h$ from $0$ to $\infty$.
\end{lemma}
\begin{proof}
Since $\eta_{\theta_2}$ is almost surely continuous (recall Remark~\ref{rem::continuity_non_boundary}), the right connected component $C$ of $\h \setminus \eta_{\theta_2}$ is almost surely a Jordan domain.  Thus $\psi$ extends as a homeomorphism $\ol{C} \to \ol{\h}$, so that $\psi(\eta_{\theta_2})$ is almost surely a continuous path in $\h$ from $0$ to $\infty$.

We will now argue that the first criterion of Proposition~\ref{prop::cont_driving_function} holds in this case.  The reason is that the only way this could fail to be true is if the following occurs.  After intersecting $\eta_{\theta_2}$, say at time $t_0$, $\eta_{\theta_1}$ enters a bounded connected component $C_0$ of $\h \setminus (\eta_{\theta_1}([0,t_0]) \cup \eta_{\theta_2})$.  Since $\eta_{\theta_1}$ lies to the right of $\eta_{\theta_2}$, this would force $\eta_{\theta_1}$ to have a self intersection upon exiting $C_0$.  This is a contradiction since $\eta_{\theta_1}$ is a simple path.

To check the second criterion of Proposition~\ref{prop::cont_driving_function}, it suffices to show that the set $I$ of times $t \in [0,\infty)$ such that $\eta_{\theta_1}(t)$ is contained in the range of $\eta_{\theta_2}$ is nowhere dense in $[0,\infty)$ almost surely.  Indeed, we note that we do not need to check that $\psi(\eta_{\theta_1})$ does not trace itself because we know that $\eta_{\theta_1}$ does not trace itself.  Since $I$ is closed, it suffices to show the event $E$ that $I$ contains an open interval has probability zero.  Suppose for sake of contradiction that $\p[E] > 0$.  Let $\CF = \sigma(\eta_{\theta_1},\eta_{\theta_2})$.  Fix an open interval $I_0 \subseteq I$ and let $T_0$ be an $\CF$-measurable random variable taking values in $[0,\infty)$ such that $\p[ T_0 \in I_0 \giv E] = 1$.  Since $\eta_{\theta_1}$ and $\eta_{\theta_2}$ are both simple paths, on $E$ we can find a sequence of points $(z_k)$ contained in the right component of $\h \setminus (\eta_{\theta_1} \cup \eta_{\theta_2})$ converging to $\eta_{\theta_1}(T_0)$.  Note that we can apply Proposition~\ref{gff::prop::cond_union_mean} to $\CC_{\eta_{\theta_1} \cup \eta_{\theta_2}}$ evaluated at $\eta_{\theta_1}(T_0)$ since $\eta_{\theta_1}(I_0)$ is connected and contains more than one point.  This leads to a contradiction since Proposition~\ref{gff::prop::cond_union_mean} thus implies $(\CC_{\eta_{\theta_1} \cup \eta_{\theta_2}}(z_k) - \chi \cdot {\rm winding})$ converges to both
\[ \lambda' - \theta_1 \chi\quad \text{and}\quad
   \lambda' - \theta_2 \chi;\]
(see Figure~\ref{fig::interval_overlap}).
\end{proof}

In the following series of remarks, we will describe results analogous to Lemma~\ref{lem::cont_loewner_driving_function} which hold for a number of different configurations of flow and counterflow lines.  In each case, the proof is roughly the same as Lemma~\ref{lem::cont_loewner_driving_function}, except for minor modifications which we are careful to point out (we will in particular not explain in each case why the relevant path does not trace itself).  Although this might seem pedantic, we felt obliged to treat each case separately for the sake of completeness.  The reader should feel free to skip the remainder of this section on a first reading.  Remark~\ref{rem::cond_mean_general_bd} also applies here: the subsequent remarks will prove the continuity of the Loewner driving function of one path given several others, restricted to the regime of boundary data in which the paths do not intersect the boundary.  Once Theorems~\ref{thm::coupling_uniqueness}-\ref{thm::monotonicity_crossing_merging} have been proven in Section~\ref{sec::uniqueness}, the arguments we present here will also work without modification in the regime of general piecewise constant boundary data.

\begin{remark}
\label{rem::cont_loewner_multiple_angles}  The result of Lemma~\ref{lem::cont_loewner_driving_function} extends to the setting of multiple flow lines.  Suppose that $\theta_1 < 0 < \theta_2$ and $\eta := \eta_0$ is the flow line with angle $0$.  Let $C$ be any connected component of $\h \setminus (\eta_{\theta_1} \cup \eta_{\theta_2})$ which lies between $\eta_{\theta_1}$ and $\eta_{\theta_2}$ and let $x_0,y_0$ be the first and last points on $\partial C$ traced by $\eta_{\theta_1}$.  Let $\psi \colon C \to \h$ a conformal map with $\psi(x_0) = 0$ and $\psi(y_0) = \infty$.  Then $\psi(\eta)$ has a continuous Loewner driving function as a curve in $\ol{\h}$.  The justification that the first criterion of Proposition~\ref{prop::cont_driving_function} holds is exactly the same as in the setting of two flow lines.  As before, we also know that $\psi$ extends as a homeomorphism $\ol{C} \to \ol{\h}$ since $C$ is a Jordan domain by the continuity of $\eta_{\theta_1}$ and $\eta_{\theta_2}$.  The proof of Lemma~\ref{lem::cont_loewner_driving_function} implies that the set of times $t$ that $\eta(t)$ is contained in the range of either $\eta_{\theta_1}$ or $\eta_{\theta_2}$ is nowhere dense in $[0,\infty)$.  Therefore the second criterion of Proposition~\ref{prop::cont_driving_function} also holds.
\end{remark}

\begin{remark}
\label{rem::cont_loewner_av} 

A version of Lemma~\ref{lem::cont_loewner_driving_function} also holds in the setting of angle varying flow lines.  In particular, we suppose that $\theta_1,\ldots,\theta_k \in \R$ and that $\eta_{\theta_1 \cdots \theta_k}^{\tau_1 \cdots \tau_k}$ is an angle varying flow line with these angles starting at $0$.  We assume $\theta_1,\ldots,\theta_k$ are chosen so that $\eta_{\theta_1 \cdots \theta_k}^{\tau_1 \cdots \tau_k}$ almost surely stays to the right of $\eta$, the zero angle flow line of $h$ (we will prove in Proposition~\ref{prop::angle_varying_monotonicity} that $\theta_1,\ldots, \theta_k < 0$ is a sufficient condition for this to hold).  We moreover assume that $|\theta_i - \theta_j| \leq \pi$ for all pairs $1 \leq i,j \leq k$ so that by Lemma~\ref{lem::av_simple_determined} we know that $\eta_{\theta_1 \cdots \theta_k}^{\tau_1 \cdots \tau_k}$ is almost surely continuous and determined by $h$ (we will relax this to $|\theta_i-\theta_j| < 2\lambda/\chi$ in Section~\ref{sec::uniqueness}).  Let $C$ be the left connected component of $\h \setminus \eta_{\theta_1 \cdots \theta_k}^{\tau_1 \cdots \tau_k}$ and let $\psi \colon C \to \h$ be a conformal map which preserves $0$ and $\infty$.  Since $C$ is a Jordan domain, $\psi$ extends as a homeomorphism to $\partial C$.  Then $\psi(\eta)$ has a continuous Loewner driving function as a path from $0$ to $\infty$.  The proof is the same as Lemma~\ref{lem::cont_loewner_driving_function}.
\end{remark}

\begin{figure}[h!]
\begin{center}
\includegraphics[scale=0.85]{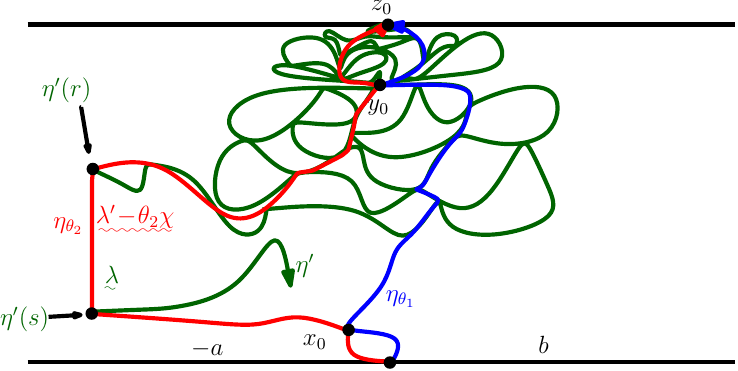}
\end{center}
\caption{\label{fig::cf_cont_driving_function}Suppose that $h$ is a GFF on the strip $\strip$, $\eta_{\theta_1}, \eta_{\theta_2}$ are flow lines with angles $\tfrac{1}{\chi}(\lambda'-\lambda)  = - \tfrac{\pi}{2} < \theta_1 < \theta_2 < \tfrac{1}{\chi}(\lambda-\lambda') = \tfrac{\pi}{2}$ so that $\eta_{\theta_1},\eta_{\theta_2}$ are almost surely contained in the range of the counterflow line $\eta'$.  Fix any bounded connected component $C$ of $\strip \setminus (\eta_{\theta_1} \cup \eta_{\theta_2})$ and let $x_0$ be the first point on $\partial C$ traced by $\eta_{\theta_1}$ and let $y_0$ be the last.  The boundary data for $\CC_{\eta_{\theta_1} \cup \eta' \cup \eta_{\theta_2}}$ is depicted above.  For every pair of rationals $r < s$, we almost surely have that $\eta'((r,s))$ does not contain a non-trivial interval of and stay to the right of $\eta_{\theta_1}$ or a non-trivial interval and stay to the left of $\eta_{\theta_2}$.  Indeed, this would lead to the contradiction that $\CC_{\eta_{\theta_1} \cup \eta' \cup \eta_{\theta_2}}$ takes on two different values on $\eta'((r,s)) \cap \partial C$.  This is the analog of Figure~\ref{fig::interval_overlap} for boundary filling counterflow lines and is the key observation for showing that $\eta'$ viewed as a continuous path in $C$ admits a continuous Loewner driving function.}
\end{figure}

\begin{remark}
\label{rem::cont_loewner_cf}
We will now describe an analog of Lemma~\ref{lem::cont_loewner_driving_function} which holds in the setting of counterflow lines.  In order to state this result, we will for convenience work with the GFF $h$ on the strip $\strip$ rather than $\h$.  We assume the boundary data for $h$ is as in Figure~\ref{fig::cf_cont_driving_function}, let $\eta'$ be the counterflow line starting at $z_0$, and $\eta_{\theta_1},\eta_{\theta_2}$ be the flow lines of $h$ starting at $0$ with angles $\theta_1,\theta_2$.  We assume that $a,b,a',b'$ are sufficiently large so that Proposition~\ref{prop::flow_counterflow_left_right} applies to $\eta_{\theta_1},\eta_{\theta_2}$, and $\eta'$.  We first consider the case $\theta_1 < \tfrac{1}{\chi}(\lambda'-\lambda) = -\tfrac{\pi}{2}$ and $\theta_2 > \tfrac{1}{\chi}(\lambda-\lambda') = \tfrac{\pi}{2}$ so that by Proposition~\ref{prop::flow_counterflow_left_right} we have that $\eta_{\theta_1}$ passes to the right of $\eta'$ and $\eta_{\theta_2}$ passes to its left.  Let $C$ be any connected component of $\strip \setminus (\eta_{\theta_1} \cup \eta_{\theta_2})$ which lies between $\eta_{\theta_1}$ and $\eta_{\theta_2}$ and $x_0$ the first point on $\partial C$ traced by $\eta_{\theta_1}$ and $y_0$ the last.  Let $\psi \colon C \to \h$ be a conformal transformation with $\psi(x_0) = \infty$ and $\psi(y_0) = 0$.  Since $\eta_{\theta_1}$ and $\eta_{\theta_2}$ are continuous, $C$ is a Jordan domain so that $\psi$ extends as a homeomorphism $\ol{C} \to \ol{\h}$.  As in Remark~\ref{rem::cont_loewner_multiple_angles}, the set $I$ of times $t$ that $\eta'(t)$ is contained in $\eta_{\theta_1} \cup \eta_{\theta_2}$ is nowhere dense in $[0,\infty$).  The reason for this is that the left (resp.\ right) boundary of $\eta'$ is $\eta_L$ (resp.\ $\eta_R$), the flow line of $h$ with angle $\tfrac{1}{\chi}(\lambda-\lambda') = \tfrac{\pi}{2}$ (resp.\ $\tfrac{1}{\chi}(\lambda'-\lambda) = -\tfrac{\pi}{2}$) (recall Proposition~\ref{prop::light_cone_construction}) and we know from the proof of Lemma~\ref{lem::cont_loewner_driving_function} that $\eta_L \cap \eta_{\theta_i}$ (resp.\ $\eta_{R} \cap \eta_{\theta_i}$) is nowhere dense in $\eta_{\theta_i}$ for $i=1,2$.  Therefore the second criterion of Proposition~\ref{prop::cont_driving_function} holds.  In order to see that the first criterion holds, suppose that $\eta'(t_0)$ is contained in the range of $\eta_{\theta_1}$.  Then $\eta'(t_0)$ is also in the right outer boundary $\eta_R$ of $\eta'$.  This implies that it is impossible for $\eta'$ to turn into a connected component of $C \setminus \eta'([0,t_0])$ which does not contain $x_0$ because it would contradict Lemma~\ref{lem::light_cone_contains_av} (see also Lemma~\ref{lem::hit_in_order}), that $\eta'$ hits points in $\eta_R$ in reverse chronological order.  A symmetrical argument applies in the case $\eta'(t_0)$ is contained in the range of $\eta_{\theta_2}$.  Therefore $\psi(\eta')$ has a continuous Loewner driving function as a path in $\h$.
\end{remark}

\begin{remark}
\label{rem::cont_loewner_cf_contains}
This is a continuation of the previous remark, in which we now consider the case that both $\theta_1 < \theta_2$ lie in the interval $I = [-\tfrac{\pi}{2}, \tfrac{\pi}{2}]$ (the case that only one of the $\theta_i$ are in $I$ follows from an analogous argument).  Let $C, x_0,y_0,\psi$ be as in the previous remark.  We will first argue that the part of $\eta'$ which traces through $C$ is a continuous path in $C$.  To see this, let $\tau = \inf\{t \geq 0 : \eta'(t) = y_0\}$ and $\sigma = \inf\{t \geq 0 : \eta'(t) = x_0\}$.  Lemma~\ref{lem::light_cone_contains_av}  (see also Lemma~\ref{lem::hit_in_order}) implies that $\eta'$ hits the points in $\eta_{\theta_1}$ in reverse chronological order.  This implies $\sigma \geq \tau$ almost surely.  Let $\eta_{C}'$ be the path with $\eta_{C}'|_{[0,\tau]} = y_0$, $\eta_{C}'|_{[\sigma,\infty)} = x_0$.  We will now describe $\eta_{C}'|_{[\tau,\sigma]}$.  Let $D = \strip \setminus \ol{C}$.  Since $D$ is open and $\eta'$ is continuous, $J = (\eta')^{-1}(D) \subseteq (0,\infty)$ is open.  We can write $J = \cup_k J_k$ where the $J_k = (a_k,b_k)$ are pairwise disjoint open intervals in $(0,\infty)$.  Suppose that $J_k \subseteq (\tau,\sigma)$.  Since $\eta'$ hits the points of $\eta_{\theta_1},\eta_{\theta_2}$ in reverse chronological order (Lemma~\ref{lem::light_cone_contains_av} and Lemma~\ref{lem::hit_in_order}), it must be that $x_k := \eta'(a_k) = \eta'(b_k)$.  We set $\eta_{C}'|_{[a_k,b_k]} \equiv x_k$ and $\eta_{C}'|_{[\tau,\sigma] \setminus J} \equiv \eta'|_{[\tau,\sigma] \setminus J}$.  Then $\eta_{C}'$ is clearly a continuous path in $\ol{C}$ which agrees with $\eta'$ at times when it is in $\ol{C}$.

We will now argue that $\psi(\eta_{C}')$ has a continuous Loewner driving function by checking the criterion of Proposition~\ref{prop::cont_driving_function}.  In order to check the first part of the proposition, it suffices to show that $\eta_{C}'((t,\infty))$ is contained in the closure of the connected component of $C \setminus \eta_{C}'((0,t))$ which contains $x_0$.  This is true because $\eta_C'((t,\infty)) = \eta'((t,\infty)) \cap C$.  Since $\eta'$ cannot cross itself and hits the points of $\eta_{\theta_1}$ and $\eta_{\theta_2}$ in reverse chronological order, it is obvious that $\eta'((t,\infty)) \cap C$ has this property.  We now turn to check the second hypothesis of the proposition.  Suppose that $0 < r < q$ are rational.  If $\eta_{C}'((r,q))$ contains a non-trivial interval of and is contained in $\partial C$, then $\eta'((r,q))$ also contains a non-trivial interval of $\partial C$ and is contained in $\ol{D}$.  This leads to a contradiction as described in Figure~\ref{fig::cf_cont_driving_function}.  Thus, almost surely, $\eta_C'$ does not contain a non-trivial interval of $\eta_{\theta_1}$ or $\eta_{\theta_2}$ in any rational time interval.  This completes the proof that $\eta'$ satisfies the second criterion of Proposition~\ref{prop::cont_driving_function}.
\end{remark}

\begin{remark}
\label{rem::cont_loewner_cf_av}
This is a continuation of the previous remark.  Let $\eta_{\theta_1 \cdots \theta_k}^{\tau_1 \cdots \tau_k}$ be an angle varying flow line of $h$ with angles $\theta_1,\ldots,\theta_k$.  We assume that the boundary data of $h$ is such that both $\eta'$ and $\eta_{\theta_1 \cdots \theta_k}^{\tau_1 \cdots \tau_k}$ almost surely intersect $\partial \strip$ only at $0$ and $z_0$.  Assume that $|\theta_i - \theta_j| \leq \pi$ for all $i,j$ so that $\eta_{\theta_1 \cdots \theta_k}^{\tau_1 \cdots \tau_k}$ is simple, continuous, and almost surely determined by $h$ by Lemma~\ref{lem::av_simple_determined}.  (This can be relaxed to $|\theta_i - \theta_j| < 2\lambda/\chi$ upon proving that such angle varying paths are almost surely continuous and determined by $h$.  This will be accomplished in Section~\ref{sec::uniqueness}).  Moreover, assume that $\eta_{\theta_1 \cdots \theta_k}^{\tau_1 \cdots \tau_k}$ stays to the right of the left boundary of $\eta'$ (we will prove in Proposition~\ref{prop::angle_varying_monotonicity} that $\theta_1, \ldots, \theta_k < \tfrac{\pi}{2}$ is a sufficient criterion for this).    Then $\eta'$, viewed as a path in the left connected component of $\strip \setminus \eta_{\theta_1 \cdots \theta_k}^{\tau_1 \cdots \tau_k}$ almost surely has a continuous Loewner driving function.  The proof of this is the same as that given in Remark~\ref{rem::cont_loewner_cf} and Remark~\ref{rem::cont_loewner_cf_contains}.
\end{remark}

There is one more configuration of paths that will be important for us: two angle varying flow lines $\eta_1,\eta_2$ and a counterflow line $\eta'$ where $\eta_1$ and $\eta_2$ actually cross each other.  We defer this case until after we study the crossing phenomenon of flow lines in Section~\ref{subsec::monotonicity_merging_crossing}.

 \section{Proofs of main theorems}
\label{sec::uniqueness}

In this section, we will complete the proof of Theorems~\ref{thm::coupling_uniqueness}--\ref{thm::monotonicity_crossing_merging}.  We will start in Section~\ref{subsec::two_boundary_force_points} by proving Theorem~\ref{thm::coupling_uniqueness} and Theorem~\ref{thm::continuity} for $\kappa \in (0,4]$ in the special case of two force points $x^L = 0^-$ and $x^R = 0^+$ with weights $\rho^L,\rho^R > -2$, respectively.  Then by an induction argument, we will deduce Theorem~\ref{thm::coupling_uniqueness} for $\kappa \in (0,4]$ in complete generality from the two force point case.  The proof of these results will also imply that the monotonicity result for flow lines established in Section~\ref{sec::non_boundary_intersecting} holds in the regime of boundary data which is constant on $(-\infty,0)$ and on $(0,\infty)$.  Next, in Section~\ref{subsec::monotonicity_merging_crossing}, we will extend the monotonicity result further to cover the case of flow lines of GFFs with general piecewise constant boundary data and then, from this, we will extract the monotonicity of angle varying flow lines.  This is one of the key tools that we will use in Section~\ref{subsec::many_boundary_force_points} to prove Theorem~\ref{thm::continuity} for $\kappa \in (0,4]$ in the setting of multiple force points, at least up until just before the continuation threshold is hit.  We also prove Theorem~\ref{thm::monotonicity_crossing_merging} in Section~\ref{subsec::monotonicity_merging_crossing}, that flow lines with the same angle almost surely merge upon intersecting and never separate and that flow lines with different angles may cross upon intersecting (depending on their relative angle and their starting points), after which they may bounce off of each other but never cross again.  The latter will then allow us to prove Theorem~\ref{thm::continuity} for $\kappa \in (0,4]$, even up to and including when the continuation threshold is hit.  We continue in Section~\ref{subsec::counterflow} by explaining the modifications necessary to prove Theorem~\ref{thm::coupling_uniqueness} and Theorem~\ref{thm::continuity} for $\kappa' > 4$.  We will also extend the light cone construction of Section~\ref{sec::non_boundary_intersecting} to the setting in which the counterflow line can intersect the boundary.  Finally, in Section~\ref{subsec::fan}, we will combine all of the machinery we have developed in this article to show that the fan $\fan$ --- the set of points accessible by flow lines of different (but constant) angles starting from an initial boundary point $x$ (recall Figures~\ref{fig::flowlines}--\ref{fig::flowlines4}, and~\ref{fig::sle64_fan}) --- almost surely has zero Lebesgue measure for $\kappa \in (0,4)$. (We remark that this follows for $\kappa \in (2,4)$ since, as suggested by the discussion at the end of Section~\ref{sec::non_boundary_intersecting}, $\fan$ is contained in a counterflow line which is an $\SLE_{\kappa'}$ type curve with $\kappa'=16/\kappa \in (4,8)$.  New arguments will be needed for $\kappa \in (0,2]$.)

\subsection{Two boundary force points}
\label{subsec::two_boundary_force_points}

\begin{figure}[h!]
\begin{center}
\includegraphics[scale=0.85]{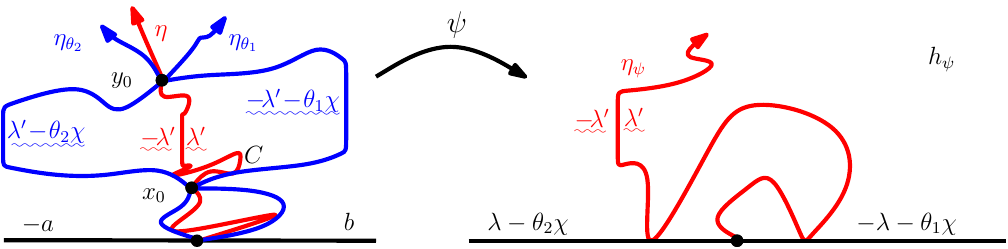}
\caption{\label{fig::flow_line_angles} Suppose that $h$ is a GFF on $\h$ with boundary data as shown in the illustration on the left hand side.  Suppose $\theta_1 < 0 < \theta_2$ and let $\eta_{\theta_i}$ be the flow line of $h$ starting from $0$ with angle $\theta_i$, i.e.\ the flow line of $h + \theta_i \chi$.  Assume $a,b$ are chosen sufficiently large so that Proposition~\ref{prop::monotonicity_non_boundary} applies to $\eta_{\theta_1},\eta_{\theta_2}$, and $\eta$, the zero angle flow line of $h$.  Then $\eta_{\theta_1}$ almost surely lies to the right of $\eta$ which in turn almost surely lies to the right of $\eta_{\theta_2}$.  We will prove in Lemma~\ref{lem::two_force_points_cond_law} and Lemma~\ref{lem::two_force_points_determined} that, conditionally on $\eta_{\theta_1},\eta_{\theta_2}$, the law of $\eta$ in every connected component $C$ of $\h \setminus (\eta_{\theta_1} \cup \eta_{\theta_2})$ which lies between $\eta_{\theta_1}$ and $\eta_{\theta_2}$ is independently that of an $\SLE_\kappa(\rho^L;\rho^R)$ process with $\rho^L = \theta_2 \tfrac{\chi}{\lambda} - 2$ and $\rho^R = -\theta_1 \tfrac{\chi}{\lambda} - 2$ and, moreover, is almost surely determined by $h|_C$.}
\end{center}
\end{figure}

\noindent{\bf Setup.}  Fix $a,b > 0$ and let $h$ be a GFF on $\h$ with boundary data as in the left side of Figure~\ref{fig::flow_line_angles}.  Fix $\theta_1 < 0 < \theta_2$ and let $\eta_{\theta_i}$ be the flow line of $h$ with angle $\theta_i$, $i=1,2$.  That is, $\eta_{\theta_i}$ is the flow line of $h+\theta_i \chi$, for $i=1,2$.  Let $\eta$ be the (zero angle) flow line of $h$.  Assume that $a,b$ are chosen sufficiently large so that Proposition~\ref{prop::monotonicity_non_boundary} applies to $\eta_{\theta_1}$, $\eta$, and $\eta_{\theta_2}$.  Thus we know that  $\eta_{\theta_1}$ lies to the right of $\eta$ which in turn lies to the right of $\eta_{\theta_2}$.  By Theorem~\ref{thm::coupling_uniqueness}, in a certain non-boundary-intersecting regime, which we proved in Section~\ref{subsec::uniqueness_non_boundary_intersecting}, we know that $\eta_{\theta_1}$, $\eta$, and $\eta_{\theta_2}$ are all almost surely determined by $h$.  Fix a connected component $C$ of $\h\setminus (\eta_{\theta_1} \cup \eta_{\theta_2})$ which lies between $\eta_{\theta_1}$ and $\eta_{\theta_2}$.  (The particular way that we select $C$ will ultimately be unimportant since what we will argue holds for all such $C$ simultaneously.  One example of an explicit rule for selecting $C$ would be to fix a positive integer $k$ and a countable, dense sequence $(r_n)$ of points in $\C$ and consider the subsequence containing those $r_i$ that lie between $\eta_{\theta_1}$ and $\eta_{\theta_2}$; we may then let $C$ be the component containing the $k$th element in the subsequence.)  Let $x_0$ be the first point in $\partial C$ traced by $\eta_{\theta_1}$ and $y_0$ the last.  Let $h_C$ be the restriction of $h$ to $C$ and $\eta_C$ the restriction of $\eta$ to the time interval in which it takes values in $\ol{C}$.  Let $\psi \colon C \to \h$ be a conformal transformation which sends $x_0$ to $0$ and $y_0$ to $\infty$ (the scale factor can be determined by an arbitrary rule --- e.g., by requiring the $k$th element in the subsequence discussed above to map to a point on the unit circle) and let $h_\psi = h \circ \psi^{-1} - \chi \arg(\psi^{-1})'$ be the GFF on $\h$ given by the coordinate change~\eqref{eqn::ac_eq_rel} of $h_C$ under $\psi$ and let $\eta_\psi$ be the flow line of $h_\psi$ starting from $0$ and targeted at $\infty$.

To complete the proof of Theorem~\ref{thm::coupling_uniqueness} for $\kappa \in (0,4]$ with two force points $\rho^L,\rho^R > -2$, we will show the following:
\begin{enumerate}
\item $\eta_\psi$ is almost surely determined by $h_\psi$.
\item $\eta_\psi \sim \SLE_\kappa(\rho^L;\rho^R)$, and by adjusting $\theta_1,\theta_2$ we can obtain any pair of weights $\rho^L, \rho^R > -2$.
\end{enumerate}
Theorem~\ref{thm::continuity} for $\kappa \in (0,4]$ with two force points $\rho^L,\rho^R > -2$ then follows by showing that $\eta_\psi$ is almost surely continuous.  We will accomplish these two steps in the following lemmas.

\begin{lemma}
\label{lem::two_force_points_cond_law}
Conditional on $\eta_{\theta_1}$, $\eta_{\theta_2}$ and $\eta$ up until the first time that it hits $\partial C$, we have that $\eta_\psi \sim \SLE_\kappa(\rho^L;\rho^R)$ where the weights $\rho^R, \rho^L$ are given by
\[ \rho^R= -\frac{ \theta_1 \chi}{\lambda} - 2 \quad\text{and}\quad \rho^L = \frac{\theta_2 \chi}{\lambda} - 2\]
and correspond to force points at $0^+$ and $0^-$, respectively.  Moreover, $\eta_\psi$ is almost surely continuous with $\lim_{t \to \infty} \eta_\psi(t) = \infty$ and $(\eta_\psi,h_\psi)$ are coupled as in Theorem~\ref{thm::coupling_existence}.
\end{lemma}
Note that Lemma~\ref{lem::two_force_points_cond_law} implies that $(\eta_\psi,h_\psi)$ is independent of $(\eta_{\theta_1},\eta_{\theta_2})$.  This fact will be important for us in the proof of Lemma~\ref{lem::two_force_points_determined}.
\begin{proof}[Proof of Lemma~\ref{lem::two_force_points_cond_law}]
By Remark~\ref{rem::cont_loewner_multiple_angles}, the uniformizing conformal maps $(g_t)$ of the unbounded connected component of $\h \setminus \eta_\psi([0,t])$ with $\lim_{z \to \infty}|g_t(z) - z| = 0$ satisfy the Loewner equation with continuous driving function $W_t$.  For a local set $A$ of $h_\psi$, let $\CC_A^\psi$ be as in Section~\ref{subsec::local_sets}.  That $\eta,\eta_{\theta_1},\eta_{\theta_2}$ are almost surely determined by $h$ and are local sets for $h$ combined with Lemma~\ref{lem::stopping_local_set} implies that $\eta_\psi([0,\tau])$ is a local set for $h_\psi$ for every $\eta$ stopping time $\tau$.  Moreover, Remark~\ref{rem::cond_mean_height} implies that $\CC_{\eta_\psi([0,\tau])}^\psi$ is the harmonic function in $\h \setminus \eta_\psi([0,\tau])$ whose boundary values are described in the right hand side of Figure~\ref{fig::flow_line_angles}.  Proposition~\ref{prop::cond_mean_continuous} implies that $\CC_{\eta_\psi([0,t])}^\psi(z)$ has a modification which is continuous in $t$ and $z$.  Consequently, Theorem~\ref{thm::martingale} implies $\eta_\psi \sim \SLE_\kappa(\rho^L;\rho^R)$ in $\h$ from $0$ to $\infty$ where the values of $\rho^L,\rho^R$ are as in the statement of the lemma (recall Figure~\ref{fig::conditional_boundary_data}).  The continuity of $\eta_\psi$ follows since $\eta$ is almost surely continuous (recall Remark~\ref{rem::continuity_non_boundary} and Proposition~\ref{prop::transience}, which gives the continuity of $\eta_\psi(t)$ as $t \to \infty$) and, as explained in Remark~\ref{rem::cont_loewner_multiple_angles}, $\psi$ extends as a homeomorphism from $\ol{C}$ to $\ol{\h}$.
\end{proof}

We now turn to show that $\eta_\psi$ is almost surely determined by $h_\psi$.

\begin{lemma}
\label{lem::two_force_points_determined}
Almost surely, $h_\psi$ determines $\eta_\psi$.
\end{lemma}
\begin{proof}
We remark again that $\eta$ lies to the left of $\eta_{\theta_1}$ and to the right of $\eta_{\theta_2}$ by Proposition~\ref{prop::monotonicity_non_boundary}.  Moreover, $\eta_{\theta_1}$, $\eta$, and $\eta_{\theta_2}$ are almost surely determined by $h$ by the non-boundary-intersecting version of Theorem~\ref{thm::coupling_uniqueness}, which was proved in Section~\ref{subsec::uniqueness_non_boundary_intersecting}.   Let $h'$ be the restriction of $h$ to $\h \setminus \ol{C}$ and $Q = (\eta_{\theta_1},\eta_{\theta_2},h')$.  Since $Q$ determines $C$ hence $\psi$, it follows from Lemma~\ref{lem::local_restriction_determine} and \cite[Theorem~8.1]{RS05} that the pair $(Q,h_\psi)$ determines the entire GFF $h$, hence also $\eta_\psi$.  Thus to show that $\eta_\psi$ is determined by $h_\psi$, it suffices to show that the pair $(\eta_\psi,h_\psi)$ is independent of $Q$ since $h_\psi$ is independent of $Q$.  It in turn suffices to show that $h'$ is independent of $(\eta_\psi,h_\psi)$ given $\eta_{\theta_1},\eta_{\theta_2}$.  The reason is that the previous lemma implies $(\eta_\psi,h_\psi)$ is independent of $(\eta_{\theta_1},\eta_{\theta_2})$ since its law conditionally on $(\eta_{\theta_1},\eta_{\theta_2})$ does not depend on $\eta_{\theta_1},\eta_{\theta_2}$.

Let $h_1'$ and $h_2'$ be the restrictions of $h'$ to the right and left sides of $\eta_{\theta_1}$ and $\eta_{\theta_2}$, respectively.  Let $U$ be the set of points in $\h \setminus (\eta_{\theta_1} \cup \eta_{\theta_2})$ which lie between~$\eta_{\theta_1}$ and~$\eta_{\theta_2}$.  We can put an ordering on the set of connected components $\CU$ of $U$ by saying that $A < B$ for $A,B \in \CU$ if and only if $\eta$ intersects $A$ before $B$.  Let $h_3',h_4'$ be the restrictions of $h$ to those components which come strictly before and after $C$ in this ordering, respectively.  By Lemma~\ref{lem::local_restriction_determine} and \cite[Theorem~8.1]{RS05}, we have that $h'$ is determined by $(h_1',\ldots,h_4')$.  Proposition~\ref{gff::prop::local_independence} (with $A_1$ given by $\eta_{\theta_1} \cup \eta_{\theta_2}$, $A_2 = \eta$, and $C$ given by the components of $\h \setminus (\eta_{\theta_1} \cup \eta_{\theta_2})$ which are to the right of $\eta_{\theta_1}$ and to the left of $\eta_{\theta_2}$) implies that the pair $(h_1',h_2')$ is independent of $(\eta_\psi, h_\psi)$ given $(h_3',h_4',\eta_{\theta_1}, \eta_{\theta_2})$.  Another application of Proposition~\ref{gff::prop::local_independence} (with $A_1$ given by $\eta_{\theta_1} \cup \eta_{\theta_2}$, $A_2$ given by $\eta$ stopped upon the last time it hits $\partial C$, and $C$ given by the components of $\h \setminus (\eta_{\theta_1} \cup \eta_{\theta_2})$ which come after $C$) implies that $h_4'$ is independent of $(\eta_\psi,h_\psi)$ given $(\eta_{\theta_1},\eta_{\theta_2},h_3')$.  Finally, Proposition~\ref{gff::prop::local_independence} (with $A_1 = \eta_{\theta_1} \cup \eta_{\theta_2}$, $A_2$ given by the union $\eta_{\theta_1}$, $\eta_{\theta_2}$, and the restriction of $\eta$ to the interval of time in which it is in $\ol{C}$, and $C$ given by those components of $\h \setminus (\eta_{\theta_1} \cup \eta_{\theta_2})$ which come before $C$; we note that $A_2$ is local for $h$ by Lemma~\ref{lem::two_force_points_cond_law} because it implies that the conditional law of $\eta$ in $\ol{C}$ does not depend on its realization up until first hitting $\partial C$ when $\eta_{\theta_1},\eta_{\theta_2}$ are fixed) $(\eta_\psi,h_\psi)$ is independent of $h_3'$ given $(\eta_{\theta_1},\eta_{\theta_2})$.  This completes the proof.
\end{proof}

The important ingredients in the proof of Lemma~\ref{lem::two_force_points_determined} are that:
\begin{enumerate}
\item the conditional law of $h$ given $\eta_{\theta_1}$, $\eta_{\theta_2}$, and $\eta$ restricted to the left and right connected components of $\h \setminus (\eta_{\theta_1} \cup \eta_{\theta_2})$ does not depend either on $\eta$ or on the restriction of $h$ to the connected components which lie between $\eta_{\theta_1}$ and $\eta_{\theta_2}$,
\item the conditional law of $\eta$ in $C$ given $\eta_{\theta_1}$, $\eta_{\theta_2}$ does not depend on $h$ restricted to the connected components whose boundaries are traced by $\eta_{\theta_1}$ and $\eta_{\theta_2}$ before $C$,
\item the conditional law of $h$ given $\eta_{\theta_1}$,$\eta_{\theta_2}$ restricted to the connected components whose boundaries are traced by $\eta_{\theta_1}$ and $\eta_{\theta_2}$ after $C$ does not depend on $\eta$ stopped upon exiting $C$, and that
\item $\eta$, $\eta_{\theta_1}$, and $\eta_{\theta_2}$ are all determined by $h$.
\end{enumerate}
By the same argument, an analogous result holds in the setting of counterflow lines (see Section~\ref{subsec::counterflow}).

By combining Lemma~\ref{lem::two_force_points_cond_law} and Lemma~\ref{lem::two_force_points_determined}, we have obtained Theorem~\ref{thm::coupling_uniqueness} and Theorem~\ref{thm::continuity} in the special case of two boundary force points, one to the left and one to the right of the $\SLE$ seed.  We will record this result in the following proposition.
\begin{proposition}
\label{prop::two_force_point_uniqueness_and_continuity}  Suppose that $\eta$ is an $\SLE_\kappa(\rho^L;\rho^R)$ process in $\h$ from $0$ to $\infty$ with $\kappa \in (0,4]$ and with weights $\rho^L,\rho^R > -2$ located at the force points $0^-,0^+$, respectively.  Then $\eta$ is almost surely continuous and $\lim_{t \to \infty} \eta(t) = \infty$.  Moreover, in the coupling of $\eta$ with a GFF $h$ as in Theorem~\ref{thm::coupling_existence}, $\eta$ is almost surely determined by~$h$.
\end{proposition}

Proposition~\ref{prop::two_force_point_uniqueness_and_continuity} implies that the flow lines of the GFF on $\h$ starting at $0$ with boundary data which is constant on $(-\infty,0)$ and on $(0,\infty)$ are almost surely defined as path valued functionals of the underlying GFF.  The technique we used to prove Proposition~\ref{prop::two_force_point_uniqueness_and_continuity} can be applied to multiple flow lines simultaneously.  We obtain as a consequence the following extension of Proposition~\ref{prop::monotonicity_non_boundary} to the regime of boundary data which is constant on $(-\infty,0)$ and constant on $(0,\infty)$.

\begin{proposition}
\label{prop::monotonicity_boundary_intersecting}
Suppose that $h$ is a GFF on $\h$ with boundary data as in the left side of Figure~\ref{fig::flow_line_angles} (though we do not restrict the values of $a$ and $b$).  Assume $\theta_1 < \theta_2$ satisfy
\[  \theta_1 > -\frac{\lambda+b}{\chi} \quad\text{and}\quad \theta_2 < \frac{a+\lambda}{\chi}.\]
With $\eta_{\theta_i}$ the flow line of $h$ with angle $\theta_i$ for $i=1,2$, we have that $\eta_{\theta_2}$ almost surely lies to the left of $\eta_{\theta_1}$.  The conditional law of $\eta_{\theta_1}$ given $\eta_{\theta_2}$ is that of an $\SLE_{\kappa}((\theta_2-\theta_1) \chi /\lambda -2;(b+\theta_1 \chi)/\lambda-1)$ independently in each of the connected components of $\h \setminus \eta_{\theta_2}$ which lie to the right of $\eta_{\theta_2}$.  Similarly, the conditional law of $\eta_{\theta_2}$ given $\eta_{\theta_1}$ is that of an $\SLE_\kappa( (a-\theta_2 \chi)/\lambda -1;(\theta_2-\theta_1)\chi/\lambda-2)$ independently in each of the connected components of $\h \setminus \eta_{\theta_1}$ which lie to the left of $\eta_{\theta_1}$.
\end{proposition}
\noindent The hypothesis on $\theta_1,\theta_2$ is to ensure that the values of the weights of the force points associated with $\eta_{\theta_1},\eta_{\theta_2}$ exceed $-2$.

We are now able to complete the proof of Theorem~\ref{thm::coupling_uniqueness} for $\kappa \in (0,4]$.  This follows from an induction argument, the absolute continuity properties of the field (Proposition~\ref{prop::gff_abs_continuity}), and the two force point case (Lemma~\ref{lem::two_force_points_determined} and Proposition~\ref{prop::two_force_point_uniqueness_and_continuity}), and is accomplished in the following lemma.

\begin{lemma}
\label{lem::functional_many_force_points}
In the setting of Theorem~\ref{thm::coupling_existence} for $\kappa \in (0,4]$, the $\SLE_\kappa(\ul{\rho}^L;\ul{\rho}^R)$ flow line $\eta$ of $h$ is almost surely determined by $h$.
\end{lemma}
\begin{proof}
Write $\ul{\rho}^L = (\rho^{1,L},\ldots,\rho^{k,L})$ and $\ul{\rho}^R = (\rho^{1,R},\ldots,\rho^{\ell,R})$.  We are going to prove the result by induction on $k,\ell$.  For simplicity of notation, we are going to assume without loss of generality that $x^{1,L} = 0^-$ and $x^{1,R} = 0^+$ by possibly adding $0$ weight force points.  Lemma~\ref{lem::two_force_points_determined} gives the desired result when $k,\ell \leq 1$.  Let $K_t$ denote the hull of $\eta$ at time $t$ and let $f_t \colon \h \setminus K_t \to \h$ be the corresponding centered Loewner map.  Assume that the statement of the lemma holds for some fixed $k,\ell \geq 1$.  We are going to prove that the result holds for $k+1$ force points to the left of $0$ and $\ell$ to the right of $0$ (and vice-versa by symmetry).  There are two possibilities: either $K_t$ does or does not accumulate in $(-\infty,x^{k+1,L}]$.  In the latter case, we are done because we can invoke Proposition~\ref{prop::gff_abs_continuity} to deduce the result from the induction hypothesis.  Suppose that we are in the former case.  Let $\tau$ be the first time $K_t$ accumulates in $(-\infty,x^{k+1,L}]$.  Note that $K|_{[0,\tau]}$ is almost surely determined by $h$ by Proposition~\ref{prop::gff_abs_continuity} and the induction hypothesis (we can apply these results to $K|_{[0,\tau_\epsilon]}$ where $\tau_\epsilon$ is the first time that $K_t$ gets within distance $\epsilon > 0$ of $(-\infty,x^{k+1,L}]$ and then send $\epsilon \to 0$).  If $\tau$ is at the continuation threshold, then we are done.  If not, we just need to show that $K|_{(\tau,\infty)}$ is almost surely determined by $h$.  Assume that the rightmost point of $K_\tau$ is contained in $[x^{j_0,R},x^{j_0+1,R})$.  Then the conditional law of $f_\tau(K_t)$ for $t \geq \tau$ given $K|_{[0,\tau]}$ is an $\SLE_\kappa(\ol{\rho}^L;\wt{\ul{\rho}}^R)$ process in $\h$ from $0$ to $\infty$ where
\begin{align*}
 \ol{\rho}^L &= \sum_{s=1}^{k+1} \rho^{s,L}
 \quad\text{and}\quad
 \wt{\rho}^{1,R} = \sum_{s=1}^{j_0} \rho^{s,R},\ 
 \wt{\rho}^{2,R} = \rho^{j_0+1,R},\ldots,
 \wt{\rho}^{\ell-j_0+1,R} = \rho^{\ell,R}.
\end{align*}
By the induction hypothesis, it thus follows that $(f_\tau(K_t) : t \geq \tau)$ is almost surely determined by $h \circ f_\tau - \chi \arg f_\tau'$, hence also by $h|_{\h \setminus K_\tau}$ given $K_\tau$.  The result now follows.
\end{proof}

The proof of Lemma~\ref{lem::functional_many_force_points} is not specific to $\SLE_\kappa(\ul{\rho})$ processes when $\kappa \in (0,4]$.  In particular, upon proving that $\SLE_{\kappa'}(\rho^L;\rho^R)$ processes are almost surely determined by $h$ for $\rho^L,\rho^R > -2$ in the coupling of Theorem~\ref{thm::coupling_existence} in Section~\ref{subsec::counterflow}, we will have completed the proof of Theorem~\ref{thm::coupling_uniqueness}.

\subsection{Monotonicity, merging, and crossing}
\label{subsec::monotonicity_merging_crossing}

Up until now, the only type of interaction between flow lines that we have considered has been when the paths have the same seed.  In this subsection, we will expand on this to complete the proof of Theorem~\ref{thm::monotonicity_crossing_merging} (contingent on a continuity assumption which will be removed upon proving Theorem~\ref{thm::continuity} for flow lines in Section~\ref{subsec::many_boundary_force_points}).  This gives a complete description of the manner in which flow lines can interact with each other and makes the phenomena observed in the simulations from the introduction (Figures~\ref{fig::flowlines}-\ref{fig::flowlines4},~\ref{fig::grid},~\ref{fig::two_fans}, and~\ref{fig::flow_line_interaction}) precise.  In particular, we will consider the following setup: we have two flow lines $\eta_{\theta_1}^{x_1}$ and $\eta_{\theta_2}^{x_2}$ of a GFF on $\h$ with piecewise constant boundary data which changes a finite number of times with angles $\theta_1$ and $\theta_2$ starting at boundary points $x_1$ and $x_2$, respectively.  We will show that if $x_1 > x_2$ and $\theta_1 < \theta_2$, then $\eta_{\theta_1}^{x_1}$ stays to the right of $\eta_{\theta_2}^{x_2}$.  This is a generalization of Proposition~\ref{prop::monotonicity_boundary_intersecting}, the monotonicity statement for boundary data which is constant on $(-\infty,0)$ and on $(0,\infty)$, to the setting of general piecewise constant boundary data and where the initial points of the flow lines can be different.  The more interesting behavior occurs when $\theta_1 = \theta_2$ or $\theta_2 < \theta_1 < \theta_2+\pi$.  In the former case, the flow lines will actually almost surely merge upon intersecting and then never separate (in obvious contrast with a Euclidean geometry).  In the latter case, upon intersecting, the flow lines will almost surely cross exactly once.  Afterwards, they may continue to intersect and bounce off of one another, but will never cross again.

\begin{figure}[h!]
\begin{center}
\subfigure[If $\theta_1 = \theta_2 = \theta$, then the flow lines almost surely merge upon intersecting and then never separate.]{
\includegraphics[scale=0.85,page=1]{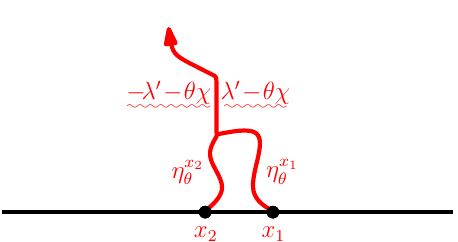}}
\hspace{0.025\textwidth}
\subfigure[If $\theta_2 < \theta_1 < \theta_2 + \pi$, then the flow lines almost surely cross upon intersecting.  Afterwards, they may bounce off of each other but will never cross again.]{
\includegraphics[scale=0.85,page=2]{figures/merging_constant_bc.pdf}}
\end{center}
\caption{\label{fig::merging_and_crossing_phases}
Suppose that $h$ is a GFF on $\h$ with piecewise constant boundary data which changes a finite number of times.  For each $x \in \R$ and angle $\theta$, let $\eta_\theta^x$ be the flow line of $h$ starting at $x$ with angle $\theta$.  If $x_1 > x_2$ and $\theta_1 = \theta_2 = \theta$, then $\eta_{\theta_1}^{x_1}$ will almost surely merge with $\eta_{\theta_2}^{x_2}$ upon intersecting (left panel).  If $\theta_2 < \theta_1 < \theta_2+\pi$, then $\eta_{\theta_1}^{x_1}$ will cross $\eta_{\theta_2}^{x_2}$ upon intersecting but will never cross back (right) panel.  If $\theta_1 < \theta_2$, then $\eta_{\theta_1}^{x_1}$ will almost surely stay to the right of $\eta_{\theta_2}^{x_2}$.}
\end{figure}

The first step to proving these results is Lemma~\ref{lem::crossing_local}, which says that the set $K$ which consists of those points of $\eta_{\theta_1}^{x_1}$ until the first time $\tau_1$ that $\eta_{\theta_1}^{x_1}$ intersects $\eta_{\theta_2}^{x_2}$ and those points of $\eta_{\theta_2}^{x_2}$ until the first time $\tau_2$ that $\eta_{\theta_2}^{x_2}$ hits $\eta_{\theta_1}^{x_1}$ is a local set for $h$.  We will in particular show that if $\tau_1,\tau_2 < \infty$, then $\eta_{\theta_1}^{x_1}(\tau_1) = \eta_{\theta_2}^{x_2}(\tau_2)$ (which is of course not in general true for continuous paths).  This is a particularly interesting example of a local set because it cannot be generated using a ``local algorithm'' that explores the values of the field along the flow lines until stopping times (without ever looking at the field off of those flow lines).  Once Lemma~\ref{lem::crossing_local} is established, we will then prove Lemma~\ref{lem::hitting_conditional_law} which gives that the conditional mean $\CC_K$ of $h$ given $\CK$ (where $\CK$ plays the same role for $K$ as $\CA$ from Section~\ref{subsec::local_sets}) does not exhibit pathological behavior in the unbounded connected component $D$ of $\h \setminus K$, even at the first intersection point of $\eta_{\theta_1}^{x_1}$ and $\eta_{\theta_2}^{x_2}$.  This in turn allows us to show that $\eta_{\theta_i}^{x_i}|_{[\tau_i,\infty)}$, $i=1,2$, is the flow line of the conditional field $h|_D$ given $\CK$ with angle $\theta_i$ starting at $\eta_{\theta_1}^{x_1}(\tau_1) = \eta_{\theta_2}^{x_2}(\tau_2)$ provided $\tau_1,\tau_2 < \infty$.  We then prove Proposition~\ref{prop::generalized_monotonicity}, which is our most general monotonicity statement, followed by Proposition~\ref{prop::angle_varying_monotonicity}, which gives the monotonicity of angle varying flow lines.  The merging and crossing behavior are proved in Proposition~\ref{prop::merging_and_crossing} and follows by using Lemma~\ref{lem::crossing_local} and Lemma~\ref{lem::hitting_conditional_law} to reduce the results to Proposition~\ref{prop::generalized_monotonicity}.  Before we prove Lemma~\ref{lem::crossing_local}, we need to record the following technical lemma.

\begin{figure}[h!]
\begin{center}
\includegraphics[scale=0.85]{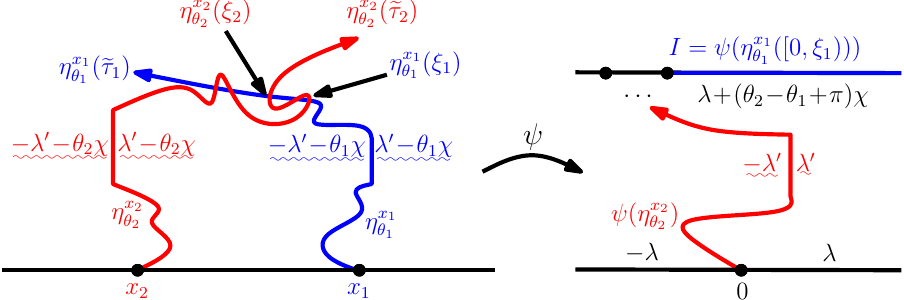}
\end{center}
\caption{\label{fig::cannot_hit_right_side}
Suppose that $h$ is a GFF on $\h$ with piecewise constant boundary data which changes a finite number of times.  Suppose that $x_1,x_2 \in \R$ with $x_2 < x_1$ and fix $\theta_1,\theta_2 \in \R$ with $\theta_1 < \theta_2 +\pi$.  Let $T_i$ be a stopping time for $\eta_{\theta_i}^{x_i}$, $i=1,2$, such that $\eta_{\theta_i}^{x_i}|_{[0,T_i]}$ is almost surely continuous.  Let $\wt{\tau}_1 \leq T_1$ be any stopping time for $\eta_{\theta_1}^{x_1}$.  We show in Lemma~\ref{lem::cannot_hit_right_side} that the following is true.  Suppose that $\wt{\tau}_2 \leq T_2$ is any stopping time for the filtration $\CF_t = \sigma(\eta_{\theta_2}^{x_2}(s) : s \leq t,\ \eta_{\theta_1}^{x_1}([0,\wt{\tau}_1]))$ such that if $\xi_2$ is the largest time $t \in [0,\wt{\tau}_2]$ so that $\eta_{\theta_2}^{x_2}(t) \in \eta_{\theta_1}^{x_1}([0,\wt{\tau}_1])$ then $\eta_{\theta_2}^{x_2}(\wt{\tau}_2)$ is contained in the unbounded connected component of $\h \setminus (\eta_{\theta_1}^{x_1}([0,\wt{\tau_1}]) \cup \eta_{\theta_2}^{x_2}([0,\xi_2]))$.  Let $\xi_1$ be the smallest time $t \in [0,\wt{\tau}_1]$ that $\eta_{\theta_1}^{x_1}(t) \in \eta_{\theta_2}^{x_2}([0,\wt{\tau}_2])$.  Then $\eta_{\theta_2}^{x_2}|_{[\wt{\tau}_2,T_2]}$ does not exit the unbounded connected component $D$ of $\h \setminus (\eta_{\theta_1}^{x_1}([0,\wt{\tau}_1]) \cup \eta_{\theta_2}^{x_2}([0,\wt{\tau}_2]))$ in the right side of $\eta_{\theta_1}^{x_1}([0,\xi_1])$.  To see this, we let $\psi$ be the conformal from $D$ to $\strip$ which takes $\eta_{\theta_2}^{x_2}(\wt{\tau}_2)$ to $0$, and the left and right sides of $\eta_{\theta_2}^{x_2}((\xi_2,\wt{\tau}_2))$ to $(-\infty,0)$ and $(0,\infty)$, respectively.  The boundary data for the GFF $\wt{h} + \theta_2 \chi$ where $\wt{h} = h|_D \circ \psi^{-1} - \chi \arg (\psi^{-1})'$ is shown on the right side.  Note that $\theta_2-\theta_1+\pi \geq 0$ so that with $I = \psi(\eta_{\theta_1}^{x_1}([0,\xi_1)))$ we have that $(\wt{h} + \theta_2 \chi)|_I \geq \lambda$.  Since $\psi(\eta_{\theta_2}^{x_2}|_{[\wt{\tau}_2,T_2]})$ is the flow line of $\wt{h} + \theta_2\chi$ starting from $0$, the result then follows from Lemma~\ref{lem::flow_cannot_hit}.}
\end{figure}

\begin{lemma}
\label{lem::cannot_hit_right_side}
Suppose that $h$ is a GFF on $\h$ with piecewise constant boundary data which changes a finite number of times.  Suppose that $x_1,x_2 \in \R$ with $x_2 < x_1$ and fix angles $\theta_1,\theta_2 \in \R$ with $\theta_1 < \theta_2 +\pi$.  Let $T_i$ be a stopping time for $\eta_{\theta_i}^{x_i}$, $i=1,2$, such that $\eta_{\theta_i}^{x_i}|_{[0,T_i]}$ is almost surely continuous.  Let $\wt{\tau}_1 \leq T_1$ be any stopping time for $\eta_{\theta_1}^{x_1}$.  Then the following is true.  Suppose that $\wt{\tau}_2 \leq T_2$ is any stopping time for the filtration $\CF_t = \sigma(\eta_{\theta_2}^{x_2}(s) : s \leq t,\ \eta_{\theta_1}^{x_1}([0,\wt{\tau}_1]))$ such that with $\xi_2$ the largest time $t \in [0,\wt{\tau}_2]$ with $\eta_{\theta_2}^{x_2}(t) \in \eta_{\theta_1}^{x_1}([0,\wt{\tau}_1])$ we have that $\eta_{\theta_2}^{x_2}(\wt{\tau}_2)$ is contained in the unbounded connected component of $\h \setminus (\eta_{\theta_1}^{x_1}([0,\wt{\tau}_1]) \cup \eta_{\theta_2}^{x_2}([0,\xi_2]))$.  Let $\xi_1$ be the smallest time $t$ that $\eta_{\theta_1}^{x_1}(t) \in \eta_{\theta_2}^{x_2}([0,\wt{\tau}_2])$.  Then $\eta_{\theta_2}^{x_2}|_{[\wt{\tau}_2,T_2]}$ either hits $\partial \h$ or $\eta_{\theta_1}^{x_1}([\xi_1,\wt{\tau}_1])$ or escapes to~$\infty$ before hitting $\eta_{\theta_1}^{x_1}([0,\xi_1))$.  The analogous result holds when the roles of $\eta_{\theta_1}^{x_1}$, $\eta_{\theta_2}^{x_2}$ are swapped.
\end{lemma}
We note that in the case that either $\theta_1$ or $\theta_2$ is outside of the range of values necessary for the flow line to be defined (i.e., there is a force point of weight less than or equal to $-2$ immediately to the left or to the right of the starting point of the path), the statement of Lemma~\ref{lem::cannot_hit_right_side} trivially holds as we can view such a flow line to be the path which is always equal to its starting point.  Lemma~\ref{lem::cannot_hit_right_side} is a consequence of Lemma~\ref{lem::flow_cannot_hit}; see Figure~\ref{fig::cannot_hit_right_side} for the setup of the proof.  Its proof is also very similar to that of Proposition~\ref{prop::flow_counterflow_left_right}.
\begin{proof}[Proof of Lemma~\ref{lem::cannot_hit_right_side}]
Note that $\eta_{\theta_1}^{x_1}([0,\wt{\tau}_1]) \cup \eta_{\theta_2}^{x_2}([0,\wt{\tau}_2])$ is a local set for $h$ by Lemma~\ref{lem::stopping_local_set}.  Let $D$ be the unbounded connected component of $\h \setminus (\eta_{\theta_1}^{x_1}([0,\wt{\tau}_1]) \cup \eta_{\theta_2}^{x_2}([0,\wt{\tau}_2]))$.  Let $\psi \colon D \to \strip$ be the conformal map which takes $\eta_{\theta_2}^{x_2}(\wt{\tau}_2)$ to $0$, the left side of $\eta_{\theta_2}^{x_2}((\xi_2,\wt{\tau}_2))$ to $(-\infty,0)$, and the right side to $(0,\infty)$.  Let $\wt{h} = h|_D \circ \psi^{-1} - \chi \arg (\psi^{-1})'$.  Then $\wt{h}$ is a GFF on $\strip$; the boundary data for $\wt{h} + \theta_2 \chi$ is depicted in the right side of Figure~\ref{fig::cannot_hit_right_side}.  Here, we are using Proposition~\ref{gff::prop::cond_union_mean} to get the boundary data for $\wt{h}$ on the image $I$ under $\psi$ of the right side of $\eta_{\theta_1}^{x_1}([0,\xi_1))$ as well as on the lower boundary $\stripbot$ (in particular, we do not try to rule out pathological behavior in the conditional mean of $h$ at intersection points of $\eta_{\theta_1}^{x_1}$ and $\eta_{\theta_2}^{x_2}$).  Since $\eta_{\theta_2}^{x_2}|_{[0,T_2]}$ is almost surely continuous, we know that the image of $\eta_{\theta_2}^{x_2}|_{[\wt{\tau}_2,T_2]}$ under $\psi$ is also continuous.  Since the boundary data of $\wt{h}+\theta_2 \chi$ on $I$ is $\lambda + (\theta_2-\theta_1+\pi)\chi \geq \lambda$, Lemma~\ref{lem::flow_cannot_hit} implies that $\psi(\eta_{\theta_2}^{x_2}|_{[\wt{\tau}_2,T_2]})$ must exit $\strip$ in $\striptop \setminus I$ (or does not exit before time $T_2$).  This corresponds to $\eta_{\theta_2}^{x_2}$ exiting $D$ either in $\partial \h$ or in $\eta_{\theta_1}^{x_1}|_{[\xi_1,\wt{\tau}_1]}$ (or not exiting at all).  This completes the proof.
\end{proof}

\begin{figure}[h!]
\begin{center}
\includegraphics[scale=0.85]{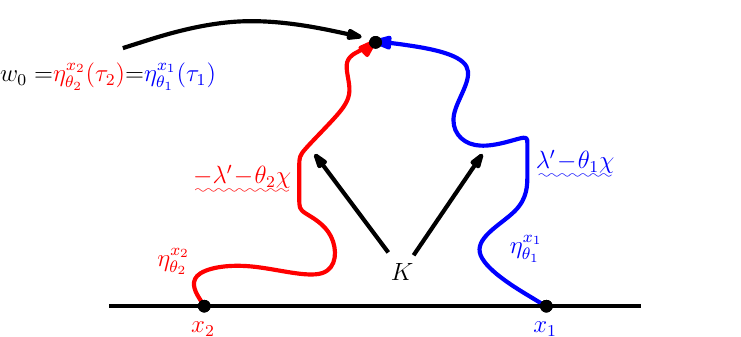}
\end{center}
\caption{\label{fig::crossing_local_proof} Suppose that $h$ is a GFF on $\h$ with piecewise constant boundary data which changes a finite number of times.  Suppose that $x_1,x_2 \in \partial \h$ with $x_2 < x_1$ and fix $\theta_1,\theta_2 \in \R$ with $\theta_1 < \theta_2 +\pi$.  For $i=1,2$, let $T_i$ be a stopping time for $\eta_{\theta_i}^{x_i}$ such that $\eta_{\theta_i}^{x_i}|_{[0,T_i]}$ is almost surely continuous.  Let $\tau_1$ be the first time that $\eta_{\theta_1}^{x_1}|_{[0,T_1]}$ hits $\eta_{\theta_2}^{x_2}|_{[0,T_2]}$ and let $\tau_2$ be the first time that $\eta_{\theta_2}^{x_2}|_{[0,T_2]}$ hits $\eta_{\theta_1}^{x_1}|_{[0,T_1]}$.  In Lemma~\ref{lem::crossing_local}, we show that $K = \eta_{\theta_1}^{x_1}([0,\tau_1 \wedge T_1]) \cup \eta_{\theta_2}^{x_2}([0,\tau_2 \wedge T_2])$ is a local set for $h$ and, if $\tau_1,\tau_2 < \infty$, then $w_0 = \eta_{\theta_1}^{x_1}(\tau_1) = \eta_{\theta_2}^{x_2}(\tau_2)$.  Let $D$ be the unbounded connected component of $\h \setminus K$.  In Lemma~\ref{lem::hitting_conditional_law}, we show that $h|_D$ given $\CK$ (where $\CK$ plays the role for $K$ of $\CA$ from Section~\ref{subsec::local_sets}) is a GFF in $D$ whose boundary data is as depicted above and that $\eta_{\theta_i}^{x_i}|_{[\tau_i,T_i]}$ is the flow line of $h|_D$ starting at $w_0$ of angle $\theta_i$ for $i=1,2$.}
\end{figure}

\begin{lemma}
\label{lem::crossing_local}
Suppose that $h$ is a GFF on $\h$ with piecewise constant boundary data which changes a finite number of times.  Fix $x_1,x_2 \in \partial \h$ with $x_2 < x_1$ and angles $\theta_1,\theta_2$ with $\theta_1 < \theta_2+\pi$.  For $i=1,2$, let $T_i$ be a stopping time for $\eta_{\theta_i}^{x_i}$ such that $\eta_{\theta_i}^{x_i}|_{[0,T_i]}$ is almost surely continuous.  Let $\tau_2$ be the first time that $\eta_{\theta_2}^{x_2}|_{[0,T_2]}$ intersects $\eta_{\theta_1}^{x_1}|_{[0,T_1]}$ and let $\tau_1$ be the first time that $\eta_{\theta_1}^{x_1}|_{[0,T_1]}$ intersects $\eta_{\theta_2}^{x_2}|_{[0,T_2]}$.  Let $K = \eta_{\theta_1}^{x_1}([0,\tau_1 \wedge T_1]) \cup \eta_{\theta_2}^{x_2}([0,\tau_2 \wedge T_2])$.  Then $K$ is a local set for $h$.  Moreover, if $\tau_i \leq T_i$ and $\tau_i < \infty$ for $i=1,2$ then $\eta_{\theta_1}^{x_1}(\tau_1) = \eta_{\theta_2}^{x_2}(\tau_2)$ and $\eta_{\theta_i}^{x_i}|_{[\tau_i,T_i]}$, for $i=1,2$, is almost surely contained in the unbounded connected component of $\h \setminus K$.
\end{lemma}

We note that in the statement of Lemma~\ref{lem::crossing_local}, the stopping times $\tau_i$, $i=1,2$, may actually be infinite in the case that $\eta_{\theta_1}^{x_1}|_{[0,T_1]}$ does not touch $\eta_{\theta_2}^{x_2}|_{[0,T_2]}$.  This can happen, for example, if one of the $\eta_{\theta_i}^{x_i}|_{[0,T_i]}$ hits a segment of $\partial \h$ after which it is not able to continue, i.e.\ the boundary data of $h + \theta_i \chi$ is at least $\lambda$ on the left side of the intersection point or is not more than $-\lambda$ on the right side.  On the other hand, the statement is vacuous in the setting in which the boundary of $h$ is such that either $\eta_{\theta_1}^{x_1}$ or $\eta_{\theta_2}^{x_2}$ immediately hits the continuation threshold (i.e., the $\rho$ value immediately to the left or right of one of the paths is not larger than $-2$).

\begin{proof}[Proof of Lemma~\ref{lem::crossing_local}]
Let $\wt{\eta}_{\theta_i}^{x_i}(t) = \eta_{\theta_i}^{x_i}(t \wedge T_i)$ for $i=1,2$.  Assume that $\tau_i < \infty$ for $i=1,2$ for otherwise the result is trivial.  We are now going to prove that $\wt{\eta}_{\theta_1}^{x_1}(\tau_1) = \wt{\eta}_{\theta_2}^{x_2}(\tau_2)$.  To see this, we apply Lemma~\ref{lem::cannot_hit_right_side} for the stopping time $\wt{\tau}_1 = T_1$ for $\wt{\eta}_{\theta_1}^{x_1}$ and for any stopping time $\tau_2 \leq \wt{\tau}_2 \leq T_2$ for the filtration $\CF_t = \sigma(\wt{\eta}_{\theta_2}^{x_2}(s) : s \leq t,\ \wt{\eta}_{\theta_1}^{x_1}([0,T_1]))$ so that the criteria of Lemma~\ref{lem::cannot_hit_right_side} hold.  Let $\xi_1$ be the first time $t$ that $\wt{\eta}_{\theta_1}^{x_1}|_{[0,T_1]}$ hits $\wt{\eta}_{\theta_2}^{x_2}([0,\wt{\tau}_2])$.  Lemma~\ref{lem::cannot_hit_right_side} implies that $\wt{\eta}_{\theta_2}^{x_2}|_{[\wt{\tau}_2,T_2]}$ cannot hit the right side of $\wt{\eta}_{\theta_1}^{x_1}([0,\xi_1))$ before hitting either $\wt{\eta}_{\theta_1}^{x_1}([\xi_1,T_1])$ or $\partial \h$.  By applying this to a countable collection of stopping times which is dense among admissible times in $[\tau_2,T_2]$ (i.e, stopping times $\wt{\tau}_2$ so that the criteria of Lemma~\ref{lem::cannot_hit_right_side} apply), we see that $\wt{\eta}_{\theta_2}^{x_2}([\tau_2,T_2])$ cannot hit the right side of $\wt{\eta}_{\theta_1}^{x_1}([0,\sigma_1))$ where $\sigma_1$ is the first time that $\wt{\eta}_{\theta_1}^{x_1}$ hits $\wt{\eta}_{\theta_2}^{x_2}(\tau_2)$.  An analogous argument with the roles of $\eta_{\theta_1}^{x_1}$ and $\eta_{\theta_2}^{x_2}$ reversed implies that $\wt{\eta}_{\theta_1}^{x_1}([\tau_1,T_1])$ cannot hit the left side of $\wt{\eta}_{\theta_2}^{x_2}([0,\sigma_2))$ where $\sigma_2$ is the first time that $\wt{\eta}_{\theta_2}^{x_2}$ hits $\wt{\eta}_{\theta_1}^{x_1}(\tau_1)$.  It thus follows that $\sigma_1 = \tau_1$, which proves the claim.

Recall that we assumed that $\tau_i < \infty$ for $i=1,2$.  Let $K = \wt{\eta}_{\theta_1}^{x_1}([0,\tau_1]) \cup \wt{\eta}_{\theta_2}^{x_2}([0,\tau_2])$.  The above argument also implies that $\wt{\eta}_{\theta_i}^{x_i}|_{[\tau_i,T_i]}$ is contained in the unbounded connected component of $\h \setminus K$ for $i=1,2$ (i.e., if $\wt{\eta}_{\theta_1}^{x_1}|_{[\tau_1,T_1]}$ is in a bounded connected component, then $\wt{\eta}_{\theta_2}^{x_2}$ hits $\wt{\eta}_{\theta_1}^{x_1}$ on its right side first).  We are now going to prove that $K$ is local by checking characterization~\eqref{it::Ucond} of Lemma~\ref{lem::local_char}.  Fix $U \subseteq \h$ open.  For $i=1,2$, let $\tau_i^U = \inf\{t \geq 0 : \wt{\eta}_{\theta_i}^{x_i}(t) \in U\}$.  Then $\wt{\eta}_{\theta_i}^{x_i}([0,\tau_i^U])$, $i=1,2$, is almost surely determined by the projection of $h$ onto $H^\perp(U)$ as a consequence of Theorem~\ref{thm::coupling_uniqueness} and that $\wt{\eta}_{\theta_i}^{x_i}([0,\tau_i^U])$, $i=1,2$, is a local set for $h$.  Since the event that $\{K \cap U \neq \emptyset\}$ is determined by $\wt{\eta}_{\theta_1}^{x_1}([0,\tau_1^U])$ and $\wt{\eta}_{\theta_2}^{x_2}([0,\tau_2^U])$, it is therefore determined by the projection of $h$ onto $H^\perp(U)$.
\end{proof}

\begin{lemma}
\label{lem::hitting_conditional_law}
Assume that we have the same setup as Lemma~\ref{lem::crossing_local}.  Let $D$ be the unbounded connected component of $\h \setminus K$ and let $\varphi \colon D \to \h$ be a conformal transformation which fixes $\infty$ and let $\Fg$ be the function which is harmonic in $\h$ with boundary conditions given by $-\lambda-\theta_i \chi$ (resp.\ $\lambda-\theta_i \chi$) on the $\varphi$ image of the left (resp.\ right) side of $\eta_i|_{[0,\tau_i \wedge T_i]}$ for $i=1,2$ and the same as $h \circ \varphi^{-1}$ on $\varphi(\partial \h)$.  Then the conditional mean $\CC_K$ of $h$ given $\CK$ (where $\CK$ plays the role for $K$ of $\CA$ as in Section~\ref{subsec::local_sets}) restricted to $D$ is $\Fg \circ \varphi - \chi \arg \varphi'$ (see Figure~\ref{fig::crossing_local_proof}).  Moreover, if $\tau_i < \infty$ for $i=1,2$ then $\eta_{\theta_i}^{x_i}|_{[\tau_i,T_i]}$ is the flow line of the conditional field $h$ given $\CK$ restricted to $D$ with angle $\theta_i$ starting at $w_0 = \eta_{\theta_1}^{x_1}(\tau_1) = \eta_{\theta_2}^{x_2}(\tau_2)$.
\end{lemma}
\begin{proof}
Lemma~\ref{lem::crossing_local} implies that $K = \eta_{\theta_1}^{x_1}([0,\tau_1]) \cup \eta_{\theta_2}^{x_2}([0,\tau_2])$ is a local set for $h$.  The claim regarding $\CC_K$ is clear if $w_0 \in \partial \h$ by Proposition~\ref{prop::cond_mean_continuous} or if $\tau_1,\tau_2 = \infty$.  To see this in the case that $w_0 \in \h$, we note that with $K_{r,s} = \eta_{\theta_1}^{x_1}([0,r]) \cup \eta_{\theta_2}^{x_2}([0,s])$ we have that $\CC_{K_{r,s}}$ does not exhibit pathological behavior whenever $s > \tau_2$ and $r < \tau_1$ by Proposition~\ref{gff::prop::cond_union_local} and Proposition~\ref{gff::prop::cond_union_mean}.  The claim then follows by using Proposition~\ref{prop::cond_mean_continuous} and that $\CC_K = \lim_{s \downarrow \tau_2} \lim_{r \uparrow \tau_1} \CC_{K_{r,s}}$ almost surely.

To see the second claim, we first note that $\eta_{\theta_i}^{x_i}|_{[\tau_i,T_i]}$, for $i=1,2$, has a continuous Loewner driving function viewed as a path in $D$.  The reason is that $\eta_{\theta_i}^{x_i}$ cannot trace itself or $\partial \h$ since it has a continuous Loewner driving function viewed as a path in $\h$ and it cannot trace or create loops with $K$ because the proof of Lemma~\ref{lem::crossing_local} implies that $\eta_{\theta_i}^{x_i}|_{[\tau_i,T_i]}$ almost surely does not hit $K$ after time $\tau_i$.  Thus the claim follows by applying Proposition~\ref{prop::cont_driving_function}.  Combining Theorem~\ref{thm::coupling_uniqueness}, Theorem~\ref{thm::martingale}, and Proposition~\ref{prop::cond_mean_continuous} completes the proof of the result.
\end{proof}

The proofs of Lemma~\ref{lem::crossing_local} and Lemma~\ref{lem::hitting_conditional_law} also apply in the following slightly more general situation.

\begin{lemma}
\label{lem::crossing_local_stopping}
Assume that we have the same setup as Lemma~\ref{lem::crossing_local}.  Suppose that $\sigma_2 \leq T_2$ is a stopping time for $\eta_{\theta_2}^{x_2}$ such that, almost surely, $\eta_{\theta_1}^{x_1}|_{[0,T_1]}$ lies to the right of $\eta_{\theta_2}^{x_2}|_{[0,\sigma_2]}$.  Let $\tau_1$ be the first time that $\eta_{\theta_1}^{x_1}|_{[0,T_1]}$ hits $\eta_{\theta_2}^{x_2}|_{[\sigma_2,T_2]}$ and let $\tau_2$ be the first time that $\eta_{\theta_2}^{x_2}|_{[\sigma_2,T_2]}$ hits $\eta_{\theta_1}^{x_1}|_{[0,T_1]}$.  Let $K = \eta_{\theta_1}^{x_1}([0,\tau_1 \wedge T_1]) \cup \eta_{\theta_2}^{x_2}([0,\tau_2 \wedge T_2])$.  Then $K$ is a local set for $h$.  If, in addition, $\tau_1,\tau_2 < \infty$ then $w_0 = \eta_{\theta_1}^{x_1}(\tau_1) = \eta_{\theta_2}^{x_2}(\tau_2)$.  Let $D$ be the unbounded connected component of $\h \setminus K$.  Then $\eta_{\theta_i}^{x_i}|_{[\tau_i,T_i]}$ for $i=1,2$ is the flow line of the conditional field $h|_D$ given $\CK$ (where $\CK$ plays the same role for $K$ as $\CA$ from Section~\ref{subsec::local_sets}) with angle $\theta_i$ starting at $w_0$.
\end{lemma}
\begin{proof}
This follows from the same argument used to establish Lemma~\ref{lem::crossing_local} and Lemma~\ref{lem::hitting_conditional_law}.  We note that the intersection points of $\eta_{\theta_1}^{x_1}([0,\tau_1))$ and $\eta_{\theta_2}^{x_2}([0,\tau_2))$ (i.e., before $\eta_{\theta_1}^{x_1}|_{[0,T_1]}$ hits $\eta_{\theta_2}^{x_2}|_{[\sigma_2,T_2]}$ and vice-versa) do not lead to singularities in the conditional mean $\CC_K$ of $h$ given $K$ in $D$ by Proposition~\ref{gff::prop::local_independence}.
\end{proof}

By combining Lemma~\ref{lem::crossing_local_stopping} with Proposition~\ref{prop::monotonicity_boundary_intersecting}, we can now prove our general monotonicity statement for flow lines.

 \begin{figure}[h!]
\begin{center}
\includegraphics[scale=0.85]{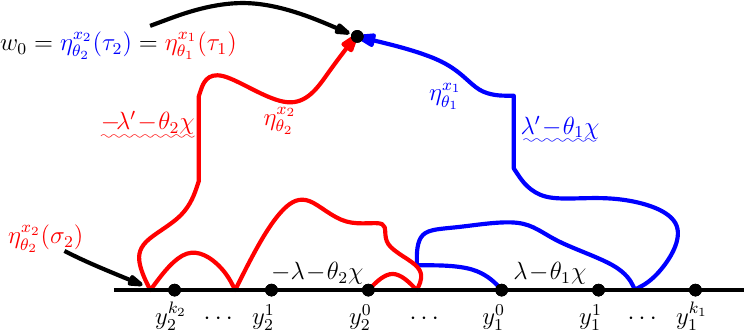}
\end{center}
\caption{\label{fig::monotonicity_induction} 
Suppose that $h$ is a GFF on $\h$ with piecewise constant boundary data which changes a finite number of times, $x_1,x_2 \in \partial \h$ with $x_1 > x_2$, and fix angles $\theta_1<\theta_2$.  Let $\ul{y}_2 = (y_2^{k_2} < \cdots < y_2^1 \leq y_2^0 = x_2)$ and $\ul{y}_1 = (y_1^{k_1} > \cdots > y_1^1 \geq y_1^0 = x_1)$.  Assume that the boundary data for $h$ changes to the left of $x_2$ only at the points of $\ul{y}_2$ and to the right of $x_1$ only at the points of $\ul{y}_1$.  Moreover, assume that the boundary data for $h$ in $[y_2^1,y_2^0)$ is $-\lambda - \theta_2 \chi$ and in $[y_1^0,y_1^1)$ is $\lambda - \theta_1 \chi$.  For $i=1,2$, let $T_i$ be a stopping time for $\eta_{\theta_i}^{x_i}$ such that $\eta_{\theta_i}^{x_i}|_{[0,T_i]}$ is almost surely continuous.  We prove in Proposition~\ref{prop::generalized_monotonicity} that $\eta_{\theta_1}^{x_1}|_{[0,T_1]}$ lies to the right of $\eta_{\theta_2}^{x_2}|_{[0,T_2]}$ by induction on $k_1,k_2$.  The result for the case $k_1,k_2 \leq 1$ is a consequence of Proposition~\ref{prop::monotonicity_boundary_intersecting} and Lemma~\ref{lem::crossing_local} (to justify growing the paths until they first hit each other).  Suppose the result holds for $k_1 = j_1$ and $k_2=j_2$, some $j_1,j_2 \geq 1$ and that $k_2=j_2+1$ and $k_1=j_1$.  Letting $\sigma_2$ be the first time $t$ that $\eta_{\theta_2}^{x_2}|_{[0,T_2]}$ hits $(-\infty,y_2^{k_2}]$, it follows from the induction hypothesis and absolute continuity (Proposition~\ref{prop::gff_abs_continuity}) that $\eta_{\theta_2}^{x_2}|_{[0,\sigma_2 \wedge T_2]}$ lies to the left of $\eta_{\theta_1}^{x_1}|_{[0,T_1]}$.  Let $\tau_1$ be the first time $t$ that $\eta_{\theta_1}^{x_1}|_{[0,T_1]}$ hits $\eta_{\theta_2}^{x_2}|_{[\sigma_2,T_2]}$ and $\tau_2$ the first time $t$ after $\sigma_2$ that $\eta_{\theta_2}^{x_2}|_{[0,T_2]}$ hits $\eta_{\theta_1}^{x_1}|_{[0,T_1]}$.  If $\tau_1=\tau_2=\infty$, there is nothing to prove.  If $\tau_1,\tau_2 < \infty$, then Lemma~\ref{lem::crossing_local_stopping} implies that $K = \eta_{\theta_1}^{x_1}([0,\tau_1]) \cup \eta_{\theta_2}^{x_2}([0,\tau_2])$ is a local set for $h$ and that $\eta_{\theta_i}^{x_i}|_{[\tau_i,T_i]}$, $i=1,2$, is the flow line of $h|_D$, $D$ the unbounded connected component of $\h \setminus K$, conditional on $\CK$ (where $\CK$ plays the role for $K$ of $\CA$ from Section~\ref{subsec::local_sets}) starting from $w_0 = \eta_{\theta_1}^{x_1}(\tau_1) = \eta_{\theta_2}^{x_2}(\tau_2)$.  The result then follows from the induction hypothesis.  We note that it could be that $w_0 \in \partial \h$, though the case $w_0 \in \h$ is depicted above.}
\end{figure}

\begin{proposition}
\label{prop::generalized_monotonicity}
Suppose that $h$ is a GFF on $\h$ with piecewise constant boundary data which changes a finite number of times, $x_1,x_2 \in \partial \h$, and fix angles $\theta_1,\theta_2$.  For $i=1,2$, let $T_i$ be a stopping time for $\eta_{\theta_i}^{x_i}$ such that $\eta_{\theta_i}^{x_i}|_{[0,T_i]}$ is almost surely continuous.  If $\theta_1 < \theta_2$ and $x_1 > x_2$, then $\eta_{\theta_1}^{x_1}|_{[0,T_1]}$ almost surely lies to the right of $\eta_{\theta_2}^{x_2}|_{[0,T_2]}$.
\end{proposition}
\begin{proof}
Let $\ul{y}_2 = (y_2^{k_2} < \cdots < y_2^1 \leq y_2^0 = x_2)$ and $\ul{y}_1 = (y_1^{k_1} > \cdots > y_1^1 \geq y_1^0 = x_1)$.  Assume that the boundary data for $h$ in $[y_2^1,y_2^0)$ is $-\lambda-\theta_2 \chi$, in $[y_1^0,y_1^1)$ is $\lambda - \theta_1 \chi$, and otherwise changes to the left of $x_2$ only at the points of $\ul{y}_2$ and changes to the right of $x_1$ only at the points of $\ul{y}_1$ (note that we are not careful to restrict or specify the places where the boundary data for $h$ changes in $(x_2,x_1)$).  We are going to prove the result by induction on $k_1,k_2$ (see Figure~\ref{fig::monotonicity_induction} for an illustration of the setup and the proof).

We first assume that $k_1,k_2 \leq 1$.  Let $\tau_1$ be the first time that $\eta_{\theta_1}^{x_1}|_{[0,T_1]}$ hits $\eta_{\theta_2}^{x_2}|_{[0,T_2]}$ and $\tau_2$ the first time that $\eta_{\theta_2}^{x_2}|_{[0,T_2]}$ hits $\eta_{\theta_1}^{x_1}|_{[0,T_1]}$.  If $\tau_1 = \tau_2 = \infty$, then the result is trivial, so we shall assume that $\tau_1,\tau_2 < \infty$.  Let $D$ be the unbounded connected component of $\h \setminus K$ where $K = \eta_{\theta_1}^{x_1}([0,\tau_1]) \cup \eta_{\theta_2}^{x_2}([0,\tau_2])$.  Lemma~\ref{lem::hitting_conditional_law} implies that $\eta_{\theta_i}^{x_i}|_{[\tau_i,T_i]}$ for $i=1,2$ is the flow line of $h|_D$ given $\CK$ (where $\CK$ plays the same role for $K$ as $\CA$ from Section~\ref{subsec::local_sets}) starting at $w_0 = \eta_{\theta_1}^{x_1}(\tau_1) = \eta_{\theta_2}^{x_2}(\tau_2)$ with angle $\theta_i$.  Thus the case $k_1,k_2 \leq 1$ follows by Proposition~\ref{prop::monotonicity_boundary_intersecting}.

Suppose that the result holds for $k_i = j_i$, some $j_1,j_2 \geq 1$.  We are now going to show that it holds for $k_1 = j_1$ and $k_2 = j_2+1$ (the argument to show that it holds for $k_1 = j_1+1$ and $k_2=j_2$ is the same).  For $i=1,2$ and $\epsilon \geq 0$, let $\sigma_i^\epsilon$ be the first time that $\eta_{\theta_i}^{x_i}|_{[0,T_i]}$ gets within distance $\epsilon$ of $(-\infty,y_2^{k_2}]$.  The induction hypothesis combined with Proposition~\ref{prop::gff_abs_continuity} implies that $\eta_{\theta_2}^{x_2}|_{[0,\sigma_2^\epsilon \wedge T_2]}$ lies to the left of $\eta_{\theta_1}^{x_1}|_{[0,\sigma_1^\epsilon \wedge T_1]}$ for each $\epsilon >0$.  Taking a limit as $\epsilon \downarrow 0$, we see that $\eta_{\theta_2}^{x_2}|_{[0,\sigma_2 \wedge T_2]}$ stays to the left of $\eta_{\theta_1}^{x_1}|_{[0,\sigma_1 \wedge T_1]}$ where $\sigma_i = \sigma_i^0$ for $i=1,2$.  Since $\eta_{\theta_1}^{x_1}$ is targeted at $\infty$ and $\eta_{\theta_1}^{x_1}|_{[0,T_1]}$ is simple, we also know that $\eta_{\theta_1}^{x_1}|_{[\sigma_1,T_1]}$ does not hit $\eta_{\theta_2}^{x_2}|_{[0,\sigma_2 \wedge T_2]}$.  Therefore $\eta_{\theta_1}^{x_1}|_{[0,T_1]}$ lies to the right of $\eta_{\theta_2}^{x_2}|_{[0,\sigma_2 \wedge T_2]}$.  Let $\tau_1$ be the first time that $\eta_{\theta_1}^{x_1}|_{[0,T_1]}$ hits $\eta_{\theta_2}^{x_2}|_{[\sigma_2,T_2]}$ and let $\tau_2$ be the first time that $\eta_{\theta_2}^{x_2}|_{[\sigma_2,T_2]}$ hits $\eta_{\theta_1}^{x_1}|_{[0,T_1]}$.  If $\tau_1=\tau_2=\infty$, then there is nothing to prove.  If $\tau_1,\tau_2 < \infty$, Lemma~\ref{lem::crossing_local_stopping} implies that $K = \eta_{\theta_1}^{x_1}([0,\tau_1]) \cup \eta_{\theta_2}^{x_2}([0,\tau_2])$ is a local set for $h$ and that $\eta_{\theta_i}^{x_i}|_{[\tau_i,T_i]}$ is the flow line of $h|_D$, $D$ the unbounded connected component of $\h \setminus K$, of angle $\theta_i$ starting from $w_0 = \eta_{\theta_1}^{x_1}(\tau_1) = \eta_{\theta_2}^{x_2}(\tau_2)$ for $i=1,2$.  Let $\psi \colon D \to \h$ be a conformal map which fixes $\infty$ and takes $w_0$ to $0$.  Then $\psi(\eta_{\theta_i}^{x_i}|_{[\tau_i,T_i]})$ for $i=1,2$ is the flow line of the GFF $h \circ \psi^{-1} - \chi \arg (\psi^{-1})'$ on $\h$  with angle $\theta_i$.  The result now follows from the induction hypothesis.
\end{proof}

Next, we will extend the result of Proposition~\ref{prop::generalized_monotonicity} to the setting of angle varying flow lines.

\begin{proposition}
\label{prop::angle_varying_monotonicity}
Suppose that $h$ is a GFF on $\h$ with piecewise constant boundary data which changes a finite number of times.  Fix angles $\theta_1,\ldots,\theta_k$ and $\wt{\theta}$ with $\wt{\theta} > \max_i \theta_i$.  Assume
\begin{equation}
\label{eqn::angle_bounds}
|\theta_i - \theta_j| < \frac{2\lambda}{\chi} \quad\text{for all}\quad 1 \leq i,j \leq k.
\end{equation}
Let $\eta := \eta_{\theta_1 \cdots \theta_k}^{\tau_1 \cdots \tau_k}$ be an angle varying flow line of $h$ starting at $0$ and let $\tau_1,\ldots,\tau_{k-1}$ be the corresponding angle change times (recall the definition from the beginning of Section~\ref{subsec::light_cone}).  Let $\wt{\eta}$ be the flow line of $h$ starting at $0$ with angle $\wt{\theta}$.  Assume that $T, \wt{T}$ are stopping times for $\eta, \wt{\eta}$, respectively, such that $\eta|_{[0,T]},\wt{\eta}_{[0,\wt{T}]}$ are both almost surely continuous.  Then $\eta|_{[0,T]}$ almost surely lies to the right of $\wt{\eta}|_{[0,\wt{T}]}$.
\end{proposition}

The hypothesis that the angles $\theta_i$ satisfy~\eqref{eqn::angle_bounds} implies that $\eta$ never crosses itself, though~$\eta$ may hit itself (it turns out that if $|\theta_i - \theta_j| \leq \pi$, then $\eta$ never hits itself).

\begin{proof}[Proof of Proposition~\ref{prop::angle_varying_monotonicity}]
We will prove the result by induction on the number of times that $\eta$ changes angles.  Proposition~\ref{prop::generalized_monotonicity} implies the result for $k=1$ (the fixed angle case).  Suppose $k \geq 2$ and the result holds for $k-1$ (which corresponds to $k-2$ angle changes).  The induction hypothesis implies that $\eta|_{[0,\tau_{k-1} \wedge T]}$ almost surely stays to the right of $\wt{\eta}|_{[0,\wt{T}]}$.  Let $\wt{\tau}$ be the first time that $\wt{\eta}|_{[0,\wt{T}]}$ hits $\eta|_{[\tau_{k-1},T]}$ and let $\tau$ be the first time that $\eta|_{[0,T]}$ hits $\wt{\eta}|_{[0,\wt{T}]}$ after time $\tau_{k-1}$.  If $\tau = \wt{\tau} = \infty$, then the desired result is trivial, so we shall assume that $\tau,\wt{\tau} < \infty$.  The argument of Lemma~\ref{lem::crossing_local_stopping} implies that $K = \eta([0,\tau]) \cup \wt{\eta}([0,\wt{\tau}])$ is a local set for $h$ and that $w_0 = \eta(\tau) = \wt{\eta}(\wt{\tau})$.  Let $D$ be the unbounded connected component of $\h \setminus K$.  Moreover, the argument of Lemma~\ref{lem::crossing_local_stopping} also implies that $\wt{\eta}|_{[\wt{\tau},\wt{T}]}$ is the flow line of $h|_D$ conditional on $\CK$ (where $\CK$ plays the same role for $K$ as $\CA$ from Section~\ref{subsec::local_sets}) with angle $\wt{\theta}$ in $D$ starting at $w_0$ and that $\eta|_{[\tau,T]}$ is the flow line of $h|_D$ conditional on $\CK$ with angle $\theta_k < \wt{\theta}$, also starting at $w_0$.  Consequently, the result follows from Proposition~\ref{prop::generalized_monotonicity}.
\end{proof}

Proposition~\ref{prop::angle_varying_monotonicity} immediately implies the following.  Suppose that $\wt{\theta}_1,\ldots,\wt{\theta}_\ell$ is another collection of angles and $\wt{\eta} = \eta_{\wt{\theta}_1 \cdots \wt{\theta}_\ell}^{\wt{\tau}_1 \cdots \wt{\tau}_\ell}$ is the angle varying flow line with corresponding angle change times $\wt{\tau}_1,\ldots,\wt{\tau}_{\ell-1}$.  If $\min_i \wt{\theta}_i > \max_i \theta_i$ and $T,\wt{T}$ are stopping times for $\eta,\wt{\eta}$, respectively, such that $\eta|_{[0,T]}$ and $\wt{\eta}_{[0,\wt{T}]}$ are both almost surely continuous then $\eta|_{[0,T]}$ almost surely lies to the right of $\wt{\eta}|_{[0,\wt{T}]}$.

Next, we will complete the proof of Theorem~\ref{thm::monotonicity_crossing_merging} (contingent on continuity hypotheses which will be removed in the next subsection) in the following proposition:

\begin{proposition}
\label{prop::merging_and_crossing}
Suppose that $h$ is a GFF on $\h$ with piecewise constant boundary data which changes a finite number of times.  Fix $x_1 > x_2$ and angles $\theta_1,\theta_2$.  For $i=1,2$ let $T_i$ be a stopping time for $\eta_{\theta_i}^{x_i}$ such that $\eta_{\theta_i}^{x_i}|_{[0,T_i]}$ is almost surely continuous.  If $\theta_2 <\theta_1 < \theta_2+\pi$, then $\eta_{\theta_1}^{x_1}|_{[0,T_1]}$ almost surely crosses $\eta_{\theta_2}^{x_2}|_{[0,T_2]}$ upon intersecting.  After crossing, $\eta_{\theta_1}^{x_1}|_{[0,T_1]}$ and $\eta_{\theta_2}^{x_2}|_{[0,T_2]}$ may continue to bounce off of each other, but will never cross again.  If $\theta_1 = \theta_2$, then $\eta_{\theta_1}^{x_1}|_{[0,T_1]}$ merges with $\eta_{\theta_2}^{x_2}|_{[0,T_2]}$ upon intersecting.
\end{proposition}
The statement of the proposition implies that $\eta_{\theta_1}^{x_1}|_{[0,T_1]}$ almost surely crosses $\eta_{\theta_2}^{x_2}|_{[0,T_2]}$ upon intersecting, but it is not necessarily true that $\eta_{\theta_1}^{x_1}|_{[0,T_1]}$ intersects $\eta_{\theta_2}^{x_2}|_{[0,T_2]}$ since  one of the flow lines may get stuck upon hitting the continuation threshold.
\begin{proof}
Let $\tau_1$ be the first time that $\eta_{\theta_1}^{x_1}|_{[0,T_1]}$ intersects $\eta_{\theta_2}^{x_2}|_{[0,T_2]}$ and $\tau_2$ the first time that $\eta_{\theta_2}^{x_2}|_{[0,T_2]}$ intersects $\eta_{\theta_1}^{x_1}|_{[0,T_1]}$.  If $\tau_1 = \tau_2 = \infty$, then the desired result is trivial, so we shall assume that $\tau_1,\tau_2 < \infty$.  Lemma~\ref{lem::crossing_local} implies that $K = \eta_{\theta_1}^{x_1}([0,\tau_1]) \cup \eta_{\theta_2}^{x_2}([0,\tau_2])$ is a local set for $h$ and that $\eta_{\theta_i}^{x_i}|_{[\tau_i,T_i]}$, $i=1,2$, is almost surely contained in the unbounded connected component $D$ of $\h \setminus K$.  By Lemma~\ref{lem::hitting_conditional_law}, we know that $\eta_{\theta_i}^{x_i}|_{[\tau_i,T_i]}$ is the flow line of $h|_D$ given $\CK$ (where $\CK$ plays the same role for $K$ as $\CA$ from Section~\ref{subsec::local_sets}) starting at $\eta_{\theta_1}^{x_1}(\tau_1) = \eta_{\theta_2}^{x_2}(\tau_2)$ with angle $\theta_i$ for $i=1,2$.  If $\theta_2 < \theta_1 < \theta_2+\pi$, then Proposition~\ref{prop::generalized_monotonicity} implies that $\eta_{\theta_1}^{x_1}|_{[\tau_1,T_1]}$ stays to the left of $\eta_{\theta_2}^{x_2}|_{[\tau_2,T_2]}$.  The merging claim comes as a consequence of Theorem~\ref{thm::coupling_uniqueness}, which implies that there is a unique flow line for each given angle.
\end{proof}

\begin{figure}[h!]
\begin{center}
\includegraphics[scale=0.85,page=1]{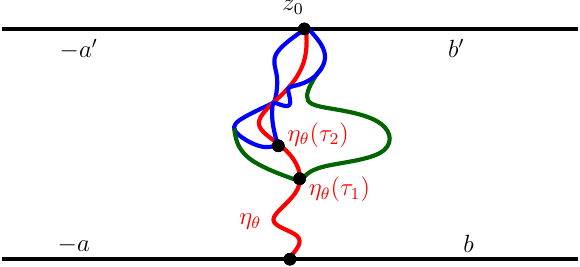}
\end{center}
\caption{\label{fig::merging_through_duality} The merging phenomenon of Figure~\ref{fig::merging_and_crossing_phases} and Proposition~\ref{prop::merging_and_crossing} can also be seen through the light cone perspective of $\SLE_{16/\kappa}$.  To see this, fix a GFF $h$ on the strip $\strip$ with the boundary data above.  Assume for simplicity that $a',b' \geq \lambda' + \pi \chi$ so that $\eta'$ almost surely does not intersect $\striptop$ except at $z_0$.  Let $\eta_\theta$ be the flow line of $h$ of angle $\theta \in (-\tfrac{\pi}{2},\tfrac{\pi}{2})$ starting at $0$.  Let $\tau_1 < \tau_2$ be stopping times for $\eta_\theta$.  Then we know by Proposition~\ref{prop::light_cone_construction} that the flow lines $\eta_i$, $i=1,2$, which run along $\eta_\theta$ until time $\tau_i$ and then flow at angle $\tfrac{1}{\chi}(\lambda-\lambda') = \tfrac{\pi}{2}$ are almost surely the left boundary of the counterflow line $\eta'$ starting at $z_0$ upon hitting $\eta_\theta(\tau_i)$.  This implies that $\eta_1,\eta_2$ almost surely merge and then never separate since $\eta'$ does not cross itself.  This gives the merging result with constant boundary data since the conditional law of $h$ given $\eta_\theta$ in the left connected component of $\strip \setminus \eta_\theta$ close to $\eta_\theta(\tau_1),\eta_\theta(\tau_2)$ looks like a GFF with constant boundary data, provided $\tau_1$ and $\tau_2$ are chosen to be very close to each other.}
\end{figure}

\begin{figure}[h!]
\begin{center}
\includegraphics[scale=0.85,page=2]{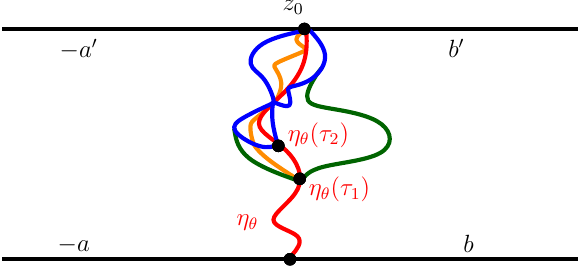}
\end{center}
\caption{\label{fig::crossing_through_duality} (Continuation of Figure~\ref{fig::merging_through_duality}.) It is also possible to see the crossing phenomenon of Figure~\ref{fig::merging_and_crossing_phases} and Proposition~\ref{prop::merging_and_crossing} through light cones and duality.  Indeed, assume we have the same setup as Figure~\ref{fig::crossing_through_duality} except we take the angle of $\eta_1$ at $\eta_\theta(\tau_1)$ to be an intermediate value in the range $(\theta,\tfrac{\pi}{2})$ (orange curve).  Then $\eta_1$ has to cross $\eta_2$ since $\eta'$ swallows $\eta_1$ (Lemma~\ref{lem::light_cone_contains_av}) and $\eta_2$ contains the left boundary of $\eta'$ when it hits $\eta_\theta(\tau_2)$ (Proposition~\ref{prop::light_cone_construction}).  It is impossible for $\eta_1$ to cross $\eta_2$ subsequently since $\eta'$ swallows the points in $\eta_1$ in reverse chronological order (Lemma~\ref{lem::light_cone_contains_av}).}
\end{figure}

\begin{remark}
\label{rem::intersect_almost_surely}
We remark that it is straightforward in the setting of Proposition~\ref{prop::merging_and_crossing} to compute the conditional law of $\eta_{\theta_1}^{x_1}$ given $\eta_{\theta_2}^{x_2}$ before $\eta_{\theta_1}^{x_1}$ crosses $\eta_{\theta_2}^{x_2}$.  For simplicity, we assume that the boundary data of $h$ is given by some constant, say $c$ (note that Proposition~\ref{prop::two_force_point_uniqueness_and_continuity} implies that $\eta_{\theta_i}^{x_i}$ is almost surely continuous for all time, $i=1,2$).  In this case, the conditional law is that of an $\SLE_{\kappa}(\rho^{1,L},\rho^{2,L};\rho^{1,R})$ process in the connected component of $\h \setminus \eta_{\theta_2}^{x_2}$ which contains $x_1$ where
\[ \rho^{1,L} = -\frac{\theta_1 \chi +c}{\lambda} - 1,\ \ \ \rho^{1,L} + \rho^{2,L} = -\frac{(\theta_1-\theta_2) \chi}{\lambda}-2,\ \ \ \rho^{1,R} = \frac{\theta_1 \chi + c}{\lambda} -1.\]
In particular, since $\theta_1 \geq \theta_2$, we have that $\rho^{1,L} + \rho^{2,L} \leq -2$.  This implies that in this case, $\eta_{\theta_1}^{x_1}$ almost surely intersects (hence crosses) $\eta_{\theta_2}^{x_2}$.  This holds more generally whenever we have boundary data which is piecewise constant and is such that $\eta_{\theta_1}^{x_1}$ and $\eta_{\theta_2}^{x_2}$ can be continued upon intersecting $\partial \h$.  The facts summarized here will be useful for us in Section~\ref{subsec::many_boundary_force_points} because we will employ them in order to prove the almost sure continuity of $\SLE_{\kappa}(\ul{\rho})$ processes right up to when the continuation threshold is hit.
\end{remark}

\begin{remark}
\label{rem::crossing_boundary_data}
It is also straightforward in the setting of Proposition~\ref{prop::merging_and_crossing} to compute the conditional law of the segment of the path $\eta_{\theta_1}^{x_1}$ after it has crossed $\eta_{\theta_2}^{x_2}|_{[0,T_2]}$.  For example, in the special case that $h$ has constant boundary data $c$ (as before, we already know that both paths are continuous from Proposition~\ref{prop::two_force_point_uniqueness_and_continuity}), the conditional law of the segment of $\eta_{\theta_1}^{x_1}$ after it has crossed $\eta_{\theta_2}^{x_2}$ given both
\begin{enumerate}
\item its realization up until hitting $\eta_{\theta_2}^{x_2}$ and
\item the entire realization of $\eta_{\theta_2}^{x_2}$
\end{enumerate}
viewed as a path in the component of $\h \setminus \eta_{\theta_2}^{x_2}$ in which it immediately enters after crossing is that of an $\SLE_\kappa(\rho^{1,L},\rho^{2,L};\rho^{1,R})$ where
\[ \rho^{1,L} = \frac{(\theta_2-\theta_1) \chi}{\lambda},\ \ \ \rho^{1,L} + \rho^{2,L} = -\frac{\theta_1\chi + c}{\lambda} - 1,\ \ \ \rho^{1,R} = \frac{(\theta_1-\theta_2) \chi}{\lambda} -2.\]
\end{remark}

\begin{remark}
\label{rem::merging_crossing_duality}
Both the merging and crossing phenomena described in Proposition~\ref{prop::merging_and_crossing} can also be seen as a consequence of the light cone construction of counter flow lines described in Section~\ref{subsec::light_cone}.  The former is explained in Figure~\ref{fig::merging_through_duality} and the latter is in Figure~\ref{fig::crossing_through_duality}.
\end{remark}

\subsection{Continuity for many boundary force points}
\label{subsec::many_boundary_force_points}

We will now complete the proof of Theorem~\ref{thm::continuity} for $\kappa \in (0,4]$ by extending the special case proved in Proposition~\ref{prop::two_force_point_uniqueness_and_continuity} to the setting of many boundary force points.  We begin by noting that absolute continuity (Proposition~\ref{prop::gff_abs_continuity}) along with the two force point case implies that in this more general setting, $\eta \sim \SLE_{\kappa}(\ul{\rho})$ is almost surely a continuous curve when it hits $\partial \h$ between force points before the continuation threshold is hit.  Indeed, in this case absolute continuity implies that $\eta$ locally evolves like an $\SLE_\kappa(\rho)$ process with just one force point with weight $\rho > -2$.  Thus to get the continuity of a general $\SLE_\kappa(\ul{\rho})$ process, we need to rule out pathological behavior when $\eta$ interacts with a force point or hits the boundary at the continuation threshold.  We will accomplish the former in the next series of lemmas, in which we systematically study the behavior of $\SLE_\kappa(\rho^{1,R},\rho^{2,R})$ processes in $\h$ from $0$ to $\infty$ with two force points located to the right of $0$.  We will show that if $\rho^{1,R}, \rho^{1,R} + \rho^{2,R} > -2$, then $\eta$ almost surely does not hit its force points.  We will prove the continuity right at the continuation threshold in Lemma~\ref{lem::cont_continuation_threshold}.  At the continuation threshold, it turns out that whether or not $\eta$ hits a particular force point depends on the sum of the weights.  This is natural to expect in view of Lemma~\ref{lem::hitting_single_point}.

The first lemma of the subsection is a simple technical result which states that the set which consists of those points where an $\SLE_\kappa(\rho)$ process with a single force point of weight $\rho > -2$ is in $\partial \h$ almost surely has zero Lebesgue measure.  This will be employed in Lemma~\ref{lem::not_intersect_force_point}, which handles the regime of $\rho^{1,R},\rho^{2,R}$ where either $|\rho^{2,R}| < 2$ or $\rho^{1,R} \in (-2,\tfrac{\kappa}{2} - 2)$.  In Lemma~\ref{lem::not_intersect_force_point2}, we will weaken the hypothesis to $\rho^{1,R},\rho^{1,R}+\rho^{2,R} > -2$.  The reason that we need to make the stronger hypothesis in Lemma~\ref{lem::not_intersect_force_point} is that its proof will proceed in analogy with the argument given in Section~\ref{subsec::two_boundary_force_points}, except rather than conditioning on flow lines with angles $\theta_1 < 0 < \theta_2$, we will condition on an angle-varying flow line.  The hypothesis that $|\rho^{2,R}| < 2$ implies that the angle-varying flow line does not cross itself.

\begin{figure}[h!]
\begin{center}
\includegraphics[scale=0.85]{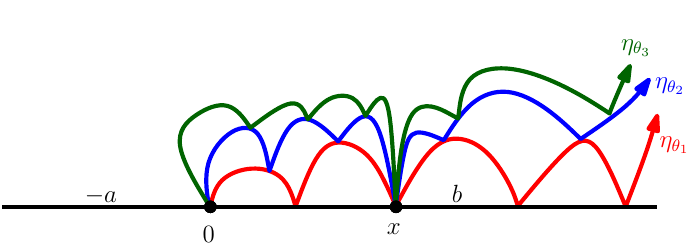}
\end{center}
\caption{\label{fig::sle_kappa_rho_boundary_intersection_zero_lebesgue_measure}  Suppose that $h$ is a GFF on $\h$ whose boundary data is as depicted above where $a,b$ are very large.  Fix $\rho^L > -2$ and $\rho^R \in (-2,\tfrac{\kappa}{2}-2)$.  Choose angles $\theta_1 < \cdots < \theta_n$ and, for each $i$, let $\eta_{\theta_i}$ be the flow line of $h$ with angle $\theta_i$.  Assume that $\eta_{\theta_1} \sim \SLE_{\kappa}(\rho_1^L;\rho^R)$ and that the angles $\theta_2,\ldots,\theta_n$ are such that, for each $k \in \{2,\ldots,n\}$, the conditional law of $\eta_{\theta_k}$ given $\eta_{\theta_{k-1}}$ is an $\SLE_{\kappa}(\rho_k^L;\rho^R)$ process in the left connected component of $\h \setminus \eta_{\theta_{k-1}}$.  By scale invariance, the probability that $\eta_{\theta_k}$ conditional on $\eta_{\theta_{k-1}}$ hits a particular point $x \neq 0$ in $\eta_{\theta_{k-1}}$ is a function of $\rho_k^L$, say $p(\rho_k^L)$, but not $x$.  This implies that the probability that $\eta_{\theta_n}$ hits $x \in \R_+$ is $p = p(\rho_1^L) \cdots p(\rho_n^L)$.  By choosing $n$ large enough, we can arrange so that $\eta_{\theta_n} \sim \SLE_\kappa(\rho_n^L;\rho_n^R)$ with $\rho_n^R \geq \tfrac{\kappa}{2}-2$.  This implies that $p = 0$, so there exists $1 \leq k_0 \leq n$ such that $p(\rho_{k_0}^L) = 0$.  From this, it is possible to see that the probability that an $\SLE_\kappa(\rho^L;\rho^R)$ process hits a particular point on $\R_+$ is zero (see Lemma~\ref{lem::sle_does_not_cover_boundary}).
}
\end{figure}

\begin{lemma}
\label{lem::sle_does_not_cover_boundary}
Suppose that $\eta \sim \SLE_\kappa(\rho^L;\rho^R)$ in $\h$ from $0$ to $\infty$ with $\rho^L > -2$ and $\rho^R \in (-2,\tfrac{\kappa}{2}-2)$ where the force points are located at $0^-$ and $0^+$, respectively.  The Lebesgue measure of $\eta \cap \partial \h$ is almost surely zero.  In particular, for any $x \in \partial \h \setminus \{0\}$, the probability that $\eta$ hits $x$ is zero.
\end{lemma}
\begin{proof}
Suppose that $h$ is a GFF on $\h$ whose boundary data is as in Figure~\ref{fig::sle_kappa_rho_boundary_intersection_zero_lebesgue_measure}.  By choosing $a,b$ very large, we can pick angles $\theta_1 < \cdots < \theta_n$ such that $\eta_{\theta_1}$ is an $\SLE_\kappa(\rho_1^L;\rho^R)$ process in $\h$ from $0$ to $\infty$ and, for each $k \in \{2,\ldots,n\}$, the law of $\eta_{\theta_k}$ given $\eta_{\theta_{k-1}}$ is an $\SLE_\kappa(\rho_k^L;\rho^R)$ process in the left connected component of $\h \setminus \eta_{\theta_{k-1}}$ (Proposition~\ref{prop::monotonicity_boundary_intersecting}) and that $\rho_k^L \geq 0$ for all $k \in \{1,\ldots,n\}$.  By the scale invariance of $\SLE_{\kappa}(\rho_k^L;\rho^R)$ processes with force points at $0^-,0^+$, there exists $p(\rho_k^L)$ such that the probability that $\eta_{\theta_k}$ hits any particular point $x \neq 0$ in $\eta_{\theta_{k-1}}$ is $p(\rho_k^L)$.  By choosing $n$ large enough, we can arrange that $\eta_{\theta_n} \sim \SLE_\kappa(\rho_n^L;\rho_n^R)$ with $\rho_n^R \geq \tfrac{\kappa}{2}-2$.  Fix $x \in \R_+$.  This implies that $\p[x \in \eta_{\theta_n}] = 0$ (see Lemma~\ref{lem::flow_cannot_hit}).  On the other hand, we also have that
\[ \p[x \in \eta_{\theta_n}] = p(\rho_1^L) \cdots p(\rho_n^L).\]
This implies there exists $1 \leq k_0 \leq n$ such that $p(\rho_{k_0}^L) = 0$, i.e.\ so that the probability that an $\SLE_\kappa(\rho_{k_0}^L;\rho^R)$ process hits any particular point $x \in \R_+$ is $0$.

To get the result for general choices of $\rho^L$, we fix $\wt{\rho}^L > \max(\rho^L;\rho_{k_0}^L)+2$ and let $h$ be a GFF on $\h$ whose boundary data is such that the zero angle flow line $\eta$ of $h$ is an $\SLE_\kappa(\wt{\rho}^L;\rho^R)$ process.  Then by choosing $\theta_1 = \lambda(\rho^L+2)/\chi$ and $\theta_2 = \lambda(\rho_{k_0}^L+2)/\chi$ and letting $\eta_{\theta_i}$ be the flow line of $h$ with angle $\theta_i$, $i=1,2$, Proposition~\ref{prop::monotonicity_boundary_intersecting} implies that the conditional law of $\eta$ given $\eta_{\theta_1}$ is an $\SLE_\kappa(\rho^L;\rho^R)$ process and the conditional law of $\eta$ given $\eta_{\theta_2}$ is an $\SLE_{\kappa}(\rho_{k_0}^L;\rho^R)$ process.  Since the probability that the latter hits any particular point $x \in \R_+$ is zero, it follows the same is true for the former.
\end{proof}

\begin{figure}[ht!]
\begin{center}
\includegraphics[scale=0.85]{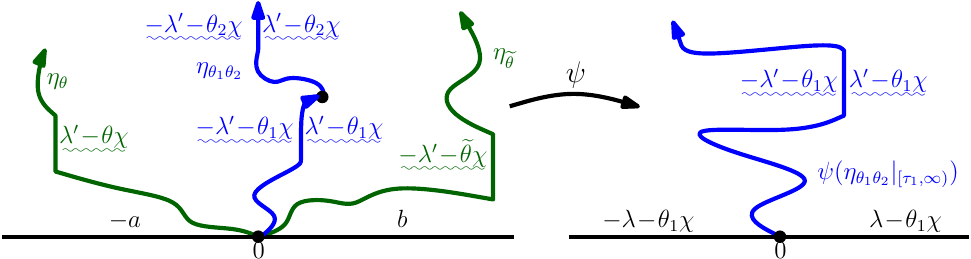}
\end{center}
\caption{\label{fig::two_angle_change_cts}  Suppose that $h$ is a GFF on $\h$ whose boundary data is as depicted above.  Assume that $|\theta_1-\theta_2| < 2\lambda/\chi$.  Let $\eta_{\theta_1 \theta_2}^{\tau_1 \tau_2}$ be the angle varying flow line of $h$ starting at $0$ with angle change time $\tau_1 > 0$ and angles $\theta_1,\theta_2$.  We can see that $\eta_{\theta_1 \theta_2}^{\tau_1 \tau_2}$ is continuous by taking $\theta$ (resp. $\wt{\theta}$) to be such that $\lambda - \theta \chi = -\lambda - \theta_1 \chi$ (resp. $-\lambda-\wt{\theta} \chi = \lambda - \theta_1 \chi$) and let $\eta_{\theta}$ (resp. $\eta_{\wt{\theta}})$ be the flow line of $h$ starting at $0$ with angle $\theta$ (resp. $\wt{\theta}$).  The conditional law of $\eta_{\theta_1 \theta_2}^{\tau_1 \tau_2}|_{[\tau_1,\infty)}$ given $\eta_{\theta_1 \theta_2}([0,\tau_1])$ and $\eta_{\theta}$ and $\eta_{\wt{\theta}}$ is that of an $\SLE_\kappa((\theta_1-\theta_2) \chi/\lambda; (\theta_2-\theta_1)\chi/\lambda)$ process (justified in the proof of Lemma~\ref{lem::av_continuous}).  Therefore Proposition~\ref{prop::two_force_point_uniqueness_and_continuity} implies that $\eta_{\theta_1 \theta_2}^{\tau_1 \tau_2}$ is almost surely continuous.
}
\end{figure}

\begin{lemma}
\label{lem::av_continuous}
Suppose that $h$ is a GFF on $\h$ whose boundary data is as depicted in Figure~\ref{fig::two_angle_change_cts}.  Fix angles $\theta_1,\theta_2$ such that $|\theta_1 - \theta_2| < 2\lambda/\chi$ and such that $b+\theta_i \chi > -\lambda$ and $-a+\theta_i \chi < \lambda$ for $i=1,2$.  Let $\eta_{\theta_1 \theta_2}^{\tau_1 \tau_2}$ be the angle varying flow line of $h$ starting from $0$ with angles $\theta_1,\theta_2$ and angle change time $\tau_1 > 0$.  Then $\eta_{\theta_1 \theta_2}^{\tau_1 \tau_2}$ is almost surely continuous.
\end{lemma}
\begin{proof}
See Figure~\ref{fig::two_angle_change_cts} for an illustration of the argument.  We pick $\theta$ (resp.\ $\wt{\theta}$) so that $\lambda - \theta\chi = -\lambda - \theta_1 \chi$, i.e. $\theta \chi = 2\lambda + \theta_1 \chi$ (resp.\ $-\lambda - \wt{\theta} \chi = \lambda - \theta_1 \chi$, i.e. $\wt{\theta} \chi = -2\lambda + \theta_1\chi$) and let $\eta_\theta$ (resp. $\eta_{\wt{\theta}}$) be the flow line of $h$ with angle $\theta$ (resp. $\wt{\theta}$) starting at $0$.  Note that $\wt{\theta} < \theta_1,\theta_2 < \theta$.  We first assume that $a,b$ are sufficiently large so that $\eta_\theta$ and $\eta_{\wt{\theta}}$ almost surely do not intersect $\partial \h$ after time $0$.  Fix $\epsilon > 0$ and let $T_\epsilon$ be the first time $t \geq \tau_1$ that $\eta_{\theta_1 \theta_2}^{\tau_1 \tau_2}|_{[\tau_1,\infty)}$ comes within distance $\epsilon$ of $\partial \h$.  Proposition~\ref{prop::two_force_point_uniqueness_and_continuity} implies that $\eta_{\theta_1 \theta_2}^{\tau_1 \tau_2}|_{[0,\tau_1]}$, $\eta_\theta$, and $\eta_{\wt{\theta}}$ are almost surely continuous and Proposition~\ref{prop::gff_abs_continuity} implies that $\eta_{\theta_1 \theta_2}^{\tau_1 \tau_2}|_{[\tau_1,T_\epsilon]}$ is almost surely continuous since its law is mutually absolutely continuous with respect to that of an $\SLE_\kappa(\wt{\rho}^L;\wt{\rho}^R)$ process in $\h \setminus \eta_{\theta_1 \theta_2}^{\tau_1 \tau_2}([0,\tau_1])$ with $\wt{\rho}^L = (\theta_1-\theta_2)\chi/\lambda$ and $\wt{\rho}^R = (\theta_2-\theta_1)\chi/\lambda$.  Consequently, Proposition~\ref{prop::angle_varying_monotonicity} implies that $\eta_{\theta_1 \theta_2}^{\tau_1 \tau_2}|_{[\tau_1,T_\epsilon]}$ stays to the right of $\eta_\theta$ and to the left of $\eta_{\wt{\theta}}$.  The argument of Remark~\ref{rem::angle_varying_mean_height} implies that the conditional expectation of $h$ given $\eta_{\theta_1 \theta_2}^{\tau_1 \tau_2}$, $\eta_\theta$, and $\eta_{\wt{\theta}}$ does not have singularities at points where $\eta_{\theta_1 \theta_2}^{\tau_1 \tau_2}$ intersects one of $\eta_\theta$ or $\eta_{\wt{\theta}}$ (and the same holds when we run $\eta_{\theta_1 \theta_2}^{\tau_1 \tau_2}|_{[0,T_\epsilon]}$ up to a stopping time by Proposition~\ref{gff::prop::cond_union_mean}).  Moreover, the argument of Remark~\ref{rem::cont_loewner_av} implies that $\eta_{\theta_1 \theta_2}^{\tau_1 \tau_2}|_{[\tau_1,T_\epsilon]}$ has a continuous Loewner driving function viewed as a path in the connected component of $\h \setminus (\eta_\theta \cup \eta_{\wt{\theta}} \cup \eta_{\theta_1 \theta_1}^{\tau_1 \tau_2}([0,\tau_1]))$ which lies between $\eta_{\theta}$ and $\eta_{\wt{\theta}}$.  Consequently,  Theorem~\ref{thm::martingale} combined with Proposition~\ref{prop::cond_mean_continuous} together imply that the conditional law of $\eta_{\theta_1 \theta_2}^{\tau_1 \tau_2}|_{[\tau_1,T_\epsilon]}$ given $\eta_{\theta_1 \theta_2}^{\tau_1 \tau_2}([0,\tau_1])$, $\eta_\theta$, and $\eta_{\wt{\theta}}$ is an $\SLE_\kappa(\wt{\rho}^{L};\wt{\rho}^{R})$ process with the same weights $\wt{\rho}^L,\wt{\rho}^R$ as before.  Lemma~\ref{lem::sle_does_not_cover_boundary} implies that the distance $\eta_{\theta_1 \theta_2}^{\tau_1 \tau_2}|_{[\tau_1,T_\epsilon]}$ comes within $0$ remains strictly positive as $\epsilon \to 0$ (since $\eta_{\theta_1 \theta_2}^{\tau_1 \tau_2}|_{[\tau_1,T_\epsilon]}$ is continuous and almost surely does not hit $0^+$ or $0^-$).  Since $\eta_{\theta}$ and $\eta_{\wt{\theta}}$ otherwise do not intersect $\partial \h$ (since we picked $a,b > 0$ large), it follows that $T_\epsilon \to \infty$ as $\epsilon \to 0$ almost surely.  Therefore $\eta_{\theta_1 \theta_2}^{\tau_1 \tau_2}$ is almost surely continuous.  To see the result for general choices of $a,b$, we apply the same argument used to prove Proposition~\ref{prop::two_force_point_uniqueness_and_continuity} (we condition on flow lines with appropriately chosen angles and use the continuity of $\eta_{\theta_1 \theta_2}^{\tau_1 \tau_2}$ for large $a,b$).
\end{proof}

\begin{figure}[ht!]
\begin{center}
\includegraphics[scale=0.85]{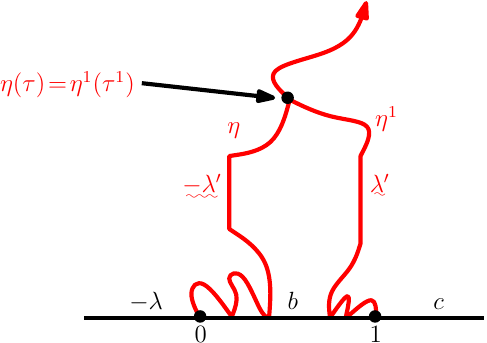}
\end{center}
\caption{\label{fig::flow_line_cannot_hit} Suppose that $h$ is a GFF whose boundary data is as depicted on the right side with $b \in (-\lambda,\lambda-\pi \chi)$ and $c \geq \lambda - \pi \chi$.  We can see that the flow line $\eta$ of $h$ does not intersect $[1,\infty)$ by noting that, if it did, it would have to intersect hence merge with the flow line $\eta^1$ of $h$ starting at $1$ (Proposition~\ref{prop::merging_and_crossing}).  Letting $\tau$ be the first time $\eta$ hits $\eta^1$ and $\tau^1$ the first time $\eta^1$ hits $\eta$, the conditional law of $\eta$ given $K = \eta([0,\tau]) \cup \eta^1([0,\tau^1])$ is that of an $\SLE_\kappa(\rho)$ process with $\rho \geq \tfrac{\kappa}{2}-2$ in the unbounded connected component of $\h \setminus K$.  By Lemma~\ref{lem::hitting_and_bouncing}, we know that such processes do not hit the boundary, which implies that $\eta$ cannot hit $[1,\infty)$.}
\end{figure}

\begin{lemma}
\label{lem::not_intersect_force_point}
Suppose $\rho^{1,R},\rho^{1,R}+\rho^{2,R} > -2$.  Additionally, assume that either
\[ |\rho^{2,R}| < 2  \quad\text{or}\quad \rho^{1,R} < \frac{\kappa}{2}-2.\]
Let $\eta$ be an $\SLE_\kappa(\rho^{1,R},\rho^{2,R})$ process in $\h$ from $0$ to $\infty$ where the force points corresponding to the weights $\rho^{1,R}$ and $\rho^{2,R}$ are located at $0^+$ and $1$, respectively.  Then $\eta$ almost surely does not hit $1$ and is generated by a continuous curve.
\end{lemma}
\begin{proof}
Assume $|\rho^{2,R}| < 2$.  Let
\[ \theta_1 = -\frac{\lambda}{\chi}\big(2 + \rho^{1,R} \big) \quad\text{and}\quad \theta_2 = -\frac{\lambda}{\chi}\big(2 + \rho^{1,R} + \rho^{2,R} \big).\]
Suppose that $h$ is a GFF with boundary data as in Figure~\ref{fig::flow_line_angles} where $a = \lambda$ and $b > 0$ is sufficiently large so that the angle-varying flow line $\eta_{\theta_1 \theta_2}^{\tau_1 \tau_2}$ with angles $\theta_1,\theta_2$ and angle change time $\tau_1 = 1$ starting at $0$ almost surely does not hit $\partial \h$ after time $0$.  Note that $|\rho^{2,R}| < 2$ implies $|\theta_1-\theta_2| < 2\lambda/\chi$, which is the condition necessary for $\eta_{\theta_1 \theta_2}^{\tau_1 \tau_2}$ not to cross itself.  Let $\eta$ be the zero angle flow line of $h$ starting at $0$.  Proposition~\ref{prop::two_force_point_uniqueness_and_continuity} implies that $\eta$ is continuous and Lemma~\ref{lem::av_continuous} implies that $\eta_{\theta_1 \theta_2}^{\tau_1 \tau_2}$ is continuous.  Thus by Proposition~\ref{prop::angle_varying_monotonicity}, we know that $\eta$ is almost surely to the left of $\eta_{\theta_1 \theta_2}^{\tau_1 \tau_2}$ because $\theta_1,\theta_2 < 0$.

Lemma~\ref{lem::av_continuous} implies that the left connected component $C$ of $\h \setminus \eta_{\theta_1 \theta_2}^{\tau_1 \tau_2}$ is almost surely a Jordan domain.  Let $\psi \colon C \to \h$ be a conformal map which fixes $0$ and $\infty$.  Then $\psi$ extends continuously to $\partial C$.  Therefore $\psi(\eta)$ is a continuous path in $\ol{\h}$.  Moreover, $\psi(\eta)$ has a continuous Loewner driving function by Remark~\ref{rem::cont_loewner_av}. (Remark~\ref{rem::cont_loewner_av} assumed $|\theta_1-\theta_2| \leq \pi$ to ensure the continuity of $\eta_{\theta_1 \theta_2}^{\tau_1 \tau_2}$ and that $\eta_{\theta_1 \theta_2}^{\tau_1 \tau_2}$ is almost surely determined by $h$.  We have now proved both of these facts in the setting we consider here.)  The boundary data for the conditional law of $h$ given $\eta_{\theta_1 \theta_2}^{\tau_1 \tau_2}$ is shown in Figure~\ref{fig::multiple_force_points_angle_change}.  Consequently, it follows from Proposition~\ref{prop::cond_mean_continuous} and Theorem~\ref{thm::martingale} that the conditional law of $\eta$ given $\eta_{\theta_1 \theta_2}^{\tau_1 \tau_2}$ is that of an $\SLE_\kappa(\rho^{1,R},\rho^{2,R})$ process.  Therefore $\SLE_\kappa(\rho^{1,R},\rho^{2,R})$ processes are continuous provided $|\rho^{2,R}| < 2$.

To see the second claim of the lemma when $|\rho^{2,R}| < 2$, we take $\max(\theta_1,\theta_2) < \theta < 0$.  Then $\eta_\theta$ lies to the left of $\eta_{\theta_1 \theta_2}^{\tau_1 \tau_2}$ and to the right of $\eta$.  Let $z_0$ be the leftmost point of the intersection of $\eta_{\theta_1 \theta_2}^{\tau_1 \tau_2}([\tau_1,\infty))$ with $\eta_{\theta_1 \theta_2}^{\tau_1 \tau_2}([0,\tau_1])$.  If $\eta_{\theta}$ does not hit $z_0$, then the desired claim follows.  If $\eta_{\theta}$ does hit $z_0$, then we know that $\eta$ almost surely does not hit $z_0$ because the conditional law of $\eta$ given $\eta_\theta$ and $\eta_{\theta_1 \theta_2}^{\tau_1 \tau_2}$ is that of an $\SLE_\kappa(- \theta \chi/\lambda-2)$ process in the left connected component of $\h \setminus \eta_\theta$ (Proposition~\ref{prop::two_force_point_uniqueness_and_continuity}), hence we can apply Lemma~\ref{lem::sle_does_not_cover_boundary}.

Alternatively, suppose $\rho^{1,R} < \tfrac{\kappa}{2}-2$.  There are two possibilities.  If $\rho^{1,R} + \rho^{2,R} \in (-2,\tfrac{\kappa}{2}-2)$ then $|\rho^{2,R}| < 2$ so that the result in this case follows as before.  Suppose $\rho^{1,R} + \rho^{2,R} \geq \tfrac{\kappa}{2}-2$.  Assume that $\eta$ is coupled with a GFF $h$ so that $\eta$ is its flow line starting from $0$.  We claim that $\eta$ almost surely does not hit $[1,\infty)$, in which case we are done because then Proposition~\ref{prop::gff_abs_continuity} implies that the law of $\eta$ (stopped upon exiting a ball of any finite size) is mutually absolutely continuous with respect to the law of an $\SLE_\kappa(\rho^{1,R},\wt{\rho}^{2,R})$ process, $\wt{\rho}^{2,R}$ such that $\rho^{1,R} + \wt{\rho}^{2,R} = \tfrac{\kappa}{2}-2$, (stopped upon exiting a ball of the same finite size) and we know from the previous argument that such processes are continuous.

To see that $\eta$ does not hit $[1,\infty)$, we let $\eta^1$ be the flow line of $h$ (with angle $0$) starting at $1$.  Let $T_\epsilon$ (resp. $T_\epsilon^1$) be the first time $t$ that $\eta$ (resp. $\eta^1$) gets within distance $\epsilon$ of $[1,\infty)$ (resp. $(-\infty,0]$) and let $T = \lim_{\epsilon \to 0} T_\epsilon$ (resp. $T^1 = \lim_{\epsilon \to 0} T_\epsilon^1$).  On the event $\{T < \infty\} \cup \{T^1 < \infty\}$, we have that $\eta|_{[0,T)}$ and $\eta^1|_{[0,T^1)}$ intersect.  Since both of the (restricted) paths are continuous by Proposition~\ref{prop::gff_abs_continuity} and Proposition~\ref{prop::two_force_point_uniqueness_and_continuity}, Proposition~\ref{prop::merging_and_crossing} then implies that $\eta^1$ and $\eta$ merge.  Let $\tau$ be the first time $\eta$ hits $\eta^1$ and $\tau^1$ be the first time $\eta^1$ hits $\eta$ and let $K = \eta([0,\tau]) \cup \eta^1([0,\tau^1])$.  The conditional law of $\eta$ given $K$ is that of an $\SLE_\kappa(\rho^{1,R}+\rho^{2,R})$ process in the unbounded connected component of $\h \setminus K$ (Proposition~\ref{prop::merging_and_crossing}).  Therefore Lemma~\ref{lem::hitting_and_bouncing} implies that $\eta$ almost surely does not hit $[1,\infty)$, which is a contradiction (see Figure~\ref{fig::flow_line_cannot_hit}).
\end{proof}

\begin{remark}
\label{rem::not_intersect_force_point_extra}
A slight modification of the proof of Lemma~\ref{lem::not_intersect_force_point} implies the continuity of $\SLE_\kappa(\rho^{1,L};\rho^{1,R},\rho^{2,R})$ processes where $\rho^{1,L} > -2$ and the same hypotheses are made on $\rho^{1,R},\rho^{2,R}$.  Indeed, this is accomplished by conditioning on the flow line with angle $(2+\rho^{1,L})\tfrac{\lambda}{\chi} > 0$.
\end{remark}

Lemma~\ref{lem::not_intersect_force_point} requires that if $\rho^{1,R} \geq \tfrac{\kappa}{2}-2$ then $\rho^{1,R} + \rho^{2,R}$ is larger than both $\rho^{1,R}-2$ and $-2$.  The purpose of the next lemma is to remove this restriction.

\begin{lemma}
\label{lem::not_intersect_force_point2}
Suppose
\[ \rho^{1,R} \geq \frac{\kappa}{2}-2 \quad\text{and}\quad \rho^{1,R}+\rho^{2,R} > -2.\]
Let $\eta$ be an $\SLE_\kappa(\rho^{1,R},\rho^{2,R})$ process in $\h$ where the force points corresponding to the weights $\rho^{1,R}$ and $\rho^{2,R}$ are located at $0^+$ and $1$, respectively.  Then $\eta$ almost surely does not hit $1$ and is generated by a continuous curve.
\end{lemma}
\begin{proof}
We may assume without loss of generality that $\rho^{1,R}+\rho^{2,R} \in (-2,\tfrac{\kappa}{2}-2)$ since if $\rho^{1,R}+\rho^{2,R} \geq \tfrac{\kappa}{2}-2$, we know by Remark~\ref{rem::continuity_non_boundary} that $\eta$ almost surely does not hit $\partial \h$ after time $0$ and is almost surely continuous.  Assume that $\eta$ is coupled with a GFF $h$ as in Theorem~\ref{thm::coupling_existence}.  By Proposition~\ref{prop::two_force_point_uniqueness_and_continuity} and Proposition~\ref{prop::gff_abs_continuity} we know that $\eta$ is continuous, at least up until just before the first time $\tau$ it accumulates in $[1,\infty)$.  Thus, we just need to show that $\eta$ is continuous at time $\tau$, that $\eta$ is a continuous curve for $t \geq \tau$, and that $\eta(\tau) \neq 1$.

To see that $\eta$ does not accumulate at $1$, we apply a conformal map $\psi$ taking $\h$ to the strip $\strip$ as in Figure~\ref{fig::criticalforintersection} with $0$ fixed, $1$ going to $+\infty$, and $\infty$ going to $-\infty$.  Then Figure~\ref{fig::hittingrange} implies $\psi(\eta)$ almost surely hits $\partial \strip$ after time $0$ when it accumulates on the upper boundary of $\strip$, which in turn implies that $\eta$ almost surely does not accumulate at $1$.  This implies that, by Proposition~\ref{prop::gff_abs_continuity} and Theorem~\ref{thm::coupling_uniqueness}, for any fixed $t_0 > 0$ the law of $\eta|_{[t_0,\infty)}$ conditional on $\eta([0,t_0])$ is absolutely continuous with respect to that of an $\SLE_\kappa(\rho^{1,R}+\rho^{2,R})$ process.  Therefore the continuity of $\eta$ up to time $\tau$ follows from Proposition~\ref{prop::two_force_point_uniqueness_and_continuity}.  After time $\tau$, we know that $\eta$ is a continuous curve since conditional on $\eta([0,\tau])$, $\eta$ evolves as an $\SLE_\kappa(\rho^{1,R}+\rho^{2,R})$ process in the unbounded connected component $C$ of $\h \setminus \eta([0,\tau])$.  We know that such processes are generated by continuous curves by Proposition~\ref{prop::two_force_point_uniqueness_and_continuity} in $\h$.  Since $C$ is a Jordan domain, the Loewner map $f_\tau$ extends as a homeomorphism to $\partial (\h \setminus K_\tau)$, hence we see that $\eta$ for $t \geq \tau$ is also a continuous curve.
\end{proof}

\begin{figure}[ht!]
\begin{center}
\includegraphics[scale=0.85]{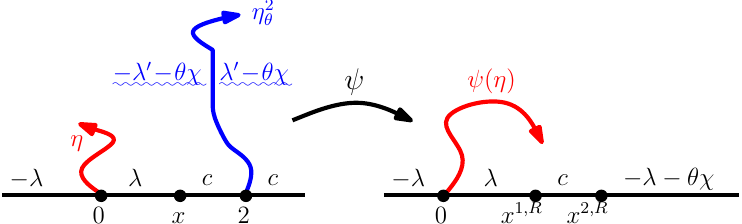}
\end{center}
\caption{\label{fig::cont_continuation_threshold1}  Suppose that $h$ is a GFF on $\h$ whose boundary data is as depicted on the left side where $c \in (-\lambda,\lambda-\pi\chi]$.  Let $\eta_\theta^2$ be the flow line of $h$ starting at $2$ with angle $\theta \in [0,\pi)$.  Note that $c+\theta \chi \in (-\lambda, \lambda)$ and $-\lambda + \theta \chi \leq -\lambda + \pi \chi$.  Thus $\eta_\theta^2$ exists (does not immediately hit the continuation threshold) and the event $E$ that $\eta_\theta^2$ is continuous for all time and does not hit $(-\infty,x)$ has positive probability (Proposition~\ref{prop::gff_abs_continuity}, Remark~\ref{rem::not_intersect_force_point_extra}, and Remark~\ref{rem::flow_cannot_hit}).  Condition on $\eta_\theta^2$, $E$ and let $C$ be the connected component of $\h \setminus \eta_\theta^2$ which contains $0$.  Let $\psi \colon C \to \h$ be the conformal map which fixes $0$, sends the leftmost intersection point $z$ of $\eta_{\theta}^2$ with $(0,\infty)$ to $x^{2,R}$, and fixes $\infty$.  We let $x^{1,R} = \psi(x)$.  Note that once we have conditioned on $\eta_\theta^2$, we can adjust $x$ between $0$ and $z$ to obtain any value of $x^{1,R}$ between $0$ and $x^{2,R}$ we like.  The boundary data for the GFF $h \circ \psi^{-1} - \chi \arg (\psi^{-1})'$ is depicted on the right side and $\psi(\eta)$ is an $\SLE_\kappa(\rho^{1,R},\rho^{2,R})$ process with $\rho^{1,R} = c/\lambda - 1$ and $\rho^{1,R}+\rho^{2,R} = -\theta \chi/\lambda - 2$.  Since $\eta$ and $\eta_{\theta}^2$ are almost surely continuous given $E$, so is $\psi(\eta)$.  By adjusting $\theta \in [0,\pi)$, we can achieve any value of $\rho^{1,R} + \rho^{2,R} \in (\tfrac{\kappa}{2}-4,-2]$ we like.  Likewise, by adjusting $c$, we can obtain any value of $\rho^{1,R} \in (-2,\tfrac{\kappa}{2}-2]$ we like.  Finally, we note that by either conditioning on an additional flow line starting at $0$ of positive angle or changing the boundary data to the left of $0$ to be smaller than $-\lambda$, we can also get the continuity of $\SLE_\kappa(\rho^{1,L};\rho^{1,R},\rho^{2,R})$ processes with $\rho^{1,R},\rho^{2,R}$ as before and any $\rho^{1,L} > -2$.
}
\end{figure}

\begin{figure}[ht!]
\begin{center}
\includegraphics[scale=0.85]{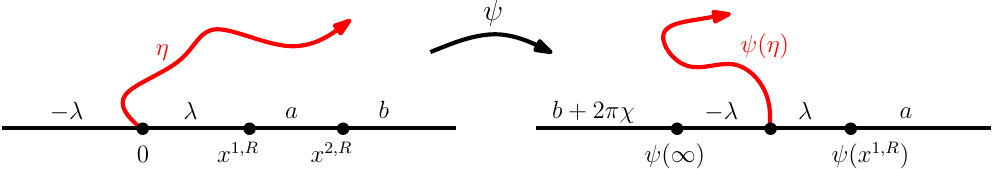}
\end{center}
\caption{\label{fig::cont_continuation_threshold2}  Suppose that $h$ is a GFF on $\h$ whose boundary data is as depicted on the left side.  We assume that $a > -\lambda$ and $b \leq -\lambda-\pi \chi$.  Then the flow line $\eta$ of $h$ is an $\SLE_\kappa(\rho^{1,R},\rho^{2,R})$ process with $\rho^{1,R} > -2$ and $\rho^{1,R} + \rho^{2,R} \leq \tfrac{\kappa}{2}-4$.  Let $\psi \colon \h \to \h$ be a conformal map which fixes $0$ and takes $x^{2,R}$ to $\infty$.  The boundary data for the GFF $h \circ \psi^{-1} - \chi \arg (\psi^{-1})'$ is depicted on the right side.  Hence $\psi(\eta)$ is an $\SLE_\kappa(\rho^{1,L};\rho^{1,R})$ process with $\rho^{1,L} \geq \tfrac{\kappa}{2}-2$ and $\rho^{1,R} > -2$ and therefore continuous by Proposition~\ref{prop::two_force_point_uniqueness_and_continuity}.  This implies the continuity of $\eta$ and that $\eta$ almost surely terminates at $x^{2,R}$ (because $\psi(\eta)$ almost surely tends to $\infty$ by Proposition~\ref{prop::two_force_point_uniqueness_and_continuity}).  If $a \in (-\lambda-\pi\chi,-\lambda]$, then we can apply the same conformal map and then get the continuity from the argument described at the end of Figure~\ref{fig::cont_continuation_threshold1} (the process one gets after applying the conformal map is an $\SLE_\kappa(\rho^{1,L};\rho^{1,R})$ with $\rho^{1,L} \geq \tfrac{\kappa}{2}-2$ and $\rho^{1,R} > \tfrac{\kappa}{2}-4$).  If $a \leq -\lambda-\pi \chi$, then we can apply a conformal map which sends both $x^{1,R}$ and $x^{2,R}$ to $(-\infty,0)$ and get the continuity from the fact that $\psi(\eta) \sim \SLE_\kappa(\wt{\rho}^{1,L},\wt{\rho}^{2,L})$ with $\wt{\rho}^{1,L} \geq \tfrac{\kappa}{2}-2$ and $\wt{\rho}^{1,L}+\wt{\rho}^{2,L}\geq \tfrac{\kappa}{2}-2$.}
\end{figure}

\begin{figure}[ht!]
\begin{center}
\includegraphics[scale=0.85]{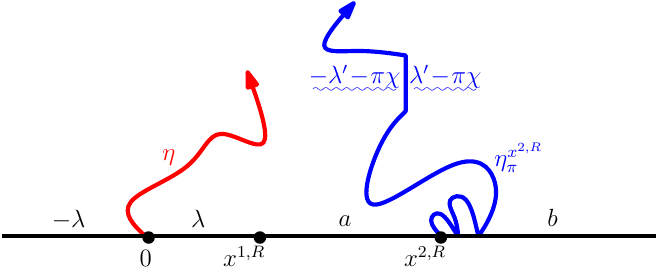}
\end{center}
\caption{\label{fig::cont_continuation_threshold3}  Suppose that $h$ is a GFF on $\h$ whose boundary data is as depicted above.  We assume that $a \leq -\lambda$ and $b > -\lambda-\pi\chi$.  Then the flow line $\eta$ of $h$ is an $\SLE_\kappa(\rho^{1,R},\rho^{2,R})$ process with $\rho^{1,R} \leq -2$ and $\rho^{1,R} + \rho^{2,R} > \tfrac{\kappa}{2}-4$.  Let $\eta_\pi^{x^{2,R}}$ be the flow line of $h$ starting at $x^{2,R}$ with angle $\pi$.  Then $\eta_\pi^{x^{2,R}}$ is an $\SLE_\kappa(\rho^{1,L},\rho^{2,L},\rho^{3,L};\rho^{1,R})$ process with $\rho^{1,L} \geq \tfrac{\kappa}{2}-2$, $\rho^{1,L}+\rho^{2,L} \leq -2$, $\rho^{1,L}+\rho^{2,L}+\rho^{3,L} = \tfrac{\kappa}{2}-2$ and $\rho^{1,R} > -2$.  There are two possibilities: either $\eta_{\pi}^{x^{2,R}}$ hits the continuation threshold upon accumulating in $[0,x^{1,R}]$ or it does not hit the continuation threshold.  In the former case, the conditional law of $\eta$ given $\eta_\pi^{x^{2,R}}$ in the leftmost connected component of $\h \setminus \eta_\pi^{x^{2,R}}$ is that of an $\SLE_\kappa(\wt{\rho}^{1,R},\wt{\rho}^{2,R})$ process with $\wt{\rho}^{1,R} = \tfrac{\kappa}{2}-2$ and $\wt{\rho}^{1,R}+\wt{\rho}^{2,R} = \rho^{1,R}+\rho^{2,R}$.  In this case, $\eta$ does not hit $\eta_\pi^{x^{2,R}}$ (Lemma~\ref{lem::hitting_interval}) and is continuous by the first part of the proof of Lemma~\ref{lem::cont_continuation_threshold} (recall Figure~\ref{fig::cont_continuation_threshold1}).  If $\eta_\pi^{x^{2,R}}$ does not hit the continuation threshold, then its law is absolutely continuous (Proposition~\ref{prop::gff_abs_continuity}) with respect to that of an $\SLE_\kappa(\tfrac{\kappa}{2}-2;\wt{\rho}^{1,R})$ process with $\wt{\rho}^{1,R} > -2$, hence continuous (Proposition~\ref{prop::two_force_point_uniqueness_and_continuity}).  In this case, the conditional law of $\eta$ given $\eta_\pi^{x^{2,R}}$ up until intersects $\eta_\pi^{x^{2,R}}$ is that of an $\SLE_\kappa(\wt{\rho}^{1,R},\wt{\rho}^{2,R})$ process with $\wt{\rho}^{1,R} = \rho^{1,R}$ and $\wt{\rho}^{1,R}+\wt{\rho}^{2,R} = \tfrac{\kappa}{2}-4$.  This implies that $\eta$ is continuous in this case by the argument of Figure~\ref{fig::cont_continuation_threshold2} (we note that on this event, $\eta$ exits $\h$ in $(x^{1,R},x^{2,R})$).}
\end{figure}

\begin{lemma}
\label{lem::cont_continuation_threshold}
Let $\eta$ be an $\SLE_\kappa(\rho^{1,R},\rho^{2,R})$ process in $\h$ starting at $0$ where the force points corresponding to the weights $\rho^{1,R}$ and $\rho^{2,R}$ are located at $0 < x^{1,R} < x^{2,R} < \infty$, respectively.  If $\rho^{1,R} \leq -2$ or $\rho^{1,R} + \rho^{2,R} \leq -2$, then $\eta$ is almost surely a continuous curve.
\end{lemma}
In order to prove Lemma~\ref{lem::cont_continuation_threshold}, we will need to consider several different cases.  These are described in Figure~\ref{fig::cont_continuation_threshold1} ($\rho^{1,R}  > -2$, $\rho^{1,R} + \rho^{2,R} \in (\tfrac{\kappa}{2}-4,-2]$), Figure~\ref{fig::cont_continuation_threshold2} ($\rho^{1,R} \in \R$, $\rho^{1,R} + \rho^{2,R} \leq \tfrac{\kappa}{2}-4$), and Figure~\ref{fig::cont_continuation_threshold3} ($\rho^{1,R} \leq -2$, $\rho^{1,R} + \rho^{2,R} > \tfrac{\kappa}{2}-4$).
\begin{proof}[Proof of Lemma~\ref{lem::cont_continuation_threshold}]
The proof of the first claim in the special case $\rho^{1,R} \in (-2,\tfrac{\kappa}{2}-2]$ and $\rho^{1,R} + \rho^{2,R} \in (\tfrac{\kappa}{2}-4,-2]$ is described in Figure~\ref{fig::cont_continuation_threshold1}.  For $\rho^{1,R} \geq \tfrac{\kappa}{2}-2$, Lemma~\ref{lem::hitting_and_bouncing} implies that $\eta$ first hits $\partial \h$ after time $0$ in $(x^{2,R},\infty)$.   Therefore the laws of the paths when $\rho^{1,R} = \tfrac{\kappa}{2}-2$ and $\rho^{1,R} > \tfrac{\kappa}{2}-2$  are mutually absolutely continuous (Proposition~\ref{prop::gff_abs_continuity}) upon hitting the continuation threshold (this is the same argument used in the proof of Lemma~\ref{lem::not_intersect_force_point2}), which completes the proof for $\rho^{1,R} > \tfrac{\kappa}{2}-2$ and $\rho^{1,R} + \rho^{2,R} \in (\tfrac{\kappa}{2}-4,2]$.  Note that these results hold if $\eta$ has an additional force point at $0^-$ of weight $\rho^{1,L} > -2$, as explained in Figure~\ref{fig::cont_continuation_threshold1}.

We now suppose that $\rho^{1,R}+\rho^{2,R} \leq \tfrac{\kappa}{2}-4$.  If $\rho^{1,R} > \tfrac{\kappa}{2}-4$, by applying a conformal map $\psi \colon \h \to \h$ which fixes $0$, sends $x^{1,R}$ to $\infty$, and $x^{2,R}$ to $-1$, we see that $\psi(\eta) \sim \SLE_\kappa(\wt{\rho}^{1,L};\rho^{1,R})$ where $\wt{\rho}^{1,L} \geq \tfrac{\kappa}{2}-2$ (this argument is described in Figure~\ref{fig::cont_continuation_threshold2}).  This puts us into the setting of the case considered in the previous paragraph.  Therefore $\psi(\eta)$ is continuous, hence also $\eta$.  If $\rho^{1,R} \leq \tfrac{\kappa}{2}-4$, then we can apply a conformal map $\psi \colon \h \to \h$ which fixes $0$ and sends both $x^{1,R}$ and $x^{2,R}$ to $(-\infty,0)$.  Then $\psi(\eta)$ is an $\SLE_\kappa(\rho^{1,L}.\rho^{2,L})$ process with $\rho^{1,L} \geq \tfrac{\kappa}{2}-2$ and $\rho^{1,L}+\rho^{2,L} \geq \tfrac{\kappa}{2}$, so the continuity in this case follows as well.

The final possibility is when $\rho^{1,R} \leq -2$ and $\rho^{1,R} + \rho^{2,R} > \tfrac{\kappa}{2}-4$.  The proof in this case is explained in the caption of Figure~\ref{fig::cont_continuation_threshold3}.
\end{proof}

\begin{figure}[ht!]
\begin{center}
\includegraphics[scale=0.85]{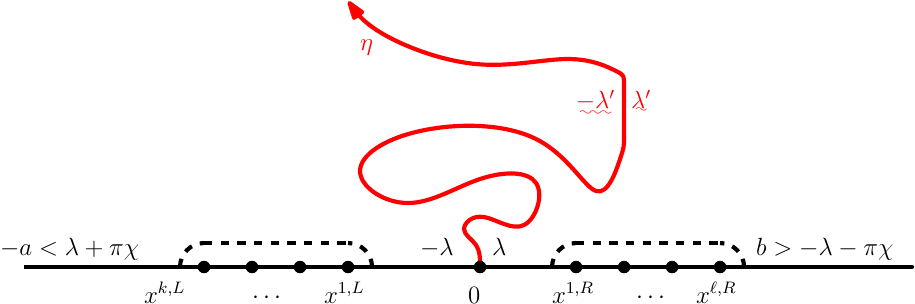}
\end{center}
\caption{\label{fig::escape_force_points} Suppose that $h$ is a GFF on $\h$ whose boundary data is as depicted above and let $\eta$ be the flow line of $h$ starting at $0$ targeted at $\infty$.  Then $\eta \sim \SLE_\kappa(\ul{\rho})$, $|\ul{\rho}^R| = \ell$ and $|\ul{\rho}^L| = k$, with $\sum_{i=1}^k \rho^{i,L}  = a/\lambda-1 > \tfrac{\kappa}{2}-4$ and $\sum_{i=1}^\ell \rho^{i,R} = b/\lambda-1 > \tfrac{\kappa}{2}-4$.  Let $U \subseteq \h$ consist of those points $z$ whose distance to $[x^{1,R},x^{\ell,R}]$ and $[x^{k,L},x^{1,L}]$ is at least $\tfrac{1}{2}$.  Then the law of $h|_U$ is mutually absolutely continuous with respect to the law of $\wt{h}|_U$ where $\wt{h}$ is a GFF on $\h$ whose boundary data agrees with $h$ in $(-\infty,x^{k,L}) \cup (x^{\ell,R},\infty)$, is $-\lambda$ in $[x^{k,L}, 0]$ and $\lambda$ in $[0,x^{\ell,R}]$.  Proposition~\ref{prop::two_force_point_uniqueness_and_continuity} and Lemmas~\ref{lem::not_intersect_force_point}--\ref{lem::cont_continuation_threshold} imply that the flow line $\wt{\eta}$ of $\wt{h}$ from $0$ targeted at $\infty$ either reaches $\infty$ or hits $(-\infty,x^{k,L}) \cup (x^{\ell,R},\infty)$ without leaving $U$ with positive probability.  Consequently, it follows from Proposition~\ref{prop::gff_abs_continuity} that the same is likewise true for $\eta$.  Moreover, from the discussion in Remark~\ref{rem::gff_abs_continuity} it is easy to see that this probability admits a positive lower bound which depends only on $|x^{k,L}|, |x^{\ell,R}|$, and $\|h|_{[x^{k,L},x^{1,L}]}\|_\infty$ and $\|h|_{[x^{1,R},x^{\ell,R}]}\|_\infty$.}
\end{figure}

\begin{lemma}
\label{lem::make_it_to_infty}
Suppose that $\eta$ is an $\SLE_\kappa(\ul{\rho})$ process in $\h$ from $0$ targeted at $\infty$, $k = |\ul{\rho}^L|$ and $\ell = |\ul{\rho}^R|$, with $\sum_{i=1}^k \rho^{i,L} > \tfrac{\kappa}{2}-4$ and $\sum_{i=1}^{\ell} \rho^{i,R} > \tfrac{\kappa}{2}-4$.  Assume that $|x^{1,L}|, |x^{1,R}| \geq 1$.  Fix $M > 0$ such that the locations of the force points $\ul{x} = (\ul{x}^L;\ul{x}^R)$ satisfy $x^{i,R}/x^{1,R} \leq M$ for all $1 \leq i \leq \ell$ and $x^{i,L}/x^{1,L} \leq M$ for all $1 \leq i \leq k$ and the weights $\ul{\rho}$ satisfy $|\rho^{i,q}| \leq M$ for all $1 \leq i \leq |\ul{\rho}^q|$ and $q \in \{L,R\}$.  Let $E_1$ be the event that either $\lim_{t \to \infty} \eta(t) = \infty$ or $\eta$ disconnects $x^{\ell,R}$ or $x^{k,L}$.  Let $E_2$ be the event that $\dist(\eta([0,\infty)), [x^{1,R}, x^{\ell,R}]) \geq \tfrac{1}{2}$ and $\dist(\eta([0,\infty)), [x^{k,L},x^{1,L}]) \geq \tfrac{1}{2}$.  With $E = E_1 \cap E_2$, we have that $\p[E] \geq \rho_0$ where $\rho_0 > 0$ depends only on $M$, $\sum_{i=1}^k \rho^{i,L}$, and $\sum_{i=1}^\ell \rho^{i,R}$.
\end{lemma}
\begin{proof}
This follows from Proposition~\ref{prop::gff_abs_continuity} and Remark~\ref{rem::gff_abs_continuity}; see Figure~\ref{fig::escape_force_points} for an explanation of the proof.
\end{proof}

\begin{figure}[ht!]
\begin{center}
\includegraphics[scale=0.85]{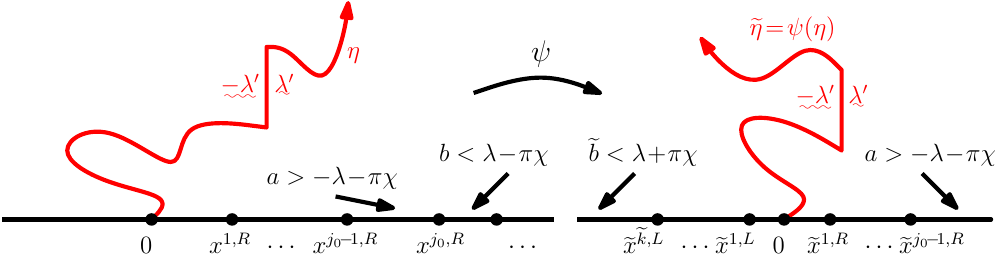}
\end{center}
\caption{\label{fig::map_to_infinity}  Let $h$ be a GFF on $\h$ whose boundary data is as depicted on the left side.  Let $\eta \sim \SLE_\kappa(\ul{\rho})$ be the flow line of $h$ starting at $0$. Suppose that $j_0$ is such that $\sum_{i=1}^{j_0-1} \rho^{i,R} > \tfrac{\kappa}{2}-4$ and $\sum_{i=1}^{j_0} \rho^{i,R} < \tfrac{\kappa}{2}-2$.  This particular choice implies that $\eta$ can hit $(x^{j_0-1,R},x^{j_0+1,R})$.  Let $\psi \colon \h \to \h$ be the conformal map which takes $\infty$ to $-1$, $x^{1,R}$ to $1$, and $x^{j_0,R}$ to $\infty$.  Then it is possible for $\wt{\eta} = \psi(\eta)$ to hit $(-\infty,\wt{x}^{k,L}) \cup (\wt{x}^{j_0-1},\infty)$ or reach $\infty$ before hitting $[\wt{x}^{k,L},\wt{x}^{1,L}] \cup [\wt{x}^{1,R},\wt{x}^{j_0-1,R}]$ where $\wt{\ul{x}}$ denotes the locations of the force points of $\wt{\eta}$.}
\end{figure}

\begin{proof}[Proof of Theorem~\ref{thm::continuity} for $\kappa \in (0,4)$]
We are going to prove the result by induction on the number of force points.  Proposition~\ref{prop::two_force_point_uniqueness_and_continuity} and Lemmas~\ref{lem::not_intersect_force_point}--\ref{lem::cont_continuation_threshold} imply the result for $\SLE_\kappa(\ul{\rho})$ processes with two force points.  Suppose that the result holds for all $\SLE_\kappa(\ul{\rho})$ processes with at most $n$ force points, some $n \geq 2$, and that $\eta \sim \SLE_\kappa(\ul{\rho})$ in $\h$ from $0$ to $\infty$ with $n+1$ force points.  If $\eta$ immediately hits the continuation threshold upon starting, there is nothing to prove.  Otherwise, running $\eta$ for a small amount of time and then applying a conformal mapping, we may assume that all of the force points are to the right of $0$; we denote their locations by $\ul{x}^R$.

Suppose that there exists $j_0 \geq 2$ such that $\sum_{i=1}^{j_0-1} \rho^{i,R} > \frac{\kappa}{2}-4$ and $\sum_{i=1}^{j_0} \rho^{i,R} < \tfrac{\kappa}{2}-2$ (if we couple $\eta$ with a GFF $h$ as in Theorem~\ref{thm::coupling_existence} on $\h$, this corresponds to the boundary data of $h$ in $[x^{j_0-1,R},x^{j_0,R})$ being larger than $-\lambda-\pi\chi$ and in $[x^{j_0,R},x^{j_0+1,R})$ being less than $\lambda - \pi \chi$).  Let $\psi \colon \h \to \h$ be the conformal map which sends $x^{1,R}$ to $1$, $x^{j_0,R}$ to $\infty$, and $\infty$ to $-1$ (see Figure~\ref{fig::map_to_infinity}).

Let $\wt{\eta} = \psi(\eta)$ and let $\ul{\wt{x}} = (\ul{\wt{x}}^L;\ul{\wt{x}}^R)$ denote the locations of the force points of $\wt{\eta}$.  Let $\wt{k} = |\ul{\wt{x}}^L|$ and note that $|\ul{\wt{x}}^R| = j_0-1$ by construction.  Let $\wt{W}_t$ be the Loewner driving function of $\wt{\eta}$, $\wt{g}_t$ the corresponding family of conformal maps, and let $\wt{V}_t^{i,q} = \wt{g}_t(\wt{x}^{i,q})$ denote the time evolution of the force points of $\wt{\eta}$ under $\wt{g}_t$.  We define stopping times as follows.  We let $\wt{\xi}_1$ be the first time $t$ that $\wt{W}_t = 0$ and let $\wt{\zeta}_1$ be the first time $t$ after $\wt{\xi}_1$ that $\wt{\eta}$ comes within distance $\tfrac{1}{2}$ of either $[\wt{V}_0^{\wt{k},L},\wt{V}_0^{1,L}]$ or $[\wt{V}_0^{1,R},\wt{V}_0^{j_0-1,R}]$.  For each $k \geq 2$, we inductively let $\wt{\xi}_k$ be the first time $t$ after $\wt{\zeta}_{k-1}$ that $\wt{W}_t = 0$ and $\wt{\zeta}_k$ the first time $t$ after $\wt{\xi}_k$ that $\wt{g}_{\wt{\xi}_{k}}(\wt{\eta}(t))$ comes within distance $\tfrac{1}{2}$ of either $[\wt{V}_{\wt{\xi}_{k}}^{\wt{k},L},\wt{V}_{\wt{\xi}_{k}}^{1,L}]$ or $[\wt{V}_{\wt{\xi}_{k}}^{1,R}, \wt{V}_{\wt{\xi}_{k}}^{j_0-1,R}]$.  Let $\wt{\tau}$ be the first time $t$ that $\wt{\eta}$ hits $(-\infty,\wt{V}_0^{2,L}]$, $[\wt{V}_0^{2,R},\infty)$, escapes to $\infty$, or hits the continuation threshold.

We are now going to show that $\wt{\eta}|_{[0,\wt{\tau} \wedge \wt{\xi}_k]}$ is almost surely continuous for every $k$.  We will argue that this holds by induction on $k$.  It holds for $k=1$ as a consequence of the induction hypothesis: if $\wt{W}_t$ is to the left or right of $0$, then the evolution of $\wt{\eta}$ is absolutely continuous with respect to the evolution of an $\SLE_\kappa(\ul{\rho})$ process with at most $n$ force points by the Girsanov theorem.  Suppose that $\wt{\eta}|_{[0,\wt{\tau} \wedge \wt{\xi}_k]}$ is continuous for some $k \geq 1$; we will argue that the same holds with $k+1$ in place of $k$.  For $t \in (\wt{\xi}_k,\wt{\zeta}_k]$, the desired continuity follows from Proposition~\ref{prop::gff_abs_continuity} and the two force point case.  For $t \in (\wt{\zeta}_k, \wt{\xi}_{k+1}]$, the claim follows by applying the Girsanov theorem and the induction hypothesis in the same manner we used to handle the case that $k=1$.

Let $E = \cup_{k} \{\wt{\tau} \leq \wt{\xi}_k\}$.  Then on $E$ we know that $\wt{\eta}|_{[0,\wt{\tau}]}$ is continuous.  Moreover, the conditional law of $\wt{\eta}|_{[\wt{\tau},\infty)}$ given $\wt{\eta}|_{[0,\wt{\tau}]}$ in the connected component $C$ of $\h \setminus \wt{\eta}([0,\wt{\tau}])$ which contains $\psi(\infty)=-1$ is that of an $\SLE_\kappa(\ul{\rho})$ process with at most $n$ force points.  Thus since $C$ is a Jordan domain, the desired result follows from the induction hypothesis.  We will complete the proof by showing that $\p[E^c] = 0$.  Since $\partial_t (\wt{V}_t^{j,R} - \wt{V}_t^{i,R}) < 0$ if $j > i$ and $\partial_t \wt{V}_t^{i,R} > 0$ for all $i$ (these facts come directly from the Loewner evolution and analogously hold when $R$ is replaced with $L$), Lemma~\ref{lem::make_it_to_infty} implies the existence of $\rho_0 > 0$ such that $\p[ \wt{\tau} < \wt{\xi}_{k+1} \giv \wt{\tau} > \wt{\xi}_k] \geq \rho_0$ for all $k$.  Therefore $\p[E^c] = 0$, as desired.

In order to complete the proof, we need to argue continuity in the case that there exists $J$ so that $\sum_{i=1}^j \rho^{i,R} \leq \tfrac{\kappa}{2}-4$ for all $1 \leq j \leq J$ and $\sum_{i=1}^j \rho^{i,R} \geq \tfrac{\kappa}{2}-2$ for all $J+1 \leq j \leq n+1$.  If $J = n+1$, we can see the continuity by applying a conformal map $\psi$ which fixes $0$ and sends all of the force points to the other side.  Indeed, then $\psi(\eta) \sim \SLE_\kappa(\ul{\rho})$ where the partial sums of the weights are all at least $\tfrac{\kappa}{2}-2$, so we can use Remark~\ref{rem::continuity_non_boundary}.  If $1 \leq J \leq n$, we can use the same argument described in Figure~\ref{fig::cont_continuation_threshold3} (condition on an auxiliary flow line at angle $\pi$ starting from $x^{J+1,R}$).  This completes the proof.
\end{proof}

Now that we have established Theorem~\ref{thm::continuity} for $\kappa \leq 4$, we can prove the almost sure continuity of angle varying flow lines.

\begin{proposition}
\label{prop::angle_varying_continuous}
Suppose that we have the same setup as Proposition~\ref{prop::angle_varying_monotonicity} (without the \emph{a priori} assumption of continuity).  Let $\theta_1,\ldots,\theta_k$ be angles satisfying~\eqref{eqn::angle_bounds}.  The angle varying flow line $\eta_{\theta_1 \cdots \theta_k}$ is almost surely a continuous path.
\end{proposition}
\begin{proof}
We prove the result by induction on $k$.  Theorem~\ref{thm::continuity}, which we have now proven for $\kappa \in (0,4]$, states that this result holds for $k=1$ (which corresponds to the constant angle case).  Suppose that $k \geq 2$ and the result holds for $k-1$.  Let $\tau_1,\ldots,\tau_{k-1}$ be the angle change times (and take $\tau_k = \infty$).  By assumption, $\eta_{\theta_1 \cdots \theta_k}^{\tau_1 \cdots \tau_k}|_{[\tau_{k-2},\tau_{k-1}]}$ given $\eta_{\theta_1 \cdots \theta_k}^{\tau_1 \cdots \tau_k}([0,\tau_{k-2}])$ evolves as an $\SLE_\kappa(\ul{\rho}^L;\ul{\rho}^R)$ process.  By induction, $\eta_{\theta_1 \cdots \theta_k}^{\tau_1 \cdots \tau_k}|_{[0,\tau_{k-2}]}$ is continuous so that a conformal map $\psi$ which takes the unbounded connected component of $\h \setminus \eta_{\theta_1 \cdots \theta_k}([0,\tau_{k-2}])$ to $\h$ with $\eta_{\theta_1 \cdots \theta_k}^{\tau_1 \cdots \tau_k}(\tau_{k-2})$ mapped to $0$ extends as a homeomorphism to the boundary.  Thus the continuity of $\eta_{\theta_1 \cdots \theta_k}^{\tau_1 \cdots \tau_k}$ follows from Theorem~\ref{thm::continuity} for $\kappa \in (0,4]$, which completes the proof of the induction step.
\end{proof}

\subsection{Counterflow lines}
\label{subsec::counterflow}

We will now explain how to modify the proofs from the previous subsections to complete the proof of Theorem~\ref{thm::coupling_uniqueness} and Theorem~\ref{thm::continuity} for $\kappa' > 4$.  Throughout this subsection, we will often work with a GFF $h$ on the strip $\strip$ in order to make the setting compatible with $\SLE$ duality (recall Section~\ref{sec::dubedat}).  We assume that the boundary data for $h$ is as in the left side of Figure~\ref{fig::counter_flow_line_crossing} and Figure~\ref{fig::counter_flow_line_angles_crossing}, where $a,b,a',b'$ are taken to be sufficiently large so that the configuration of flow and counterflow lines we consider almost surely does not interact with $\partial \strip$.

We let $\eta'$ be the counterflow line of $h$ starting at $z_0$.  The proof follows a strategy similar to but more involved than what we employed for $\kappa \in (0,4]$.  In Section~\ref{subsec::counterflow_two_boundary_force_points}, we will focus on the case with two boundary force points.  It turns out that in order to generate an $\SLE_{\kappa'}(\rho^L;\rho^R)$ process for arbitrary choices of $\rho^L,\rho^R > -2$ by conditioning on auxiliary flow lines in a manner similar to that used for $\kappa \in (0,4]$, we are already led to consider the law of $\eta'$ conditional on two angle varying flow lines (for $\kappa \in (0,4]$, we only had to consider angle varying trajectories when generalizing the two force point case to the many force point case).  Extending these results from two force points to many force points also follows a similar but more elaborate version of the strategy we used for $\kappa \in (0,4]$, since we will need to consider four different cases as opposed to three.  This is carried out in Section~\ref{subsec::counterflow_many_boundary_force_points}.  Finally, in Section~\ref{subsubsec::light_cone_general}, we will explain how to extend the light cone construction to the setting of general $\SLE_{\kappa'}(\ul{\rho}^L;\ul{\rho}^R)$ processes and, in particular, obtain general forms of $\SLE$ duality.

\subsubsection{Two boundary force points}
\label{subsec::counterflow_two_boundary_force_points}

We are now going to prove Theorem~\ref{thm::coupling_uniqueness} and Theorem~\ref{thm::continuity} for counterflow lines with two boundary force points.  The proof is a bit more elaborate than what we employed for flow lines because we will need to consider different types of configurations of flow and counterflow lines depending on the values of $\rho^L,\rho^R$.  In the first step, we will handle the case that $\rho^L > -2$ and $\rho^R \geq \tfrac{\kappa'}{2}-4$ (and vice-versa) --- recall that $\tfrac{\kappa'}{2}-4$ is the threshold at which $\eta'$ becomes boundary filling.  This will be accomplished in Lemma~\ref{lem::cfl_cont_det_part1} by considering a configuration consisting of two flow lines in addition to $\eta'$ (see Figure~\ref{fig::counter_flow_line_crossing}).  In the second step, accomplished in Lemma~\ref{lem::cfl_cont_det_part2} (see Figure~\ref{fig::counter_flow_line_angles_crossing}), we will take care of the case that $\rho^L,\rho^R  \in (-2,\tfrac{\kappa'}{2}-4)$ using a configuration which consists of two angle varying flow lines in addition to $\eta'$.  Combining these two lemmas completes the proof of Theorem~\ref{thm::coupling_uniqueness} for $\kappa ' > 4$ with many boundary force points (recall that the proof of Lemma~\ref{lem::functional_many_force_points} was not flow line specific) and Theorem~\ref{thm::continuity} for $\kappa' > 4$ with two boundary force points with weight exceeding $-2$, one on each side of the counterflow line seed.

\begin{figure}[h]
\begin{center}
\includegraphics[scale=0.85]{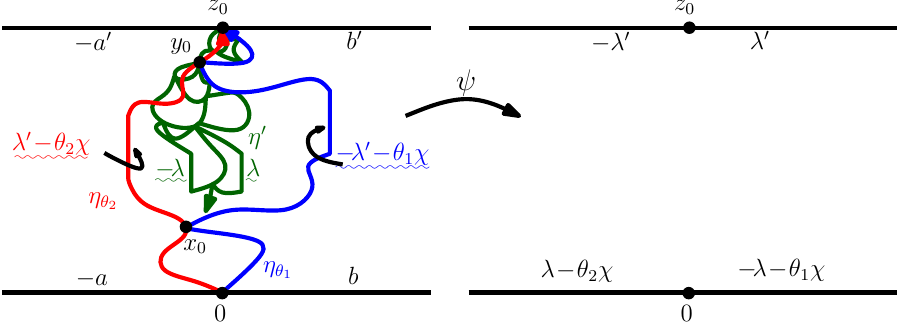}
\caption{\label{fig::counter_flow_line_crossing}
Let $h$ be a GFF on the strip $\strip$ whose boundary data is depicted in the left panel.  Fix $\theta_1 < \theta_2$ and let $\eta_{\theta_i}$ be the flow line of $h$ starting at $0$ with angle $\theta_i$.  Suppose that $C$ is any connected component of $\strip \setminus (\eta_{\theta_1} \cup \eta_{\theta_2})$ which lies between $\eta_{\theta_1}$ and $\eta_{\theta_2}$.  We assume that both $\theta_1 < \tfrac{\pi}{2}$ and $\theta_2 > -\tfrac{\pi}{2}$; this choice implies that $\eta'$ almost surely intersects $C$.  Let $x_0$ be the first point on $\partial C$ to be traced by $\eta_{\theta_1},\eta_{\theta_2}$ and $y_0$ the last.  Fix a stopping time $\tau'$ for $\CF_t = \sigma(\eta'(s) : s \leq t, \eta_{\theta_1},\eta_{\theta_2})$ such that $\eta'(\tau') \in C$ almost surely.  The boundary data for the conditional law of $h$ given $\eta_{\theta_1},\eta_{\theta_2}$ and $\eta'([0,\tau'])$ in $C$ is depicted in the left panel.  Let $\psi$ be a conformal map which takes the connected component of $C \setminus \eta'([0,\tau'])$ which contains $x_0$ to $\strip$ with $\eta'(\tau')$ taken to $z_0$ and $x_0$ taken to $0$.  The boundary data for the GFF $h \circ \psi^{-1} - \chi \arg (\psi^{-1})'$ on $\strip$ is depicted on the right side.  From this, we can read off the conditional law of $\eta'$ viewed as a path in $C$ given $\eta_{\theta_1},\eta_{\theta_2}$: it is an $\SLE_{\kappa'}(\rho^L;\rho^R)$ process where $\rho^L = (1/2 + \theta_2/\pi)(\kappa'/2-2)- 2$ and $\rho^R = (1/2-\theta_1/\pi)(\kappa'/2-2) - 2$.
}
\end{center}
\end{figure}

\begin{lemma}
\label{lem::cfl_cont_det_part1}
Suppose that $\rho^L > -2$ and $\rho^R \geq \tfrac{\kappa'}{2}-4$ or $\rho^L \geq \tfrac{\kappa'}{2}-4$ and $\rho^R > -2$.  In the coupling of an $\SLE_{\kappa'}(\rho^L;\rho^R)$ process $\eta_0'$ with a GFF $h_0$ as in Theorem~\ref{thm::coupling_existence}, $\eta_0'$ is almost surely determined by $h_0$.  Moreover, $\eta_0'$ is almost surely a continuous path.
\end{lemma}
The reason that we used the notation $\eta_0'$ and $h_0$ in the statement of Lemma~\ref{lem::cfl_cont_det_part1} is to avoid confusing $\eta_0'$ with $\eta'$ and $h_0$ with $h$.  Recall also from Remark~\ref{rem::continuity_non_boundary} that by absolute continuity (the Girsanov theorem) we know that non-boundary intersecting $\SLE_{\kappa'}(\ul{\rho})$ processes are almost surely continuous, at least up until just before terminating (or tending to $\infty$ if the terminal point is at $\infty$).  The proof of Lemma~\ref{lem::cfl_cont_det_part1} allows us to deduce the continuity of such processes even upon terminating by reducing the result to the transience of $\SLE_{\kappa'}$ processes established in \cite{RS05}.  This is accomplished by picking angles $\theta_1,\theta_2$ so that the conditional law of $\eta'$ is an $\SLE_{\kappa'}$ process given flow lines $\eta_{\theta_1}, \eta_{\theta_2}$ in each of the connected components of $\h \setminus (\eta_{\theta_1} \cup \eta_{\theta_2})$ which lie between $\eta_{\theta_1}$ and $\eta_{\theta_2}$.
\begin{proof}[Proof of Lemma~\ref{lem::cfl_cont_det_part1}]
We suppose that we have the setup described in Figure~\ref{fig::counter_flow_line_crossing}.  That is, we fix $\theta_1 < \theta_2$ and let $\eta_{\theta_i}$ be the flow line of $h$ with angle $\theta_i$ starting from $0$ and let $\eta'$ be the counterflow line starting from $z_0$.   We assume that $a,a',b,b'$ are large enough so that $\eta_{\theta_1},\eta_{\theta_2},$ and $\eta'$ almost surely do not intersect $\partial \strip$ except at their initial and terminal points.  We also assume that $\theta_1 < \tfrac{\pi}{2}$ and $\theta_2 > -\tfrac{\pi}{2}$.  This implies that $\eta_{\theta_1}$ lies to the right of the left boundary of $\eta'$ and likewise that $\eta_{\theta_2}$ lies to the left of the right boundary of $\eta'$ (recall Proposition~\ref{prop::flow_counterflow_left_right} and Proposition~\ref{prop::monotonicity_non_boundary}).  Figure~\ref{fig::counter_flow_line_crossing} describes the conditional mean $\CC_{A(\tau')}$ of $h$ given $A(\tau')$ where $A(t) = \eta_{\theta_1} \cup \eta'([0,t]) \cup \eta_{\theta_2}$ and where $\tau'$ is any stopping time for the filtration $\CF_t = \sigma(\eta'(s) : s \leq t,\ \ \eta_{\theta_1}, \eta_{\theta_2})$.  Indeed, we know that $A(\tau')$ is a local set for $h$ by Lemma~\ref{lem::stopping_local_set}.  Moreover, Remark~\ref{rem::cond_mean_height_cf} and Remark~\ref{rem::cond_mean_height_cf_contained} imply that $\CC_{A(\tau')}$ does not exhibit pathological behavior at points where any pair of $\eta_{\theta_1},\eta',\eta_{\theta_2}$ intersect.

Recall also Remark~\ref{rem::cont_loewner_cf} and Remark~\ref{rem::cont_loewner_cf_contains}, which imply that $\eta'$ has a continuous Loewner driving function viewed as a path in each of the connected components of $\strip \setminus (\eta_{\theta_1} \cup \eta_{\theta_2})$ which lie between $\eta_{\theta_1},\eta_{\theta_2}$.  Note that if either $\theta_1 >- \tfrac{\pi}{2}$ or $\theta_2 < \tfrac{\pi}{2}$ so that one of the $\eta_{\theta_i}$ is actually contained in the range of $\eta'$ (Lemma~\ref{lem::light_cone_contains_av}), we need to interpret what it means for $\eta'$ to be a path in one of these complementary connected components.  This is explained in complete detail in Remark~\ref{rem::cont_loewner_cf_contains}.  Applying Theorem~\ref{thm::martingale} and Proposition~\ref{prop::cond_mean_continuous}, we find that the conditional law of $\eta'$ given $\eta_{\theta_1}$ and $\eta_{\theta_2}$ in each of the connected components of $\strip \setminus (\eta_{\theta_1} \cup \eta_{\theta_2})$ which lie between $\eta_{\theta_1}$ and $\eta_{\theta_2}$ is independently an $\SLE_\kappa(\rho^L;\rho^R)$ process where
\begin{equation}
\label{eqn::cfl_cont_det_part1_rho}
 \rho^L = \left(\frac{1}{2} + \frac{\theta_2}{\pi}\right)\left(\frac{\kappa'}{2}-2\right) - 2
   \quad\text{and}\quad
   \rho^R = \left(\frac{1}{2}- \frac{\theta_1}{\pi}\right)\left(\frac{\kappa'}{2}-2\right) -2.
\end{equation}
Indeed, the values of $\rho^L,\rho^R$ are determined by solving the equations:
\[ -\lambda'(1+\rho^L) = \lambda - (\theta_2+\pi) \chi \quad\text{and}\quad \lambda'(1+\rho^R) = -\lambda + (-\theta_1+\pi) \chi\]
(see Figure~\ref{fig::counter_flow_line_crossing} and recall Figure~\ref{fig::conditional_boundary_data}).  The continuity statement of the lemma follows by adjusting $\theta_1,\theta_2$ appropriately and noting that each complementary component is almost surely a Jordan domain by the almost sure continuity of $\eta_{\theta_1}$ and $\eta_{\theta_2}$.  The first statement of the lemma follows from the same argument as in the proof of Lemma~\ref{lem::two_force_points_determined}.
\end{proof}

Note that in~\eqref{eqn::cfl_cont_det_part1_rho}, as $\theta_2 \downarrow -\tfrac{\pi}{2}$ --- which corresponds to $\eta_{\theta_2}$ approaching the right boundary of $\eta'$ --- we have that $\rho^L \downarrow -2$.  Likewise, as $\theta_1 \uparrow \tfrac{\pi}{2}$ --- which corresponds to $\eta_{\theta_1}$ approaching the left boundary of $\eta'$ --- we have that $\rho^R \downarrow -2$.  The constraint $\theta_1 < \theta_2$ means that it is not possible to obtain the full range of $\rho^L,\rho^R > -2$ values by computing the conditional law of $\eta'$ given configurations of flow lines of this type.

\begin{figure}[h!]
\begin{center}
\includegraphics[scale=0.85]{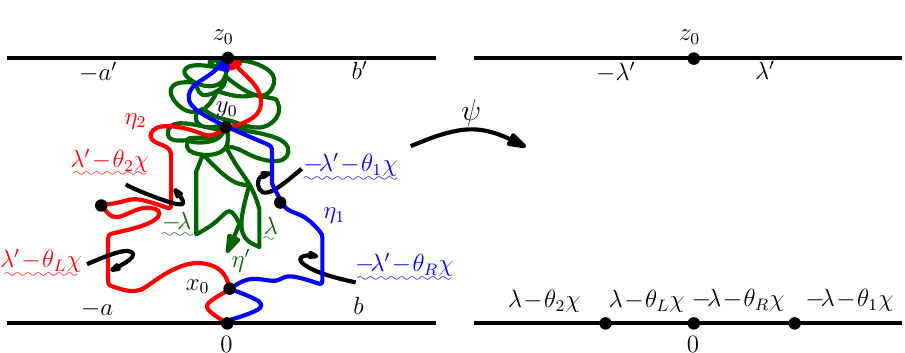}
\caption{\label{fig::counter_flow_line_angles_crossing}
Assume that $h$ is a GFF on the strip $\strip$ whose boundary data is depicted in the left panel.  Let $\theta_L = \tfrac{\pi}{2}$ and $\theta_R = - \tfrac{\pi}{2}$.  Suppose that $|\theta_1-\theta_R| \leq \pi$ and $|\theta_2-\theta_L| \leq \pi$.  Let $\eta_1 := \eta_{\theta_R \theta_1}^{\tau_1^1 \tau_2^1}$ and $\eta_2 := \eta_{\theta_L \theta_2}^{\tau_1^2 \tau_2^2}$ be angle varying flow lines of $h$ with the aforementioned angles.  We take the angle change times for both $\eta_1$ and $\eta_2$ to be when they first hit unit capacity as seen from $z_0$.  If $\theta_1 > \theta_R$, then $\eta'$ will cross $\eta_1$ and likewise if $\theta_2 < \theta_L$ then $\eta'$ will cross $\eta_2$.  We assume that the boundary data of $h$ is large enough so that $\eta_1,\eta_2,\eta'$ intersect $\partial \strip$ only at their initial and terminal points.  Let $y_0$ be the first point where $\eta_1$ and $\eta_2$ intersect after both $\eta_1$ and $\eta_2$ change directions and are traveling with angles $\theta_1$ and $\theta_2$, respectively.  Let $C$ be the connected component of $\strip \setminus (\eta_1 \cup \eta_2)$ which lies between $\eta_1$ and $\eta_2$ such that $y_0$ is the last point on $\partial C$ to be traced by $\eta_1,\eta_2$; let $x_0$ be the first.  Fix a stopping time $\tau'$ for $\CF_t = \sigma(\eta'(s) : s \leq t,\ \ \eta_1,\eta_2)$ such that $\eta'(\tau') \in C$ almost surely.  The boundary data for the conditional law of $h$ given $\eta_1,\eta_2,$ and $\eta'([0,\tau'])$ is depicted in the left panel.  Let $\psi$ be a conformal map which takes the connected component of $C \setminus \eta'([0,\tau'])$ which contains $x_0$ to $\strip$ with $\eta'(\tau')$ taken to $z_0$ and $x_0$ taken to $0$.  The boundary data for the GFF $h \circ \psi^{-1} - \chi \arg (\psi^{-1})'$ on $\strip$ is depicted on the right side.  From this, we can read off the conditional law of $\eta'$ viewed as a path in $C$ given $\eta_1,\eta_2$.  It is an $\SLE_{\kappa'}(\rho^{1,L},\rho^{2,L};\rho^{1,R},\rho^{2,R})$ process where the weights of the force points are given in~\eqref{eqn::cfl_cont_det_part2_rho}.
}
\end{center}
\end{figure}

In the next lemma, we are going to explain how to extend this method to the case that $\rho^L,\rho^R \in (-2,\tfrac{\kappa'}{2}-4)$, which corresponds to an $\SLE_{\kappa'}$ with both weights below the critical boundary filling threshold.

\begin{lemma}
\label{lem::cfl_cont_det_part2}
Suppose that $\rho^L,\rho^R \in (-2,\tfrac{\kappa'}{2}-4)$.  Then in the coupling of an $\SLE_{\kappa'}(\rho^L;\rho^R)$ process $\eta_0'$ with a GFF $h_0$ as in Theorem~\ref{thm::coupling_existence}, we have that $\eta_0'$ is almost surely determined by $h_0$.  Moreover, $\eta_0'$ is almost surely continuous.
\end{lemma}
\begin{proof}
In order to prove the lemma, we first need to understand the conditional mean of $h$ given $\eta'$ and two angle varying flow lines (which have the possibility of crossing each other).  Specifically, we let $\theta_L= \tfrac{\pi}{2}$ and $\theta_R = -\theta_L = -\tfrac{\pi}{2}$.  Suppose that $\theta_1,\theta_2$ are such that $|\theta_1 - \theta_R| \leq \pi$ and $|\theta_2 - \theta_L| \leq \pi$.  Let $\eta_1 := \eta_{\theta_R \theta_1}^{\tau_1^1 \tau_2^1}$ be an angle varying flow line with angles $\theta_R,\theta_1$ and $\eta_2 := \eta_{\theta_L \theta_2}^{\tau_1^2 \tau_2^2}$ be an angle varying flow line with angles $\theta_L,\theta_2$.  We take the angle change time for both $\eta_1,\eta_2$ to be when the curves first reach unit capacity (the actual choice here is not important).  Our hypotheses on $\theta_1,\theta_2$ imply that $\eta_1,\eta_2$ are simple (but may cross each other).  We assume $a,b,a',b'$ are large enough so that $\eta_1,\eta_2,\eta'$ almost surely do not intersect $\partial \strip$ except at their starting and terminal points.

It follows from Proposition~\ref{prop::merging_and_crossing} that if $\eta_1,\eta_2$ do cross, they cross precisely once, after which they may bounce off of one another.  Let $A(t) = \eta_1 \cup \eta'([0,t]) \cup \eta_2$ and $\CF_t = \sigma(\eta'(s) : s \leq t,\ \ \eta_{1}, \eta_2)$.  By Lemma~\ref{lem::stopping_local_set}, we know that $A(\tau')$ is a local set for $h$ for every $\CF_t$ stopping time $\tau'$.  The boundary data for $\CC_{A(\tau')}$ is described in the left panel of Figure~\ref{fig::counter_flow_line_angles_crossing}, depicted in the case that $\eta_1,\eta_2$ actually do cross.  The justification of this follows from exactly the same argument as in Remark~\ref{rem::cond_mean_height_cf_angle_varying}.  Let $C$ be the connected component of $\strip \setminus (\eta_1 \cup \eta_2)$ which lies between $\eta_1$ and $\eta_2$ such that the last point on $\partial C$ traced by $\eta_1$ is the point where $\eta_1,\eta_2$ first intersect after changing angles.  It follows from the same argument as Remark~\ref{rem::cont_loewner_cf_av} that $\eta'$ viewed as a path in $C$ from $y_0$ to $x_0$ (recall Remark~\ref{rem::cont_loewner_cf_contains}), the last and first points on $\partial C$ traced by $\eta_1$ and $\eta_2$, respectively, has a continuous Loewner driving function.  Consequently, it follows from Theorem~\ref{thm::martingale} and Proposition~\ref{prop::cond_mean_continuous} that the conditional law of $\eta'$ in $C$ given $\eta_1,\eta_2$ is an $\SLE_{\kappa'}(\ul{\rho}^L;\ul{\rho}^R)$, with the weights given by
\begin{align}
  \rho^{1,L} = \left(\frac{1}{2} + \frac{\theta_2}{\pi} \right)\left(\frac{\kappa'}{2}-2\right) & - 2,\ \ \
  \rho^{1,R} = \left(\frac{1}{2} - \frac{\theta_1}{\pi}\right) \left(\frac{\kappa'}{2}-2\right) - 2, \label{eqn::cfl_cont_det_part2_rho}\\
  \rho^{1,q} + &\rho^{2,q} = \frac{\kappa'}{2}-4 \quad\text{for}\quad q \in \{L,R\}. \notag
\end{align}
(see the right panel of Figure~\ref{fig::counter_flow_line_angles_crossing}, the values of $\rho$ follow from the same argument explained in Figure~\ref{fig::counter_flow_line_crossing}; recall also Figure~\ref{fig::criticalforintersection}).  For the final expression, we used $\theta_L = \tfrac{1}{\chi}(\lambda - \lambda')$ so that
\[ \frac{(\pi+\theta_L) \chi}{\lambda'} - \frac{\lambda}{\lambda'} - 1 = \frac{\kappa'}{2} - 1 -1 = \frac{\kappa'}{2}-4.\]
By choosing $\theta_2 \in (\theta_R,\theta_L)$ we can obtain any value of $\rho^{1,L} \in (-2,\tfrac{\kappa'}{2}-4)$ we like.  Likewise, by choosing $\theta_1 \in (\theta_R,\theta_L)$ we can obtain any value of $\rho^{1,R} \in (-2, \tfrac{\kappa'}{2}-4)$ we desire.  Therefore the continuity statement of the lemma follows by adjusting $\theta_1,\theta_2$ appropriately, using the almost sure continuity of $\eta_1,\eta_2$ to get that $C$ is almost surely a Jordan domain, and then applying the absolute continuity of the field (Proposition~\ref{prop::gff_abs_continuity}).  The first statement of the lemma follows from the same argument as the proof of Lemma~\ref{lem::two_force_points_determined} along with another application of Proposition~\ref{prop::gff_abs_continuity}.
\end{proof}

\subsubsection{Many boundary force points}
\label{subsec::counterflow_many_boundary_force_points}

The reduction of the statement of Theorem~\ref{thm::coupling_uniqueness} for counterflow lines to the two boundary force point case is exactly the same as the analogous reduction for flow lines, which was given in the proof of Lemma~\ref{lem::functional_many_force_points}.  This means that Theorem~\ref{thm::coupling_uniqueness} for $\kappa' > 4$ follows from Lemma~\ref{lem::cfl_cont_det_part1} and Lemma~\ref{lem::cfl_cont_det_part2}.  Thus we are left to complete the proof of Theorem~\ref{thm::continuity} for counterflow lines with many boundary force points.  Just as in the case of flow lines, it suffices to prove the continuity of counterflow lines with two boundary force points on the same side of $0$ (recall that the proof of Lemma~\ref{lem::make_it_to_infty} and Theorem~\ref{thm::continuity} for $\kappa \in (0,4]$ was not flow line specific).  The proof in this setting is more involved, though.  The reason is that for certain ranges of $\rho$ values, a counterflow line will almost surely hit a force point even before the continuation threshold is reached (flow and counterflow lines only hit force points with positive probability when the partial sum of the weights is at most $\tfrac{\kappa}{2}-4 \leq -2$ and $\tfrac{\kappa'}{2}-4 > -2$, respectively).  This leads us to consider four different types of local behavior:
\begin{enumerate}
\item $\rho^{1,R} > -2$, $\rho^{1,R} + \rho^{2,R} > -2$ with $|\rho^{2,R}| < \tfrac{\kappa'}{2}$ (Lemma~\ref{lem::cfl_cont_case1}),
\item $\rho^{1,R} \in (-2, \tfrac{\kappa'}{2}-2)$, $\rho^{1,R} + \rho^{2,R} \geq \tfrac{\kappa'}{2}-2$ (Lemma~\ref{lem::cfl_cont_case2}),
\item $\rho^{1,R} \geq \tfrac{\kappa'}{2}-2$, $\rho^{1,R} + \rho^{2,R} \in (-2, \tfrac{\kappa'}{2}-2)$ (Lemma~\ref{lem::cfl_cont_case3}), and
\item at least one of $\rho^{1,R} \leq -2$ or $\rho^{1,R} + \rho^{2,R} \leq -2$ (Lemma~\ref{lem::cfl_continuation_threshold}).
\end{enumerate}

In the following sequence of lemmas, it may appear to the reader that the roles of the superscripts ``$L$'' and ``$R$'' have been reversed.  The reason is that the results will be stated for a counterflow line growing from the bottom of $\h$ as opposed to the top of the strip $\strip$, so everything is rotated by $180$ degrees.

\begin{lemma}
\label{lem::cfl_cont_case1}
Suppose that $\eta_0'$ is an $\SLE_{\kappa'}(\rho^{1,R},\rho^{2,R})$ process in $\h$ from $0$ to $\infty$ with force points located at $-1$ and $-2$ with weights satisfying $\rho^{1,R}, \rho^{1,R} + \rho^{2,R} > -2$ and $|\rho^{2,R}| < \tfrac{\kappa'}{2}$.  Then $\eta_0'$ is almost surely continuous.
\end{lemma}
\begin{proof}
Suppose that $h$ is a GFF on the strip $\strip$ with the same boundary data as in Figure~\ref{fig::counter_flow_line_angles_crossing}.  Let $\eta_{\theta_1 \theta_2}^{\tau_1 \tau_2}$ be an angle varying flow line with angles $\theta_1,\theta_2$ such that $|\theta_1 - \theta_2| < \frac{2\lambda}{\chi}$.  We assume that $\theta_1,\theta_2 < \tfrac{\pi}{2}$ so that $\eta_{\theta_1 \theta_2}^{\tau_1 \tau_2}$ almost surely stays to the right of the left boundary of the counterflow line $\eta'$ (recall Proposition~\ref{prop::angle_varying_monotonicity} as well as Proposition~\ref{prop::flow_counterflow_left_right}).  Assume that $a,b,a',b'$ are sufficiently large so that $\eta'$ and $\eta_{\theta_1 \theta_2}^{\tau_1 \tau_2}$ intersect $\partial \strip$ only at $0$ and $z_0$.  Arguing as in the proof of Lemma~\ref{lem::cfl_cont_det_part2}, the conditional law of $\eta'$ viewed as a path in the left connected component $C$ of $\strip \setminus \eta_{\theta_1 \theta_2}^{\tau_1 \tau_2}$ is that of an $\SLE_\kappa(\ul{\rho}^L;\ul{\rho}^R)$ where
\begin{align*}
  \rho^{1,R} = \left(\frac{1}{2} - \frac{\theta_2}{\pi} \right)\left(\frac{\kappa'}{2}-2\right) - 2 \quad\text{and}\quad
  \rho^{1,R} + \rho^{2,R} = \left(\frac{1}{2} - \frac{\theta_1}{\pi} \right)\left(\frac{\kappa'}{2}-2\right) - 2.
\end{align*}
Consequently, it is not difficult to see that by adjusting the angles $\theta_1,\theta_2$, we can obtain any combination of values of $\rho^{1,R}, \rho^{2,R}$ such that $\rho^{1,R} > -2$, $\rho^{1,R}+\rho^{2,R} > -2$, and $|\rho^{2,R}| < \tfrac{\kappa'}{2}$ (the restriction on $|\rho^{2,R}|$ comes from the restriction $|\theta_1 - \theta_2| < \tfrac{2\lambda}{\chi}$).  This completes the proof because we know that $\eta'$ viewed as a path in $C$ is almost surely continuous and, in particular, is continuous when it interacts with the force points corresponding to the weights $\rho^{1,R},\rho^{2,R}$.
\end{proof}

\begin{figure}[h!]
\begin{center}
\includegraphics[scale=0.85]{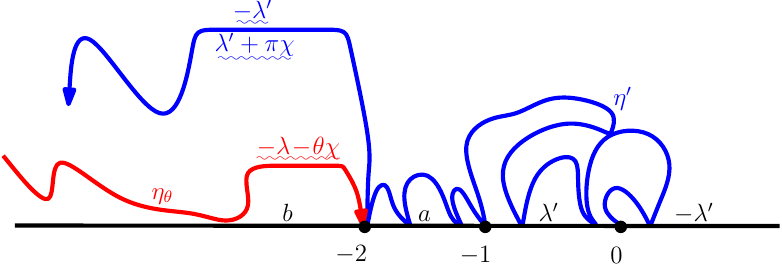}
\caption{\label{fig::counter_flow_line_small_big}
The configuration of paths used to prove the continuity of an $\SLE_{\kappa'}(\rho^{1,R},\rho^{2,R})$ process with $\rho^{1,R} \in (-2,\tfrac{\kappa'}{2}-2)$ and $\rho^{1,R} + \rho^{2,R} \geq \tfrac{\kappa'}{2}-2$ (the roles of ``$R$'' and ``$L$'' are flipped since we are growing the counterflow line from the bottom of $\h$ rather than the top of $\strip$, so everything is rotated by $180$ degrees).  We suppose that $h$ is a GFF on $\h$ with the boundary data depicted above.  We assume that $a = \lambda'(1+\rho^{1,R}) \in (-\lambda', \lambda' + \pi \chi)$ and $b = \lambda'(1+\rho^{1,R}+\rho^{2,R}) \geq \lambda'+\pi\chi$.  We let $\eta'$ be the counterflow line of $h$ starting at $0$ and $\eta_\theta$ be the flow line of $h$ starting at $\infty$ with angle $\theta$.  Taking $\theta$ so that $-\lambda-\theta \chi = \lambda' + \pi \chi$, we see that the conditional law of $\eta'$ given $\eta_\theta$ in the unbounded connected component $C$ of $\h \setminus \eta_\theta([0,\tau])$, $\tau$ the first time $\eta_\theta$ hits $[-2,\infty)$, is an $\SLE_{\kappa'}(\wt{\rho}^{1,R},\wt{\rho}^{2,R})$ process with $\wt{\rho}^{1,R} = \rho^{1,R}$ and $\wt{\rho}^{1,R} + \wt{\rho}^{2,R} = \tfrac{\kappa'}{2}-2$.  Since $\eta_\theta$ is continuous, $C$ is a Jordan domain, so the continuity of $\eta'$ follows from the case $|\rho^{2,R}| < \tfrac{\kappa'}{2}$.  We note that the precise location that $\eta_\theta$ hits $[-2,-1]$ depends on the choice of $a$.  There are choices of $a$ for which $\eta_\theta$ almost surely hits $-2$ first, which is the case shown in the illustration, and there are choices of $a$ for which $\eta_\theta$ hits somewhere in $(-2,-1)$ first.
}
\end{center}
\end{figure}

\begin{lemma}
\label{lem::cfl_cont_case2}
Suppose that $\eta_0'$ is an $\SLE_{\kappa'}(\rho^{1,R},\rho^{2,R})$ process with force points located at $-1$ and $-2$ satisfying $\rho^{1,R} \in (-2,\tfrac{\kappa'}{2}-2)$ and $\rho^{1,R}+\rho^{2,R} \geq \tfrac{\kappa'}{2}-2$.  Then $\eta_0'$ is almost surely continuous.
\end{lemma}
\begin{proof}
See Figure~\ref{fig::counter_flow_line_small_big} for an explanation of the proof.  The one aspect of the proof which was skipped in the caption is that we did not explain why the conditional law of $\eta'$ given $\eta_\theta$ is actually an $\SLE_{\kappa'}(\wt{\rho}^{1,R},\wt{\rho}^{2,R})$ process with $\wt{\rho}^{1,R} = \rho^{1,R}$ and $\wt{\rho}^{1,R} + \wt{\rho}^{2,R} = \tfrac{\kappa'}{2}-2$.  This follows from an application of Theorem~\ref{thm::martingale}.  In order to justify the usage of this result, we just need to explain why the conditional mean $\CC_{A(t)}$ of $h$ given $A(t) = \eta_\theta([0,\tau]) \cup \eta'([0,t])$, $\tau$ the first time that $\eta_\theta$ hits $[-2,\infty)$, does not exhibit pathological behavior and is continuous in $t > 0$ as well as why $\eta'$ has a continuous Loewner driving function viewed as a path in the unbounded connected component $C$ of $\h \setminus \eta_\theta([0,\tau])$.  The latter holds up until the first time $\tau'$ that $\eta'$ hits $\eta_{\theta} \setminus \partial \h$ since $\eta'$ itself has a continuous Loewner driving function.  The former also holds up until time $\tau'$ by Lemma~\ref{lem::stopping_local_set} and Proposition~\ref{gff::prop::cond_union_mean}.  Since we know that $\SLE_{\kappa'}(\wt{\rho}^{1,R},\wt{\rho}^{2,R})$ processes are almost surely continuous by Lemma~\ref{lem::cfl_cont_case1}, we thus have the continuity of $\eta'$ up until either $\tau'$.  Lemma~\ref{lem::flow_cannot_hit} (see also Remark~\ref{rem::flow_cannot_hit}) implies that $\tau' = \infty$ almost surely.
\end{proof}

\begin{figure}[h!]
\begin{center}
\includegraphics[scale=0.85]{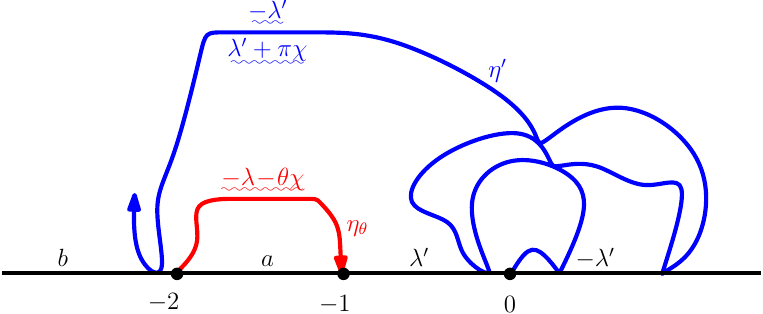}
\caption{\label{fig::counter_flow_line_big_small} The configuration of paths used to prove the continuity of an $\SLE_{\kappa'}(\rho^{1,R},\rho^{2,R})$ counterflow line with $\rho^{1,R} \geq \tfrac{\kappa'}{2}-2$ and $\rho^{1,R} + \rho^{2,R} \in (-2,\tfrac{\kappa'}{2}-2)$.  We suppose that $h$ is a GFF on $\h$ with the boundary data depicted above.  Here, we assume that $a = \lambda'(1+\rho^{1,R}) \geq \lambda' + \pi \chi$ and $b = \lambda'(1+\rho^{1,R}+\rho^{2,R}) \in (-\lambda',\lambda'+\pi\chi)$.  We let $\eta'$ be the counterflow line of $h$ starting at $0$ and let $\eta_\theta$ be the flow line of $h$ starting at $-2$ with angle $\theta$.  Taking $\theta$ so that $-\lambda-\theta \chi = \lambda' + \pi \chi$, we see that the conditional law of $\eta'$ given $\eta_\theta$ in the unbounded connected component $C$ of $\h \setminus \eta_\theta([0,\tau])$, $\tau$ the first time $\eta_\theta$ hits $[-1,\infty)$, is an $\SLE_{\kappa'}(\wt{\rho}^{1,R},\wt{\rho}^{2,R})$ process with $\wt{\rho}^{1,R} = \tfrac{\kappa'}{2}-2$ and $\wt{\rho}^{1,R} + \wt{\rho}^{2,R} = \rho^{1,R} + \rho^{2,R}$.  Since $\eta_\theta$ is continuous, $C$ is a Jordan domain, so the continuity of $\eta'$ follows from the case $|\rho^{2,R}| < \tfrac{\kappa'}{2}$.
}
\end{center}
\end{figure}

\begin{lemma}
\label{lem::cfl_cont_case3}
Suppose that $\eta'$ is an $\SLE_{\kappa'}(\rho^{1,R},\rho^{2,R})$ process with $\rho^{1,R} \geq \tfrac{\kappa'}{2}-2$ and $\rho^{1,R}+\rho^{2,R} \in (-2,\tfrac{\kappa'}{2}-2)$.  Then $\eta'$ is almost surely continuous.
\end{lemma}
\begin{proof}
See Figure~\ref{fig::counter_flow_line_big_small} for an explanation of the proof.  As in the proof of Lemma~\ref{lem::cfl_cont_case2}, we did not explain in the caption why the conditional law of $\eta'$ given $\eta_\theta$ is an $\SLE_{\kappa'}(\wt{\rho}^{1,R},\wt{\rho}^{2,R})$ process with $\wt{\rho}^{1,R} = \rho^{1,R}$ and $\wt{\rho}^{1,R} + \wt{\rho}^{2,R} = \tfrac{\kappa'}{2}-2$.  This follows from Theorem~\ref{thm::martingale} using an argument similar to what we employed for Lemma~\ref{lem::cfl_cont_case2}.  Indeed, we need to explain why the conditional mean $\CC_{A(t)}$ of $h$ given $A(t) = \eta_{\theta}([0,\tau]) \cup \eta'([0,t])$, $\tau$ the first time that $\eta_\theta$ hits $[-1,\infty)$, does not exhibit pathological behavior and is continuous in $t$ as well as why $\eta'$ has a continuous Loewner driving function viewed as a path in the unbounded connected component $C$ of $\h \setminus \eta_\theta$.  The latter holds up until the first time $\tau'$ that $\eta'$ hits $\eta_{\theta} \setminus \partial \h$ since $\eta'$ has a continuous Loewner driving function.  The former also holds up until time $\tau'$ by Lemma~\ref{lem::stopping_local_set} and Proposition~\ref{gff::prop::cond_union_mean}.  Since we know that $\SLE_{\kappa'}(\wt{\rho}^{1,R},\wt{\rho}^{2,R})$ processes are almost surely continuous by Lemma~\ref{lem::cfl_cont_case1}, we thus have the continuity of $\eta'$ up until $\tau'$.  This allows us to apply Lemma~\ref{lem::flow_cannot_hit} to $\eta'|_{[0,\tau')}$, which in turn implies that $\p[\tau' = \infty]~{=1}$.
\end{proof}

\begin{figure}[h!]
\begin{center}
\includegraphics[scale=0.85]{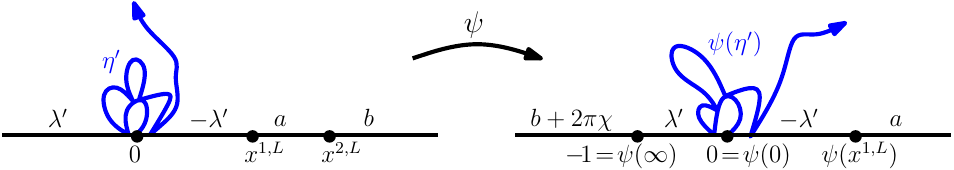}
\end{center}
\caption{\label{fig::cont_continuation_threshold_cfl}  Suppose that $h$ is a GFF on $\h$ whose boundary data is as depicted on the left side.  We assume that at least one of $a \geq \lambda'$ or $b \geq \lambda'$.  Then the counterflow line $\eta'$ of $h$ is an $\SLE_{\kappa'}(\rho^{1,L},\rho^{2,L})$ process with at least one of $\rho^{1,L} \leq -2$ or $\rho^{2,L} \leq -2$.  Assume, for example, that $\rho^{1,L} > -2$ and $\rho^{1,L}+\rho^{2,L} \leq -2$.  We can see that $\eta'$ is almost surely continuous by applying the conformal map $\psi \colon \h \to \h$ which fixes $0$, takes $x^{2,L}$ to $\infty$, and $\infty$ to $-1$.  The boundary data for the GFF $h \circ \psi^{-1} - \chi \arg (\psi^{-1})'$ is depicted on the right side.  Hence $\psi(\eta')$ is an $\SLE_{\kappa'}(\rho^{1,L};\rho^{1,R})$ process with $\rho^{1,R} > \kappa'-4 > -2$ and $\rho^{1,L} > -2$ and therefore continuous by Lemma~\ref{lem::cfl_cont_det_part1}.  This implies the continuity of $\eta'$ and that $\eta'$ almost surely terminates at $1$ because $\psi(\eta')$ is almost surely transient.  If both $\rho^{1,L} \leq -2$ and $\rho^{1,L}+\rho^{2,L} \leq -2$, the same argument works except we apply a conformal map which switches the sides of both $x^{1,L}$ and $x^{2,L}$ (as opposed to just $x^{2,L}$).}
\end{figure}

\begin{figure}[h!]
\begin{center}
\includegraphics[scale=0.85]{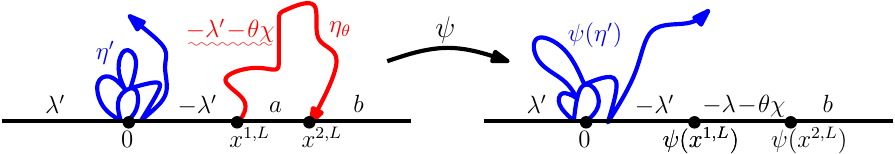}
\end{center}
\caption{\label{fig::cont_continuation_threshold_cfl2} (Continuation of Figure~\ref{fig::cont_continuation_threshold_cfl}.)  Suppose that $\h$ is a GFF whose boundary data is as depicted on the left side where $a \leq \lambda'$ and $b > \lambda'$.  Then $\eta'$ is an $\SLE_{\kappa'}(\rho^{1,L},\rho^{2,L})$ process with $\rho^{1,L} \leq -2$ and $\rho^{1,L}+\rho^{2,L} > -2$.  Let $\eta$ be the flow line of $h$ starting at $x^{1,L}$ with angle $\theta$.  Taking $\theta = -\tfrac{3}{2} \pi$, we know that $\eta_\theta$ lies to the right of $\eta'$ and that $\eta_\theta$ is almost surely continuous.  There are two possibilities.  Either $\eta_\theta$ first hits $(-\infty,0)$ or $(x^{2,L},\infty)$, say at the time $\tau$.  In the former case, the conditional law of $\eta'$ given $\eta_\theta([0,\tau])$ is that of an $\SLE_{\kappa'}(\tfrac{\kappa'}{2}-4)$ process, hence continuous.  In the latter case, the conditional law of $\eta'$ given $\eta_\theta([0,\tau])$ is that of an $\SLE_{\kappa'}(\tfrac{\kappa'}{2}-4,\wt{\rho}^{2,L})$ process where $\tfrac{\kappa'}{2}-4+\wt{\rho}^{2,L} = \rho^{1,L}+\rho^{2,L} > -2$ , hence continuous.}
\end{figure}

\begin{lemma}
\label{lem::cfl_continuation_threshold}
Suppose that $\eta'$ is an $\SLE_{\kappa'}(\rho^{1,L},\rho^{2,L})$ process with either $\rho^{1,L} \leq -2$ or $\rho^{1,L}+\rho^{2,L} \leq -2$ with force points located at $0 < x^{1,L} < x^{2,L}$.  Then $\eta'$ is almost surely continuous.
\end{lemma}
\begin{proof}
The proof is explained in Figure~\ref{fig::cont_continuation_threshold_cfl} and Figure~\ref{fig::cont_continuation_threshold_cfl2}.
\end{proof}

\begin{proof}[Proof of Theorem~\ref{thm::continuity} for $\kappa > 4$]
Exactly the same as the proof for $\kappa \in (0,4]$ (recall Lemma~\ref{lem::make_it_to_infty}).
\end{proof}

\subsubsection{Light cones with general boundary data}
\label{subsubsec::light_cone_general}

\begin{figure}[ht!]
\begin{center}
\subfigure[]{\includegraphics[scale=0.85,page=1]{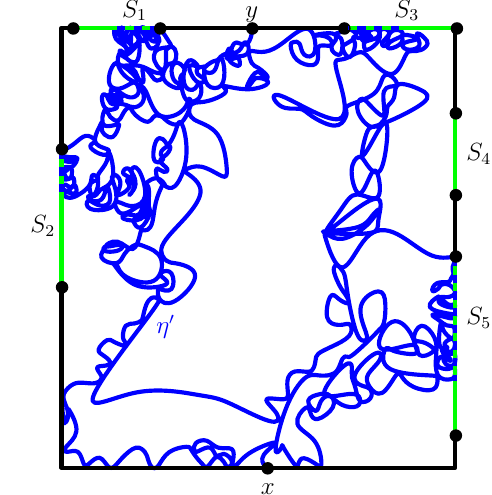}}
\hspace{0.015\textwidth}
\subfigure[]{\includegraphics[scale=0.85,page=2]{figures/duality_boundary_filling2.pdf}}
\end{center}
\vspace{-0.03\textheight}
\caption{\label{fig::lightcone_boundary_filling} { Suppose that $h$ is a GFF on a Jordan domain $D$ and $x,y \in \partial D$ are distinct.  Let $\eta'$ be the counterflow line of $h$ starting at $x$ aimed at $y$.  Let $K = K_L \cup K_R$ be the outer boundary of $\eta'$, $K_L$ and $K_R$ its left and right sides, respectively, and let $I$ be the interior of $K_L  \cap \partial D$.  We suppose that the event $E = \{I \neq \emptyset\}$ that $\eta'$ fills a segment of the left side of $\partial D$ has positive probability, though we emphasize that this does not mean that $\eta'$ {\em traces} a segment of $\partial D$---which would yield a discontinuous Loewner driving function---with positive probability.  In the illustrations above, $\eta'$ fills parts of $S_1,\ldots,S_5$ with positive probability (but with positive probability does not hit any of $S_1,\ldots,S_5$).  The connected component of $K_L \setminus I$ which contains $x$ is given by the flow line $\eta_L$ of $h$ with angle $\tfrac{\pi}{2}$ starting at $x$ (left panel).  On $E$, $\eta_L$ hits the continuation threshold before hitting $y$ (in the illustration above, this happens when $\eta_L$ hits $S_2$).  On $E$ it is possible to describe $K_L$ completely in terms of flow lines using the following algorithm.  First, we flow along $\eta_L$ starting at $x$ until the continuation threshold is reached, say at time $\tau_1$, and let $z_1 = \eta_L(\tau_1)$.  Second, we trace along $\partial D$ in the clockwise direction until the first point $w_1$ where it is possible to flow starting at $w_1$ with angle $\tfrac{\pi}{2}$ without immediately hitting the continuation threshold.  Third, we flow from $w_1$ until the continuation threshold is hit again.  We then repeat this until $y$ is eventually hit.  This is depicted in the right panel above, where three iterations of this algorithm are needed to reach $y$ and are indicated by the colors red, yellow, and purple, respectively.}
}
\end{figure}

We are now going to explain how the light cone construction extends to the setting of general piecewise constant boundary data.  Recall that Remark~\ref{rem::light_cone_contains_av_general} and Remark~\ref{rem::light_cone_construction_general} from Section~\ref{subsec::light_cone} imply that the missing ingredients to prove that the light cone construction for counterflow lines is applicable in this general setting are:
\begin{enumerate}
\item the continuity of $\SLE_{\kappa'}(\ul{\rho})$ processes for general weights $\ul{\rho}$ and
\item the continuity of angle varying flow lines.
\end{enumerate}
We have at this point in the article established both of these results, which completes the proof of Theorem~\ref{thm::lightconeroughstatement}.  We remark that the light cone is a bit different if $\eta'$ fills some segment of the boundary, say on its left side (see Figure~\ref{fig::lightcone_boundary_filling}); let $K_L$ be the left boundary of $\eta'$.  The reason is that, in this case, $K_L$ is no longer a flow line, though it is still possible to express $K_L$ as a union of flow lines with angle $\theta_L = \tfrac{\pi}{2}$ and boundary segments.  In particular, if $\eta'$ does not fill the boundary all of the way until it hits its terminal point, say $x$, then the connected component of the closure of $K_L \setminus \partial D$ which contains $x$ is given by the flow line starting from $x$ with angle $\theta_L=\tfrac{\pi}{2}$.  The same is likewise true if the roles of left and right are swapped.

Suppose that $\eta'$ is non-boundary filling, i.e. $\sum_{i=1}^j \rho^{i,q} > \tfrac{\kappa'}{2}-4$ for all $1 \leq j \leq |\ul{\rho}^q|$ and $q \in \{L,R\}$, so that the left and right boundaries $\eta_L$ and $\eta_R$ of $\eta'$ are given by flow lines with angles $\tfrac{\pi}{2}$ and $-\tfrac{\pi}{2}$, respectively.  Then we can write down the conditional law of $\eta'$ given $\eta_L$ and $\eta_R$. (This is referred to as ``strong duality'' in \cite{DUB_PART}; see also \cite[Section~8]{DUB_PART} for related results).

\begin{proposition}
\label{prop::cfl_given_outer_boundary}
Suppose that $\eta'$ is an $\SLE_{\kappa'}(\ul{\rho}^L;\ul{\rho}^R)$ process on a Jordan domain $D$ from $y$ to $x$ with $x,y \in \partial D$ distinct.  Assume $\sum_{i=1}^j \rho^{i,q} > \tfrac{\kappa'}{2}-4$ for all $1 \leq j \leq |\ul{\rho}^q|$ and $q \in \{L,R\}$.  Then the conditional law of $\eta'$ given its left and right boundaries $\eta_L$ and $\eta_R$ is that of an $\SLE_{\kappa'}(\tfrac{\kappa'}{2}-4;\tfrac{\kappa'}{2}-4)$ process independently in each of the connected components of $D \setminus (\eta_L \cup \eta_R)$ which lie between $\eta_L$ and $\eta_R$.
\end{proposition}
\begin{proof}
This follows from the same proof used to establish the continuity of $\SLE_{\kappa'}(\rho^L;\rho^R)$ processes for $\rho^L,\rho^R \geq \tfrac{\kappa'}{2}-4$ and is given explicitly in Lemma~\ref{lem::cfl_cont_det_part1}.  The only difference is that the proof of Lemma~\ref{lem::cfl_cont_det_part1} required $\eta'$ not to hit the boundary (with the exception of its initial and terminal points).  The reason for this is that, at that point in the article, we had not yet established the continuity of general boundary hitting counterflow lines.  Now that this has been proved, we can repeat the same argument again to get the proposition.
\end{proof}

If there exists a boundary point $z$ which $\eta'$ almost surely hits, then we can use the light cone construction to describe the outer boundary $\eta_z^1$ of $\eta'$ upon hitting $z$ as well as compute the conditional law of $\eta'$ given $\eta_z^1$ before and after hitting $z$.  This is formulated in the following proposition (see also Figure~\ref{fig::duality2} and Figure~\ref{fig::duality3}).

\begin{figure}[h!]
\begin{center}
\includegraphics[scale=0.85]{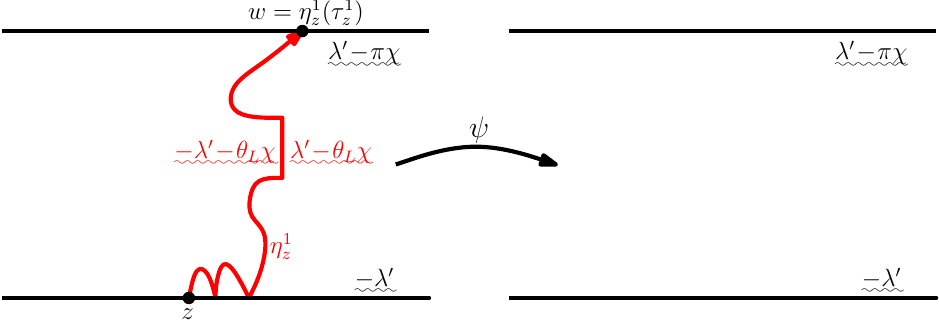}
\end{center}
\caption{\label{fig::duality2} Suppose that $h$ is a GFF on $\strip$ whose boundary data is depicted above and fix $z$ in the lower boundary $\stripbot$ of $\strip$.  Then the counterflow line $\eta' \sim \SLE_{\kappa'}(\tfrac{\kappa'}{2}-4;\tfrac{\kappa'}{2}-4)$ of $h$ from $\infty$ to $-\infty$ almost surely hits $z$, say at time $\tau_z'$.  The left boundary of $\eta'([0,\tau_z'])$ is almost surely equal to the flow line $\eta_z^1$ of $h$ starting at $z$ with angle $\theta_L = \tfrac{\pi}{2}$ stopped at time $\tau_z^1$, the first time it hits the upper boundary $\striptop$ of $\strip$.  The connected components of $\strip \setminus \eta_z^1([0,\tau_z^1])$ which lie to the right of $\eta_z^1([0,\tau_z^1])$ are visited by $\eta'$ in the reverse order that their boundaries are traced by $\eta_z^1$ (recall Lemma~\ref{lem::light_cone_contains_av} and Remark~\ref{rem::light_cone_contains_av_general}).  The right and left most points where the boundary of such a component intersects $\stripbot$ are the entrance and exit points of $\eta'$.  The conditional law of $h$ given $\eta_z^1([0,\tau_z^1])$ in each such component is (independently) the same as $h$ itself, up to a conformal change of coordinates which preserves the entrance and exit points of $\eta'$ and the conditional law of $\eta'$ is (independently) an $\SLE_{\kappa'}(\tfrac{\kappa'}{2}-4;\tfrac{\kappa'}{2}-4)$ process.
}
\end{figure}

\begin{figure}[h!]
\begin{center}
\includegraphics[scale=0.85]{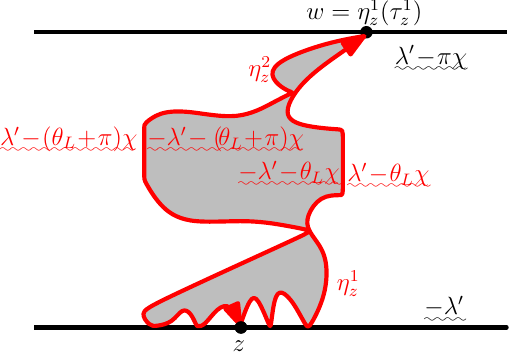}
\end{center}
\caption{\label{fig::duality3}(Continuation of Figure~\ref{fig::duality2})  Moreover, $\eta'([\tau_z',\infty))$ almost surely stays to the left of $\eta_z^1([0,\tau_z^1])$ and is the counterflow line of $h$ given $\eta_z^1([0,\tau_z^1])$ starting at $z$ and running to $-\infty$.  Let $w = \eta_z^1(\tau_z^1)$.  Since $\eta'$ is boundary filling and cannot enter into the loops it creates with itself and the boundary, the first point on $\striptop$ that $\eta'$ hits after $\tau_z'$ is $w$.  The left boundary of $\eta'|_{[\tau_z',\infty)}$ is given by the flow line $\eta_z^2$ of $h$ given $\eta_z^1([0,\tau_z^1])$ in the left connected component of $\strip \setminus \eta_z^1([0,\tau_z^1])$ stopped at the time $\tau_z^2$ that it first hits $z$ (Proposition~\ref{prop::strong_duality_boundary} and Theorem~\ref{thm::lightconeroughstatement}).  The order in which $\eta'$ hits those connected components which lie to the left of $\eta_z^2([0,\tau_z^2])$ is determined by the reverse chronological order that $\eta_z^2$ traces their boundary (Lemma~\ref{lem::light_cone_contains_av} and Remark~\ref{rem::light_cone_contains_av_general}) and the conditional law of $\eta'$ in each is independently an $\SLE_{\kappa'}(\tfrac{\kappa'}{2}-4;\tfrac{\kappa'}{2}-4)$ process.}
\end{figure}

\begin{proposition}
\label{prop::strong_duality_boundary}
Suppose that $h$ is a GFF on a Jordan domain $D$ and $x,y \in \partial D$ are distinct.  Let $\eta'$ be the counterflow line of $h$ from $y$ to $x$.  Suppose that $z \in \partial D$ is such that the first time $\tau_z'$ that $\eta'$ hits $z$ is finite almost surely.  If $z$ is on the right side of $\partial D$, then the outer boundary of $\eta'([0,\tau_z'])$ is given by the flow line $\eta_z^1$ of $h$ starting at $z$ with angle $\theta_L = \tfrac{\pi}{2}$.  Let $C$ be a connected component of $D \setminus \eta_z^1$ which lies to the right of $\eta_z^1$.  Then $\eta'|_{[0,\tau_z']}$ given $\eta_z^1$ in $C$ is equal to the counterflow line of the GFF given by conditioning $h$ on $\eta_z^1$ and restricting to $C$ starting from the point where $\eta'$ first enters $C$.  Let $C_x$ be the connected component of $D \setminus \eta_z^1$ which contains $x$.  Then $\eta'|_{[\tau_z',\infty)}$ given $\eta_z^1$ is equal to the counterflow line of the GFF starting at $z$ given by conditioning $h$ on $\eta_z^1$ and restricting to $C_x$.  Analogous results hold when the roles of left and right are swapped.
\end{proposition}
\begin{proof}
The statement regarding the law of the outer boundary of $\eta'$ upon hitting $z$ follows from Theorem~\ref{thm::lightconeroughstatement} by viewing $\eta'$ as a counterflow line from $y$ to $z$.  Thus, to complete the proof of the proposition, we just need to deduce the conditional law of $\eta'$ given $\eta_z^1$.  This follows the same strategy we used to compute the conditional law of a counterflow line given a flow line used in Section~\ref{subsec::counterflow_two_boundary_force_points} and Section~\ref{subsec::counterflow_many_boundary_force_points}.  In particular, we know that $\eta'$ has a continuous Loewner driving function viewed as a path in each of the complementary connected components of $\eta_z^1$ using the same argument described in Remark~\ref{rem::cont_loewner_cf} and Remark~\ref{rem::cont_loewner_cf_contains}.  Moreover, the conditional mean of $h$ given $\eta'([0,\tau])$ and $\eta_z^1$, $\tau$ any stopping time for the filtration generated by $\eta'(s)$ for $s \leq t$ and $\eta_z^1$, does not exhibit pathological behavior at intersection points of $\eta_z^1$ and $\eta'$ using the same technique as described in Remark~\ref{rem::cond_mean_height_cf}.  The desired result then follows by invoking Theorem~\ref{thm::martingale} and Proposition~\ref{prop::cond_mean_continuous}.
\end{proof}

We chose not to write down the precise law of $\eta'$ given $\eta_z^1$ in the statement of Proposition~\ref{prop::strong_duality_boundary}, though in general this is very easy to do.  One special case of this result that will be especially important for us in a subsequent work is illustrated in Figure~\ref{fig::duality2} and Figure~\ref{fig::duality3} and stated precisely in the following proposition:

\begin{proposition}
\label{lem::counterflow_given_boundary_on_hit}
Suppose that $D$ is a Jordan domain and $x,y \in \partial D$ are distinct.  Let $\eta' \sim \SLE_{\kappa'}(\tfrac{\kappa'}{2}-4;\tfrac{\kappa'}{2}-4)$ from $y$ to $x$ in $D$.  Suppose that $z \in \partial D$ is in the right boundary of $D$.  Then the conditional law of $\eta'$ given its left boundary $\eta_z^1$ upon hitting $z$ is an $\SLE_{\kappa'}(\tfrac{\kappa'}{2}-4;\tfrac{\kappa'}{2}-4)$ process independently in each of the connected components of $D \setminus \eta_z^1$ which lie to the right of $\eta_z^1$.  Let $C_x$ be the connected component of $D \setminus \eta_z^1$ which contains $x$.  Then $\eta'$ restricted to $C_x$ is equal to the counterflow line of the conditional GFF $h|_{C_x}$ given $\eta_z^1$.  Let $\eta_z^2$ be the flow line of $h|_{C_x}$ starting at the first point $w$ where $\eta_z^1$ hits the left side of $\partial D$ with angle $\tfrac{\pi}{2}$.  Then $\eta_z^2$ is the left boundary of $\eta'$ restricted to $C_x$.  Moreover, the conditional law of $\eta'$ in $C_x$ given $\eta_z^2$ is independently that of an $\SLE_{\kappa'}(\tfrac{\kappa'}{2}-4;\tfrac{\kappa'}{2}-4)$ process in each of the connected components of $C_x \setminus \eta_z^2$ which lie to the left of $\eta_z^2$.  Analogous results likewise hold when the roles of left and right are swapped and the angle $\tfrac{\pi}{2}$ is replaced with $-\tfrac{\pi}{2}$.
\end{proposition}
\begin{proof}
This is a special case of Proposition~\ref{prop::strong_duality_boundary}.  See Figure~\ref{fig::duality2} and Figure~\ref{fig::duality3} for further explanation as to why these are the correct weights for the conditional law of~$\eta'$.
\end{proof}

\subsection{The fan is not space filling}
\label{subsec::fan}

Suppose that $h$ is a GFF on the infinite strip $\strip$ with boundary data as in Figure~\ref{fig::fan_counterflow}.  We will first assume that $a,b \geq \lambda - \tfrac{\pi}{2} \chi = \lambda'$ and that $a',b' \geq \lambda' + \pi \chi$ so that the counterflow line $\eta'$ of $h$ starting from $z_0$ almost surely hits $\stripbot$ only when it exits at $0$ and does not hit $\striptop$ except where it starts at $z_0$.  Recall from Section~\ref{subsec::light_cone} that the fan $\fan$ is the closure of the union of the ranges of any collection of flow lines $\eta_\theta$ of $h$ where $\theta$ ranges over a countable, dense subset of $[-\tfrac{\pi}{2},\tfrac{\pi}{2}]$ (recall also the simulations from Figures~\ref{fig::flowlines}--\ref{fig::flowlines4}).  By Lemma~\ref{lem::light_cone_contains_av}, we know that the range of $\eta'$ almost surely contains $\fan$.  The main purpose of this subsection is to establish the following proposition, which implies that $\fan$ almost surely has zero Lebesgue measure for all $\kappa \in (0,4)$ (recall Figure~\ref{fig::sle64_fan}):

\begin{figure}[h!]
\begin{center}
\includegraphics[scale=0.85]{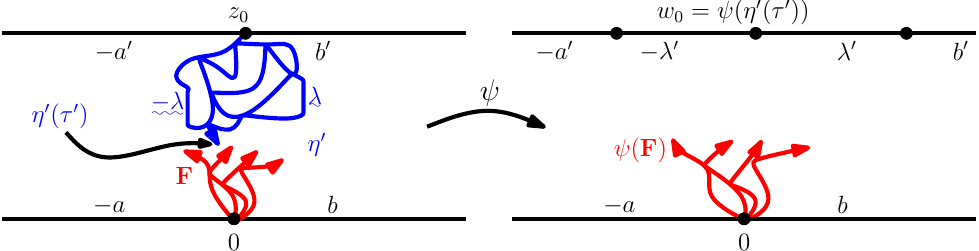}
\end{center}
\caption{\label{fig::fan_counterflow} The setup for Proposition~\ref{prop::fan_does_not_hit}.  Suppose that $h$ is a GFF on the strip $\strip$ with the boundary data depicted in the left hand side above.  We assume that $a,b \geq \lambda-\tfrac{\pi}{2} \chi = \lambda'$ and that $a',b' \geq \lambda'+\pi \chi$ so that the counterflow line $\eta'$ of $h$ starting at $z_0$ intersects $\partial \strip$ only at $z_0$ and $0$, its starting and terminal points, respsectively.  Let $\tau'$ be any stopping time for the counterflow line $\eta'$ of $h$ starting at $z_0$ such that $\eta'(\tau') \neq 0$ almost surely.  We will prove that $\eta'(\tau')$ almost surely is not contained $\fan$.  To prove this, we let $\psi$ be the conformal map which takes the unbounded connected component of $\strip \setminus \eta'([0,\tau'])$ back to $\strip$ which fixes $\pm \infty$ and $0$.  Let $w_0 = \psi(\eta'(\tau')) \in \striptop$.  The boundary data for the GFF $\wt{h} := h \circ \psi^{-1} - \chi \arg (\psi^{-1})'$ is depicted on the right side.  We show that the fan of $\wt{h}$ almost surely does not contain $w_0$.}
\end{figure}

\begin{proposition}
\label{prop::fan_does_not_hit}
Suppose that we have a GFF $h$ on $\strip$ whose boundary data is as in Figure~\ref{fig::fan_counterflow} with $a,b \geq \lambda - \tfrac{\pi}{2} \chi = \lambda'$ and $a',b' \geq \lambda' + \pi \chi$.  Let $\tau'$ be any $\eta'$ stopping time such that $\eta'(\tau') \neq 0$ almost surely.  Then we have that $\p[\eta'(\tau') \in \fan] = 0$.  In particular, the Lebesgue measure of $\fan$ is almost surely zero.
\end{proposition}

\begin{figure}[ht!]
\begin{center}
\subfigure[The shielding fan.]{
\includegraphics[scale=0.85]{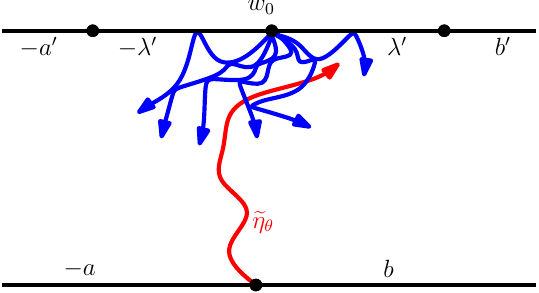}}
\subfigure[Since flow lines can only cross each other at most one time, it follows that $\wt{\eta}_\theta$ can only intersect at most one pocket between each pair $\wt{\eta}_i^{w_0}$, $\wt{\eta}_{i+1}^{w_0}$.]{
\includegraphics[scale=0.85]{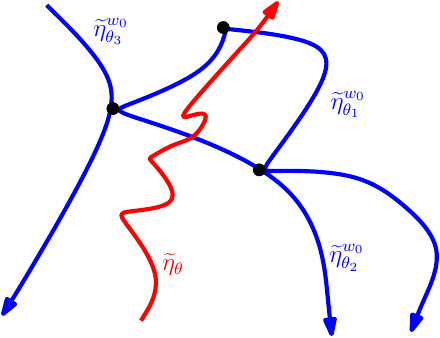}}
\end{center}
\caption{\label{fig::fan_bd} Suppose we have the same setup as the right panel of Figure~\ref{fig::fan_counterflow}.  We take $n = n(\kappa) \in \N$ flow lines $\wt{\eta}_{\theta_1}^{w_0},\ldots,\wt{\eta}_{\theta_n}^{w_0}$ with angles $\theta_1,\ldots,\theta_n$ so that $\wt{\eta}_{\theta_i}^{w_0}$ almost surely intersects both of its neighbors (or $\striptop$ if $i=1$ or $i=n$).  Fix $r > 0$ such that $B(w_0,r) \cap \partial \strip$ is almost surely contained in the part of $\striptop$ where the boundary data of $\wt{h}$ is either $\lambda'$ or $-\lambda'$.  For each $i$, we let $\wt{\tau}_i$ be the first time that $\wt{\eta}_{\theta_i}^{w_0}$ first exits $B(w_0,r)$.  Lemma~\ref{lem::hit_boundary_infinitely} implies that each of the $\wt{\eta}_{\theta_i}^{w_0}$ intersects its neighbors almost surely infinitely many times in every neighborhood of $w_0$.  Take any flow line $\wt{\eta}_\theta$ of $\wt{h}$ starting at $0$ with initial angle $\theta$.  By Proposition~\ref{prop::generalized_monotonicity} and Proposition~\ref{prop::merging_and_crossing}, $\wt{\eta}_\theta$ can only intersect at most one pocket between each pair $\wt{\eta}_{\theta_i}^{w_0},\wt{\eta}_{\theta_{i+1}}^{w_0}$.  Thus the set of points that $\wt{\eta}_\theta$ can access is contained in the set of pockets between pairs $\wt{\eta}_{\theta_i}^{w_0},\wt{\eta}_{\theta_{i+1}}^{w_0}$ which are connected to the unbounded connected component of $\strip \setminus \cup_{i=1}^n \wt{\eta}_{\theta_i}^{w_0}$ by a chain of at most $n$ such pockets.  Therefore $\fan$ almost surely does not contain $w_0$.}
\end{figure}

Before we proceed to the proof of Proposition~\ref{prop::fan_does_not_hit}, we need to record the following simple fact about $\SLE_\kappa(\rho^L;\rho^R)$ processes.  In what follows, $|\cdot|$ is used to denote counting measure. 
\begin{lemma}
\label{lem::hit_boundary_infinitely}
Suppose that $\eta$ is an $\SLE_\kappa(\rho^L;\rho^R)$ process in $\h$ with $\rho^L,\rho^R \in (-2,\tfrac{\kappa}{2}-2)$ and with the force points located at $0^-,0^+$, respectively.  For every $t > 0$, we have that both $|\eta([0,t]) \cap \R_-| = \infty$ and $|\eta([0,t]) \cap \R_+| = \infty$ almost surely.
\end{lemma}
\begin{proof}
It is obvious that $|\eta([0,t]) \cap \R_-| = \infty$ for all $t > 0$ almost surely when $\rho^R = 0$ because in this case $W_t - V_t^L$ evolves as a positive multiple of a boundary intersecting Bessel process (Section~\ref{sec::sle}).  This remains true for $\rho^R \in (-2,\tfrac{\kappa}{2}-2)$ because we can couple $\eta$ with a GFF $h$ so that $\eta$ is the flow line of $h$.  As in Section~\ref{subsec::two_boundary_force_points}, we can condition on a flow line $\eta_\theta$ of $h$ with $\theta$ chosen so that $\rho^R=-\theta \chi/\lambda -2$. Then the law of $\eta$ conditional on $\eta_\theta$ is an $\SLE_\kappa(\rho^L;\rho^R)$ process, which proves our claim.  Reversing the roles of $\rho^L$ and $\rho^R$ gives that $|\eta([0,t]) \cap \R_+| = \infty$ for all $t > 0$ almost surely, as well.
\end{proof}

We can now proceed to the proof of Proposition~\ref{prop::fan_does_not_hit}.  The idea is to construct a ``shield'' consisting of a finite number of flow lines at $w_0$, the image of $\eta'(\tau')$ under the conformal map of the unbounded connected component of $\strip \setminus \eta'([0,\tau'])$ back to $\strip$ which fixes $\pm \infty$ and $0$.  This is described in Figure~\ref{fig::fan_bd}.

\begin{proof}[Proof of Proposition~\ref{prop::fan_does_not_hit}]
Fix any stopping time $\tau'$ for $\eta'$ such that $\eta'(\tau') \neq 0$ almost surely.  Let $\psi$ be the conformal map which takes the unbounded connected component of $\strip \setminus \eta'([0,\tau'])$ back to $\strip$ and fixes $\pm \infty$ and $0$.  Let $w_0 = \psi(\eta'(\tau'))$ and note that $w_0 \in \striptop$.  Then the boundary data for the GFF $\wt{h} = h \circ \psi^{-1} - \chi \arg (\psi^{-1})'$ is depicted in the right panel of Figure~\ref{fig::fan_counterflow}.  Fix $r > 0$ such that $B(w_0,r) \cap \partial \strip$ is contained in the image of the outer boundary of $\eta'([0,\tau'])$ under $\psi$ in $\striptop$.  By Proposition~\ref{prop::gff_abs_continuity}, $\wt{h}|_{B(w_0,r)}$ is mutually absolutely continuous with respect to a GFF on $\strip$ whose boundary data is constant $-\lambda'$ on the part of $\striptop$ which lies to the left of $w_0$ and constant $\lambda'$ on the part of $\striptop$ which lies to the right of $w_0$.  Therefore there exists $n = n(\kappa) \in \N$ and angles $\theta_1,\ldots,\theta_n$ such that with $\wt{\eta}_{\theta_i}^{w_0}$ the flow line of $\wt{h}$ starting from $w_0$ with angle $\theta_i$, we have that $\wt{\eta}_{\theta_i}^{w_0}$ almost surely intersects both $\wt{\eta}_{\theta_{i-1}}^{w_0}$ and $\wt{\eta}_{\theta_{i+1}}^{w_0}$ at infinitely many points for each $i$.  For each $i$, let $\wt{\tau}_i$ be the first time $t$ that $\wt{\eta}_{\theta_i}^{w_0}$ first exits $B(w_0,r)$.  Note that $\cup_{i=1}^n \wt{\eta}_{\theta_i}^{w_0}([0,\wt{\tau}_i])$ is a local set for $\wt{h}$ by Proposition~\ref{gff::prop::cond_union_local} since each $\wt{\eta}_{\theta_i}^{w_0}([0,\wt{\tau}_i])$ is local and almost surely determined by $\wt{h}$ (Theorem~\ref{thm::coupling_uniqueness}).

Let $U_1$ be the union of the set of connected components of $U_0 = (B(w_0,r) \cap \strip) \setminus \cup_{i=1}^n \wt{\eta}_{\theta_i}^{w_0}([0,\tau_i])$ whose boundary intersects $\partial B(w_0,r)$.  Inductively let $U_k$ for $k \geq 2$ be the union of those connected components of $U_0$ whose boundary intersects the boundary of $U_{k-1}$.  Finally, let $U = \cup_{i=1}^n U_i$.  Lemma~\ref{lem::hit_boundary_infinitely} implies that $\ol{U}$ does not contain $w_0$.

We next claim that, almost surely, $\wt{\eta}_\theta = \psi(\eta_\theta)$ for each $\theta \in [-\tfrac{\pi}{2},\tfrac{\pi}{2}]$ cannot traverse $\ol{U}$ and, therefore, cannot hit $w_0$.  Once we have established this, the proof of the proposition will be complete.  Fix $\theta \in [-\tfrac{\pi}{2},\tfrac{\pi}{2}]$.  Note that $\wt{\eta}_\theta$ can hit only one side of each $\wt{\eta}_{\theta_i}([0,\wt{\tau}_i])$ and, upon hitting $\wt{\eta}_{\theta_i}^{w_0}([0,\wt{\tau}_i])$ will cross but cannot cross back (see Figure~\ref{fig::fan_bd} and Proposition~\ref{prop::merging_and_crossing}).  Since $\wt{\eta}_\theta$ must cross one of the $\wt{\eta}_{\theta_i}^{w_0}([0,\wt{\tau}_i])$ when it passes from $U_k$ to $U_{k+1}$, it follows that $\wt{\eta}_\theta$ cannot enter $U_{n+1}$ and therefore cannot traverse $\ol{U}$.
\end{proof}

\bibliographystyle{hmralphaabbrv}
\addcontentsline{toc}{section}{References}
\bibliography{sle_kappa_rho}

\bigskip

\filbreak
\begingroup
\small
\parindent=0pt

\bigskip
\vtop{
\hsize=5.3in
Microsoft Research\\
One Microsoft Way\\
Redmond, WA, USA }

\bigskip
\vtop{
\hsize=5.3in
Department of Mathematics\\
Massachusetts Institute of Technology\\
Cambridge, MA, USA } \endgroup \filbreak \end{document}